\documentclass[12pt]{article}

\usepackage{mathrsfs,amsthm,graphicx,color,verbatim,bbm,amsmath,amsfonts,amssymb,newclude,nicefrac,amsfonts,graphicx,geometry,enumerate,hyperref}
\usepackage[latin1]{inputenc}
\theoremstyle{plain}
\newtheorem{theorem}{Theorem}[section]
\newtheorem{lemma}[theorem]{Lemma}
\newtheorem{corollary}[theorem]{Corollary}
\newtheorem{proposition}[theorem]{Proposition}

\theoremstyle{definition}

\begin{document}

\newcommand{\E}{\mathbb{E}\!}
\newcommand{\ES}{\mathbb{E}}
\renewcommand{\P}{\mathbb{P}}
\newcommand{\R}{\mathbb{R}}
\newcommand{\N}{\mathbb{N}}
\newcommand{\smallsum}{\textstyle\sum}
\newcommand{\tr}{\operatorname{trace}}
\newcommand{\citationand}{\&}
\newcommand{\dt}[1][t]{\, \mathrm{d} #1}
\newcommand{\grid}{\,\mathcal{P}}
\newcommand{\euler}{Z^{\,N}}
\newcommand{\leuler}{\tilde{Z}^{\,N}}
\newcommand{\exteuler}{\bar{Z}^{\,N}}

\newcommand{\floor}[2]{\lfloor #1\rfloor_{#2}}
\newcommand{\ceil}[2]{\lceil #1\rceil_{#2}}
\newcommand{\mesh}[2]{\|#1\|^{#2}_{\grid_T}}
\newcommand{\bigbrack}[2]{\big( #1\big)^{#2}}
\newcommand{\bignorm}[3]{\bnl #1\bnr^{#2}_{#3}}
\newcommand{\bigsharp}[2]{\big[ #1\big]^{#2}}
\newcommand{\lpn}[3]{\mathcal{L}^{#1}(#2;#3)}
\newcommand{\lpnb}[3]{L^{#1}(#2;#3)}
\newcommand{\eulerm}[1]{Z^{\,\Theta_{#1}}}
\newcommand{\eulerpart}[1]{Z^{#1}}

\newcommand{\expeuler}[1]{Z^{\,\text{exp},#1}}
\newcommand{\impeuler}[1]{Z^{\,\text{imp},#1}}

\newcommand{\embed}[2]{\kappa^I_{#1,\,#2}}

\newcommand{\groupC}{ C }
\newcommand{\diagC}{ C }
\newcommand{\psiC}{ \Xi }
\newcommand{\driftC}{{\bf y}}
\newcommand{\diffusionC}{{\bf z}}
\newcommand{\power}{ q }

\newcommand{\set}{{\mathbb{V}}}

\newcommand{\resolvent}[2]{R_{#2}( #1 )}
\providecommand{\Cb}[1]{{C_b^{#1}}}
\newcommand{\BDG}[2]{\Upsilon_{#1}}
\newcommand{\SGchi}[1]{\chi_{#1}}
\newcommand{\SGkappa}[1]{\chi_{#1}}

\title{Weak convergence rates for numerical approximations of stochastic partial differential equations with nonlinear diffusion coefficients in UMD Banach spaces}

\author{Mario Hefter, Arnulf Jentzen, and Ryan Kurniawan}

\maketitle

\begin{abstract}
Strong convergence rates for numerical approximations of semilinear stochastic partial differential equations (SPDEs) with smooth and regular nonlinearities are well understood in the literature. Weak convergence rates for numerical approximations of such SPDEs have been investigated for about two decades and are still not yet fully understood. In particular, no essentially sharp weak convergence rates are known for temporal or spatial numerical approximations of space-time white noise driven SPDEs with nonlinear multiplication operators in the diffusion coefficients. In this article we overcome this problem by establishing essentially sharp weak convergence rates for exponential Euler approximations of semilinear SPDEs with nonlinear multiplication operators in the diffusion coefficients. Key ingredients of our approach are applications of the mild It\^{o} type formula in UMD Banach spaces with type 2. 
\end{abstract}

\section{Introduction}
\label{sec:intro}

This article investigates weak convergence rates for time-discrete numerical approximations of semilinear stochastic partial differential equations (SPDEs). 
In the case of finite dimensional stochastic ordinary differential equations (SODEs) with smooth 
and regular nonlinearities both strong and numerically weak convergence rates of numerical approximations are well understood in the literature; see, e.g., the monographs 
Kloeden \& Platen~\cite{kp92} and Milstein~\cite{m95}.
The situation is different in the case 
of SPDEs.
While strong convergence rates for numerical approximations 
of semilinear SPDEs with smooth and regular nonlinearities are well understood in 
the literature, 
weak convergence rates for numerical approximations of 
such SPDEs have been investigated for about two decades and are still not yet fully understood.
More specifically, to the best of our knowledge, there exist no result in the scientific literature which establishes essentially sharp weak convergence rates for temporal or spatial numerical approximations in the case of space-time white noise driven SPDEs with nonlinear multiplication operators in the diffusion coefficients.
In this paper we overcome this problem in the case of time-discrete exponential Euler approximations for SPDEs (cf., e.g., Lord \& Rougemont~\cite[Section~3]{lr04}, Cohen \& Gauckler~\cite[Section~2.2]{CohenGauckler2012}, and Wang~\cite[Section~1]{Wang2014b}), which is illustrated in the following theorem.  
\begin{theorem}
	\label{thm:intro.weak.rates.stoch.heat}
	Let $T\in(0,\infty)$, $p\in[2,\infty)$, 
	let $f,b\colon\R\to\R$ 
	and 
	$\varphi\colon\lpnb{p}{(0,1)}{\R}\to\R$ be four times continuously differentiable functions with globally Lipschitz continuous and globally bounded derivatives, 
	let $\xi\colon(0,1)\to\R$ be a $\mathcal{B}((0,1))$/$\mathcal{B}(\R)$-measurable and globally bounded function, 
	 let $ ( \Omega, \mathcal{F}, \P ) $ be a probability space with a normal filtration $( \mathcal{F}_t )_{ t \in [0,T] }$, 
	 let $ ( W_t )_{ t \in [0,T] } $ be an $\operatorname{Id}_{\lpnb{2}{(0,1)}{\R}}$-cylindrical $ ( \Omega, \mathcal{F}, \P, ( \mathcal{F}_t )_{ t \in [0,T] } ) $-Wiener process, 
	 let $X\colon[0,T]\times\Omega\to\lpnb{p}{(0,1)}{\R}$ be a continuous $( \mathcal{F}_t )_{ t \in [0,T] }$-adapted mild solution process of the SPDE
	 \begin{equation}
	 \label{eq:intro.stoch.heat.eq}
	   dX_t(x)=
	   \big[ \tfrac{\partial^2}{\partial x^2} X_t(x) + f(X_t(x)) \big] \, dt
	   +
	   b(X_t(x)) \, dW_t(x)
	 \end{equation}
	 with $X_t(0)=X_t(1)=0$ and $X_0(x)=\xi(x)$ for $t\in[0,T]$, $x\in(0,1)$, and for every $N\in\N$ let
	 $Y^N\colon\{0,1,\ldots,N\}\times\Omega\to\lpnb{p}{(0,1)}{\R}$ be a time-discrete exponential Euler approximation for the SPDE~\eqref{eq:intro.stoch.heat.eq} with time step size $\nicefrac{T}{N}$ (see, e.g., item~\eqref{item:exist.exp.Euler} of Theorem~\ref{thm:stochastic.heat.eq} in Section~\ref{sec:temporal.rate.stoch.heat} below).
	 Then for every $\varepsilon\in(0,\infty)$ there exists a real number $C\in\R$ such that for all $N\in\N$ it holds that 
	 \begin{equation}
	     \left|
	     \ES\big[ 
	     \varphi( X_T )
	     \big]
	     -
	     \ES\big[ 
	     \varphi( Y^N_N )
	     \big]
	     \right|
	     \leq
	     C \cdot
	     N^{
	     	( \varepsilon-\nicefrac{1}{2} )
	     }
	     .
	 \end{equation}
\end{theorem}

Theorem~\ref{thm:intro.weak.rates.stoch.heat} is an immediate consequence of Theorem~\ref{thm:stochastic.heat.eq} below. 
Theorem~\ref{thm:intro.weak.rates.stoch.heat} establishes for every arbitrarily small $\varepsilon\in(0,\infty)$ the weak convergence rate $\nicefrac{1}{2}-\varepsilon$ for the exponential Euler approximations 
$Y^N$, $N\in\N$, (see, e.g., Lord \& Rougemont~\cite[Section~3]{lr04}, Celledoni et al.~\cite[Section~2]{CelledoniCohenOwren2008}, Cohen \& Gauckler~\cite[Section~2.2]{CohenGauckler2012}, Lord \& Tambue~\cite[Section~2.3]{LordTambue2013},  Wang~\cite[Section~1]{Wang2014b}, and item~\eqref{item:exist.exp.Euler} of Theorem~\ref{thm:stochastic.heat.eq} below) in the case of the SPDE~\eqref{eq:intro.stoch.heat.eq}. 
We would like to point out that the rate $\nicefrac{1}{2}-\varepsilon$ can, in general, not be essentially improved.
More specifically, Corollary~9.8 in~\cite{JentzenKurniawan2015arXiv} and Theorem~\ref{thm:intro.weak.rates.stoch.heat} above prove in the case 
$ \forall \, x \in \R \colon f(x)=0 $, 
$ \forall \, x \in \R \colon b(x)=1 $, 
and 
$ \forall \, y \in (0,1) \colon \xi(y)=0 $ 
that there exist a four times continuously differentiable function
$\varphi\colon\lpnb{p}{(0,1)}{\R}\to\R$ 
such that for all $\varepsilon\in(0,\infty)$ there exist real numbers 
$c,C \in (0,\infty)$ 
such that for all $N\in\N$ it holds that 
\begin{equation}
  c \cdot N^{-\nicefrac{1}{2}}
  \leq
	     \left|
	     \ES\big[ 
	     \varphi( X_T )
	     \big]
	     -
	     \ES\big[ 
	     \varphi( Y^N_N )
	     \big]
	     \right|
  \leq
  C \cdot N^{(\varepsilon-\nicefrac{1}{2})}  
\end{equation}
(cf., e.g., also the references mentioned in the overview article in M{\"u}ller-Gronbach \& Ritter~\cite{mr08} and in Section~9 in~\cite{JentzenKurniawan2015arXiv} for further lower bound results for numerical approximations of the SPDE~\eqref{eq:intro.stoch.heat.eq}).
The literature also contains a series of other results which establish essentially sharp weak convergence rates for temporal or spatial numerical approximations of the SPDE~\eqref{eq:intro.stoch.heat.eq} in the case where the diffusion coefficient function 
$b\colon\R\to\R$ 
is affine linear, that is, in the case where it holds that 
\begin{equation}
\label{eq:affine.linear}
  \exists \, \alpha, \beta \in \R \colon
  \forall \, x \in \R \colon
  b(x)=\alpha x+\beta
\end{equation}
(cf., e.g., 
\cite{s03, 
	h03b, 
	dd06, 
	dp09, 
	GeissertKovacsLarsson2009, 
	DorsekTeichmann2010, 
	h10c, 
	Debussche2011, 
	kll11, 
	Doersek2012, 
	Brehier2012b, 
	DorsekTeichmannVeluscek2013, 
	WangGan2013WeakHeatAdditiveNoise,
	LindnerSchilling2013, 
	KovacsLarssonLindgren2013BIT,
	AnderssonLarsson2013,
	BrehierKopec2013, 
	AnderssonKruseLarsson2013, 
	Kruse2014_PhD_Thesis, 
	Brehier2014,
	KovacsPrintems2014, 
	Kopec2014_PhD_Thesis, 
	Wang2014Weak, 
	ConusJentzenKurniawan2014arXiv, 
	Wang2014b, 
	AnderssonKovacsLarsson2014,
	JentzenKurniawan2015arXiv}).
To the best of our knowledge, Theorem~\ref{thm:intro.weak.rates.stoch.heat} is the first result in the scientific literature which establishes an essentially sharp weak convergence rate for the SPDE~\eqref{eq:intro.stoch.heat.eq} in the case where~\eqref{eq:affine.linear} is not fulfilled.
Note that Theorem~\ref{thm:intro.weak.rates.stoch.heat} above and Theorem~\ref{thm:stochastic.heat.eq} below prove weak convergence rates only for exponential Euler approximations (see, e.g., item~\eqref{item:cor.convergence.rates} of Theorem~\ref{thm:stochastic.heat.eq} below).
However, their methods of proof extend to other kinds of numerical approximations for SPDEs such as linear-implicit Euler approximations (see, e.g., Da Prato et al.~\cite[Section~3.3.1]{DaPratoJentzenRoeckner2012}).
Our proof of Theorem~\ref{thm:intro.weak.rates.stoch.heat} above and Theorem~\ref{thm:stochastic.heat.eq} below, respectively, is based on the proof in~\cite{JentzenKurniawan2015arXiv} by extending the proof of Theorem~1.1 in~\cite{JentzenKurniawan2015arXiv} from Hilbert spaces to UMD Banach spaces (cf., e.g., Brze{\'z}niak~\cite[(2.5)]{b97b} and Van Neerven, Veraar, \& Weis~\cite[(2.3)]{vvw08}) with type 2 (cf.\ Sections~2--4 and~6--8 in~\cite{JentzenKurniawan2015arXiv} with Sections~2--7 below).

\subsection{Notation}
\label{sec:notation}

Throughout this article the following notation is frequently used. 
For every set $ A $ 
we denote by 
$ \mathcal{P}(A) $ the power set of $ A $.
For every set $ A $ we denote by 
$ \#_A \in \{\infty,0,1,2,\ldots\} $ 
the number of elements of $ A $.
For all sets $ A $ and $ B $ we denote by 
$ \mathbb{M}(A,B) $ 
the set of all functions from $A$ to $B$.
For all
measurable spaces
$
( A, \mathcal{A} )
$
and
$
( B, \mathcal{B} )
$
we denote by
$
\mathcal{M}( \mathcal{A}, \mathcal{B} )
$
the set of
$ \mathcal{A} $/$ \mathcal{B} $-measurable
functions.
For every Borel measurable set $ A \in \mathcal{B}( \R ) $
we denote by $ \lambda_A \colon \mathcal{B}( A ) \to [0,\infty] $
the Lebesgue-Borel measure on $ A $.
We denote by 
$ \lfloor \cdot \rfloor_h \colon \R \to \R $,
$ h \in (0,\infty) $,
the functions which satisfy for all
$ h \in (0,\infty) $, $ t \in \R $
that 
$
  \lfloor t \rfloor_h =
  \max(
    (-\infty,t]
    \cap
    \{ 0, h , - h , 
$
$
    2 h , - 2 h , \dots \}
  )
$. 
We denote by 
$
  \mathcal{E}_r
  \colon
  [0,\infty)
  \rightarrow
  [0,\infty)
$,
$
  r \in (0,\infty)
$,
the functions which satisfy 
for all $ x \in [0,\infty) $, $ r \in (0,\infty) $
that
$
  \mathcal{E}_r(x)
  =
  \big[
    \sum^{ \infty }_{ n = 0 }
    \frac{ 
        x^{ 2 n }
        \,
        \Gamma(r)^n
    }{
      \Gamma( n r + 1 ) 
    }
  \big]^{
    \nicefrac{ 1 }{ 2 }
  }
$
(cf.\ \cite[Chapter~7]{h81} and \cite[Section~1.2]{ConusJentzenKurniawan2014arXiv}). 
For all $ \R $-Banach spaces
$ ( V , \left\| \cdot \right\|_V ) $
and
$ ( W , \left\| \cdot \right\|_W ) $
with 
$ \#_V > 1 $
and every natural number $ n \in \N=\{1,2,\ldots\} $
we denote by
$
\left| \cdot \right|_{ \Cb{n}( V, W ) }
\colon
C^n( V, W ) \to [0,\infty]
$
and
$
\left\| \cdot \right\|_{ \Cb{n}( V, W ) }
\colon
C^n( V, W ) \to [0,\infty]
$
the functions which satisfy 
for all $ f \in C^n( V, W ) $
that
\begin{equation}
\label{eq:Cb.def}
\begin{split}
\left| f \right|_{
	\Cb{n}( V, W )
}
& =
\sup_{
	x \in V
}
\left\|
f^{ (n) }( x )
\right\|_{
	L^{ (n) }( V, W )
}
,
\qquad
\left\| f \right\|_{
	\Cb{n}( V, W )
}
=
\|f(0)\|_W
+
\sum_{ k = 1 }^n
\left| f \right|_{ \Cb{k}(V,W) }
\end{split}
\end{equation}
and we denote by
$
\Cb{n}( V, W )
$
the set given by
$
\Cb{n}( V, W ) =
\{ f \in C^n( V, W ) \colon \left\| f \right\|_{ \Cb{n}( V, W ) } < \infty \}
$.
For all $ \R $-Banach spaces
$ ( V , \left\| \cdot \right\|_V ) $
and
$ ( W , \left\| \cdot \right\|_W ) $
with
$ \#_V > 1 $
and every nonnegative integer $ n \in \N_0=\N\cup\{0\}=\{0,1,2,\ldots\} $
we denote by
$
\left| \cdot \right|_{
	\operatorname{Lip}^n( V, W )
}
\colon 
C^n( V, W )
\to [0,\infty]
$
and 
$
\left\| \cdot \right\|_{
	\operatorname{Lip}^n( V, W )
}
\colon 
C^n( V, W )
\to [0,\infty]
$
the functions which satisfy
for all $ f \in C^n( V, W ) $
that
\begin{equation}
\label{eq:Lip.def}
\begin{split}
\left| f \right|_{ 
	\operatorname{Lip}^n( V, W )
}
&=
\begin{cases}
\sup_{ 
	\substack{
		x, y \in V ,\,
		x \neq y
	}
}
\left(
\frac{
	\left\| f( x ) - f( y ) \right\|_W
}{
\left\| x - y \right\|_V
}
\right)
&
\colon
n = 0
\\
\sup_{ 
	\substack{
		x, y \in V ,\,
		x \neq y
	}
}
\left(
\frac{
	\| f^{ (n) }( x ) - f^{ (n) }( y ) \|_{ L^{ (n) }( V, W ) }
}{
\left\| x - y \right\|_V
}
\right)
&
\colon
n \in \N
\end{cases}
,
\\
\left\| f \right\|_{
	\operatorname{Lip}^n( V, W )
}
&
=
\|f(0)\|_W
+
\sum_{ k = 0 }^n
\left| f \right|_{ \operatorname{Lip}^k(V,W) }
\end{split}
\end{equation}
and we denote by
$
\operatorname{Lip}^n( V, W )
$
the set given by
$
\operatorname{Lip}^n( V, W ) =
\{ f \in C^n( V, W ) \colon \left\| f \right\|_{ \operatorname{Lip}^n( V, W ) } < \infty \}
$.
For every separable $\R$-Hilbert space 
$ 
( U, 
\left< \cdot, \cdot \right>_U, 
\left\| \cdot \right\|_U ) 
$ 
and every $\R$-Banach space 
$ 
( V, \left\| \cdot \right\|_V ) 
$ 
we denote by $ \gamma(U,V) $ the $\R$-Banach space of $\gamma$-radonifying operators from $U$ to $V$ (see, e.g., \cite[Section~2]{vvw08}).
For every measure space $ ( \Omega , \mathcal{F}, \mu ) $,
every measurable space $ ( S , \mathcal{S} ) $,
every set $ R $, 
and every function
$ f \colon \Omega \to R $
we denote by
$
   \left[ f \right]_{
     \mu, \mathcal{S}
   }
$
the set given by
\begin{equation}
   \left[ f \right]_{
     \mu, \mathcal{S}
   }
   =
   \left\{
     g \in \mathcal{M}( \mathcal{F}, \mathcal{S} )
     \colon
     (
     \exists \, A \in \mathcal{F} \colon
     \mu(A) = 0 
     \text{ and }
     \{ \omega \in \Omega \colon f(\omega) \neq g(\omega) \}
     \subseteq A
     )
   \right\}
   .
\end{equation}
For every measure space $ ( \Omega , \mathcal{F}, \mu ) $, 
every measurable space $ ( S , \mathcal{S} ) $, 
and every set $ R $ 
we do as usual often not distinguish between
a function 
$ f \colon \Omega \to R $ 
and its equivalence class $ \left[ f \right]_{ \mu, \mathcal{S} } $.

\subsection{General setting}
\label{sec:global_setting}
Throughout this article the following
setting is frequently used.
Consider the notation in Section~\ref{sec:notation}, 
let 
$ 
  ( V, \left\| \cdot \right\|_V ) 
$ 
and 
$ 
( \mathcal{V}, \left\| \cdot \right\|_{\mathcal{V}} ) 
$ 
be separable UMD $\R$-Banach spaces with type 2, 
let
$ 
  ( U, 
    \left< \cdot, \cdot \right>_U, 
$
$
  \left\| \cdot \right\|_U ) 
$ 
be a separable $ \R $-Hilbert space, 
let 
$ T \in (0,\infty) $, 
$ \eta \in \R $, 
$ 
\angle 
=
\left\{
( t_1, t_2 ) \in [0,T]^2 \colon t_1 < t_2
\right\}
$, 
let $ ( \Omega, \mathcal{F}, \P ) $
be a probability space with a normal filtration 
$( \mathcal{F}_t )_{ t \in [0,T] }$,
let
$
( W_t )_{ t \in [0,T] }
$
be an $ \operatorname{Id}_U $-cylindrical 
$ ( \Omega , \mathcal{F}, \P, ( \mathcal{F}_t )_{ t \in [0,T] } ) $-Wiener process, 
and for every $p\in[2,\infty)$
let 
$
  \BDG{p}{} \in [0,\infty)
$
be the real number given by
\begin{equation}
  \BDG{p}{}
  =
  \sup\left(\left\{
    \frac{
    	\|\int^T_0 X_t \,dW_t\|_{\lpnb{p}{\P}{V}}
    	}{
    	(\int^T_0 \|X_t\|^2_{\lpn{p}{\P}{\gamma(U,V)}} \,dt)^{\nicefrac{1}{2}}
    	}
    \colon
    \substack{
    	(\mathcal{F}_t)_{t\in[0,T]}/\mathcal{B}(\gamma(U,V))\text{-predictable}
    	\\
    	X\colon[0,T]\times\Omega\to\gamma(U,V)\text{ with}
    	\\
    	\int^T_0 \|X_t\|^2_{\lpn{p}{\P}{\gamma(U,V)}} \,dt
    	\in (0,\infty)
    	}
  \right\}\right)
\end{equation}
(cf., e.g., \cite[Corollary~3.10]{vvw07}).

\subsection{An auxiliary lemma}

Throughout this article we frequently use the following elementary lemma (see, e.g., \cite[Lemma~2.3]{bvvw08} and~\cite[Lemma~2.2]{CoxJentzenKurniawanPusnik2016}).

\begin{lemma}
\label{lem:gamma.estimate}
 	Consider the notation in Section~\ref{sec:notation}, 
 	let $ ( U, \langle \cdot, \cdot \rangle_U, \left \| \cdot \right \|_U ) $
 	be a separable $ \R $-Hilbert space, 
 	let $ ( V, \left \| \cdot \right \|_V ) $ and $ ( {\mathcal{V}}, \left \| \cdot \right \|_{\mathcal{V}} ) $
 	be $ \R $-Banach spaces,
 	and let
 	$ \beta \in L^{(2)}(V, {\mathcal{V}}) $.
 	Then
 	\begin{enumerate}[(i)]
 		\item it holds for all 	$ A_1, A_2 \in \gamma (U, V ) $
 		and all 
 		orthonormal sets
 		$ \mathbb{U} \subseteq U $ of $ U $
 		that
 		there exists a unique $ v \in {\mathcal{V}} $
 		such that  
 		\begin{equation}
 		\inf_{ \substack{ I \subseteq \mathbb{U}, \\ \#_I < \infty } }
 		\sup_{ \substack{ I \subseteq J \subseteq \mathbb{U}, \\ \#_J < \infty } }
 		\bigg\| v - \smallsum\limits_{ u \in J } \beta( A_1 u, A_2 u ) \bigg\|_{\mathcal{V}} = 0
 		,
 		\end{equation}
 		\item it holds for all orthonormal bases
 		$ \mathbb{U}_1, \mathbb{U}_2 \subseteq U $
 		of $ U $ that
 		\begin{equation}
 		\sum_{ u \in \mathbb{U}_1 }
 		\beta( A_1 u, A_2 u )
 		=
 		\sum_{ u \in \mathbb{U}_2 }
 		\beta( A_1 u, A_2 u )
 		,
 		\end{equation}
 		\item 
 		it holds for all
 		$ A_1, A_2 \in \gamma (U, V ) $ 
 		and all
 		orthonormal sets $ \mathbb{U} \subseteq U $ of $ U $ that
 		\begin{equation}
 		\bigg \| \smallsum\limits_{u \in \mathbb{U} } \beta( A_1 u, A_2 u ) \bigg \|_{\mathcal{V}}
 		\leq
 		\| \beta \|_{L^{(2)}(V, {\mathcal{V}})}
 		\| A_1 \|_{\gamma(U,V)}
 		\| A_2 \|_{\gamma(U,V)}
 		,
 		\end{equation}
 		and
 		\item it holds for all 
 		orthonormal sets $ \mathbb{U} \subseteq U $ of $ U $
 		that
 		\begin{equation}
 		\begin{split}
 		\bigg( 
 		\gamma(U,V) \times \gamma(U,V) \ni (A_1, A_2)
 		\mapsto
 		\smallsum\limits_{u \in \mathbb{U} }
 		\beta ( A_1 u, A_2 u ) \in {\mathcal{V}}
 		\bigg)\in L^{(2)}( \gamma(U, V), {\mathcal{V}} ) 
 		.
 		\end{split}
 		\end{equation}
 	\end{enumerate}
\end{lemma}

\section{Strong a priori estimates for SPDEs}
\label{sec:strong_a_priori}

\subsection{Setting}
\label{sec:setting_strong_apriori}

Assume the setting in Section~\ref{sec:global_setting}, 
let 
$ p \in [ 2, \infty ) $, 
$ \vartheta \in [ 0, 1 ) $, 
$ \driftC, \diffusionC \in [ 0, \infty ) $, 
let
$
  X
  \colon
  [ 0, T ] \times \Omega
  \to V
$ 
be a stochastic process 
with 
$
  \sup_{ s \in [ 0, T ] } 
  \| X_s \|_{ \lpn{ p }{ \P }{ V } }
  < \infty
$, 
and for every $ t \in ( 0, T ] $ let 
$
  Y^t \colon [ 0, t ] \times \Omega
  \to V
$
and
$
  Z^t \colon [ 0, t ] \times \Omega
  \to \gamma( U, V )
$ 
be $ ( \mathcal{F}_s )_{ s \in [ 0, t ] } $-predictable stochastic processes 
which satisfy for all 
$ s \in ( 0, t ) $ 
that 
\begin{equation}
\label{eq:apriori_assumption}
    \|
    Y^t_s
    \|_{ \lpn{p}{\P}{V} }
    \leq
  \tfrac{
    \driftC
    \sup_{ u \in [ 0, s ] }
    \| X_u \|_{ \lpn{ p }{ \P }{ V } }
  }{
    ( t - s )^{ \vartheta }
  }
  \qquad 
  \text{and}
  \qquad
    \|
    Z^t_s
    \|_{ \lpn{p}{\P}{\gamma(U,V)} }
    \leq
  \tfrac{
    \diffusionC
    \sup_{ u \in [ 0, s ] }
    \| X_u \|_{ \lpn{ p }{ \P }{ V } }
  }{
    ( t - s )^{ \nicefrac{ \vartheta }{ 2 } }
  }
  .
\end{equation}

\subsection{A strong a priori estimate}
\label{sec:strong.estimate}

\begin{proposition}[A strong a priori estimate]
\label{prop:strong_apriori_estimate}
Assume the setting in Section~\ref{sec:setting_strong_apriori}. 
Then 
\begin{enumerate}[(i)]
\item
it holds for all $ t \in [0,T] $ that
$
  \P\big(
    \int_0^t
    \| Y^t_s \|_V
    +
    \| Z^t_s \|^2_{ \gamma(U,V) }
    \, ds
    < \infty
  \big)
  = 1
$
and 
\item
it holds that
\begin{equation}
\label{eq:apriori_estimate}
\begin{split}
&
  \sup_{ t \in [ 0, T ] }
  \| X_t \|_{ \lpn{ p }{ \P }{ V } }
\leq
  \sqrt{2} \,
  \mathcal{E}_{ ( 1 - \vartheta ) }
  \!\left[
    \tfrac{ 
      \driftC 
      \sqrt{2} \, 
      T^{ ( 1 - \vartheta ) } 
    }{ 
      \sqrt{ 1 - \vartheta } 
    }
    +
    \diffusionC 
    \BDG{p}{V}
    \sqrt{ 
    	2
    	T^{ ( 1 - \vartheta ) } 
    }
  \right]
\\ & \cdot
  \sup_{ t \in [ 0, T ] }
  \left\|
    X_t
    -
    \left[
    \int^t_0
    Y^t_s
    \, ds
    +
    \int^t_0
    Z^t_s
    \, dW_s
    \right]
  \right\|_{ \lpnb{ p }{ \P }{ V } }
\leq
  \left[
    1
    +
    \tfrac{
      \driftC \, T^{ ( 1 - \vartheta ) }
    }{ 
      ( 1 - \vartheta ) 
    }
    +
    \tfrac{
      \diffusionC \BDG{p}{V}
      \sqrt{
        T^{ ( 1 - \vartheta ) }
      }
    }{ 
      \sqrt{ 
        ( 1 - \vartheta ) 
      }
    }
  \right]
\\ & \cdot
  \sqrt{2} \,
  \mathcal{E}_{ ( 1 - \vartheta ) }\!\left[
    \tfrac{ 
      \driftC 
      \sqrt{2} \, 
      T^{ ( 1 - \vartheta ) } 
    }{ 
      \sqrt{ 1 - \vartheta } 
    }
    +
    \diffusionC 
    \BDG{p}{V}
    \sqrt{ 
      2
      T^{ ( 1 - \vartheta ) } 
    }
  \right]
  \sup_{ t \in [ 0, T ] }
  \|
    X_t
  \|_{ \lpn{ p }{ \P }{ V } }
  < \infty
  .
\end{split}
\end{equation}
\end{enumerate}
\end{proposition}
\begin{proof}
We first observe that 
\eqref{eq:apriori_assumption},
H\"{o}lder's inequality,
and the assumption that
$
  \sup_{ s \in [0,T] } \| X_s \|_{ \mathcal{L}^p( \P; V ) }
$
$
  < \infty 
$
imply that for all 
$ t \in [ 0, T ] $ 
it holds that 
\begin{equation}
\label{eq:apriori_F}
\begin{split}
&
  \int^t_0
    \|
      Y^t_s
    \|_{ \lpn{ p }{ \P }{ V } }
    \,
  ds
\leq
  \driftC
  \int^t_0
  \frac{
    \sup_{ v \in [ 0, s ] }
    \|
      X_v
    \|_{ \lpn{ p }{ \P }{ V } }
  }{
    ( t - s )^\vartheta
  }
  \, ds
\\ & 
\leq
  \driftC
  \left[
  \frac{ t^{ ( 1 - \vartheta ) } }{ ( 1 - \vartheta ) }
  \int^{ t }_0
  \frac{
    \sup_{ v \in [ 0, s ] }
    \|
      X_v
    \|^2_{ \lpn{ p }{ \P }{ V } }
  }{
    ( t - s )^\vartheta
  }
  \, ds
  \right]^{ 1 / 2 }
  < \infty
\end{split}
\end{equation}
and
\begin{equation}
\label{eq:apriori_B}
\begin{split}
&
  \left[
    \int^t_0
      \|
        Z^t_s
      \|^2_{ 
        \lpn{ p }{ \P }{ \gamma(U,V) } 
      }
      \,
    ds
    \,
  \right]^{ 1 / 2 }
\leq
  \diffusionC
  \left[
    \int^{ t }_0
    \frac{
      \sup_{ v \in [ 0, s ] }
      \|
        X_v
      \|^2_{ \lpn{ p }{ \P }{ V } }
    }{
      ( t - s )^\vartheta
    }
    \, ds
  \right]^{ 1 / 2 }
  < \infty
  .
\end{split}
\end{equation}
Combining~\eqref{eq:apriori_F}--\eqref{eq:apriori_B} and the assumption that 
$ p \geq 2 $ proves that 
for all $ t \in [0,T] $ it holds that
$
  \int_0^t
  \| Y^t_s \|_{ \mathcal{L}^1( \P; V ) }
  +
  \| Z^t_s \|^2_{ \mathcal{L}^2( \P; \gamma(U,V) ) }
  \, ds
  < \infty
$.
This, in turn, shows that for all 
$ t \in [ 0, T ] $ 
it holds $ \P $-a.s.\ that 
\begin{equation}
\label{eq:finite_integral}
    \int_0^t
    \| Y^t_s \|_V
    +
    \| Z^t_s \|^2_{ \gamma(U,V) }
    \, ds
    < \infty
    .
\end{equation}
It thus remains to prove \eqref{eq:apriori_estimate} to complete
the proof of Proposition~\ref{prop:strong_apriori_estimate}.
For this observe that \eqref{eq:apriori_F}--\eqref{eq:finite_integral}
imply that for all $ t \in [0,T] $
it holds that
\begin{equation}
\label{eq:apriori_combined}
\begin{split}
&
  \left\|
    \int^t_0
    Y^t_s
    \, ds
  \right\|_{ \lpn{ p }{ \P }{ V } }
+
  \left\|
    \int^t_0
    Z^t_s
    \, dW_s
  \right\|_{ \lpnb{ p }{ \P }{ V } }
\\ & \leq
  \left[
    \frac{ \driftC \, t^{ \nicefrac{( 1 - \vartheta )}{2} } }{ \sqrt{ 1 - \vartheta } }
    +
    \diffusionC 
    \BDG{p}{V}
    \,
  \right]
  \left[
    \int^t_0
    \frac{
      \sup_{ v \in [ 0, s ] }
      \|
        X_v
      \|^2_{ \lpn{ p }{ \P }{ V } }
    }{
      ( t - s )^\vartheta
    }
    \, ds
  \right]^{ 1 / 2 }.
\end{split}
\end{equation}
Next we observe 
that for all 
$ t, u \in [ 0, T ] $ 
with $ t \leq u $ 
it holds that 
\begin{equation}
\label{eq:apriori_pre-gronwall}
\begin{split}
&
  \int^{ t }_0
  \frac{
  \sup_{ v \in [ 0, s ] }
  \big\|
    X_v
  \big\|^2_{ \lpn{ p }{ \P }{ V } }
  }{
    ( t - s )^\vartheta
  }
  \, ds
=
  \int^{ u }_{ u - t }
  \frac{
  \sup_{ v \in [ 0, s - u + t ] }
  \big\|
    X_v
  \big\|^2_{ \lpn{ p }{ \P }{ V } }
  }{
    ( u - s )^\vartheta
  }
  \, ds
\\ & \leq
  \int^{ u }_{ u - t }
  \frac{
  \sup_{ v \in [ 0, s ] }
  \big\|
    X_v
  \big\|^2_{ \lpn{ p }{ \P }{ V } }
  }{
    ( u - s )^\vartheta
  }
  \, ds
\leq
  \int^{ u }_0
  \frac{
  \sup_{ v \in [ 0, s ] }
  \big\|
    X_v
  \big\|^2_{ \lpn{ p }{ \P }{ V } }
  }{
    ( u - s )^\vartheta
  }
  \, ds
  .
\end{split}
\end{equation}
Moreover, we note that Minkowski's inequality ensures that for all 
$ t \in [ 0, T ] $ 
it holds that 
\begin{equation}
\begin{split}
\label{eq:apriori_decomposition}
&
  \| X_t \|_{ \lpn{ p }{ \P }{ V } }
\\ & \leq
  \Bigg\|
    X_t
    -
    \left[
    \int^t_0
    Y^t_s
    \, ds
    +
    \int^t_0
    Z^t_s
    \, dW_s
    \right]
  \Bigg\|_{ \lpnb{ p }{ \P }{ V } }
+
  \Bigg\|
    \int^t_0
    Y^t_s
    \, ds
  \Bigg\|_{ \lpn{ p }{ \P }{ V } }
+
  \Bigg\|
    \int^t_0
    Z^t_s
    \, dW_s
  \Bigg\|_{ \lpnb{ p }{ \P }{ V } }
  .
\end{split}
\end{equation}
Combining \eqref{eq:apriori_combined}--\eqref{eq:apriori_decomposition} with the fact that
$
  \forall \, a, b \in \R \colon \left( a + b \right)^2 \leq 2 a^2 + 2 b^2 
$
proves that for all $ u \in [ 0, T ] $ 
it holds that 
\allowdisplaybreaks
\begin{equation}
\begin{split}
  \sup_{ t \in [ 0, u ] }
  \| X_t \|^2_{ \lpn{ p }{ \P }{ V } }
&\leq
  2 \, \sup_{ t \in [ 0, T ] }
  \Bigg\|
    X_t
    -
    \left[
    \int^t_0
    Y^t_s
    \, ds
    +
    \int^t_0
    Z^t_s
    \, dW_s
    \right]
  \Bigg\|^2_{ \lpnb{ p }{ \P }{ V } }
\\ & \quad +
  2\left[
  \frac{ \driftC \, T^{ \nicefrac{( 1 - \vartheta )}{2} } }{ \sqrt{ 1 - \vartheta } }
  +
    \diffusionC 
    \BDG{p}{V}
  \right]^2
  \int^u_0
  \frac{
  \sup_{ t \in [ 0, s ] }
  \big\|
    X_t
  \big\|^2_{ \lpn{ p }{ \P }{ V } }
  }{
    ( u - s )^\vartheta
  }
  \, ds
  .
\end{split}
\end{equation}
Combining this and the assumption that
$
  \sup_{ s \in [0,T] }
  \| X_s \|_{ \mathcal{L}^p( \P; V ) } < \infty
$
with 
the generalized Gronwall lemma in Chapter 7 in Henry~\cite{h81}
(see, e.g., also Andersson et al.~\cite[Lemma~2.6]{AnderssonJentzenKurniawan2015arXiv})
proves the first inequality in \eqref{eq:apriori_estimate}.
In the next step we note that 
\eqref{eq:apriori_combined}
implies that 
\begin{equation}
\begin{split}
&
  \sup_{ t \in [0,T] }
  \left\|
    X_t
    -
    \left[
    \int^t_0
    Y^t_s
    \, ds
    +
    \int^t_0
    Z^t_s
    \, dW_s
    \right]
  \right\|_{ 
    \lpnb{ p }{ \P }{ V } 
  }
\\ & \leq
  \sup_{ t \in [0,T] }
  \left\|
    X_t
  \right\|_{
    \lpn{ p }{ \P }{ V } 
  }
  +
  \sup_{ t \in [0,T] }
  \left[
  \left\|
    \int^t_0
    Y^t_s
    \, ds
  \right\|_{
    \lpn{ p }{ \P }{ V } 
  }
  +
  \left\|
    \int^t_0
    Z^t_s
    \, dW_s
  \right\|_{
    \lpnb{ p }{ \P }{ V } 
  }
  \right]
\\ & \leq
  \left[
  1
  +
  \frac{
  \driftC \, T^{ ( 1 - \vartheta ) }
  }{ ( 1 - \vartheta ) }
  +
    \diffusionC 
    \BDG{p}{V}
  \sqrt{
  \frac{
    T^{ ( 1 - \vartheta ) }
  }{ ( 1 - \vartheta ) }
  }
  \right]
  \sup_{ t \in [ 0, T ] }
  \big\|
    X_t
  \big\|_{ \lpn{ p }{ \P }{ V } }
  .
\end{split}
\end{equation}
This proves the second inequality 
in \eqref{eq:apriori_estimate}.
The third inequality in \eqref{eq:apriori_estimate}
is an immediate consequence of the 
assumption that
\begin{equation}
  \sup_{ s \in [0,T] }
  \| X_s \|_{
    \mathcal{L}^p( \P; V )
  }
  < \infty
  .
\end{equation}
The proof of Proposition~\ref{prop:strong_apriori_estimate}
is thus completed.
\end{proof}

\section{Strong perturbations for SPDEs}
\label{sec:strong_perturbation}

\subsection{Setting}
\label{sec:setting_strong_perturbation}

Assume the setting in Section~\ref{sec:global_setting}, 
let 
$ p \in [ 2, \infty ) $, 
$ \vartheta \in [ 0, 1 ) $, 
$ \driftC, \diffusionC \in [ 0, \infty ) $, 
let
$
  X, \bar{X}
  \colon
  [ 0, T ] \times \Omega
  \to V
$ 
be 
stochastic processes 
with 
$
  \sup_{ s \in [ 0, T ] } 
  \| X_s - \bar{X}_s \|_{ \lpn{ p }{ \P }{ V } }
  < \infty
$, 
and for every $ t \in ( 0, T ] $ let 
$
  Y^t, \bar{Y}^t \colon [ 0, t ] \times \Omega
  \to V
$, 
$
  Z^t, \bar{Z}^t \colon [ 0, t ] \times \Omega
  \to \gamma(U,V)
$ 
be $ ( \mathcal{F}_s )_{ s \in [ 0, t ] } $-predictable stochastic processes which satisfy for all 
$ s \in ( 0, t ) $ that 
$
  \P\big(
  \int^t_0
  \| Y^t_r \|_V
  +
  \| \bar{Y}^t_r \|_V
  +
  \| Z^t_r \|^2_{ \gamma(U,V) }
  +
  \| \bar{Z}^t_r \|^2_{ \gamma(U,V) }
  \, dr
  < \infty
  \big)
  =1
$
and  
\begin{equation}
\label{eq:difference_assumption}
    \|
    Y^t_s - \bar{Y}^t_s
    \|_{ \lpn{p}{\P}{V} }
    \leq
  \tfrac{
    \driftC
    \sup_{ u \in [ 0, s ] }
    \| X_u - \bar{X}_u \|_{ \lpn{ p }{ \P }{ V } }
  }{
    ( t - s )^{ \vartheta }
  }
    ,
\quad
    \|
    Z^t_s - \bar{Z}^t_s
    \|_{ \lpn{p}{\P}{\gamma(U,V)} }
    \leq
  \tfrac{
    \diffusionC
    \sup_{ u \in [ 0, s ] }
    \| X_u - \bar{X}_u \|_{ \lpn{ p }{ \P }{ V } }
  }{
    ( t - s )^{ \nicefrac{ \vartheta }{ 2 } }
  }
  .
\end{equation}

\subsection{Strong perturbation estimates}
\label{sec:strong_perturbation_estimates}

The following result, Corollary~\ref{prop:general_perturb}, is an immediate consequence of Proposition~\ref{prop:strong_apriori_estimate} in Subsection~\ref{sec:strong.estimate} above.

\begin{corollary}[A strong perturbation estimate]
\label{prop:general_perturb}
Assume the setting in Section~\ref{sec:setting_strong_perturbation}. 
Then 
\begin{equation}
\label{eq:perturbation_estimate}
\begin{split}
&
  \sup_{ t \in [ 0, T ] }
  \| X_t - \bar{X}_t \|_{ \lpn{ p }{ \P }{ V } }
\leq
  \sqrt{2} \,
  \mathcal{E}_{ ( 1 - \vartheta ) }\!\left[
    \tfrac{ 
      \driftC 
      \sqrt{2} \, 
      T^{ ( 1 - \vartheta ) } 
    }{ 
      \sqrt{ 1 - \vartheta } 
    }
    +
  \diffusionC \BDG{p}{V}
    \sqrt{ 
      2
      T^{ ( 1 - \vartheta ) } 
    }
  \right]
\\ & \cdot
  \sup_{ t \in [ 0, T ] }
  \left\|
    X_t
    -
    \left[
    \int^t_0
    Y^t_s
    \, ds
    +
    \int^t_0
    Z^t_s
    \, dW_s
    \right]
+
    \left[
    \int^t_0
    \bar{Y}^t_s
    \, ds
    +
    \int^t_0
    \bar{Z}^t_s
    \, dW_s
    \right]
    -
    \bar{X}_t
  \right\|_{ \lpnb{ p }{ \P }{ V } }
\\ & \leq
  \sqrt{2} \,
  \mathcal{E}_{ ( 1 - \vartheta ) }\!\left[
    \tfrac{ 
      \driftC 
      \sqrt{2} \, 
      T^{ ( 1 - \vartheta ) } 
    }{ 
      \sqrt{ 1 - \vartheta } 
    }
    +
  \diffusionC \BDG{p}{V}
    \sqrt{ 
      2
      T^{ ( 1 - \vartheta ) } 
    }
  \right]
  \left[
  1
  +
  \tfrac{
    \driftC \, T^{ ( 1 - \vartheta ) }
  }{ 
    ( 1 - \vartheta ) 
  }
  +
  \diffusionC \BDG{p}{V}
  \sqrt{
    \tfrac{
      T^{ ( 1 - \vartheta ) }
    }{ ( 1 - \vartheta ) }
  }
  \right]
\\ & 
  \cdot
  \sup_{ t \in [ 0, T ] }
  \|
    X_t
    -
    \bar{X}_t
  \|_{ \lpn{ p }{ \P }{ V } }
  < \infty
  .
\end{split}
\end{equation}
\end{corollary}

The next result, Corollary~\ref{cor:initial_perturbation}, follows directly from Corollary~\ref{prop:general_perturb} above.

\begin{corollary}
\label{cor:initial_perturbation}
Assume the setting in Section~\ref{sec:setting_strong_perturbation}, 
let 
$
  S \in
  \mathbbm{M}( [0,T], L(V) )
$,
and assume that for all 
$ t \in [ 0, T ] $ 
it holds $ \P $-a.s.\ that
\begin{equation}
  X_t
  =
    S_t\, X_0
  + 
    \int_0^t Y^t_s \, ds
  + 
    \int_0^t Z^t_s \, dW_s, 
    \qquad
  \bar{X}_t
  =
    S_t\, \bar{X}_0
  + 
    \int_0^t \bar{Y}^t_s \, ds
  + 
    \int_0^t \bar{Z}^t_s \, dW_s. 
\end{equation}
Then 
\begin{equation}
\begin{split}
&
  \sup_{ t \in [ 0, T ] }
  \| X_t - \bar{X}_t \|_{ \lpn{ p }{ \P }{ V } }
\\ & \leq
  \sqrt{2} \, 
  \left[
  \sup_{ t \in [ 0, T ] }
  \| S_t \|_{ L( V ) }
  \right]
  \|
    X_0
    -
    \bar{X}_0
  \|_{ \lpn{ p }{ \P }{ V } } 
  \,
  \mathcal{E}_{ ( 1 - \vartheta ) }
  \!\left[
    \tfrac{ 
      \driftC 
      \sqrt{2} 
      \, T^{ ( 1 - \vartheta ) } 
    }{ 
      \sqrt{ 1 - \vartheta } 
    }
    +
  \diffusionC \BDG{p}{V}
    \sqrt{ 
      2
      T^{ ( 1 - \vartheta ) } 
    }
  \right]
  .
\end{split}
\end{equation}
\end{corollary}

\section{Strong convergence of mollified solutions for SPDEs}
\label{sec:strong_convergence}

\subsection{Setting}
\label{sec:setting_strong_convergence}

Assume the setting in Section~\ref{sec:global_setting}, 
let
$
A \colon D(A)
\subseteq
V \rightarrow V
$
be a generator of a strongly continuous analytic semigroup
with 
$
\operatorname{spectrum}( A )
\subseteq
\{
z \in \mathbb{C}
\colon
\text{Re}( z ) < \eta
\}
$,
let
$
(
V_r
,
\left\| \cdot \right\|_{ V_r }
)
$,
$ r \in \R $,
be a family of interpolation spaces associated to
$
\eta - A
$ (cf., e.g., \cite[Section~3.7]{sy02}), 
let 
$ p \in [ 2, \infty ) $, $ \vartheta \in [0,1) $, 
$
  \varPi \in \mathcal{M}\big( \mathcal{B}( [ 0, T ] ), \mathcal{B}( [ 0, T ] ) \big)
$,
$ 
  ( \groupC_r )_{ 
    r \in 
    [ 0, 1 ] 
  } 
  \subseteq [ 1 , \infty ) 
$, 
$
  F \in 
  \operatorname{Lip}^0( V , V_{ - \vartheta } ) 
$, 
$
  B \in 
  \operatorname{Lip}^0( 
    V, 
    \gamma( 
      U, 
      V_{ 
        - \nicefrac{ \vartheta }{ 2 } 
      } 
    ) 
  ) 
$, 
$ 
  L \in 
  \mathcal{M}\big( 
    \mathcal{B}( \angle ) , 
    \mathcal{B}( L( V_{ - 1 } ) )
  \big)
$ 
satisfy for all $ t \in [ 0, T ] $ 
that 
$ \varPi( t ) \leq t $
and for all 
$ ( s, t ) \in \big( \angle \cap (0,T]^2 \big) $, 
$ \rho \in [ 0, 1 ) $ 
that 
$
  L_{ 0, t }( V ) \subseteq V
$, 
$
  L_{ s, t }( V_{ -\rho } ) \subseteq V
$, 
and  
$
  \|
    L_{ s, t }
  \|_{ L( V_{ - \rho }, V ) }
  \leq
  \groupC_{ \rho }
  \,
  ( t - s )^{ - \rho }
$, 
let 
$
\SGchi{r}
\in [1,\infty)
$, 
$r\in[0,1]$, 
be the real numbers 
which satisfy for all 
$r\in[0,1]$
that 
$
\SGchi{r}
=
\max\{
1
,
\sup_{ t \in (0,T] }
t^r
\,
\|
( \eta - A )^r
e^{ t A }
\|_{ L( V ) }
,
\sup_{ t \in (0,T] }
t^{-r}
\,
\|
( \eta - A )^{ - r }
( e^{ t A } - \operatorname{Id}_V )
\|_{ L( V ) }
\}
$
(cf., e.g., \cite[Lemma~11.36]{rr93}), 
and let 
$ 
  Y^\kappa \colon [0,T] \times \Omega \to V
$, 
$ \kappa \in [ 0, T ] $, 
be $ ( \mathcal{F}_t )_{ t \in [0,T] } $-predictable 
stochastic processes which satisfy 
for all $ \kappa \in [0,T] $
that
$
  \sup_{ t \in [0,T] }
  \| Y^\kappa_{\Pi(t)} \|_{ \lpn{p}{\P}{V} } 
$
$
  < \infty
$ 
and which satisfy 
that for all 
$ \kappa \in [0,T] $,
$ t \in (0,T] $ 
it holds $ \P $-a.s.\ that 
$
  Y^{ \kappa }_0 = Y^0_0
$
and 
\begin{equation}
  Y^\kappa_t
  = 
    L_{ 0, t } \, Y^\kappa_0 
  + 
    \int_0^t L_{s,t} \, e^{ \kappa A } F( Y^\kappa_{ \varPi(s) } ) \, ds
  + 
    \int_0^t L_{s,t} \, e^{ \kappa A } B( Y^\kappa_{ \varPi(s) } ) \, dW_s
  .
\end{equation}

\subsection{A priori bounds for the non-mollified process}
\label{secstrong:a_priori}

\begin{lemma}
\label{lem:strong_apriori_bound}
Assume the setting in Section~\ref{sec:setting_strong_convergence} 
and let $ \kappa \in [ 0, T ] $. 
Then 
$
  \sup_{ t \in [0,T] }
  \| Y^\kappa_t \|_{ \lpn{p}{\P}{V} } 
$
$
  < \infty
  .
$ 
\end{lemma}
\begin{proof}
We observe that for all 
$ t \in ( 0, T ] $ 
it holds that 
\begin{equation}
\begin{split}
&
  \| Y^\kappa_t \|_{ \lpn{p}{\P}{V} }
\leq
  \| L_{ 0, t } \, Y^\kappa_0 \|_{ \lpn{p}{\P}{V} }
  +
    \int^t_0
  \|
    L_{s,t} \, e^{ \kappa A } F( Y^\kappa_{ \varPi(s) } )
  \|_{ \lpn{p}{\P}{V} }
    \, ds
\\ & +
  \BDG{p}{V}
  \left[
    \int^t_0
  \|
    L_{s,t} \, e^{ \kappa A } B( Y^\kappa_{ \varPi(s) } )
  \|^2_{ \lpn{p}{\P}{ \gamma(U,V) } }
    \, ds
  \right]^{ 1/2 }
\\ & \leq
  \SGchi{0} \,
  \| Y^\kappa_0 \|_{ \lpn{p}{\P}{V} }
  +
  \int^t_0
  \tfrac{
    \SGchi{0} \, \groupC_\vartheta
    \,
    \| F( Y^\kappa_{ \Pi(s) } ) \|_{ \lpn{p}{\P}{ V_{ -\vartheta } } }
  }{ ( t - s )^\vartheta }
  \, ds
\\ & \quad +
  \BDG{p}{V}
  \left[
    \int^t_0
  \tfrac{
    |\SGchi{0}|^2 \, |\groupC_{ \nicefrac{\vartheta}{2} }|^2
    \,
    \| B( Y^\kappa_{ \Pi(s) } ) \|^2_{ \lpn{p}{\P}{ \gamma( U, V_{ -\nicefrac{\vartheta}{2} } ) } }
  }{ ( t - s )^\vartheta }
    \, ds
  \right]^{ 1/2 }
\\ & \leq
  \left[
    \SGchi{0}
    +
    \tfrac{
      \SGchi{0} \, \groupC_\vartheta \, T^{ ( 1 - \vartheta ) } \,
      \| F \|_{ \operatorname{Lip}^0( V, V_{ -\vartheta } ) }
    }{ ( 1 - \vartheta ) }
    +
    \tfrac{
      \SGchi{0} \, \groupC_{ \vartheta/2 } \BDG{p}{V}
      \sqrt{
        T^{ ( 1 - \vartheta ) }
      }
      \,
      \| B \|_{ \operatorname{Lip}^0( V, \gamma( U, V_{ -\vartheta/2 } ) ) }
    }{
      \sqrt{ 1 - \vartheta }
    }
  \right]
\\ & \quad \cdot
  \sup_{ s \in [0,T] }
  \|
  \max\{
    1
    ,
    \| Y^\kappa_{ \Pi(s) } \|_V
  \}
  \|_{ \lpn{p}{\P}{\R} } 
  .
\end{split}
\end{equation}
This and the fact that 
$
  \sup_{ t \in [0,T] }
  \|
  \max\{
    1
    ,
    \| Y^\kappa_{ \Pi(t) } \|_V
  \}
  \|_{ \lpn{p}{\P}{\R} } 
  \leq
  1
  +
  \sup_{ t \in [0,T] }
  \|
  Y^\kappa_{ \Pi(t) }
  \|_{ \lpn{p}{\P}{V} } 
  < \infty
$
complete the proof of Lemma~\ref{lem:strong_apriori_bound}.
\end{proof}

\begin{proposition}[An a priori bound for the non-mollified process]
\label{prop:numerics_Lp_bound}
Assume the setting in Section~\ref{sec:setting_strong_convergence}. 
Then 
\begin{equation}
\begin{split}
&
  \sup_{ t \in [0,T] }
  \| Y^0_t \|_{
   \lpn{ p }{ \P }{ V }
  }
\leq
  \sqrt{2} 
  \,
  \bigg[
    \sup_{ t \in ( 0, T ] }
    \max\!\big\{
      1 
      , 
      \| 
        L_{ 0, t } 
      \|_{ L(V) }
    \big\}
    \,
    \| 
      Y^0_0 
    \|_{ \lpn{p}{\P}{V} }
\\ & 
    +
    \tfrac{
      \groupC_\vartheta 
      \,
      T^{ ( 1 - \vartheta ) } \,
      \| F(0) \|_{ V_{ -\vartheta } }
    }{
      ( 1 - \vartheta )
    }
    +
    \tfrac{
      \groupC_{ \vartheta / 2 } \BDG{p}{V}
      \sqrt{
        T^{ ( 1 - \vartheta ) }
      }
      \,
      \| B(0) \|_{ \gamma( U, V_{ - \vartheta / 2 } ) }
    }{
      \sqrt{
        1 - \vartheta
      }
    }
  \bigg]
\\ & \cdot
  \mathcal{E}_{ ( 1 - \vartheta ) }\!\left[
  \tfrac{ \sqrt{2} \, \groupC_\vartheta \, T^{( 1 - \vartheta )} \,   | F |_{ \operatorname{Lip}^0( V, V_{ -\vartheta } ) } }{ \sqrt{ 1 - \vartheta } }
  +
  \BDG{p}{V} \,
  \groupC_{ \nicefrac{ \vartheta }{ 2 } } \,
  \sqrt{ 2 T^{( 1 - \vartheta )} } \,
  | B |_{ 
  	\operatorname{Lip}^0( V, \gamma( U, V_{ -\nicefrac{ \vartheta }{ 2 } } ) ) 
  }
  \right]
  < \infty
  .
\end{split}
\end{equation}
\end{proposition}
\begin{proof}
Throughout this proof let 
$
  \tilde{L}
  \colon
  \{ ( t_1, t_2 ) \in [ 0, T ]^2 \colon t_1 \leq t_2 \}
  \to L( V_{-1} )
$ 
be the function which satisfies for all 
$ t_0 \in [0,T] $, 
$ ( t_1, t_2 ) \in \angle $,
$ v \in V_{ - 1 } $
that 
\begin{equation}
  \tilde{L}_{ t_1, t_2 } v = L_{ t_1, t_2 } v
  \qquad\text{and}\qquad
  \tilde{L}_{t_0,t_0} = \operatorname{Id}_{ V_{ -1 } }
  .
\end{equation}
Combining Corollary~\ref{prop:general_perturb} and Lemma~\ref{lem:strong_apriori_bound}
shows\footnote{with 
$ \bar{X}_t = 0 $,
$ \bar{Y}^t_s = \tilde{L}_{ s, t } F(0) $,
$ \bar{Z}^t_s = \tilde{L}_{ s, t } B(0) $
for $ s \in (0,t) $, $ t \in (0,T] $
in the notation of Corollary~\ref{prop:general_perturb}} 
that 
\begin{equation}
\begin{split}
\label{eq:numerics_lp_bound}
&
  \sup_{ t \in [ 0, T ] }
  \left\|
    Y^0_t
  \right\|_{ \lpn{ p }{ \P }{ V } }
\leq
  \sqrt{2}
  \sup_{ t \in [ 0, T ] }
  \left\|
    \tilde{L}_{ 0, t } \, Y^0_0
    +
    \int^t_0
    \tilde{L}_{ s, t } \,
    F(0)
    \, ds
    +
    \int^t_0
    \tilde{L}_{ s, t } \,
    B(0)
    \, dW_s
  \right\|_{ \lpnb{ p }{ \P }{ V } }
\\ & \cdot
  \mathcal{E}_{ ( 1 - \vartheta ) }\!\left[
    \tfrac{ \sqrt{2} \, T^{( 1 - \vartheta )} }{ \sqrt{ 1 - \vartheta } }
    \groupC_\vartheta \,
    | F |_{ \operatorname{Lip}^0( V, V_{ -\vartheta } ) }
    +
    \sqrt{ 2 T^{( 1 - \vartheta )} } \,
    \BDG{p}{V} \,
    \groupC_{ \nicefrac{ \vartheta }{ 2 } } \,
    | B |_{ 
      \operatorname{Lip}^0( V, \gamma( U, V_{ -\nicefrac{ \vartheta }{ 2 } } ) ) 
    }
  \right]
  .
\end{split}
\end{equation}
Combining \eqref{eq:numerics_lp_bound}
with the triangle inequality
completes the proof of Proposition~\ref{prop:numerics_Lp_bound}.
\end{proof}

\subsection{A strong convergence result}
\label{secstrong:strong}

\begin{proposition}[A bound on the difference between the mollified and the non-mollified processes]
\label{prop:strong_convergence_numerics}
Assume the setting in Section~\ref{sec:setting_strong_convergence} and let 
$ \kappa \in [ 0, T ] $, 
$ \rho \in [ 0, \frac{ 1 - \vartheta }{ 2 } ) $. 
Then 
\begin{align}
&
\label{eq:strong_convergence_numerics}
  \sup_{ t \in [0,T] }
  \left\|
    Y^0_t - Y^{ \kappa }_t
  \right\|_{ 
    \lpn{ p }{ \P }{ V } 
  }
\leq
  \tfrac{ 
    2 \, \kappa^\rho 
  }{ 
    T^\rho 
  }
  \bigg[
    \sup\nolimits_{ t \in ( 0, T ] }
    \max\!\big\{
      1
      , 
      \| 
        L_{ 0, t } 
      \|_{ L(V) }
    \big\}
    \,
    \max\{ 1, \| Y^0_0 \|_{ \lpn{p}{\P}{V} } \}
\nonumber
\\ & +
    \tfrac{
      \SGchi{\rho}
      \, 
      \groupC_\vartheta 
      \, 
      \groupC_{ \rho + \vartheta }
      \,
      T^{ ( 1 - \vartheta ) }
      \,
      \| 
        F 
      \|_{ 
        \operatorname{Lip}^0( V, V_{ - \vartheta } ) 
      }
    }{
      ( 1 - \vartheta - \rho )
    }
    +
    \tfrac{
      \BDG{p}{V} \,
      \SGchi{\rho} \, 
      \groupC_{ \vartheta / 2 } \, 
      \groupC_{ \rho + \vartheta / 2 }
      \sqrt{
        T^{ ( 1 - \vartheta ) }
      }
      \,
      \| B \|_{ 
        \operatorname{Lip}^0( 
          V, \gamma( U, V_{ - \vartheta / 2 } ) 
        ) 
      }
    }{
      \sqrt{ 
        1 - \vartheta - 2 \rho
      }
    }
  \bigg]^2
\\ & \cdot
\nonumber
\bigg|
\mathcal{E}_{ ( 1 - \vartheta ) }\bigg[
\tfrac{ 
	\sqrt{2} \, 
	T^{ ( 1 - \vartheta ) } \, 
	\SGchi{0} \, 
	\groupC_\vartheta
}{ 
\sqrt{ 1 - \vartheta } 
}
	| F |_{ 
		\operatorname{Lip}^0( V, V_{ - \vartheta } ) 
	}
+
\BDG{p}{V}
\sqrt{ 
	2
	T^{ ( 1 - \vartheta ) } 
}
\,
\SGchi{0} \, 
\groupC_{ \vartheta / 2 } \,
| B |_{ 
	\operatorname{Lip}^0( 
	V, 
	\gamma( U, V_{ - \vartheta / 2 } ) 
	) 
}
\bigg]
\bigg|^2
  .
\end{align}
\end{proposition}
\begin{proof}
First of all, we observe that Lemma~\ref{lem:strong_apriori_bound} 
allows us to apply Corollary~\ref{prop:general_perturb} 
to obtain\footnote{with 
$ \bar{X}_t = Y^\kappa_t $, 
$ \bar{Y}^t_s = L_{ s, t } \, e^{ \kappa A } F( Y^\kappa_{\varPi(s)} ) $, 
$ \bar{Z}^t_s = L_{ s, t } \, e^{ \kappa A } B( Y^\kappa_{\varPi(s)} ) $ 
for $ s \in ( 0, t ) $, $ t \in ( 0, T ] $ 
in the notation of Corollary~\ref{prop:general_perturb}} that 
\begin{equation}\label{eq:strong_converge_num_decompose}
\begin{split}
&
  \sup_{ t \in [0,T] }
  \left\|
    Y^0_t - Y^\kappa_t
  \right\|_{ \lpn{ p }{ \P }{ V } }
\\ & \leq
  \mathcal{E}_{ ( 1 - \vartheta ) }
  \!\left[
    \groupC_\vartheta
    \,
    | 
      e^{ \kappa A } F 
    |_{ 
      \operatorname{Lip}^0( V, V_{ -\vartheta } ) 
    }
    \tfrac{ 
      \sqrt{2} \, 
      T^{
        ( 1 - \vartheta )
      } 
    }{ 
      \sqrt{ 1 - \vartheta } 
    }
    +
    \groupC_{ 
      \nicefrac{ \vartheta }{ 2 } 
    } 
    \,
    \BDG{p}{V}
    | e^{ \kappa A } B |_{ 
      \operatorname{Lip}^0( 
        V, 
        \gamma( 
          U, 
          V_{ - \vartheta / 2 } 
        ) 
      ) 
    }
    \sqrt{ 
      2
      T^{ 
        ( 1 - \vartheta )
      } 
    }
  \right]
\\ & \cdot
  \sqrt{2}
  \sup_{ t \in [ 0, T ] }
  \left\|
    \int^t_0
    L_{ s, t } \,
    \big( \operatorname{Id}_V - e^{ \kappa A } \big)
    F( Y^0_{ \varPi( s ) } )
    \, ds
    +
    \int^t_0
    L_{ s, t } \,
    \big( \operatorname{Id}_V - e^{ \kappa A } \big)
    B( Y^0_{ \varPi( s ) } )
    \, dW_s
  \right\|_{ \lpnb{ p }{ \P }{ V } }
  .
\end{split}
\end{equation}
Moreover, we observe that for all 
$ t \in ( 0, T ] $ 
it holds that 
\begin{equation}\label{eq:strong_converge_num_F}
\begin{split}
&
  \left\|
    \int^t_0
    L_{ s, t } 
    \left( 
      \operatorname{Id}_V - e^{ \kappa A } 
    \right)
    F( Y^0_{ \varPi( s ) } )
    \, ds
  \right\|_{ \lpn{ p }{ \P }{ V } }
\leq
  \int^t_0
  \frac{
    \SGchi{\rho} \, \groupC_{ \rho + \vartheta } \, \kappa^\rho
  }{ ( t - s )^{ ( \rho + \vartheta ) } } \,
  \| F( Y^0_{ \varPi( s ) } ) \|_{ \lpn{ p }{ \P }{ V_{ -\vartheta } } }
  \, ds
\\ & \leq
  \tfrac{
    \SGchi{\rho} \, \groupC_{ \rho + \vartheta } \,
    t^{ ( 1 - \vartheta - \rho ) }
  }{
    ( 1 - \vartheta - \rho )
  }
  \,
  \| F \|_{ \operatorname{Lip}^0( V, V_{ -\vartheta } ) }
    \sup_{ s \in [ 0, T ] }
  \max\!\left\{
    1,
    \| Y^0_s \|_{ \lpn{ p }{ \P }{ V } }
  \right\}
  \kappa^\rho
  .
\end{split}
\end{equation}
In addition, Lemma~\ref{lem:strong_apriori_bound} ensures that for all 
$ t \in (0,T] $ 
it holds that 
\begin{equation}\label{eq:strong_converge_num_B}
\begin{split}
&
  \Bigg\|
    \int^t_0
    L_{ s, t } \,
    \big( 
      \operatorname{Id}_V - e^{ \kappa A } 
    \big)
    B( Y^0_{ \varPi( s ) } )
    \, dW_s
  \Bigg\|_{ \lpnb{ p }{ \P }{ V } }
\\ & \leq
  \BDG{p}{V}
  \left[
  \int^t_0
  \frac{
    |
      \SGchi{\rho} \, 
      \groupC_{ \rho + \vartheta / 2 } \, 
      \kappa^\rho
    |^2
  }{ 
    ( t - s )^{ 
      ( 2 \rho + \vartheta ) 
    } 
  } \,
  \| 
    B( Y^0_{ \varPi( s ) } ) 
  \|^2_{ 
    \lpn{ p }{ \P }{ 
      \gamma( U, V_{ - \vartheta / 2 } ) 
    } 
  }
  \, ds
  \right]^{ 1 / 2 }
\\ & \leq
  \tfrac{
    \BDG{p}{V} \,
    \SGchi{\rho} \, 
    \groupC_{ \rho + \vartheta / 2 }
    \sqrt{ 
      t^{ ( 1 - \vartheta - 2 \rho ) } 
    }
  }{
    \sqrt{ 1 - \vartheta - 2 \rho }
  }
  \,
  \| B \|_{ 
    \operatorname{Lip}^0( 
      V, 
      \gamma( U, V_{ - \vartheta / 2 } ) 
    ) 
  }
  \sup_{ s \in [ 0, T ] }
  \max\!\left\{
    1,
    \| Y^0_s \|_{ \lpn{ p }{ \P }{ V } }
  \right\}
  \kappa^\rho.
\end{split}
\end{equation}
Putting~\eqref{eq:strong_converge_num_F}
and~\eqref{eq:strong_converge_num_B}
into~\eqref{eq:strong_converge_num_decompose} 
yields that
\begin{equation}\label{eq:strong_converge_num}
\begin{split}
&
  \sup_{ t \in [0,T] }
  \left\|
    Y^0_t - Y^\kappa_t
  \right\|_{ \lpn{ p }{ \P }{ V } }
\leq
  \sqrt{2} \,
  \kappa^\rho 
  \sup_{ t \in [ 0, T ] }
  \max\!\left\{
    1,
    \| Y^0_t \|_{ \lpn{ p }{ \P }{ V } }
  \right\}
\\ & \cdot
  \mathcal{E}_{ ( 1 - \vartheta ) }\!\left[
    \tfrac{ 
      \sqrt{2} \, 
      T^{ ( 1 - \vartheta ) } \, 
      \SGchi{0} \, 
      \groupC_\vartheta 
    }{ 
      \sqrt{ 1 - \vartheta } 
    }
    | F |_{ 
      \operatorname{Lip}^0( V, V_{ - \vartheta } ) 
    }
    +
    \BDG{p}{V}
    \sqrt{ 
      2
      T^{ ( 1 - \vartheta ) } 
    }
    \,
    \SGchi{0} \, 
    \groupC_{ \vartheta / 2 } \,
    | B |_{ 
      \operatorname{Lip}^0( 
        V, 
        \gamma( U, V_{ - \vartheta / 2 } ) 
      ) 
    }
  \right]
\\ & \cdot
  \left[
    \tfrac{
      \SGchi{\rho} \, 
      \groupC_{ \rho + \vartheta } \,
      T^{ ( 1 - \vartheta - \rho ) }
    }{
      ( 1 - \vartheta - \rho )
    }
    \,
    \| F \|_{ \operatorname{Lip}^0( V, V_{ -\vartheta } ) }
    +
    \tfrac{
      \BDG{p}{V} \,
      \SGchi{\rho} \, 
      \groupC_{ \rho + \vartheta / 2 } 
      \sqrt{ 
        T^{ ( 1 - \vartheta - 2 \rho ) } 
      }
    }{
      \sqrt{ 1 - \vartheta - 2 \rho }
    }
    \,
    \| B \|_{ 
      \operatorname{Lip}^0( 
        V, 
        \gamma( U, V_{ - \vartheta / 2 } ) 
      ) 
    }
  \right]
  .
\end{split}
\end{equation}
Combining 
Proposition~\ref{prop:numerics_Lp_bound} 
and \eqref{eq:strong_converge_num} 
proves that
\begin{equation}
\begin{split}
&
  \left\|
    Y^0_T - Y^{ \kappa }_T
  \right\|_{ \lpn{ p }{ \P }{ V } }
\leq
  2 
  \, 
  \kappa^\rho 
  \,
  \bigg[
    \sup_{ t \in ( 0, T ] }
    \max\!\big\{
      1
      , 
      \| 
        L_{ 0, t } 
      \|_{ L(V) }
    \big\}
    \,
    \max\{ 1, \| Y^0_0 \|_{ \lpn{p}{\P}{V} } \}
\\ & +
    \tfrac{
      \groupC_\vartheta \,
      T^{ ( 1 - \vartheta ) } \,
      \| F(0) \|_{ V_{ -\vartheta } }
    }{
      ( 1 - \vartheta )
    }
    +
    \tfrac{
      \BDG{p}{V} \,
      \groupC_{ \vartheta / 2 }
      \sqrt{
        T^{ ( 1 - \vartheta ) }
      }
      \,
      \| B(0) \|_{ \gamma( U, V_{ - \vartheta / 2 } ) }
    }{
      \sqrt{
        1 - \vartheta
      }
    }
  \bigg]
\\ & \cdot
  \bigg[
    \tfrac{
      \SGchi{\rho} \, 
      \groupC_{ \rho + \vartheta } \,
      T^{ ( 1 - \vartheta - \rho ) } \,
      \| F \|_{ \operatorname{Lip}^0( 
        V, V_{ - \vartheta } ) 
      }
    }{
      ( 1 - \vartheta - \rho )
    }
    +
    \tfrac{
      \BDG{p}{V} \,
      \SGchi{\rho}
      \, 
      \groupC_{ 
        \rho 
        + 
        \vartheta / 2 
      }
      \sqrt{ 
        T^{ ( 1 - \vartheta - 2 \rho ) } 
      }
      \,
      \| B \|_{ 
        \operatorname{Lip}^0( 
          V, \gamma( U, V_{ - \vartheta / 2 } ) 
        ) 
      }
    }{
      \sqrt{ 1 - \vartheta - 2 \rho }
    }
  \bigg]
\\ & \cdot
  \bigg|
    \mathcal{E}_{ ( 1 - \vartheta ) }\bigg[
      \tfrac{ 
        \sqrt{2} \, 
        T^{ ( 1 - \vartheta ) } \, 
        \SGchi{0} \, 
        \groupC_\vartheta \,
        | F |_{ 
          \operatorname{Lip}^0( V, V_{ - \vartheta } ) 
        }
      }{ 
        \sqrt{ 1 - \vartheta } 
      }
      +
      \BDG{p}{V}
      \sqrt{ 
        2
        T^{ ( 1 - \vartheta ) } 
      }
      \,
      \SGchi{0} \, 
      \groupC_{ \vartheta / 2 } \,
      | B |_{ 
        \operatorname{Lip}^0( 
          V, 
          \gamma( U, V_{ - \vartheta / 2 } ) 
        ) 
      }
    \bigg]
  \bigg|^2
  .
\end{split}
\end{equation}
Hence, we obtain that
\begin{equation}
\label{eq:strong_convergence_conclude}
\begin{split}
&
  \left\|
    Y^0_T - Y^{ \kappa }_T
  \right\|_{ \lpn{ p }{ \P }{ V } }
\leq
  \tfrac{
    2 
    \, 
    \kappa^\rho 
  }{
    T^{ \rho }
  }
  \,
  \bigg[
  \sup_{ t \in ( 0, T ] }
  \max\!\big\{
  1
  , 
  \| 
  L_{ 0, t } 
  \|_{ L(V) }
  \big\}
  \,
  \max\{ 1, \| Y^0_0 \|_{ \lpn{p}{\P}{V} } \}
  \\ & +
  \tfrac{
  	\groupC_\vartheta \,
  	T^{ ( 1 - \vartheta ) } \,
  	\| F(0) \|_{ V_{ -\vartheta } }
  }{
  ( 1 - \vartheta )
}
+
\tfrac{
	\BDG{p}{V} \,
	\groupC_{ \vartheta / 2 }
	\sqrt{
		T^{ ( 1 - \vartheta ) }
	}
	\,
	\| B(0) \|_{ \gamma( U, V_{ - \vartheta / 2 } ) }
}{
\sqrt{
	1 - \vartheta
}
}
\bigg]
\\ & \cdot
\bigg[
\tfrac{
	\SGchi{\rho} \, 
	\groupC_{ \rho + \vartheta } \,
	T^{ ( 1 - \vartheta ) } \,
	\| F \|_{ \operatorname{Lip}^0( 
		V, V_{ - \vartheta } ) 
	}
}{
( 1 - \vartheta - \rho )
}
+
\tfrac{
	\BDG{p}{V} \,
	\SGchi{\rho}
	\, 
	\groupC_{ 
		\rho 
		+ 
		\vartheta / 2 
	}
	\sqrt{ 
		T^{ ( 1 - \vartheta ) } 
	}
	\,
	\| B \|_{ 
		\operatorname{Lip}^0( 
		V, \gamma( U, V_{ - \vartheta / 2 } ) 
		) 
	}
}{
\sqrt{ 1 - \vartheta - 2 \rho }
}
\bigg]
\\ & \cdot
\bigg|
\mathcal{E}_{ ( 1 - \vartheta ) }\bigg[
\tfrac{ 
	\sqrt{2} \, 
	T^{ ( 1 - \vartheta ) } \, 
	\SGchi{0} \, 
	\groupC_\vartheta \,
	| F |_{ 
		\operatorname{Lip}^0( V, V_{ - \vartheta } ) 
	}
}{ 
\sqrt{ 1 - \vartheta } 
}
+
\BDG{p}{V}
\sqrt{ 
	2
	T^{ ( 1 - \vartheta ) } 
}
\,
\SGchi{0} \, 
\groupC_{ \vartheta / 2 } \,
| B |_{ 
	\operatorname{Lip}^0( 
	V, 
	\gamma( U, V_{ - \vartheta / 2 } ) 
	) 
}
\bigg]
\bigg|^2
.
\end{split}
\end{equation}
This implies 
\eqref{eq:strong_convergence_numerics}.
The proof of 
Proposition~\ref{prop:strong_convergence_numerics} 
is thus completed.
\end{proof}

\section{Weak temporal regularity and analysis of the weak distance between exponential Euler 
approximations of SPDEs and their semilinear integrated counterparts}
\label{sec:weak_temporal_regularity}

\subsection{Setting}
\label{sec:setting_weak_temporal_regularity}

Assume the setting in Section~\ref{sec:global_setting},
let $ \mathbb{U} \subseteq U $ be an orthonormal basis of $ U $, 
let
$
A \colon D(A)
\subseteq
V \rightarrow V
$
be a generator of a strongly continuous analytic semigroup
with 
$
\operatorname{spectrum}( A )
\subseteq
\{
z \in \mathbb{C}
\colon
\text{Re}( z ) < \eta
\}
$,
let
$
(
V_r
,
\left\| \cdot \right\|_{ V_r }
)
$,
$ r \in \R $,
be a family of interpolation spaces associated to
$
\eta - A
$, 
let 
$ h \in (0,\infty) $, 
$ p \in [ 2, \infty ) $, 
$ \vartheta \in [0,1) $, 
$
  F \in 
  \operatorname{Lip}^0( V , V_{ - \vartheta } ) 
$, 
$
  B \in 
  \operatorname{Lip}^0( 
    V, 
    \gamma( 
      U, 
      V_{ 
        - \vartheta / 2 
      } 
    ) 
  ) 
$, 
let 
$ 
  ( B^b )_{ b \in \mathbb{U} } \subseteq C( V, V_{ -\nicefrac{\vartheta}{2} } ) 
$ 
be the functions which satisfy for all 
$ b \in \mathbb{U} $, 
$ v \in V $ 
that 
$
     B^b( v ) 
    = 
      B( 
        v 
      )
      \,
      b
$, 
let $ \varsigma_{ F, B } \in \R $ 
be the real number given by 
$
  \varsigma_{ F, B }
  =
    \max\{
      1,
      \| F \|_{ \operatorname{Lip}^0( V, V_{ -\vartheta } ) },
      \| B \|^2_{ \operatorname{Lip}^0( V, \gamma( U, V_{ - \vartheta / 2 } ) ) } 
    \}
$, 
let 
$
\SGchi{r}
\in [1,\infty)
$, 
$r\in[-1,1]$, 
be the real numbers 
which satisfy for all 
$r\in[-1,1]$
that 
$
\SGchi{r}
=
\max\{
1
,
\sup_{ t \in (0,T] }
t^{\max\{r,0\}}
\,
\|
( \eta - A )^r
e^{ t A }
\|_{ L( V ) }
,
\sup_{ t \in (0,T] }
t^{-\max\{r,0\}}
$
$
\|
( \eta - A )^{ - \max\{r,0\} }
( e^{ t A } - \operatorname{Id}_V )
\|_{ L( V ) }
\}
$, 
let 
$ 
  Y, \bar{Y} \colon [0,T] \times \Omega \to V
$ 
be 
$
  ( \mathcal{F}_t )_{ t \in [0,T] }
$-predictable stochastic processes
which satisfy 
$
  \| Y_0 \|_{ \lpn{p}{\P}{V} } 
  < \infty
$
and
$
  \bar{Y}_0 = Y_0
$
and which satisfy that 
for all $ t \in (0,T] $ 
it holds $ \P $-a.s.\ that 
\begin{equation} 
  Y_t
  = 
    e^{tA}\, Y_0 
  + 
    \int_0^t e^{ (t-\lfloor s \rfloor_h)A }\, F( Y_{ \floor{ s }{ h } } ) \, ds
  + 
    \int_0^t e^{ (t-\lfloor s \rfloor_h)A }\, B( Y_{ \floor{ s }{ h } } ) \, dW_s, 
\end{equation} 
\begin{equation} 
  \bar{Y}_t
  = 
    e^{ t A } \, \bar{Y}_0 
  + 
    \int_0^t e^{ ( t - s )A } \, F( Y_{ \floor{ s }{ h } } ) \, ds
  + 
    \int_0^t e^{ ( t - s )A } \, B( Y_{ \floor{ s }{ h } } ) \, dW_s 
    ,
\end{equation}
and let 
$
  ( K_r )_{ r \in [ 0, \infty ) }
  \subseteq
  [ 0, \infty ]
$ 
be the extended real numbers
which satisfy 
for all $ r \in [0,\infty) $
that
$
  K_r
  =
  \sup_{ s, t \in [ 0, T ] }
  \ES\big[\!
    \max\{
    1,
    \| \bar{Y}_s \|^r_V,
    \| Y_t \|^r_V
    \}
  \big]
$.

\subsection{Weak temporal regularity of semilinear integrated exponential Euler approximations}

In Proposition~\ref{prop:weak_temporal_regularity_1st} below
we establish a weak temporal regularity result for the process $ \bar{Y} $
in Subsection~\ref{sec:setting_weak_temporal_regularity}.
The proof of Proposition~\ref{prop:weak_temporal_regularity_1st} uses the
following elementary result.

\begin{lemma}
\label{lem:Kp_estimate}
Assume the setting in Section~\ref{sec:setting_weak_temporal_regularity}. Then 
\begin{align}
\label{eq:Kp_estimate}
&
  \sup_{ r \in [ 0, p ] }
  K_r
  =
  K_p
\\ & \leq
\nonumber
  \left[
    \SGchi{0} \,
    \max\{
      1
      ,
      \| Y_0 \|_{\lpn{p}{\P}{V}}
    \}
    +
    \tfrac{
    \SGchi{\vartheta} \,
    \|F\|_{\operatorname{Lip}^0(V, V_{-\vartheta})} \,
    T^{ ( 1 - \vartheta ) }
    }{( 1 - \vartheta )}
  +
  \tfrac{
    \BDG{p}{} \,
    \SGchi{\nicefrac{\vartheta}{2}} \,
    \sqrt{T^{ ( 1 - \vartheta ) }} \,
    \|B\|_{\operatorname{Lip}^0(V, \gamma( U, V_{-\nicefrac{\vartheta}{2}} ))}
  }{\sqrt{1 - \vartheta}}
  \right]^{2p}
\\ & \cdot
\nonumber
  2^{ ( \frac{p}{2}+1 ) }
  \left|
  \mathcal{E}_{ ( 1 - \vartheta ) }\!\left[
    \tfrac{
      \sqrt{ 2 }
      \,
      \SGchi{ \vartheta }
      \,
      T^{ ( 1 - \vartheta ) }
      \,
      |
        F
      |_{
        \operatorname{Lip}^0( V, V_{ - \vartheta } )
      }
    }{
      \sqrt{1 - \vartheta}
    }
    +
    \BDG{p}{} \,
    \SGchi{ 
      \nicefrac{\vartheta}{2}
    }
    \sqrt{
      2 T^{ ( 1 - \vartheta ) }
    } \,
    |
      B
    |_{
      \operatorname{Lip}^0( V, \gamma( U, V_{ - \nicefrac{\vartheta}{2} } ) )
    }
  \right]
  \right|^p
  < \infty
  .
\end{align}
\end{lemma}
\begin{proof}
We first observe that the equality in~\eqref{eq:Kp_estimate} follows from the fact that 
for all $ x \in V $, $ r, s \in [ 0, \infty ) $ with $ r \leq s $ 
it holds that 
$
  \max\{
    1
    ,
    \| x \|^r_V
  \}
  \leq
  \max\{
    1
    ,
    \| x \|^s_V
  \}
$.
Moreover, we note that the second inequality in~\eqref{eq:Kp_estimate} is an immediate consequence 
of the assumption that 
$
  \| Y_0 \|_{ \lpn{p}{\P}{V} }
  < \infty
$. 
It thus remains to prove the first inequality in~\eqref{eq:Kp_estimate}. 
For this we claim that for all 
$
  k \in \{ 0, 1, \ldots, \nicefrac{\floor{T}{h}}{h} \}
$ 
it holds that
\begin{equation}
\label{eq:step.estimate}
  \| Y_{kh} \|_{ \lpn{p}{\P}{V} }
  < \infty
  . 
\end{equation}

We now prove~\eqref{eq:step.estimate} by induction on 
$
  k \in \{ 0, 1, \ldots, \nicefrac{\floor{T}{h}}{h} \}
$.
The assumption that 
$
  \|Y_0\|_{ \lpn{p}{\P}{V} }
  < \infty
$ 
establishes~\eqref{eq:step.estimate} in the base case $k=0$. 
For the induction step 
$
  \N_0 \cap (-\infty,\nicefrac{\floor{T}{h}}{h}) \ni k 
  \to k+1 \in \N \cap [0,\nicefrac{\floor{T}{h}}{h}]
$ 
assume that there exists a nonnegative integer 
$
  k \in \N_0 \cap (-\infty,\nicefrac{\floor{T}{h}}{h})
$ 
such that~\eqref{eq:step.estimate} holds for 
$k=0$, 
$k=1,\ldots,k=k$.
This ensures that
\begin{equation}
\label{eq:numerics_bound_induction}
\begin{split}
&
  \| Y_{(k+1)h} \|_{ \lpn{p}{\P}{V} }
\\ & \leq
  \| e^{ (k+1)hA } \, Y_0 \|_{ \lpn{p}{\P}{V} }
  + 
    \left\|
    \int_0^{(k+1)h} 
    e^{ ((k+1)h-\floor{s}{h}) }\, F( Y_{ \floor{ s }{ h } } ) 
    \, ds
    \right\|_{ \lpn{p}{\P}{V} }
\\&+ 
    \left\|
    \int_0^{(k+1)h} 
    e^{ ((k+1)h-\floor{s}{h}) }\, B( Y_{ \floor{ s }{ h } } ) 
    \, dW_s
    \right\|_{ \lpnb{p}{\P}{V} } 
\\ & \leq
  \SGchi{0} \, 
  \| Y_0 \|_{ \lpn{p}{\P}{V} }
  +
  \int^{(k+1)h}_0
  \|
  e^{ ((k+1)h-\floor{s}{h}) }\, F( Y_{ \floor{ s }{ h } } )
  \|_{ \lpn{p}{\P}{V} }
  \, ds
\\ & +
  \BDG{p}{}
  \left[
  \int^{(k+1)h}_0
  \|
  e^{ ((k+1)h-\floor{s}{h}) }\, B( Y_{ \floor{ s }{ h } } )
  \|^2_{ \lpn{p}{\P}{\gamma(U,V)} }
  \, ds
  \right]^{1/2}
\\ & \leq
  \SGchi{0} \, 
  \| Y_0 \|_{ \lpn{p}{\P}{V} }
  +
  \SGchi{\vartheta} \,
  \|F\|_{\operatorname{Lip}^0(V, V_{-\vartheta})}
  \left[
  \max_{ j \in \{ 0, 1, \ldots, k \} }
  \max\{
    1
    ,
    \| Y_{ jh } \|_{ \lpn{p}{\P}{V} }
  \}
  \right]
  \int^{(k+1)h}_0
  \tfrac{1}{( (k+1)h - s )^\vartheta}
  \, ds
\\ & +
  \BDG{p}{} \,
  \SGchi{\nicefrac{\vartheta}{2}} \,
  \|B\|_{\operatorname{Lip}^0(V, \gamma( U, V_{-\nicefrac{\vartheta}{2}} ))}
  \left[
  \max_{ j \in \{ 0, 1, \ldots, k \} }
  \max\{
    1
    ,
    \| Y_{ jh } \|_{ \lpn{p}{\P}{V} }
  \}
  \right]
  \left[
  \int^{(k+1)h}_0
  \tfrac{1}{( (k+1)h - s )^\vartheta}
  \, ds
  \right]^{1/2}
\\ & \leq
  \left[
    \SGchi{0}
    +
    \tfrac{
    \SGchi{\vartheta} \,
    \|F\|_{\operatorname{Lip}^0(V, V_{-\vartheta})} \,
    |(k+1)h|^{ ( 1 - \vartheta ) }
    }{( 1 - \vartheta )}
  +
  \tfrac{
    \BDG{p}{} \,
    \SGchi{\nicefrac{\vartheta}{2}} \,
    |(k+1)h|^{ \nicefrac{( 1 - \vartheta )}{2} } \,
    \|B\|_{\operatorname{Lip}^0(V, \gamma( U, V_{-\nicefrac{\vartheta}{2}} ))}
  }{\sqrt{1 - \vartheta}}
  \right]
\\ & \cdot
  \max_{ j \in \{ 0, 1, \ldots, k \} }
  \max\{
    1
    ,
    \| Y_{ jh } \|_{ \lpn{p}{\P}{V} }
  \}
  < \infty
  .
\end{split}
\end{equation}
This proves~\eqref{eq:step.estimate} in the case $k+1$.
Induction hence proves~\eqref{eq:step.estimate}.

In the next step we observe that~\eqref{eq:step.estimate} shows that 
\begin{equation}
\label{eq:apriori.discrete.time}
  \sup_{ t \in [ 0, T ] }
  \| Y_{\floor{t}{h}} \|_{ \lpn{p}{\P}{V} }
  =
  \max_{ k \in \{ 0, 1, \ldots, \nicefrac{\floor{T}{h}}{h} \} }
  \| Y_{kh} \|_{ \lpn{p}{\P}{V} }
  < \infty
  . 
\end{equation}
Proposition~\ref{prop:numerics_Lp_bound} 
hence yields\footnote{with 
$ \kappa = 0 $, 
$ L_{0,t} = e^{tA} $, 
$ L_{s,t} = e^{(t-\floor{s}{h})A} $, 
$ \Pi(s) = \floor{s}{h} $
for $ (s,t) \in ( \angle \cap (0,T]^2 ) $ 
in the notation of Proposition~\ref{prop:numerics_Lp_bound}} 
that
\begin{equation}
\label{eq:numerics_bound}
\begin{split}
&
  \sup_{t \in [ 0, T ]}
  \| Y_t \|_{ \lpn{p}{\P}{V} }
  \leq
  \sqrt{2}
\\ & \cdot
  \left[
    \SGchi{0} \,
    \| Y_0 \|_{ \lpn{p}{\P}{V} }
    +
    \tfrac{
      \SGchi{\vartheta} \,
      T^{ ( 1 - \vartheta ) }
      \| F(0) \|_{ V_{ -\vartheta } }
    }{
      ( 1 - \vartheta )
    }
      +
      \BDG{p}{} \,
      \SGchi{ \nicefrac{ \vartheta }{ 2 } }
      \sqrt{
      \tfrac{
      T^{ ( 1 - \vartheta ) }
      }{
      ( 1 - \vartheta )
      }
      }
      \| B(0) \|_{ \gamma( U, V_{ -\nicefrac{\vartheta}{2} } ) }
  \right]
\\ & \cdot
  \mathcal{E}_{ ( 1 - \vartheta ) }\!\left[
    \tfrac{
      \sqrt{ 2 }
      \,
      \SGchi{ \vartheta }
      \,
      T^{ ( 1 - \vartheta ) }
      \,
      |
        F
      |_{
        \operatorname{Lip}^0( V, V_{ - \vartheta } )
      }
    }{
      \sqrt{1 - \vartheta}
    }
    +
    \BDG{p}{} \,
    \SGchi{ 
      \nicefrac{\vartheta}{2}
    }
    \sqrt{
      2 T^{ ( 1 - \vartheta ) }
    } \,
    |
      B
    |_{
      \operatorname{Lip}^0( V, \gamma( U, V_{ - \nicefrac{\vartheta}{2} } ) )
    }
  \right]
  .
\end{split}
\end{equation}
Furthermore, note that~\eqref{eq:apriori.discrete.time} ensures that 
\begin{equation}
\label{eq:integrated_numerics_bound}
\begin{split}
&
  \sup_{ t \in [ 0, T ] }
  \| \bar{Y}_t \|_{\lpn{p}{\P}{V}}
  \leq
  \sup_{ t \in [0,T] }
  \big\|
  \max\{
    1
    ,
    \| Y_t \|_V
  \}
  \big\|_{ \lpn{p}{\P}{\R} }
\\ & \cdot
  \left[
    \SGchi{0}
    +
    \tfrac{
    \SGchi{\vartheta} \,
    \|F\|_{\operatorname{Lip}^0(V, V_{-\vartheta})} \,
    T^{ ( 1 - \vartheta ) }
    }{( 1 - \vartheta )}
  +
  \tfrac{
    \BDG{p}{} \,
    \SGchi{\nicefrac{\vartheta}{2}} \,
    \sqrt{T^{ ( 1 - \vartheta ) }} \,
    \|B\|_{\operatorname{Lip}^0(V, \gamma( U, V_{-\nicefrac{\vartheta}{2}} ))}
  }{\sqrt{1 - \vartheta}}
  \right]
  .
\end{split}
\end{equation}
Moreover, we observe that for all $ s, t \in [ 0, T ] $ 
it holds that 
\begin{equation}
\begin{split}
  \ES\big[\!
    \max\{
    1,
    \| \bar{Y}_s \|^p_V,
    \| Y_t \|^p_V
    \}
  \big]
& \leq
  \ES\big[
  \| \bar{Y}_s \|^p_V
  \big]
  +
  \ES\big[\!
  \max\{
    1
    ,
    \| Y_t \|^p_V
  \}
  \big]
\\ & \leq 
  \sup_{ u \in [0,T] }
  \| \bar{Y}_u \|^p_{ \lpn{p}{\P}{V} }
  +
  \sup_{ u \in [0,T] }
  \big\|\!
  \max\{
    1
    ,
    \| Y_u \|_V
  \}
  \big\|^p_{ \lpn{p}{\P}{\R} }
  .
\end{split}
\end{equation}
This together with \eqref{eq:numerics_bound} and \eqref{eq:integrated_numerics_bound} 
proves the first inequality in~\eqref{eq:Kp_estimate}. 
The proof of Lemma~\ref{lem:Kp_estimate} is thus completed.
\end{proof}

\begin{proposition}
\label{prop:weak_temporal_regularity_1st}
Assume the setting in Section~\ref{sec:setting_weak_temporal_regularity} and let 
$ \psiC \in [ 0, \infty ) $, 
$ \power \in [ 0, \infty ) \cap ( -\infty, p - 3 ] $, 
$ \rho \in [ 0, 1 - \vartheta ) $, 
$ \psi = ( \psi(x,y) )_{ x, y \in V } \in C^2( V \times V, \mathcal{V} ) $ 
satisfy for all 
$ x_1, x_2, y \in V $, 
$ i, j \in \{ 0, 1, 2 \} $
with 
$ i + j \leq 2 $
that 
\begin{equation}
\begin{split}
  \big\|
    \big(
     \tfrac{ \partial^{ (i + j) } }{ \partial x^i \partial y^j  }
     \psi
    \big)
    ( x_1, y )
  -
    \big(
     \tfrac{ \partial^{ (i + j) } }{ \partial x^i \partial y^j  }
     \psi
    \big)
    ( x_2, y )
  \big\|_{ L^{ (i + j) }( V, \mathcal{V} ) }
& \leq
  \psiC
  \max\{ 1, \| x_1 \|^\power_V, \| x_2 \|^\power_V, \| y \|^\power_V \}
  \left\| 
    x_1 - x_2 
  \right\|_V
  .
\end{split}
\end{equation}
Then it holds for all 
$ ( s, t ) \in \angle $ 
that 
$
  \ES\big[
    \| 
      \psi( \bar{Y}_t, Y_s )
      -
      \psi( \bar{Y}_s, Y_s ) 
    \|_{\mathcal{V}}
  \big]
  < \infty
$ 
and 
\begin{equation}
\label{eq:prop_temporal_reg}
\begin{split}
&
  \left\|
  \E\left[
    \psi(
      \bar{Y}_{ t } ,
      Y_{ s }
    )
  -
    \psi(
      \bar{Y}_{ s } ,
      Y_{ s }
    )
  \right]
  \right\|_{\mathcal{V}}
\leq
  \psiC \,
  | \SGchi{0} |^{ ( \power + 1 ) } \, 
  | \SGchi{\rho} |^2 \,
  \varsigma_{ F, B } \,
  K_{ \power + 3 } 
  \,
  ( t - s )^\rho
\\ & \cdot
  \bigg[
    \tfrac{
      2^{ \rho }
    }{
      t^{ \rho }
    }
    +
    \tfrac{
      \left(
        2 \, 
        \SGchi{\vartheta} 
        +
        \SGchi{ \rho + \vartheta }
        +
        2 \, 
        | \SGchi{ \vartheta / 2 } |^2
        + 
        2 \,
        \SGchi{ \rho + \vartheta / 2 }
        \,
        \SGchi{ \vartheta / 2 }
      \right)
      \,
      s^{ ( 1 - \vartheta - \rho ) }
      +
      \left( 
        \SGchi{\vartheta} 
        +
        \frac{ 1 }{ 2 }
        | \SGchi{ \nicefrac{ \vartheta }{ 2 } } |^2 
      \right)
      \,
      \left| t - s \right|^{ ( 1 - \vartheta - \rho ) }
    }{ 
      ( 1 - \vartheta - \rho ) }
  \bigg]
  .
\end{split}
\end{equation}
\end{proposition}

\begin{proof}
Throughout this proof 
let 
$
  ( g_r )_{ r \in [ 0, \infty ) }
  \subseteq
  C( V, \R )
$ 
be the functions
which satisfy 
for all $ r \in [0,\infty) $, $ x \in V $
that
$
  g_r( x )
  =
  \max\{
  1, \| x \|^r_V
  \}
$ 
and let 
$
  \psi_{1,0} \colon V \times V \to L( V, \mathcal{V} ) 
$, 
$
  \psi_{0,1} \colon V \times V \to L( V, \mathcal{V} ) 
$, 
$
  \psi_{2,0} \colon V \times V \to L^{(2)}( V, \mathcal{V} ) 
$, 
$
  \psi_{0,2} \colon V \times V \to L^{(2)}( V, \mathcal{V} ) 
$, 
$
  \psi_{1,1} \colon V \times V \to L^{(2)}( V, \mathcal{V} ) 
$
be the functions which satisfy
for all 
$ x, y, v_1, v_2 \in V $ 
that 
\begin{equation}
\begin{split}
  \psi_{1,0}( x, y ) \, v_1
  =
  \big(\tfrac{ \partial }{ \partial x } \psi\big)( x, y ) \, v_1
  ,&
  \qquad
  \psi_{0,1}( x, y ) \, v_1
  =
  \big(\tfrac{ \partial }{ \partial y } \psi\big)( x, y ) \, v_1
  ,
\\
  \psi_{2,0}( x, y )( v_1, v_2 )
  =
  \big(
  \tfrac{ \partial^2 }{ \partial x^2 } \psi\big)( x, y )
  ( v_1, v_2 )
  ,&
  \qquad
  \psi_{0,2}( x, y )( v_1, v_2 )
  =
  \big(
  \tfrac{ \partial^2 }{ \partial y^2 } \psi\big)( x, y )
  ( v_1, v_2 )
  ,
\\
  \psi_{ 1, 1 }( x, y )( v_1, v_2 )
  =
  \big(
  \tfrac{ \partial }{ \partial y } 
  \tfrac{ \partial }{ \partial x } 
  \psi\big)( x, y )
  ( v_1, v_2 )
  .&
\end{split}
\end{equation}
Next we observe that Lemma~\ref{lem:Kp_estimate}
and the assumption that $ q \leq p - 3 $ ensure that $ K_{ q + 1 } \leq K_{ q + 3 } < \infty $.
Combining this with the fact that
\begin{equation}
  \forall \, x_1, x_2, y \in V \colon
  \| \psi( x_1, y ) - \psi( x_2, y ) \|_{\mathcal{V}}
  \leq 
  2 \, \psiC
  \max\!\big\{ 
    1, \| x_1 \|^{ q + 1 }_V , \| x_2 \|^{ q + 1 }_V , \| y \|^{ q + 1 }_V 
  \big\}
\end{equation}
shows that 
for all $ ( s, t ) \in \angle $ 
it holds that
$
  \ES\big[
    \| 
      \psi( \bar{Y}_t, Y_s )
      -
      \psi( \bar{Y}_s, Y_s ) 
    \|_{\mathcal{V}}
  \big]
  < \infty
$.
It thus remains to prove \eqref{eq:prop_temporal_reg}.
To do so, we apply the mild It\^{o} formula 
in~\cite{CoxJentzenKurniawanPusnik2016}.
More formally,
an application of 
Proposition~3.11 in~\cite{CoxJentzenKurniawanPusnik2016}
shows
that for all 
$ ( s, t ) \in \angle $ 
it holds that 
$
  \ES\big[
    \| 
      \psi( e^{ ( t - s )A } \, \bar{Y}_s, Y_s )
      -
      \psi( \bar{Y}_s, Y_s ) 
    \|_{\mathcal{V}}
  \big]
  < \infty
$ 
and 
\begin{equation}
\label{eq:mild_ito_outer_1st}
\begin{split}
&
  \left\|
  \E\left[
    \psi(
      \bar{Y}_{ t } ,
      Y_{ s }
    )
  -
    \psi(
      \bar{Y}_{ s } ,
      Y_{ s }
    )
  \right]
  \right\|_{\mathcal{V}}
  \leq
  \left\|\E\left[
    \psi(
      e^{ ( t - s )A } \, \bar{Y}_{ s } ,
      Y_{ s }
    \big)
    -
    \psi(
      \bar{Y}_{ s } ,
      Y_{ s }
    )
  \right]\right\|_{\mathcal{V}}
\\ & +
  \int^{ t }_{ s }
  \E\left[
  \left\|
    \psi_{1,0}(
      e^{ ( t - r )A } \, \bar{Y}_r,
      Y_{ s }
    )
    \,
    e^{ ( t - r )A } 
    F(
      Y_{ \floor{ r }{ h } }
    )
  \right\|_{\mathcal{V}}
  \right]
  dr
\\ & +
  \int^{ t }_{ s }
  \E\left[\left\|
   \tfrac{ 1 }{ 2 }
  {\smallsum\limits_{ b \in \mathbb{ U } }}
    \psi_{2,0}(
      e^{ ( t - r )A } \, \bar{Y}_r ,
      Y_{ s }
    )\big(
      e^{ ( t - r ) A } \,
      B^b( Y_{ \floor{ r }{ h } } )
    ,
      e^{ ( t - r ) A } \,
      B^b( Y_{ \floor{ r }{ h } } )
  \big)
  \right\|_{\mathcal{V}}\right]
  dr .
\end{split}
\end{equation}
In the following we establish suitable estimates 
for the three summands appearing on the right hand side of \eqref{eq:mild_ito_outer_1st}.
Combining these estimates with \eqref{eq:mild_ito_outer_1st} will then
allow us to establish \eqref{eq:prop_temporal_reg}.
We begin with the second and the third summands on 
the right hand side of \eqref{eq:mild_ito_outer_1st}.
We note that the assumption that
\begin{equation}
  \forall \, x_1, x_2 , y \in V \colon
  \left\|
    \psi( x_1 , y )
    -
    \psi( x_2, y )
  \right\|_{\mathcal{V}}
  \leq
  \psiC
  \max\{
    1,
    \| x_1 \|^\power_V
    ,
    \| x_2 \|^\power_V
    ,
    \| y \|^\power_V
  \}
  \,
  \| x_1 - x_2 \|_V
\end{equation}
implies that
$
  \forall \, x , y \in V \colon
  \|
    \psi_{ 1, 0 }( x, y )
  \|_{ L( V, \mathcal{V} ) }
  \leq
  \psiC
  \max\{
    1,
    \| x \|^\power_V,
    \| y \|^\power_V
  \}
$.
This, in turn, proves that for all 
$ ( r, t ) \in \angle $, 
$ u, v, w \in V $ 
it holds that 
\begin{equation}
\begin{split}
\label{eq:drift_1st_mild_ito}
&
  \left\|
    \psi_{ 1 , 0 }(
      e^{ ( t - r ) A } \, u ,
      v
    )
    \, e^{ ( t - r ) A } \,
    F( w )
  \right\|_{\mathcal{V}}
\\ & \leq
  \psiC
  \left| \SGchi{0} \right|^{ \power }
  \max\!\left\{ 
    1 ,
    \| u \|_V^{ \power }
    ,
    \| v \|_V^{ \power }
  \right\}
  \| 
    e^{ ( t - r ) A }
  \|_{
    L( V, V_{ \vartheta } )
  }
  \left\| F( w ) \right\|_{ V_{ - \vartheta } }
\\ & \leq
  \frac{
    \psiC \,
    | \SGchi{0} |^\power \, 
    \SGchi{\vartheta}
    \max\{
      1, \| u \|^\power_V, \| v \|^\power_V
    \}
    \| F \|_{ \operatorname{Lip}^0( V, V_{ -\vartheta } ) }   
    \,
    g_1( w )
  }{
    ( t - r )^\vartheta
  }
  .
\end{split}
\end{equation}
Next we observe that the assumption that
\begin{equation}
  \forall \, x_1, x_2 , y \in V \colon
  \left\|
    \psi_{ 1,0 }( x_1 , y )
    -
    \psi_{ 1,0 }( x_2 , y )
  \right\|_{ 
    L( V, \mathcal{V} ) 
  }
  \leq
  \psiC
  \max\{
    1,
    \| x_1 \|^\power_V
    ,
    \| x_2 \|^\power_V
    ,
    \| y \|^\power_V
  \}
  \| x_1 - x_2 \|_V
\end{equation}
shows that
$
  \forall \, x , y \in V \colon
  \|
    \psi_{ 2, 0 }( x, y )
  \|_{ L^{ (2) }( V, \mathcal{V} ) }
  \leq
  \psiC
  \max\{
    1,
    \| x \|^\power_V,
    \| y \|^\power_V
  \}
$.
This and Lemma~\ref{lem:gamma.estimate} prove that for all 
$ ( r, t ) \in \angle $, 
$ u, v, w \in V $ 
it holds that 
\begin{equation}
\begin{split}
\label{eq:diffusion_1st_mild_ito}
&
  \tfrac{ 1 }{ 2 }
  \Big\|
  \smallsum\limits_{ b \in \mathbb{ U } }
  \displaystyle
    \psi_{2,0}\big(
     e^{ ( t - r ) A } \, u,
    v
    \big)
  \big(
    e^{ ( t - r ) A } \,
    B^b( w )
    ,
    e^{ ( t - r ) A } \,
    B^b( w )
  \big)
  \Big\|_{\mathcal{V}}
\\ & \leq 
    \tfrac{ 1 }{ 2 } \,
  \psiC \,
  | \SGchi{0} |^{ \power }
  \,
  \max\{
    1, \| u \|^\power_V, \| v \|^\power_V
  \}
  \,
  \|
    e^{ ( t - r ) A }
    B( w )
  \|^2_{
    \gamma(U,V)
  }
\\ & \leq
  \frac{
    \psiC \,
    | \SGchi{0} |^\power
    \,
    | \SGchi{ \nicefrac{ \vartheta }{ 2 } } |^2
    \max\{
      1, \| u \|^\power_V, \| v \|^\power_V
    \}
    \,
    \| B \|^2_{ \operatorname{Lip}^0( V, \gamma( U, V_{ - \vartheta / 2 } ) ) } \,
    g_2( w )
  }{
    2( t - r )^\vartheta
  }
  .
\end{split}
\end{equation}
Furthermore, we note
that H\"{o}lder's inequality implies that 
for all 
$ r, l \in ( 0, \infty ) $, 
$ s, t \in [ 0, T ] $ 
it holds that 
\begin{equation}
\begin{split}
&
  \ES\!\left[
    \max\!\left\{
      1,
      \| \bar{Y}_s \|^r_V,
      \| Y_t \|^r_V
    \right\}
    g_l( Y_{ \floor{ s }{ h } } )
  \right]
\\ & \leq
  \left(
    \sup_{ u, v \in [ 0, T ] }
    \left\|
    \max\!\left\{
      1,
      \| \bar{Y}_u \|^r_V,
      \| Y_v \|^r_V
    \right\}
    \right\|_{ \lpn{ 1 + \nicefrac{l}{r} }{ \P }{ \R } }
  \right)
  \left(
    \sup_{ u \in [ 0, T ] }
    \left\|
    \max\!\left\{
      1,
      \| Y_u \|^l_V
    \right\}
    \right\|_{ \lpn{ 1 + \nicefrac{r}{l} }{ \P }{ \R } }
  \right)
\\ & \leq
  | K_{ r + l } |^{ \frac{ 1 }{ 1 + \nicefrac{l}{r} } }
  \,
  | K_{ r + l } |^{ \frac{ 1 }{ 1 + \nicefrac{r}{l} } }
  =
  K_{ r + l }
  .
\end{split}
\end{equation}
This and the fact that for all 
$ l \in [0,\infty) $
it holds that 
$
  \sup_{ s \in [0,T] }
  \ES\big[
    g_l( Y_{ \floor{ s }{h} } )
  \big]
  \leq
  K_l
$ 
prove that
for all 
$ r, l \in [ 0, \infty ) $, 
$ s, t \in [ 0, T ] $ 
it holds that 
\begin{equation}
\label{eq:optimal_moment_1}
\begin{split}
&
  \ES\!\left[
    \max\!\left\{
      1,
      \| \bar{Y}_s \|^r_V,
      \| Y_t \|^r_V
    \right\}
    g_l( Y_{ \floor{ s }{ h } } )
  \right]
\leq
  K_{ r + l }
  .
\end{split}
\end{equation}
Combining \eqref{eq:drift_1st_mild_ito}, \eqref{eq:diffusion_1st_mild_ito}, 
and \eqref{eq:optimal_moment_1} 
implies that for all 
$ ( s, t ) \in \angle $ 
it holds that 
\begin{equation}
\label{eq:outer_summands}
\begin{split}
&
  \int^{ t }_{ s }
  \E\left[\left\|
    \psi_{1,0}(
      e^{ ( t - r )A } \, \bar{Y}_r,
      Y_{ s }
    )
    \,
    e^{ ( t - r )A } 
    F(
      Y_{ \floor{ r }{ h } }
    )
  \right\|_{\mathcal{V}}\right]
  dr
\\ & +
  {\displaystyle
  \int^{ t }_{ s }}
  \E\left[\left\|
   \tfrac{ 1 }{ 2 }
  {\smallsum\limits_{ b \in \mathbb{ U } }}
    \psi_{2,0}(
      e^{ ( t - r )A } \, \bar{Y}_r ,
      Y_{ s }
    )\big(
      e^{ ( t - r ) A } \,
      B^b( Y_{ \floor{ r }{ h } } )
    ,
      e^{ ( t - r ) A } \,
      B^b( Y_{ \floor{ r }{ h } } )
  \big)
  \right\|_{\mathcal{V}}\right]
  dr
\\ & \leq
    \psiC \,
    | \SGchi{0} |^\power
    \left(
      \SGchi{\vartheta}
      \,
      \| F \|_{
        \operatorname{Lip}^0( V, V_{ - \vartheta } )
      }
      +
      \tfrac{ 1 }{ 2 }
      \,
      | \SGchi{ \nicefrac{ \vartheta }{ 2 } } |^2 
      \,
      \| B \|^2_{ \operatorname{Lip}^0( V, \gamma( U, V_{ - \vartheta / 2 } ) ) } 
    \right)
    K_{ \power + 2 }
    \int_s^t
    \frac{
      1
    }{
      ( t - r )^{ \vartheta }
    }
    \,
    dr
\\ & \leq
  \frac{
    \psiC \,
    | \SGchi{0} |^\power
    \left(
      \SGchi{\vartheta}
      +
      \frac{ 1 }{ 2 }
      | \SGchi{ \nicefrac{ \vartheta }{ 2 } } |^2 
    \right)
    \varsigma_{ F, B } 
    \,
    K_{ \power + 2 }
    \left( t - s \right)^{ ( 1 - \vartheta ) }
  }{
    ( 1 - \vartheta )
  }
  .
\end{split}
\end{equation}
Inequality~\eqref{eq:outer_summands} provides us an
appropriate estimate for the second and the third summand on 
the right hand side of \eqref{eq:mild_ito_outer_1st}.
It thus remains to provide a suitable estimate 
for the first summand on the right hand side of 
\eqref{eq:mild_ito_outer_1st}.
For this we will employ 
the mild It\^{o} formula 
in~\cite{CoxJentzenKurniawanPusnik2016} again
and this will allow us to obtain an appropriate upper bound for  
$
  \big\|
  \ES\big[
    \psi(
      e^{ ( t - s ) A } \, \bar{Y}_s ,
      Y_s
    )
    -
    \psi(
      \bar{Y}_{ s } ,
      Y_{ s }
    )
  \big]
  \big\|_{\mathcal{V}}
$ 
for 
$ ( s, t ) \in \angle $.
More formally, let 
$ \tilde{F}_{ r, s, t } \colon V \times V \times V \to \mathcal{V} $, 
$ r \in [ 0, s ) $, 
$ s \in ( 0, t ) $, 
$ t \in ( 0, T ] $, 
be the functions which satisfy 
for all 
$ t \in ( 0, T ] $, 
$ s \in ( 0, t ) $, 
$ r \in [ 0, s ) $, 
$ u,v, w \in V $ 
that 
\begin{equation}
\begin{split}
\label{eq:tilde_F_def_1st}
  \tilde{F}_{ r,s,t }(u,v,w)
& =
  \psi_{1,0}\big(
    e^{ ( t - r )A } \, u,
    e^{(s-r)A} \, v
  \big) \,
  e^{ ( t - r )A } \, F( w )
  -
  \psi_{1,0}\big(
    e^{ ( s - r )A } \, u,
    e^{(s-r)A} \, v
  \big) \,
  e^{ ( s - r )A } \, F( w )
\\ & 
  +
  \left[
    \psi_{0,1}\big(
      e^{ ( t - r )A } \, u,
      e^{(s-r)A} \, v
    \big)
    -
    \psi_{0,1}\big(
      e^{ ( s - r )A } \, u,
      e^{(s-r)A} \, v
    \big)
  \right]
  e^{(s-\floor{r}{h})A} \, F( w )
\end{split}
\end{equation} 
and let 
$ \tilde{B}_{ r, s, t } \colon V \times V \times V \to \mathcal{V} $, 
$ r \in [ 0, s ) $, 
$ s \in ( 0, t ) $, 
$ t \in ( 0, T ] $, 
be the functions which satisfy
for all 
$ t \in ( 0, T ] $, 
$ s \in ( 0, t ) $, 
$ r \in [ 0, s ) $, 
$ u,v, w \in V $ 
that 
\begin{equation}
\begin{split}
&
  \tilde{B}_{r,s,t}( u, v, w )
=
  \tfrac{1}{2}
  \smallsum\limits_{ b \in \mathbb{U} }
  \psi_{2,0}\big(
     e^{ ( t - r )A } \, u,
     e^{(s-r)A} \, v
  \big)
  \big(
  e^{ ( t - r )A } \, B^b( w ),
  e^{ ( t - r )A } \, B^b( w )
  \big)
\\ &
-
  \tfrac{1}{2}
  \smallsum\limits_{ b \in \mathbb{U} }
  \psi_{2,0}\big(
     e^{ ( s - r )A } \, u,
     e^{(s-r)A} \, v
  \big)
  \big(
  e^{ ( s - r )A } \, B^b( w ),
  e^{ ( s - r )A } \, B^b( w )
  \big)
\\ & 
  +
  \tfrac{1}{2}
  \smallsum\limits_{ b \in \mathbb{U} }
  \left[
    \psi_{0,2}\big(
      e^{ ( t - r )A } \, u,
      e^{(s-r)A} \, v
    \big)
    -
    \psi_{0,2}\big(
      e^{ ( s - r )A } \, u,
      e^{(s-r)A} \, v
    \big)
  \right]\!\big(
    e^{(s-\floor{r}{h})A} \, B^b( w )
    ,
    e^{(s-\floor{r}{h})A} \, B^b( w )
  \big)
\\ & +
  \smallsum\limits_{ b \in \mathbb{U} }
  \psi_{1,1}\big(
     e^{ ( t - r )A } \, u,
     e^{(s-r)A} \, v
  \big)
  \big(
  e^{ ( t - r )A } \, B^b( w ),
  e^{(s-\floor{r}{h})A} \, B^b( w )
  \big)
\\ & 
-
  \smallsum\limits_{ b \in \mathbb{U} }
  \psi_{1,1}\big(
     e^{ ( s - r )A } \, u,
     e^{(s-r)A} \, v
  \big)
  \big(
  e^{ ( s - r )A } \, B^b( w ),
  e^{(s-\floor{r}{h})A} \, B^b( w )
  \big)
  .
\end{split}
\end{equation}
An application of
Proposition~3.11
in~\cite{CoxJentzenKurniawanPusnik2016}
shows
that for all 
$ t \in (0,T] $, $ s \in (0,t) $
it holds that 
\begin{equation}
\label{eq:mild_ito_inner_1st}
\begin{split}
&
  \left\|\E\left[
    \psi(
      e^{ ( t - s )A } \, \bar{Y}_{ s } ,
      Y_{ s }
    )
    -
    \psi(
      \bar{Y}_{ s } ,
      Y_{ s }
    )
  \right]\right\|_{\mathcal{V}}
\leq
  \E\left[\left\|
    \psi\big(
      e^{ t A } \, Y_0
      ,
      e^{ s A } \, Y_0
    \big)
    -
    \psi\big(
      e^{ s A } \, Y_0
      ,
      e^{ s A } \, Y_0
    \big)
  \right\|_{\mathcal{V}}\right]
\\ & +
  \int^{ s }_0
  \ES\big[\|
    \tilde{F}_{r,s,t}\big(
      \bar{Y}_r , 
      Y_r , 
      Y_{ \floor{ r }{ h } }
    \big)
  \|_{\mathcal{V}}\big]
  \,
  dr
  +
  \int^{ s }_0
  \ES\big[\|
    \tilde{B}_{r,s,t}\big(
      \bar{Y}_r , 
      Y_r , 
      Y_{ \floor{ r }{ h } }
    \big)
  \|_{\mathcal{V}}\big]
  \,
  dr
  .
\end{split}
\end{equation}
In the next step we estimate the summands on the right hand side of~\eqref{eq:mild_ito_inner_1st}. 
We observe that for all 
$ t \in (0,T] $, $ s \in (0,t) $
it holds that 
\begin{equation}
\begin{split}
&
  \left\|
    \psi\big(
    e^{ t A } \, Y_0,
    e^{ s A } \, Y_0
    \big)
    -
    \psi\big(
    e^{ s A } \, Y_0,
    e^{ s A } \, Y_0
    \big)
  \right\|_{\mathcal{V}}
\\ & \leq 
  \psiC
  \max\!\left\{ 
    1 ,
    \| e^{ t A } \, Y_0 \|_V^{ \power }
    ,
    \| e^{ s A } \, Y_0 \|_V^{ \power }
  \right\}
  \| e^{ t A } \, Y_0 - e^{ s A } \, Y_0 \|_V
\\ & \leq
  \psiC \,
  |\SGchi{0}|^{ \power }
  \,
  g_\power( Y_0 ) \,
  \| e^{ t A } - e^{ s A } \|_{ L( V ) }
  \,
  \| Y_0 \|_V
\leq
  \psiC
  \,
  | \SGchi{0} |^{ \power }
  \,
  g_{ \power + 1 }( Y_0 ) \,
  \tfrac{
    | \SGchi{\rho} |^2 \,
    ( t - s )^\rho
  }{
    s^{ \rho }
  }
  .
\end{split}
\end{equation}
This and the fact that 
$
  \ES\big[ 
    g_{ \power + 1 }( Y_0 )
  \big]
  \leq
  K_{ \power + 1 }
$ 
imply that for all 
$ t \in (0,T] $, $ s \in (0,t) $
it holds that 
\begin{equation}
\label{eq:1st_summand_inner}
\begin{split}
&
  \E\left[
  \|
    \psi\big(
      e^{ t A } \, Y_0
      ,
      e^{ s A } \, Y_0
    \big)
    -
    \psi\big(
      e^{ s A } \, Y_0
      ,
      e^{ s A } \, Y_0
    \big)
  \|_{\mathcal{V}}
  \right]
\leq
  \psiC
  \,
  | \SGchi{0} |^{ \power }
  \,
  K_{ \power + 1 } \,
  \tfrac{
    |\SGchi{\rho}|^2
  }{
    s^{ \rho }
  }
  \left( t - s \right)^{ \rho }
  .
\end{split}
\end{equation}
Inequality \eqref{eq:1st_summand_inner} provides us an appropriate estimate for 
the first summand on the right hand side of~\eqref{eq:mild_ito_inner_1st}. 
In the next step we establish a suitable bound for the second summand on the right hand side of~\eqref{eq:mild_ito_inner_1st}. 
Note that for all 
$ t \in ( 0, T ] $, 
$ s \in ( 0, t ) $, 
$ r \in [ 0, s ) $, 
$ u,v, w \in V $ 
it holds that 
\begin{equation}
\label{eq:mild_ito_2nd_F_A}
\begin{split}
&
  \left\|
    \psi_{1,0}\big(
     e^{ ( t - r )A } \, u,
     e^{ ( s - r )A } \, v
    \big) 
    \,
    e^{ ( t - r )A } \, F( w )
    -
    \psi_{1,0}\big(
     e^{ ( s - r )A } \, u,
     e^{ ( s - r )A } \, v
    \big) 
    \,
    e^{ ( s - r )A } \, F( w )
  \right\|_{\mathcal{V}}
\\ & \leq
  \left\|
    \big[
      \psi_{1,0}\big(
         e^{ ( t - r )A } \, u,
         e^{ ( s - r )A } \, v
      \big)
      -
      \psi_{1,0}\big(
        e^{ ( s - r )A } \, u,
        e^{ ( s - r )A } \, v
      \big)
    \big]
    \,
    e^{ ( t - r )A } \, F( w )
  \right\|_{\mathcal{V}}
\\ & +
  \left\|
    \psi_{1,0}\big(
     e^{ ( s - r )A } \, u,
     e^{ ( s - r )A } \, v
    \big) \,
    e^{ ( s - r )A } 
    \left(
      e^{ ( t - s )A } - \operatorname{Id}_V 
    \right)
    F( w )
  \right\|_{\mathcal{V}}
\\ & \leq
  \psiC
  \,
  \max\!\big\{
    1
    , 
    \| e^{ ( t - r ) A } u \|^\power_V
    , 
    \| e^{ ( s - r ) A } u \|^\power_V
    , 
    \| e^{ ( s - r )A } v \|^\power_V
  \big\}
    \left\| 
      e^{ ( s - r ) A }
      \left(
        e^{ ( t - s ) A }
        -
        \operatorname{Id}_V
      \right)
      u
    \right\|_V
\\&\cdot
  \| 
    e^{ ( t - r ) A } F( w )
  \|_{
    V
  }
+
    \psiC 
    \,
    \max\!\big\{
      1, \| e^{ ( s - r ) A } u \|^\power_V, \| e^{ ( s - r )A } v \|^\power_V
    \big\}
    \,
    \left\| 
      e^{ ( s - r ) A }
      \left(
        e^{ ( t - s ) A }
        -
        \operatorname{Id}_V
      \right)
      F( w )
    \right\|_V
\\ & \leq
  \frac{
    \psiC
    \,
    | \SGchi{0} |^\power
    \max\!\big\{
      1, \| u \|^\power_V, \| v \|^\power_V
    \big\}
    \,
    | \SGchi{\rho} |^2
    \,
    ( t - s )^\rho
    \,
    \| u \|_V 
    \,
    \SGchi{\vartheta}
    \,
    \| F \|_{ \operatorname{Lip}^0( V, V_{ -\vartheta } ) }
    \, 
    g_1( w )
  }{
    ( s - r )^{ \rho }
    \,
    ( t - r )^{ \vartheta }
  }
\\ & +
  \frac{
    \psiC \,
    | \SGchi{0} |^\power
    \max\!\big\{
      1, \| u \|^\power_V, \| v \|^\power_V
    \big\}
    \,
    \SGchi{ \rho + \vartheta } 
    \,
    \SGchi{\rho} 
    \,
    ( t - s )^\rho
    \,
    \| F \|_{ \operatorname{Lip}^0( V, V_{ -\vartheta } ) } 
    \,
    g_1( w )
  }{
    ( s - r )^{ ( \rho + \vartheta ) }
  }
  ,
\end{split}
\end{equation}
\begin{equation}
\label{eq:mild_ito_2nd_F_B}
\begin{split}
&
  \left\|
    \left[
      \psi_{0,1}\big(
        e^{ ( t - r )A } \, u
        ,
        e^{ ( s - r )A } \, v
      \big)
      -
      \psi_{0,1}\big(
        e^{ ( s - r )A } \, u
        ,
        e^{ ( s - r )A } \, v
      \big)
    \right]
    e^{ ( s - \floor{r}{h} )A } \, F( w )
  \right\|_{\mathcal{V}}
\\ & \leq
    \psiC
    \,
    | \SGchi{0} |^\power
    \max\!\big\{
      1, \| u \|^\power_V, \| v \|^\power_V
    \big\}
    \,
    \left\| 
      e^{ ( s - r ) A }
      \left(
        e^{ ( t - s ) A }
        -
        \operatorname{Id}_V
      \right)
      u
    \right\|_V
    \,
    \| e^{ ( s - \floor{r}{h} )A } \|_{ 
      L(
        V_{ - \vartheta }
        ,
        V
      )
    }
    \,
    \| F( w ) \|_{
      V_{ - \vartheta }
    }
\\ & \leq
  \frac{
    \psiC
    \,
    | \SGchi{0} |^\power
    \max\!\big\{
      1, \| u \|^\power_V, \| v \|^\power_V
    \big\}
    \,
    | \SGchi{\rho} |^2
    \,
    ( t - s )^\rho
    \,
    \| u \|_V 
    \,
    \SGchi{\vartheta}
    \,
    \| F \|_{ \operatorname{Lip}^0( V, V_{ -\vartheta } ) } 
    \,
    g_1( w )
  }{
    ( s - r )^{ ( \rho + \vartheta ) }
  }
\\ & \leq
  \frac{
    \psiC
    \,
    | \SGchi{0} |^\power
    \,
    | \SGchi{\rho} |^2
    \,
    \SGchi{\vartheta}
    \max\!\big\{
      1, \| u \|^{ \power + 1 }_V, \| v \|^{ \power + 1 }_V
    \big\}
    \,
    \| F \|_{ \operatorname{Lip}^0( V, V_{ -\vartheta } ) } 
    \,
    g_1( w )
    \,
    ( t - s )^\rho
  }{
    ( s - r )^{ ( \rho + \vartheta ) }
  }
  .
\end{split}
\end{equation}
Inequalities~\eqref{eq:mild_ito_2nd_F_A} and \eqref{eq:mild_ito_2nd_F_B} prove that for all 
$ t \in ( 0, T ] $, 
$ s \in ( 0, t ) $, 
$ r \in [ 0, s ) $, 
$ u,v, w \in V $ 
it holds that 
\begin{equation}
\begin{split}
&
  \| 
    \tilde{F}_{r,s,t}( u, v, w ) 
  \|_{\mathcal{V}}
\\ & \leq
  \psiC
  \,
  | \SGchi{0} |^{ \power }
  \left[
    \frac{
      | \SGchi{\rho} |^2
      \,
      \SGchi{\vartheta}
    }{
      ( s - r )^{ \rho }
      \,
      ( t - r )^{ \vartheta }
    }
    +
    \frac{ 
      \SGchi{ \rho + \vartheta } 
      \,
      \SGchi{\rho} 
    }{
      ( s - r )^{ ( \rho + \vartheta ) }
    }
    +
    \frac{
      | \SGchi{\rho} |^2
      \,
      \SGchi{\vartheta}
    }{
      ( s - r )^{ ( \rho + \vartheta ) }
    }
  \right]
\\ & \quad \cdot
  \max\!\big\{
    1, \| u \|^{\power+1}_V, \| v \|^{\power+1}_V
  \big\}
  \,
  \| F \|_{ \operatorname{Lip}^0( V, V_{ -\vartheta } ) } 
  \,
  g_1( w )
  \,
  ( t - s )^\rho
\\ & \leq
  \psiC
  \,
  | \SGchi{0} |^{ \power }
  \left[
    \frac{
      \SGchi{\rho}
      \,
      (
        2 \, \SGchi{\rho} \, \SGchi{\vartheta}
        +
        \SGchi{ \rho + \vartheta }
      )
    }{
      ( s - r )^{ ( \rho + \vartheta ) }
    } 
  \right]
\\&\quad\cdot
  \max\!\big\{
    1, \| u \|^{\power+1}_V, \| v \|^{\power+1}_V
  \big\}
  \,
  \| F \|_{ \operatorname{Lip}^0( V, V_{ -\vartheta } ) } 
  \,
  g_1( w )
  \,
  ( t - s )^\rho
  .
\end{split}
\end{equation}
This and \eqref{eq:optimal_moment_1} 
prove that for all 
$ t \in (0,T] $, $ s \in (0,t) $
it holds that 
\begin{equation}
\label{eq:F_summand_inner}
\begin{split}
&
  \int^{ s }_0
  \ES\big[\|
    \tilde{F}_{r,s,t}\big(
      \bar{Y}_r , 
      Y_r , 
      Y_{ \floor{ r }{ h } }
    \big)
  \|_{\mathcal{V}}\big]
  \,
  dr
\\ & \leq
  \frac{
    \psiC
    \,
    | \SGchi{0} |^q
    \,
    \SGchi{\rho}
    \,
    \big(
      2 \, \SGchi{\rho} \, \SGchi{\vartheta} 
      +
      \SGchi{ \rho + \vartheta }
    \big)
  }{
    ( 1 - \vartheta - \rho )
  }
  \,
  \| F \|_{ \operatorname{Lip}^0( V, V_{ -\vartheta } ) } 
  \,
  K_{ \power + 2 }
  \,
  ( t - s )^\rho 
  \,
  s^{ ( 1 - \vartheta - \rho ) }
  .
\end{split}
\end{equation}
Next we provide an appropriate bound for the third summand on the right hand side of~\eqref{eq:mild_ito_inner_1st}. 
Observe that Lemma~\ref{lem:gamma.estimate} shows that
for all 
$ t \in ( 0, T ] $, 
$ s \in ( 0, t ) $, 
$ r \in [ 0, s ) $, 
$ u,v, w \in V $ 
it holds that 
\allowdisplaybreaks
\begin{align}
\label{eq:prop_temporal_reg_termB1}
&
\nonumber
  \Big\|
  \smallsum\limits_{ b \in \mathbb{U} }
  \psi_{2,0}\big(
     e^{ ( t - r )A } \, u,
     e^{ ( s - r )A }\, v
  \big)
  \big(
  e^{ ( t - r )A } \, B^b( w ),
  e^{ ( t - r )A } \, B^b( w )
  \big)
\\ &
\nonumber
  -
  \smallsum\limits_{ b \in \mathbb{U} }
  \psi_{2,0}\big(
     e^{ ( s - r )A } \, u,
     e^{ ( s - r )A }\, v
  \big)
  \big(
  e^{ ( s - r )A } \, B^b( w ),
  e^{ ( s - r )A } \, B^b( w )
  \big)
  \Big\|_{\mathcal{V}}
\\ & \leq
\nonumber
  \Big\|
  \smallsum\limits_{ b \in \mathbb{U} }
    \big[
      \psi_{2,0}\big(
         e^{ ( t - r )A } \, u,
         e^{ ( s - r )A }\, v
      \big)
      -
      \psi_{2,0}\big(
        e^{ ( s - r )A } \, u,
        e^{ ( s - r )A }\, v
      \big)
    \big]
    \big(
      e^{ ( t - r )A } \, B^b( w )
      ,
      e^{ ( t - r )A } \, B^b( w )
    \big)
  \Big\|_{\mathcal{V}}
\\ & +
  \Big\|
  \smallsum\limits_{ b \in \mathbb{U} }
    \psi_{2,0}\big(
      e^{ ( s - r )A } \, u
      ,
      e^{ ( s - r )A }\, v
    \big)
    \big(
      (
        e^{ ( t - r )A } + e^{ ( s - r )A }
      )
      \, 
      B^b( w )
      ,
      e^{ ( s - r ) A } 
      \,
      (
        e^{ ( t - s ) A } - \operatorname{Id}_V 
      )
      \,
      B^b( w )
    \big)
  \Big\|_{\mathcal{V}}
\\ & \leq
\nonumber
  \frac{
    \psiC
    \,
    | \SGchi{0} |^\power
    \max\{
      1, \| u \|^\power_V, \| v \|^\power_V
    \}
    \,
    | \SGchi{\rho} |^2
    \,
    ( t - s )^\rho
    \,
    \| u \|_V 
    \,
    | \SGchi{ \nicefrac{ \vartheta }{ 2 } } |^2
    \,
    \| B \|^2_{ \operatorname{Lip}^0( V, \gamma( U, V_{ - \vartheta / 2 } ) ) }
    \, 
    g_2( w )
  }{
    ( s - r )^{ \rho }
    \,
    ( t - r )^{ \vartheta }
  }
\\ & +
\nonumber
  \frac{
    \psiC 
    \,
    | \SGchi{0} |^\power
    \max\{
      1, \| u \|^\power_V, \| v \|^\power_V
    \}
    \,
    2 
    \,
    \SGchi{ \nicefrac{ \vartheta }{ 2 } } 
    \, 
    \SGchi{ \rho + \nicefrac{ \vartheta }{ 2 } } 
    \,
    \SGchi{\rho}
    \,
    ( t - s )^\rho
    \,
    \| B \|^2_{ 
      \operatorname{Lip}^0( 
        V, 
        \gamma( U, V_{ - \vartheta / 2 } ) 
      ) 
    }
    \,
    g_2( w )
  }{
    ( s - r )^{ ( \rho + \vartheta ) }
  }
  ,
\end{align}
\begin{equation}
\begin{split}
&
  \Big\|
    \smallsum\limits_{ b \in \mathbb{U} }
    \left[
      \psi_{0,2}\big(
        e^{ ( t - r )A } \, u,
        e^{ ( s - r )A }\, v
      \big)
      -
      \psi_{0,2}\big(
        e^{ ( s - r )A } \, u,
        e^{ ( s - r )A }\, v
      \big)
    \right]
    \!
    \big(
      e^{ ( s - \floor{r}{h} )A } \, B^b( w )
      ,
      e^{ ( s - \floor{r}{h} )A } \, B^b( w )
    \big)
  \Big\|_{\mathcal{V}}
\\ & \leq
  \frac{
    \psiC
    \,
    | \SGchi{0} |^\power
    \max\!\big\{
      1, \| u \|^\power_V, \| v \|^\power_V
    \big\}
    \,
    | \SGchi{\rho} |^2
    \,
    ( t - s )^\rho
    \,
    \| u \|_V 
    \,
    | \SGchi{ \nicefrac{ \vartheta }{ 2 } } |^2
    \,
    \| B \|^2_{ \operatorname{Lip}^0( V, \gamma( U, V_{ -\nicefrac{ \vartheta }{ 2 } } ) ) }\,
    g_2( w )
  }{
    ( s - r )^{ ( \rho + \vartheta ) }
  }
  ,
\end{split}
\end{equation}
\allowdisplaybreaks
\begin{align}
\label{eq:prop_temporal_reg_termB3}
&
\nonumber
  \Big\|
  \smallsum\limits_{ b \in \mathbb{U} }
  \psi_{1,1}\big(
     e^{ ( t - r )A } \, u,
     e^{ ( s - r )A }\, v
  \big)
  \big(
  e^{ ( t - r )A } \, B^b( w ),
  e^{ ( s - \floor{r}{h} )A } \, B^b( w )
  \big)
\\ & 
\nonumber
-
  \smallsum\limits_{ b \in \mathbb{U} }
  \psi_{1,1}\big(
     e^{ ( s - r )A } \, u,
     e^{ ( s - r )A }\, v
  \big)
  \big(
    e^{ ( s - r )A } \, B^b( w )
    ,
    e^{ ( s - \floor{r}{h} )A } \, B^b( w )
  \big)
  \Big\|_{\mathcal{V}}
\\ & \leq
\nonumber
  \Big\|
  \smallsum\limits_{ b \in \mathbb{U} }
  \big[
    \psi_{1,1}\big(
      e^{ ( t - r )A } \, u
      ,
      e^{ ( s - r )A }\, v
    \big)
    -
    \psi_{1,1}\big(
      e^{ ( s - r )A } \, u
      ,
      e^{ ( s - r )A }\, v
    \big)
  \big]
  \big(
    e^{ ( t - r )A } \, B^b( w ),
    e^{ ( s - \floor{r}{h} )A } \, B^b( w )
  \big)
  \Big\|_{\mathcal{V}}
\\ & \quad +
  \Big\|
  \smallsum\limits_{ b \in \mathbb{U} }
    \psi_{1,1}\big(
      e^{ ( s - r )A } \, u,
      e^{ ( s - r )A }\, v
    \big)
    \big(
      e^{ ( s - r )A } 
      \,
      (
        e^{ ( t - s ) A } - \operatorname{Id}_V
      ) 
      \,
      B^b( w )
      , 
      e^{ ( s - \floor{r}{h} )A }\, 
      B^b( w )
    \big)
  \Big\|_{\mathcal{V}}
\\ & \leq
\nonumber
  \frac{
    \psiC
    \,
    | \SGchi{0} |^\power
    \max\{
      1, \| u \|^\power_V, \| v \|^\power_V
    \}
    \,
    | \SGchi{\rho} |^2
    \,
    ( t - s )^\rho
    \,
    \| u \|_V 
    \,
    | \SGchi{ \nicefrac{ \vartheta }{ 2 } } |^2
    \,
    \| B \|^2_{ 
      \operatorname{Lip}^0( V, \gamma( U, V_{ -\nicefrac{ \vartheta }{ 2 } } ) ) 
    }
    \, 
    g_2( w )
  }{
    ( s - r )^{ 
      ( \rho + \nicefrac{ \vartheta }{ 2 } ) 
    }
    \,
    ( t - r )^{ 
      \nicefrac{ \vartheta }{ 2 } 
    }
  }
\\ & \quad +
\nonumber
  \frac{
    \psiC 
    \,
    | \SGchi{0} |^\power
    \max\{
      1, \| u \|^\power_V, \| v \|^\power_V
    \}
    \,
    \SGchi{ \rho + \nicefrac{ \vartheta }{ 2 } } 
    \,
    \SGchi{\rho}
    \,
    ( t - s )^\rho
    \,
    \SGchi{ \nicefrac{ \vartheta }{ 2 } } 
    \,
    \| B \|^2_{ 
      \operatorname{Lip}^0( V, \gamma( U, V_{ - \vartheta / 2 } ) ) 
    }
    \,
    g_2( w )
  }{
    ( s - r )^{ ( \rho + \vartheta ) }
  }
  .
\end{align}
Inequalities~\eqref{eq:prop_temporal_reg_termB1}--\eqref{eq:prop_temporal_reg_termB3} 
imply that 
for all 
$ t \in ( 0, T ] $, 
$ s \in ( 0, t ) $, 
$ r \in [ 0, s ) $, 
$ u,v, w \in V $ 
it holds that 
\begin{equation}
\begin{split}
&
  \| 
    \tilde{B}_{r,s,t}( u, v, w ) 
  \|_{\mathcal{V}}
\leq
  \psiC 
  \,
  | \SGchi{0} |^\power
  \max\!\big\{
    1, \| u \|^{ \power + 1 }_V, \| v \|^{ \power + 1 }_V
  \big\}
  \left(
    t - s 
  \right)^{ \rho }
  \| B \|^2_{ 
    \operatorname{Lip}^0( V, \gamma( U, V_{ - \vartheta / 2 } ) ) 
  }
  \,
  g_2( w )
\\ & 
  \cdot
  \left[
    \tfrac{
      \frac{ 1 }{ 2 } 
      \,
      | \SGchi{ \rho } |^2
      \,
      | \SGchi{ \nicefrac{ \vartheta }{ 2 } } |^2
    }{
      ( s - r )^{ \rho }
      \,
      ( t - r )^{ \vartheta }
    }
    +
    \tfrac{
      \SGchi{ \nicefrac{ \vartheta }{ 2 } } 
      \, 
      \SGchi{ \rho + \nicefrac{ \vartheta }{ 2 } } 
      \,
      \SGchi{\rho}
    }{
      ( s - r )^{ ( \rho + \vartheta ) }
    }
    +
    \tfrac{
      \frac{ 1 }{ 2 } 
      \,
      | \SGchi{ \rho } |^2
      \,
      | \SGchi{ \nicefrac{ \vartheta }{ 2 } } |^2
    }{
      ( s - r )^{ ( \rho + \vartheta ) }
    }
    +
    \tfrac{
      | \SGchi{ \rho } |^2
      \,
      | \SGchi{ \nicefrac{ \vartheta }{ 2 } } |^2
    }{
      ( s - r )^{ 
        ( 
          \rho + 
          \nicefrac{ \vartheta }{ 2 }
        ) 
      }
      \,
      ( t - r )^{
        \nicefrac{ \vartheta }{ 2 }
      }
    }
    +
    \tfrac{
      \SGchi{ \rho + \nicefrac{ \vartheta }{ 2 } } 
      \, 
      \SGchi{\rho}
      \,
      \SGchi{ \nicefrac{ \vartheta }{ 2 } } 
    }{
      ( s - r )^{ ( \rho + \vartheta ) }
    }
  \right]
\\ & \leq
  \psiC 
  \,
  | \SGchi{0} |^\power
  \max\!\big\{
    1, \| u \|^{ \power + 1 }_V, \| v \|^{ \power + 1 }_V
  \big\}
  \left(
    t - s 
  \right)^{ \rho }
  \| B \|^2_{ 
    \operatorname{Lip}^0( V, \gamma( U, V_{ - \vartheta / 2 } ) ) 
  }
  \,
  g_2( w )
\\ & \cdot
  \frac{
    2 
    \, 
    \SGchi{\rho} 
    \,
    \SGchi{ 
      \nicefrac{ \vartheta }{ 2 } 
    }
    \,
    \big(
      \SGchi{\rho} 
      \, 
      \SGchi{ \nicefrac{ \vartheta }{ 2 } } 
      + 
      \SGchi{ 
        \rho + \nicefrac{ \vartheta }{ 2 } 
      }
    \big)
  }{ 
    ( s - r )^{ ( \rho + \vartheta ) } 
  }
  .
\end{split}
\end{equation}
This and \eqref{eq:optimal_moment_1} prove that for all 
$ t \in (0,T] $, $ s \in (0,t) $
it holds that 
\begin{equation}\label{eq:B_summand_inner}
\begin{split}
&
  \int^{ s }_0
  \big\|
  \ES\big[
    \tilde{B}_{r,s,t}(
      \bar{Y}_r, 
      Y_r, 
      Y_{ \floor{ r }{ h } }
    )
  \big]
  \big\|_{\mathcal{V}}
  \,
  dr
\\ & \leq
    \psiC 
    \,
    | \SGchi{0} |^\power
    \,
    K_{ \power + 3 }
    \,
  ( t - s )^\rho 
  \, 
  \| B \|^2_{ 
    \operatorname{Lip}^0( V, \gamma( U, V_{ -\nicefrac{ \vartheta }{ 2 } } ) ) 
  }
  \,
  \frac{
    2 
    \, 
    \SGchi{\rho} 
    \,
    \SGchi{ \nicefrac{ \vartheta }{ 2 } }
    \,
    \big(
      \SGchi{\rho} \, \SGchi{ \nicefrac{ \vartheta }{ 2 } } 
      + 
      \SGchi{ \rho + \nicefrac{ \vartheta }{ 2 } }
    \big)
    \,
    s^{ ( 1 - \vartheta - \rho ) }
  }{ 
    ( 1 - \vartheta - \rho ) 
  }
  .
\end{split}
\end{equation}
Combining \eqref{eq:mild_ito_inner_1st}
with the estimates 
\eqref{eq:1st_summand_inner}, \eqref{eq:F_summand_inner}, 
and \eqref{eq:B_summand_inner} 
yields that 
for all $ t \in (0,T] $, $ s \in (0,t) $ it holds that 
\begin{equation}
\label{eq:prop_temporal_reg_term1A}
\begin{split}
&
  \left\|
  \E\left[
    \psi(
      e^{ ( t - s )A } \, \bar{Y}_{ s } ,
      Y_{ s }
    )
    -
    \psi(
      \bar{Y}_{ s } ,
      Y_{ s }
    )
  \right]
  \right\|_{\mathcal{V}}
\leq
  \psiC
  \,
  | \SGchi{0} |^{ \power }
  \,
  K_{ \power + 1 } \,
  \tfrac{
    |\SGchi{\rho}|^2
  }{
    s^{ \rho }
  }
  \left( t - s \right)^{ \rho }
\\ & 
  +
  \frac{
    \psiC
    \,
    | \SGchi{0} |^q
    \,
    \SGchi{\rho}
    \left(
      2 \, \SGchi{\rho} \SGchi{\vartheta} 
      +
      \SGchi{ \rho + \vartheta }
    \right)
  }{
    ( 1 - \vartheta - \rho )
  }
  \,
  \| F \|_{ \operatorname{Lip}^0( V, V_{ -\vartheta } ) } 
  \,
  K_{ \power + 2 }
  \,
  ( t - s )^\rho 
  \,
  s^{ ( 1 - \vartheta - \rho ) }
\\ & +
  \frac{
    2 
    \, 
    \psiC 
    \,
    | \SGchi{0} |^\power
    \,
    \SGchi{\rho} 
    \,
    \SGchi{ \nicefrac{ \vartheta }{ 2 } }
    \,
    \big(
      \SGchi{\rho} \, \SGchi{ \nicefrac{ \vartheta }{ 2 } } 
      + 
      \SGchi{ \rho + \nicefrac{ \vartheta }{ 2 } }
    \big)
  }{ 
    ( 1 - \vartheta - \rho ) 
  }
  \, 
  \| B \|^2_{ 
    \operatorname{Lip}^0( V, \gamma( U, V_{ -\nicefrac{ \vartheta }{ 2 } } ) ) 
  }
  \,
    K_{ \power + 3 }
    \,
  ( t - s )^\rho 
    \,
    s^{ ( 1 - \vartheta - \rho ) }
\\ & \leq
  \psiC
  \,
  | \SGchi{0} |^q
  \,
  \varsigma_{ F, B }
  \,
  K_{ q + 3 }
  \left( t - s \right)^{ \rho }
  \left[ 
    \tfrac{ 
      | \SGchi{ \rho } |^2 
    }{
      s^{ \rho }
    }
    +
    \tfrac{
      \SGchi{\rho}
      \,
      \left(
        2 \, \SGchi{\rho} \SGchi{\vartheta} 
        +
        \SGchi{ \rho + \vartheta }
        +
        2 \, 
        \SGchi{\rho} \, | \SGchi{ \vartheta / 2 } |^2
        + 
        2 \,
        \SGchi{ \rho + \vartheta / 2 }
        \,
        \SGchi{ \vartheta / 2 }
      \right)
      \,
      s^{
        ( 1 - \vartheta - \rho )
      }
    }{
      ( 1 - \vartheta - \rho )
    }
  \right]
\\ & \leq
  \psiC
  \,
  | \SGchi{0} |^q
  \,
  | \SGchi{ \rho } |^2 
  \,
  \varsigma_{ F, B }
  \,
  K_{ q + 3 }
  \left( t - s \right)^{ \rho }
  \left[ 
    \tfrac{ 
      1
    }{
      s^{ \rho }
    }
    +
    \tfrac{
      \left(
        2 \, 
        \SGchi{\vartheta} 
        +
        \SGchi{ \rho + \vartheta }
        +
        2 \, 
        | \SGchi{ \vartheta / 2 } |^2
        + 
        2 \,
        \SGchi{ \rho + \vartheta / 2 }
        \,
        \SGchi{ \vartheta / 2 }
      \right)
      \,
      s^{
        ( 1 - \vartheta - \rho )
      }
    }{
      ( 1 - \vartheta - \rho )
    }
  \right]
  .
\end{split}
\end{equation}
In addition, we note that 
for all $ (s,t) \in \angle $ 
it holds that 
\begin{equation}
\label{eq:initial_value_t=0_1st}
\begin{split}
&
  \left\|
  \E\left[
    \psi(
      e^{ (t - s) A } \, \bar{Y}_s ,
      Y_s
    )
    -
    \psi(
      \bar{Y}_s ,
      Y_s
    )
  \right]
  \right\|_{\mathcal{V}}
\\ & \leq
  \psiC
  \,
  \E\left[ 
    \max\!\left\{ 
      1 ,
      \| e^{ ( t - s ) A } \bar{Y}_s \|^{ \power }_V
      ,
      \| \bar{Y}_s \|^{ \power }_V
      ,
      \| Y_s \|^{ \power }_V
    \right\}
    \|
      e^{ ( t - s ) A } - \operatorname{Id}_V 
    \|_{ L(V) }
    \,
    \| 
      \bar{Y}_s
    \|_V
  \right]
\\ & \leq
  \psiC
  \left|
  \SGchi{0}
  \right|^{ ( \power + 1 ) }
  \E\left[ 
    \max\!\left\{ 
      1 ,
      \| \bar{Y}_s \|^{ \power }_V
      ,
      \| Y_s \|^{ \power }_V
    \right\}
    \| 
      \bar{Y}_s 
    \|_V
  \right]
\\ & \leq
  \psiC
  \left|
  \SGchi{0}
  \right|^{ ( \power + 1 ) }
  \E\left[ 
    \max\!\left\{ 
      1 ,
      \| \bar{Y}_s \|^{ \power + 1 }_V
      ,
      \| Y_s \|^{ \power + 1 }_V
    \right\}
  \right]
\leq
  \psiC \, 
  | \SGchi{0} |^{ ( \power + 1 ) }
  K_{ \power + 1 }
  .
\end{split}
\end{equation}
Combining this with 
\eqref{eq:prop_temporal_reg_term1A}
proves that for all
$ ( s, t ) \in \angle $
it holds that
\begin{equation}
\begin{split}
&
  \left\|
  \E\left[
    \psi(
      e^{ (t - s) A } \, \bar{Y}_s ,
      Y_s
    )
    -
    \psi(
      \bar{Y}_s ,
      Y_s
    )
  \right]
  \right\|_{\mathcal{V}}
\\ & \leq
  \psiC
  \,
  | \SGchi{0} |^{ ( \power + 1 ) }
  \,
  | \SGchi{ \rho } |^2 
  \,
  \varsigma_{ F, B }
  \,
  K_{ q + 3 }
  \left[
    \min\!\left\{ 
      1 ,
      \tfrac{ 
        \left( t - s \right)^{ \rho }
      }{
        s^{ \rho }
      }
    \right\}
    +
    \tfrac{
      \left( t - s \right)^{ \rho }
      \,
      \left(
        2 \, 
        \SGchi{\vartheta} 
        +
        \SGchi{ \rho + \vartheta }
        +
        2 \, 
        | \SGchi{ \vartheta / 2 } |^2
        + 
        2 \,
        \SGchi{ \rho + \vartheta / 2 }
        \,
        \SGchi{ \vartheta / 2 }
      \right)
      \,
      s^{
        ( 1 - \vartheta - \rho )
      }
    }{
      ( 1 - \vartheta - \rho )
    }
  \right]
\\ & =
  \psiC
  \,
  | \SGchi{0} |^{ ( \power + 1 ) }
  \,
  | \SGchi{ \rho } |^2 
  \,
  \varsigma_{ F, B }
  \,
  K_{ q + 3 }
\\ & \quad \cdot
  \left[
    \mathbbm{1}_{
      [ \frac{ t }{ 2 } , T ]
    }( s )
    \cdot
      \tfrac{ 
        \left( t - s \right)^{ \rho }
      }{
        s^{ \rho }
      }
    +
    \mathbbm{1}_{
      [ 0, \frac{ t }{ 2 } )
    }( s )
    \cdot
    \tfrac{
      \left( t - s \right)^{ \rho }
    }{
      \left( t - s \right)^{ \rho }
    }
    +
    \tfrac{
        \left( t - s \right)^{ \rho }
        \,
      \left(
        2 \, 
        \SGchi{\vartheta} 
        +
        \SGchi{ \rho + \vartheta }
        +
        2 \, 
        | \SGchi{ \vartheta / 2 } |^2
        + 
        2 \,
        \SGchi{ \rho + \vartheta / 2 }
        \,
        \SGchi{ \vartheta / 2 }
      \right)
      \,
      s^{
        ( 1 - \vartheta - \rho )
      }
    }{
      ( 1 - \vartheta - \rho )
    }
  \right]
\\ & \leq
  \psiC
  \,
  | \SGchi{0} |^{ ( \power + 1 ) }
  \,
  | \SGchi{ \rho } |^2 
  \,
  \varsigma_{ F, B }
  \,
  K_{ q + 3 }
  \left[
      \tfrac{ 
        \left( t - s \right)^{ \rho }
      }{
        \left( t / 2 \right)^{ \rho }
      }
    +
    \tfrac{
        \left( t - s \right)^{ \rho }
        \,
      \left(
        2 \, 
        \SGchi{\vartheta} 
        +
        \SGchi{ \rho + \vartheta }
        +
        2 \, 
        | \SGchi{ \vartheta / 2 } |^2
        + 
        2 \,
        \SGchi{ \rho + \vartheta / 2 }
        \,
        \SGchi{ \vartheta / 2 }
      \right)
      \,
      s^{
        ( 1 - \vartheta - \rho )
      }
    }{
      ( 1 - \vartheta - \rho )
    }
  \right]
  .
\end{split}
\end{equation}
Combining this, \eqref{eq:outer_summands}, 
and \eqref{eq:mild_ito_outer_1st} establishes that
for all $ ( s, t ) \in \angle $
it holds that
\begin{equation}
\begin{split}
\label{eq:mild_ito_2nd_case_1st}
&
  \left\|
  \E\left[
    \psi\big(
    \bar{Y}_{ t },
    Y_{ s }
    \big)
    -
    \psi\big(
    \bar{Y}_{ s },
    Y_{ s }
    \big)
  \right]
  \right\|_{\mathcal{V}}
\leq
  \psiC \,
  | \SGchi{0} |^{ ( \power + 1 ) } \, 
  | \SGchi{\rho} |^2
  \,
  \varsigma_{ F, B } \,
  K_{ \power + 3 } 
  \,
  ( t - s )^\rho
\\ & \cdot
  \bigg[
    \big|
      \tfrac{
        2
      }{
        t
      }
    \big|^\rho
    +
    \tfrac{
      \left(
        2 
        \, 
        \SGchi{\vartheta} 
        +
        \SGchi{ \rho + \vartheta }
        +
        2 
        \, 
        | \SGchi{ \vartheta / 2 } |^2
        + 
        2 
        \,
        \SGchi{ \rho + \vartheta / 2 }
        \,
        \SGchi{ \vartheta / 2 }
      \right)
      \,
      s^{ ( 1 - \vartheta - \rho ) }
      +
      \left( 
        \SGchi{\vartheta} 
        +
        \frac{ 1 }{ 2 }
        | \SGchi{ \nicefrac{ \vartheta }{ 2 } } |^2 
      \right)
      \,
      \left| 
        t - s 
      \right|^{ 
        ( 1 - \vartheta - \rho ) 
      }
    }{ 
      ( 1 - \vartheta - \rho ) 
    }
  \bigg]
  .
\end{split}
\end{equation}
The proof of Proposition~\ref{prop:weak_temporal_regularity_1st} is thus completed.
\end{proof}

\subsection{Analysis of the weak distance between exponential Euler approximations and their semilinear integrated counterparts}
\label{sec:weak_distance_semilinear}

\begin{lemma}[Analysis of the analytically weak but probabilistically strong distance between exponential Euler approximations and their semilinear integrated counterparts]
\label{lem:euler_integrated_strong}
Assume the setting in Section~\ref{sec:setting_weak_temporal_regularity} 
and let 
$ \rho \in [ 0, 1 ) $, 
$ 
  \varrho \in 
  \big[ 
    0, 
    1 - \max\{\frac{ 1 + \vartheta }{ 2 } - \rho,0\} 
  \big) 
$, 
$ t \in (0,T] $.
Then 
\begin{align}
&
  \| Y_t - \bar{ Y }_t \|_{ \lpn{ p }{ \P }{ V_{ -\rho } } }
\leq
  \left| 
    K_p 
  \right|^{ 
    \frac{ 1 }{ p } 
  } \,
  \SGchi{\varrho} \,
  h^\varrho
\\ & \cdot
\nonumber
  \left[
    \frac{
      \SGchi{ \varrho+\vartheta-\rho } 
      \, 
      t^{ 
        1 - \max\{\vartheta + \varrho - \rho,0\}
      }
      \|
        F
      \|_{
        \operatorname{Lip}^0( V, V_{ -\vartheta } )
      } 
    }{ 
      ( 1 - \max\{\vartheta + \varrho - \rho,0\} ) 
    }
    +
    \frac{
      \BDG{p}{} \,
      \SGchi{ \varrho+\nicefrac{\vartheta}{2}-\rho } 
      \,
      \sqrt{ 
        t^{ 
          1 - \max\{\vartheta + 2 \varrho - 2 \rho,0\}
        } 
      }
      \,
      \|
        B
      \|_{
        \operatorname{Lip}^0(
          V, \gamma( U, V_{ - \vartheta / 2 } )
        )
      } 
  }{ \sqrt{ 1 - \max\{\vartheta + 2\varrho - 2\rho,0\} } }
  \right]
  .
\end{align}
\end{lemma}
\begin{proof}
First of all, we observe that 
\begin{equation}\label{eq:strong_integral_decomposition}
\begin{split}
  \left\|
    Y_t - \bar{Y}_t
  \right\|_{
    \lpn{p}{\P}{V_{ -\rho }}
  }
&\leq
  \left\|
    \int^t_0
    \left(
      e^{ ( t - \floor{s}{h} ) A } - e^{ ( t - s ) A }
    \right)
    F( Y_{ \floor{ s }{ h } } )
    \, ds
  \right\|_{
    \lpn{p}{\P}{V_{ -\rho }}
  }
\\ & \quad +
  \left\|
    \int^t_0
    \left(
      e^{ ( t - \floor{s}{h} ) A } 
      - 
      e^{ ( t - s ) A }
    \right)
    B( Y_{ \floor{ s }{ h } } )
    \, dW_s
  \right\|_{
    \lpnb{p}{\P}{ V_{ - \rho } }
  }
  .
\end{split}
\end{equation}
Next we note that 
\begin{equation}
\label{eq:strong_integral_decomposition_2nd}
\begin{split}
&
  \left\|
    \int^t_0
    \left(
      e^{ ( t - \floor{s}{h} ) A } - e^{ ( t - s ) A }
    \right)
    F( Y_{ \floor{ s }{ h } } )
    \, ds
  \right\|_{
    \lpn{p}{\P}{ V_{ - \rho } }
  }
 \\ & 
\leq
  \int^t_0
  \left\|
    (
      e^{ ( t - \floor{s}{h} ) A } - e^{ ( t - s ) A }
    )
    F( Y_{ \floor{ s }{ h } } )
  \right\|_{
    \lpn{p}{\P}{V_{ -\rho }}
  }
  ds
\\ & \leq
  \int^t_0
  \frac{
    \SGchi{ \varrho+\vartheta-\rho } \,
    \SGchi{\varrho}
    \, h^\varrho
    \,
    \|
      F( 
        Y_{ \floor{ s }{ h } } 
      )
    \|_{
      \lpn{p}{\P}{V_{ -\vartheta }}
    }
  }{ 
    ( t - s )^{ \max\{\vartheta + \varrho - \rho,0\} } 
  }
  \, ds
\\&\leq
  \frac{
    \SGchi{ \varrho+\vartheta-\rho } \,
    \SGchi{\varrho}
    \, 
    t^{ 1 - \max\{\vartheta + \varrho - \rho,0\} }
  }{ 
    ( 
      1 - \max\{\vartheta + \varrho - \rho,0\}
    ) 
  }
  \left\|
    F
  \right\|_{
    \operatorname{Lip}^0( V, V_{ -\vartheta } )
  }
  | K_p |^{ \frac{ 1 }{ p } } \,
  h^\varrho
  .
\end{split}
\end{equation}
Moreover, Lemma~\ref{lem:Kp_estimate} ensures that $K_p<\infty$. 
This assures that 
\begin{equation}
\label{eq:strong_integral_decomposition_3rd}
\begin{split}
  &\left\|
  \int^t_0
  (
  e^{ ( t - \floor{s}{h} ) A } - e^{ ( t - s ) A }
  )
  B\big( Y_{ \floor{ s }{ h } } \big)
  \, dW_s
  \right\|_{
  \lpnb{p}{\P}{V_{ -\rho }}
  }
\\ & \leq
  \BDG{p}{}
  \left[
    \int^t_0
    \left\|
      \left(
        e^{ ( t - \floor{s}{h} ) A } - e^{ ( t - s ) A }
      \right)
      B( Y_{ \floor{ s }{ h } } )
    \right\|_{
      \lpn{p}{\P}{
        \gamma( U, V_{ -\rho } )
      }
    }^2
    ds
  \right]^{ 1 / 2 }
\\ & \leq
  \BDG{p}{}
  \left[
    \int^t_0
    \frac{
    |\SGchi{ \varrho+\nicefrac{ \vartheta }{ 2 }-\rho }\,\SGchi{\varrho}|^2 \, h^{2\varrho} \,
    \|
      B( Y_{ \floor{ s }{ h } } )
    \|_{
      \lpn{p}{\P}{
        \gamma( U, V_{ -\nicefrac{ \vartheta }{ 2 } } )
      }
    }^2
    }{ ( t - s )^{ \max\{\vartheta + 2\varrho - 2\rho,0\} } }
    \, ds
  \right]^{ 1 / 2 }
\\ & \leq
  \frac{
  \BDG{p}{} \,
  \SGchi{ \varrho+\nicefrac{ \vartheta }{ 2 }-\rho }\,\SGchi{\varrho} \,
  \sqrt{ t^{ 1 - \max\{\vartheta + 2\varrho - 2\rho,0\} } }
  }{ \sqrt{ 1 - \max\{\vartheta + 2\varrho - 2\rho,0\} } }
  \left\|
  B
  \right\|_{
  \operatorname{Lip}^0(
  V, \gamma( U, V_{ -\nicefrac{ \vartheta }{ 2 } } )
  )
  }
  | K_p |^{ \frac{ 1 }{ p } } \, h^\varrho
  .
\end{split}
\end{equation}
Combining \eqref{eq:strong_integral_decomposition}--\eqref{eq:strong_integral_decomposition_3rd} 
completes the proof of Lemma~\ref{lem:euler_integrated_strong}.
\end{proof}

\begin{proposition}[Weak distance between exponential Euler approximations and 
their semilinear integrated counterparts]
\label{prop:weak_temporal_regularity_2nd}
Assume the setting in Section~\ref{sec:setting_weak_temporal_regularity} and let 
$ \psiC \in [ 0, \infty ) $, 
$ \power \in [ 0, \infty ) \cap ( -\infty, p - 3 ] $, 
$ \rho \in [ 0, 1 - \vartheta ) $, 
$ \psi = ( \psi(x,y) )_{ x, y \in V } \in C^2( V \times V, {\mathcal{V}} ) $ 
satisfy for all 
$ x, y_1, y_2 \in V $, 
$ i, j \in \N_0 $ 
with 
$ i + j \leq 2 $
that 
\begin{equation} 
  \big\|
    \big(
     \tfrac{ \partial^{ ( i + j ) } }{ \partial x^i \partial y^j  }
     \psi
    \big)
    ( x, y_1 )
  -
    \big(
     \tfrac{ \partial^{ ( i + j ) } }{ \partial x^i \partial y^j  }
     \psi
    \big)
    ( x, y_2 )
  \big\|_{ L^{ ( i + j ) }( V, {\mathcal{V}} ) }
\leq
  \psiC
  \,
  \max\{ 1, \| x \|^\power_V, \| y_1 \|^\power_V, \| y_2 \|^\power_V \}
  \,
  \| y_1 - y_2 \|_V
  .
\end{equation}
Then it holds for all $ t \in ( 0, T ] $ that 
$
  \ES\big[
    \| 
      \psi( \bar{Y}_t, Y_t )
      -
      \psi( \bar{Y}_t, \bar{Y}_t ) 
    \|_{\mathcal{V}}
  \big]
  < \infty
$ 
and 
\begin{equation}
\label{eq:weak_distance}
\begin{split}
&
  \left\|
  \E\left[
    \psi\big(
    \bar{Y}_{ t },
    Y_{ t }
    \big)
    -
    \psi\big(
    \bar{Y}_{ t },
    \bar{Y}_{ t }
    \big)
  \right]\right\|_{\mathcal{V}}
\leq
    \psiC
    \,
    | \SGchi{0} |^\power
    \,
    \SGchi{\rho} \,
    \varsigma_{ F, B } 
    \,
    K_{ \power + 3 } 
    \, 
    h^\rho
\\&\cdot
  \frac{
    t^{ ( 1 - \vartheta - \rho ) }
  }{ 
    ( 1 - \vartheta - \rho ) 
  }
  \Bigg[
    \SGchi{ \rho+\vartheta } + 2 \, \SGchi{ \nicefrac{ \vartheta }{ 2 } } \, \SGchi{ \rho+\nicefrac{ \vartheta }{ 2 } }
    +
    2 \, \SGchi{\rho} \,
    ( | \SGchi{ \nicefrac{ \vartheta }{ 2 } } |^2 + \SGchi{\vartheta} )
\\ & 
    \cdot
    \bigg(
      \frac{
        \SGchi{ \vartheta } \, t^{ ( 1 - \vartheta ) } \,
        \| F \|_{ \operatorname{Lip}^0( V, V_{ -\vartheta } ) }
      }{ 
        ( 1 - \vartheta ) 
      }
      +
      \frac{
        \BDG{\power+3}{} \,
        \SGchi{ \nicefrac{\vartheta}{2} } 
        \, 
        \sqrt{ 
          t^{ ( 1 - \vartheta ) } 
        }
        \,
        \| B \|_{ \operatorname{Lip}^0( V, \gamma( U, V_{ - \vartheta / 2 } ) ) }
      }{
        \sqrt{ 1 - \vartheta }
      }
    \bigg)
  \Bigg]
  .
\end{split}
\end{equation}
\end{proposition}
\begin{proof}
Throughout this proof 
let 
$
  ( g_r )_{ r \in [ 0, \infty ) }
  \subseteq
  C( V, \R )
$ 
be the functions
which satisfy
for all $ r \in [0,\infty) $, $ x \in V $
that
$
  g_r( x )
  =
  \max\{
  1, \| x \|^r_V
  \}
$
and let 
$
  \psi_{1,0} \colon V \times V \to L( V, {\mathcal{V}} ) 
$, 
$
  \psi_{0,1} \colon V \times V \to L( V, {\mathcal{V}} ) 
$, 
$
  \psi_{2,0} \colon V \times V \to L^{(2)}( V, {\mathcal{V}} ) 
$, 
$
  \psi_{0,2} \colon V \times V \to L^{(2)}( V, {\mathcal{V}} ) 
$, 
$
  \psi_{1,1} \colon V \times V \to L^{(2)}( V, {\mathcal{V}} ) 
$
be the functions which satisfy
for all 
$ x, y, v_1, v_2 \in V $ 
that 
\begin{equation}
\begin{split}
\psi_{1,0}( x, y ) \, v_1
=
\big(\tfrac{ \partial }{ \partial x } \psi\big)( x, y ) \, v_1
,&
\qquad
\psi_{0,1}( x, y ) \, v_1
=
\big(\tfrac{ \partial }{ \partial y } \psi\big)( x, y ) \, v_1
,
\\
\psi_{2,0}( x, y )( v_1, v_2 )
=
\big(
\tfrac{ \partial^2 }{ \partial x^2 } \psi\big)( x, y )
( v_1, v_2 )
,&
\qquad
\psi_{0,2}( x, y )( v_1, v_2 )
=
\big(
\tfrac{ \partial^2 }{ \partial y^2 } \psi\big)( x, y )
( v_1, v_2 )
,
\\
\psi_{ 1, 1 }( x, y )( v_1, v_2 )
=
\big(
\tfrac{ \partial }{ \partial y } 
\tfrac{ \partial }{ \partial x } 
\psi\big)( x, y )
( v_1, v_2 )
.&
\end{split}
\end{equation}
Next we observe that Lemma~\ref{lem:Kp_estimate}
and the assumption that $ q \leq p - 3 $ ensure that $ K_{ q + 1 } \leq K_{ q + 3 } < \infty $.
Combining this with the fact that
\begin{equation}
  \forall \, x, y_1, y_2 \in V \colon
  \| 
    \psi( x, y_1 ) 
    - 
    \psi( x, y_2 ) 
  \|_{\mathcal{V}}
  \leq 
  2 \, \psiC
  \max\!\big\{ 
    1, \| x \|^{ q + 1 }_V , \| y_1 \|^{ q + 1 }_V , \| y_2 \|^{ q + 1 }_V 
  \big\}
\end{equation}
shows that for all $ t \in ( 0, T ] $ it holds that 
$
  \ES\big[
    \| 
      \psi( \bar{Y}_t, Y_t )
      -
      \psi( \bar{Y}_t, \bar{Y}_t ) 
    \|_{\mathcal{V}}
  \big]
  < \infty
$.
It thus remains to prove \eqref{eq:weak_distance}.
To do so, we make use of the mild It\^{o} formula in~\cite{CoxJentzenKurniawanPusnik2016}. 
For this let 
$ \tilde{ F }_{ s, t } \colon V \times V \times V \to {\mathcal{V}} $, 
$ ( s, t ) \in \angle $, 
be the functions which satisfy
for all 
$ ( s, t ) \in \angle $, 
$ u, v, w \in V $ 
that 
\begin{equation}
\begin{split}
  \tilde{F}_{ s, t }( u, v, w )
& =
  \big[
  \psi_{1,0}
  \big(
  e^{ ( t - s )A } \, u,
  e^{ ( t - s )A } \, v
  \big)
  -
  \psi_{1,0}
  \big(
  e^{ ( t - s )A } \, u,
  e^{ ( t - s )A } \, u
  \big)
  \big]
  \,
  e^{ ( t - s )A } \, F( w )
\\ & +
  \psi_{0,1}
  \big(
  e^{ ( t - s )A } \, u,
  e^{ ( t - s )A } \, v
  \big)
  \,
  e^{ ( t - \floor{s}{h} )A } \, F( w )
  -
  \psi_{0,1}
  \big(
  e^{ ( t - s )A } \, u,
  e^{ ( t - s )A } \, u
  \big)
  \,
  e^{ ( t - s )A } \, F( w )
\end{split}
\end{equation}
and let 
$ \tilde{ B }_{ s, t } \colon V \times V \times V \to {\mathcal{V}} $, 
$ ( s, t ) \in \angle $, 
be the functions which satisfy
for all 
$ ( s, t ) \in \angle $, 
$ u, v, w \in V $ 
that 
\begin{equation}
\begin{split}
&
  \tilde{B}_{ s, t }( u, v, w )
\\ & =
  \tfrac{ 1 }{ 2 }
  \smallsum\limits_{ b \in \mathbb{ U 	} }
  \big[
  \psi_{ 2,0 }
  \big(
  e^{ ( t - s )A } \, u,
  e^{ ( t - s )A } \, v
  \big)
  -
  \psi_{ 2,0 }
  \big(
  e^{ ( t - s )A } \, u,
  e^{ ( t - s )A } \, u
  \big)
  \big]
  \big(
    e^{ ( t - s )A } \, B^b( w ),
    e^{ ( t - s )A } \, B^b( w )
  \big)
\\ & +
  \tfrac{ 1 }{ 2 }
  \smallsum\limits_{ b \in \mathbb{ U 	} }
  \psi_{ 0,2 }
  \big(
  e^{ ( t - s )A } \, u,
  e^{ ( t - s )A } \, v
  \big)
  \big(
    e^{ ( t - \floor{s}{h} )A } \, B^b( w ),
    e^{ ( t - \floor{s}{h} )A } \, B^b( w )
  \big)
\\ & 
-
  \tfrac{ 1 }{ 2 }
  \smallsum\limits_{ b \in \mathbb{ U 	} }
  \psi_{ 0,2 }
  \big(
  e^{ ( t - s )A } \, u,
  e^{ ( t - s )A } \, u
  \big)
  \big(
    e^{ ( t - s )A } \, B^b( w ),
    e^{ ( t - s )A } \, B^b( w )
  \big)
\\ & +
  \smallsum\limits_{ b \in \mathbb{ U } }
  \psi_{ 1,1 }
  \big(
  e^{ ( t - s )A } \, u,
  e^{ ( t - s )A } \, v
  \big)
  \big(
    e^{ ( t - s )A } \, B^b( w ),
    e^{ ( t - \floor{s}{h} )A } \, B^b( w )
  \big)
\\ & 
  -
  \smallsum\limits_{ b \in \mathbb{ U } }
  \psi_{ 1,1 }
  \big(
  e^{ ( t - s )A } \, u,
  e^{ ( t - s )A } \, u
  \big)
  \big(
    e^{ ( t - s )A } \, B^b( w ),
    e^{ ( t - s )A } \, B^b( w )
  \big)
  .
\end{split}
\end{equation}
An application of Proposition~3.11 in~\cite{CoxJentzenKurniawanPusnik2016}
shows
that for all $ t \in ( 0, T ] $ it holds that 
\begin{equation}
\label{eq:mild_ito_2nd}
\begin{split}
&
  \left\|
  \E\left[
    \psi\big(
      \bar{Y}_{ t } ,
      Y_{ t }
    \big)
    -
    \psi\big(
      \bar{Y}_{ t } ,
      \bar{Y}_{ t }
    \big)
  \right]\right\|_{\mathcal{V}}
\leq
  \int^{ t }_0
  \E\left[
    \big\|
      \tilde{F}_{ s, t }\big(
        \bar{Y}_s, Y_s, Y_{ \floor{ s }{ h } }
      \big)
    \big\|_{\mathcal{V}}
  \right]
  +
  \E\left[
    \big\|
      \tilde{B}_{ s, t }\big(
        \bar{Y}_s, Y_s, Y_{ \floor{ s }{ h } }
      \big)
    \big\|_{\mathcal{V}}
  \right]
  ds
  .
\end{split}
\end{equation}
In the following we establish suitable estimates for the two integrands on the right hand side of \eqref{eq:mild_ito_2nd}. 
We begin with the first integrand on the right hand side of \eqref{eq:mild_ito_2nd}. 
Observe that for all 
$ ( s, t ) \in \angle $, 
$ u, v, w \in V $ 
it holds that 
\begin{equation}
\label{eq:tilde_F_1st_decompose}
\begin{split}
&
  \left\|
  \big[
    \psi_{1,0}\big(
      e^{ ( t - s )A } \, u,
      e^{ ( t - s )A } \, v
    \big)
    -
    \psi_{1,0}\big(
      e^{ ( t - s )A } \, u,
      e^{ ( t - s )A } \, u
    \big)
  \big]
  \, e^{ ( t - s ) A } \, F( w )
  \right\|_{\mathcal{V}}
\\ & \leq
  \psiC
  \max\!\big\{
    1,
    \| e^{ ( t - s )A } \, v \|^\power_V,
    \| e^{ ( t - s )A } \, u \|^\power_V
  \big\}
  \|
  e^{ ( t - s )A }(v-u)
  \|_V
  \,
  \|
  e^{ ( t - s )A } \, F( w )
  \|_V
\\ & \leq
  \frac{
    \psiC
    \,
    | \SGchi{0} |^\power
    \max\{
      1,
      \| u \|^\power_V,
      \| v \|^\power_V
    \}
      \| e^{ ( t - s )A } ( v - u ) \|_V
    \,
    \SGchi{\vartheta} 
    \,
    \| F( w ) \|_{ V_{ -\vartheta } }
  }{
    ( t - s )^\vartheta
  }
\\ & \leq
  \frac{
    \psiC
    \,
    | \SGchi{0} |^\power
    \,
    \SGchi{\vartheta} \,
    \SGchi{\rho} \,
    \max\!\big\{
      1,
      \| u \|^\power_V,
      \| v \|^\power_V
    \big\}
    \,
    \| v - u \|_{ V_{ -\rho } }
    \| F \|_{ \operatorname{Lip}^0( V, V_{ -\vartheta } ) }
    \, g_1( w )
  }{
    ( t - s )^{(\rho+\vartheta)}
  }
  .
\end{split}
\end{equation}
Moreover, we note that the assumption that 
\begin{equation}
  \forall \, x, y_1, y_2 \in V \colon 
  \|
    \psi
    ( x, y_1 )
  -
    \psi
    ( x, y_2 )
  \|_{\mathcal{V}}
\leq
  \psiC
  \max\!\big\{ 
    1, \| x \|^\power_V, \| y_1 \|^\power_V, \| y_2 \|^\power_V 
  \big\}
  \,
  \| y_1 - y_2 \|_V
\end{equation}
implies that 
for all $ x, y \in V $ 
it holds that
$
  \| \psi_{ 0, 1 }( x, y ) \|_{ L( V, {\mathcal{V}} ) }
  \leq
  \psiC \,
  \max\{
    1
    ,
    \| x \|^\power_V
    ,
    \| y \|^\power_V
  \}
$.
This, in turn, proves that for all 
$ ( s, t ) \in \angle $, 
$ u, v, w \in V $ 
it holds that 
\allowdisplaybreaks
\begin{align}
\label{eq:tilde_F_2nd_decompose}
&
  \left\|
  \psi_{0,1}
  \big(
  e^{ ( t - s )A } \, u,
  e^{ ( t - s )A } \, v
  \big)
  \,
  e^{ ( t - \floor{s}{h} )A } \, F( w )
  -
  \psi_{0,1}
  \big(
  e^{ ( t - s )A } \, u,
  e^{ ( t - s )A } \, u
  \big)
  \,
  e^{ ( t - s )A } \, F( w )
  \right\|_{\mathcal{V}}
\\ & \leq
\nonumber
  \left\|
  \psi_{0,1}
  \big(
  e^{ ( t - s )A } \, u,
  e^{ ( t - s )A } \, v
  \big)
  \big[
    e^{ ( t - \floor{s}{h} )A }
    -
    e^{ ( t - s )A }
  \big]
  F( w )
  \right\|_{\mathcal{V}}
\\ & +
\nonumber
  \left\|
  \big[
    \psi_{0,1}
    \big(
    e^{ ( t - s )A } \, u,
    e^{ ( t - s )A } \, v
    \big)
    -
    \psi_{0,1}
    \big(
    e^{ ( t - s )A } \, u,
    e^{ ( t - s )A } \, u
    \big)
  \big]
  \,
  e^{ ( t - s )A } \, F( w )
  \right\|_{\mathcal{V}}
\\ & \leq
\nonumber
  \psiC \,
  \max\!\big\{
    1,
    \| e^{ ( t - s )A } \, v \|^\power_V,
    \| e^{ ( t - s )A } \, u \|^\power_V
  \big\} 
  \,
  \big\|
  \big[
    e^{ ( t - \floor{s}{h} )A }
    -
    e^{ ( t - s )A }
  \big]
  F( w )
  \big\|_V
\\ & +
\nonumber
  \psiC \,
  \max\!\big\{
    1,
    \| e^{ ( t - s )A } \, v \|^\power_V,
    \| e^{ ( t - s )A } \, u \|^\power_V
  \big\}
  \,
  \|
    e^{ ( t - s )A }(v-u)
  \|_V 
  \,
  \| e^{ ( t - s )A } \, F(w) \|_V
\\ & \leq
\nonumber
  \psiC \,
  | \SGchi{0} |^\power
  \max\{
    1,
    \| u \|^\power_V,
    \| v \|^\power_V
  \}
  \,
  \| F \|_{ \operatorname{Lip}^0( V, V_{ -\vartheta } ) }
  \, g_1( w )
\\ & \cdot
\nonumber
  \Bigg[
    \frac{
    \SGchi{\rho+\vartheta} \, \SGchi{\rho}
    }{(t-s)^{(\rho+\vartheta)}} \,
    h^\rho
    +
    \frac{
    \SGchi{\vartheta} \,
    \SGchi{\rho} \,
    \| v - u \|_{ V_{ -\rho } }
    }{
    ( t - s )^{(\rho+\vartheta)}
    }
  \Bigg]
  .
\end{align}
Inequalities \eqref{eq:tilde_F_1st_decompose} and \eqref{eq:tilde_F_2nd_decompose} imply that for all 
$ ( s, t ) \in \angle $, 
$ u, v, w \in V $ 
it holds that 
\begin{equation}\label{eq:absolute_tilde_F_2nd}
\begin{split}
  \| 
    \tilde{F}_{ s, t }( u, v, w )
  \|_{\mathcal{V}}
& \leq
  \psiC
  \,
  | \SGchi{0} |^\power \,
  \SGchi{\rho} \,
  \max\!\big\{
    1,
    \| u \|^\power_V ,
    \| v \|^\power_V
  \big\}
  \,
  \| F \|_{ \operatorname{Lip}^0( V, V_{ -\vartheta } ) }
  \;
  g_1( w )
\\ & \quad \cdot
  \Bigg[
    \frac{
      \SGchi{\rho+\vartheta}
    }{ 
      ( t - s )^{ ( \rho + \vartheta ) } 
    } \,
    h^\rho
    +
    \frac{
      2 \, \SGchi{\vartheta} \,
      \| v - u \|_{ V_{ -\rho } }
    }{
      { ( t - s )^{ ( \rho + \vartheta ) } }
    }
  \Bigg]
  .
\end{split}
\end{equation}
Next we estimate the second integrand on the right hand side of~\eqref{eq:mild_ito_2nd}. 
Next we observe that Lemma~\ref{lem:gamma.estimate} shows that for all 
$ ( s, t ) \in \angle $, 
$ u, v, w \in V $ 
it holds that 
\begin{equation}
\label{eq:weak_distance_B_1}
\begin{split}
&
  \Big\|
  \smallsum\limits_{ b \in \mathbb{ U 	} }
  \big[
  \psi_{2,0}
  \big(
  e^{ ( t - s )A } \, u,
  e^{ ( t - s )A } \, v
  \big)
  -
  \psi_{2,0}
  \big(
  e^{ ( t - s )A } \, u,
  e^{ ( t - s )A } \, u
  \big)
  \big]
  \big(
    e^{ ( t - s )A } \, B^b( w ),
    e^{ ( t - s )A } \, B^b( w )
  \big)
  \Big\|_{\mathcal{V}}
\\ & \leq
  \frac{
    \psiC
    \,
    | \SGchi{0} |^\power
    \,
    |
      \SGchi{ \nicefrac{ \vartheta }{ 2 } }
    |^2 \,
      \SGchi{\rho} \,
    \max\{
      1,
      \| u \|^\power_V,
      \| v \|^\power_V
    \}
    \,
      \| v - u \|_{ V_{ -\rho } } \,
    \| B \|^2_{ \operatorname{Lip}^0( V, \gamma( U, V_{ -\nicefrac{ \vartheta }{ 2 } } ) ) }
    \; 
    g_2( w )
  }{
    ( t - s )^{(\rho+\vartheta)}
  }
  ,
\end{split}
\end{equation}
\allowdisplaybreaks
\begin{align}
&
  \Big\|
  \smallsum\limits_{ b \in \mathbb{ U 	} }
  \big[
  \psi_{0,2}
  \big(
  e^{ ( t - s )A } \, u,
  e^{ ( t - s )A } \, v
  \big)
  \big(
    e^{ ( t - \floor{s}{h} )A } \, B^b( w ),
    e^{ ( t - \floor{s}{h} )A } \, B^b( w )
  \big)
\\ & 
\nonumber
  -
  \psi_{0,2}
  \big(
  e^{ ( t - s )A } \, u,
  e^{ ( t - s )A } \, u
  \big)
  \big(
    e^{ ( t - s )A } \, B^b( w ),
    e^{ ( t - s )A } \, B^b( w )
  \big)\big]
  \Big\|_{\mathcal{V}}
\\ & \leq
\nonumber
  \Big\|
  \smallsum\limits_{ b \in \mathbb{ U 	} }
    \psi_{0,2}\big(
      e^{ ( t - s )A } \, u,
      e^{ ( t - s )A } \, v
    \big)
    \big(
      [e^{ ( t - \floor{s}{h} )A }
      +
      e^{ ( t - s )A }]
      \, B^b( w ),
      [e^{ ( t - \floor{s}{h} )A }
      -
      e^{ ( t - s )A }]
      \, B^b( w )
    \big)
  \Big\|_{\mathcal{V}}
\\ & +
\nonumber
  \Big\|
  \smallsum\limits_{ b \in \mathbb{ U 	} }
  \big[
    \psi_{0,2}\big(
      e^{ ( t - s )A } \, u,
      e^{ ( t - s )A } \, v
    \big)
    -
    \psi_{0,2}\big(
      e^{ ( t - s )A } \, u,
      e^{ ( t - s )A } \, u
    \big)
  \big]
  \big(
    e^{ ( t - s )A } \, B^b( w ),
    e^{ ( t - s )A } \, B^b( w )
  \big)
  \Big\|_{\mathcal{V}}
\\ & \leq
\nonumber
  \psiC
  \,
  | \SGchi{0} |^\power
    \max\{
      1,
      \| u \|^\power_V,
      \| v \|^\power_V
    \}
  \Big[
  \|
    e^{ ( t - s )A }(v-u)
  \|_V
  \,
  \|
    e^{ ( t - s )A } \, B( w )
  \|^2_{ \gamma(U,V) }
\\ & +
\nonumber
  \|
    (e^{ ( t - \floor{s}{h} )A }
    +
    e^{ ( t - s )A })
    \, B( w )
  \|_{ \gamma(U,V) }
  \,
  \|
    (e^{ ( t - \floor{s}{h} )A }
    -
    e^{ ( t - s )A })
    \, B( w )
  \|_{ \gamma(U,V) }
  \Big]
\\ & \leq
\nonumber
  \psiC
  \,
  | \SGchi{0} |^\power
  \max\{
    1,
    \| u \|^\power_V,
    \| v \|^\power_V
  \}
  \,
  \| B \|^2_{ \operatorname{Lip}^0( V, \gamma( U, V_{ -\nicefrac{ \vartheta }{ 2 } } ) ) } \,
  g_2( w )
\\ & \cdot
\nonumber
  \Bigg[
  \frac{
    2 \, \SGchi{\nicefrac{\vartheta}{2}} \, 
    \SGchi{\rho+\nicefrac{\vartheta}{2}} \, 
    \SGchi{\rho}
  }{
    ( t - s )^{(\rho+\vartheta)}
  } \,
  h^\rho
+
  \frac{
    | \SGchi{ \nicefrac{ \vartheta }{ 2 } } |^2 \,
    \SGchi{\rho} \,
    \| v - u \|_{ V_{ -\rho } }
  }{
    ( t - s )^{(\rho+\vartheta)}
  }
  \Bigg],
\end{align}
and 
\allowdisplaybreaks
\begin{align}
\label{eq:weak_distance_B_3}
&
  \Big\|
  \smallsum\limits_{ b \in \mathbb{ U } }
  \big[
  \psi_{1,1}
  \big(
  e^{ ( t - s )A } \, u,
  e^{ ( t - s )A } \, v
  \big)
  \big(
    e^{ ( t - s )A } \, B^b( w ),
    e^{ ( t - \floor{s}{h} )A } \, B^b( w )
  \big)
\\ & 
\nonumber
  -
  \psi_{1,1}
  \big(
  e^{ ( t - s )A } \, u,
  e^{ ( t - s )A } \, u
  \big)
  \big(
    e^{ ( t - s )A } \, B^b( w ),
    e^{ ( t - s )A } \, B^b( w )
  \big)\big]
  \Big\|_{\mathcal{V}}
\\ & \leq
\nonumber
  \Big\|
  \smallsum\limits_{ b \in \mathbb{ U } }
  \psi_{1,1}\big(
    e^{ ( t - s )A } \, u,
    e^{ ( t - s )A } \, v
  \big)
  \big(
    e^{ ( t - s )A } \, B^b( w ),
    [e^{ ( t - \floor{s}{h} )A }
    -
    e^{ ( t - s )A }
    ] \, B^b( w )
  \big)
  \Big\|_{\mathcal{V}}
\\ & +
\nonumber
  \Big\|
  \smallsum\limits_{ b \in \mathbb{ U } }
  \big[
  \psi_{1,1}
  \big(
  e^{ ( t - s )A } \, u,
  e^{ ( t - s )A } \, v
  \big)
  -
  \psi_{1,1}
  \big(
  e^{ ( t - s )A } \, u,
  e^{ ( t - s )A } \, u
  \big)\big]
  \big(
    e^{ ( t - s )A } \, B^b( w ),
    e^{ ( t - s )A } \, B^b( w )
  \big)
  \Big\|_{\mathcal{V}}
\\ & \leq
\nonumber
  \psiC
  \,
  | \SGchi{0} |^\power
    \max\{
      1,
      \| u \|^\power_V,
      \| v \|^\power_V
    \}
  \Big[
  \|
    e^{ ( t - s )A }(v-u)
  \|_V
  \,
  \|
    e^{ ( t - s )A } \, B( w )
  \|^2_{ \gamma(U,V) }
\\ & +
\nonumber
  \|
  e^{ ( t - s )A } \, B( w )
  \|_{ \gamma( U,V ) }
  \,
  \|
    [e^{ ( t - \floor{s}{h} )A }
    -
    e^{ ( t - s )A }
    ] \, B( w )
  \|_{ \gamma( U,V ) }
  \Big]
\\ & \leq
\nonumber
  \psiC
  \,
  | \SGchi{0} |^\power
    \max\{
      1,
      \| u \|^\power_V,
      \| v \|^\power_V
    \}
    \,
  \| B \|^2_{ \operatorname{Lip}^0( V, \gamma( U, V_{ -\nicefrac{ \vartheta }{ 2 } } ) ) } 
  \,
  g_2( w )
\\ & \cdot
\nonumber
  \Bigg[
  \frac{
    \SGchi{ \nicefrac{ \vartheta }{ 2 } }
    \,
    \SGchi{ \rho+\nicefrac{ \vartheta }{ 2 } }
    \,
    \SGchi{\rho}
  }{
    ( t - s )^{ (\rho+\vartheta) }
  } \,
  h^\rho
  +
  \frac{
    | \SGchi{ \nicefrac{ \vartheta }{ 2 } } |^2 \,
      \SGchi{\rho}
      \,
      \| v - u \|_{ V_{ -\rho } }
  }{
    ( t - s )^{ (\rho+\vartheta) }
  }
  \Bigg]
  .
\end{align}
Combining \eqref{eq:weak_distance_B_1}--\eqref{eq:weak_distance_B_3} implies 
that for all $ ( s, t ) \in \angle $, 
$ u, v, w \in V $ 
it holds that 
\begin{equation}\label{eq:absolute_tilde_B_2nd}
\begin{split}
&
  \|\tilde{B}_{ s, t }( u, v, w )\|_{\mathcal{V}}
\\ & \leq
  2 
  \, 
  \psiC
  \,
  | \SGchi{0} |^\power
  \,
  \SGchi{ \nicefrac{ \vartheta }{ 2 } } \,
  \SGchi{\rho} \,
  \max\{
    1,
    \| u \|^\power_V,
    \| v \|^\power_V
  \}
  \,
  \| B \|^2_{ \operatorname{Lip}^0( V, \gamma( U, V_{ -\nicefrac{ \vartheta }{ 2 } } ) ) } 
  \,
  g_2( w )
\\ & \cdot
  \Bigg[
      \frac{
       \SGchi{ \rho+\nicefrac{ \vartheta }{ 2 } }
      }{
       ( t - s )^{ ( \rho + \vartheta ) }
      } \,
    h^\rho
+
    \frac{
      \SGchi{ \nicefrac{ \vartheta }{ 2 } }
      \,
      \| v - u \|_{ V_{ -\rho } }
    }{
      ( t - s )^{ ( \rho + \vartheta ) }
    }
  \Bigg].
\end{split}
\end{equation}
Next observe that \eqref{eq:absolute_tilde_F_2nd} and \eqref{eq:absolute_tilde_B_2nd} show that for all 
$ ( s, t ) \in \angle $, 
$ u, v, w \in V $ 
it holds that 
\begin{equation}\label{eq:absolute_sum_FB_tilde_2nd}
\begin{split}
&
  \|
    \tilde{F}_{ s, t }( u, v, w )
  \|_{\mathcal{V}}
  +
  \|
    \tilde{B}_{ s, t }( u, v, w )
  \|_{\mathcal{V}}
\leq
  \psiC
  \,
  | \SGchi{0} |^\power \,
  \SGchi{\rho} \,
    \max\{
      1,
      \| u \|^\power_V,
      \| v \|^\power_V
    \}
    \,
  \varsigma_{ F, B } 
  \,
  g_2( w )
\\ & \cdot
  \Bigg[
    \left[
      \frac{
        \SGchi{ \rho+\vartheta } + 2 \, \SGchi{ \nicefrac{ \vartheta }{ 2 } } \, \SGchi{ \rho+\nicefrac{ \vartheta }{ 2 } }
      }{
        ( t - s )^{ ( \rho + \vartheta ) }
      } 
    \right] 
    h^\rho
    +
    \frac{
      2 
      \, 
      ( 
        | \SGchi{ \nicefrac{ \vartheta }{ 2 } } |^2 + \SGchi{\vartheta} 
      ) 
      \, 
      \| v - u \|_{ 
        V_{ -\rho } 
      }
    }{
      ( t - s )^{ ( \rho + \vartheta ) }
    }
  \Bigg]
  .
\end{split}
\end{equation}
In addition, note that H\"{o}lder's inequality ensures that 
for all 
$ r \in ( 0, \infty ) $, 
$ s \in [ 0, T ] $ 
it holds that 
\begin{equation}
\label{eq:optimal_moment_2}
\begin{split}
&
  \ES\big[\!
  \max\!\big\{
    1,
    \| \bar{Y}_s \|^r_V,
    \| Y_s \|^r_V
  \big\} \,
  g_2\big( Y_{ \floor{ s }{ h } } \big)
  \big]
\\ & \leq
  \left(
  \sup_{ u, v \in [ 0, T ] }
  \big\|\!
  \max\!\big\{
    1,
    \| \bar{Y}_u \|^r_V,
    \| Y_v \|^r_V
  \big\}
  \big\|_{ \lpn{ 1 + \nicefrac{2}{r} }{ \P }{ \R } }
  \right)
  \left(
  \sup_{ u \in [ 0, T ] }
  \big\|
  \!
  \max\!\big\{
    1,
    \| Y_u \|^2_V
  \big\}
  \big\|_{ \lpn{ 1 + \nicefrac{r}{2} }{ \P }{ \R } }
  \right)
\\ & \leq
  | K_{ r + 2 } |^{ \frac{ 1 }{ 1 + \nicefrac{2}{r} } }
  \,
  | K_{ r + 2 } |^{ \frac{ 1 }{ 1 + \nicefrac{r}{2} } }
  = K_{ r + 2 },
\end{split}
\end{equation}
\begin{equation}
\begin{split}
  \ES\big[
    g_2(
      Y_{ \floor{ s }{ h } }
    )
    \,
    \|
      Y_s - \bar{Y}_s
    \|_{ V_{ -\rho } }
  \big]
& \leq
  \|
    g_2(
      Y_{ \floor{ s }{ h } }
    )
  \|_{ \lpn{ \nicefrac{ 3 }{ 2 } }{ \P }{ \R } } 
  \,
  \|
    Y_s - \bar{Y}_s
  \|_{ \lpn{ 3 }{ \P }{ V_{ -\rho } } }
\\ & \leq
  | K_{ 3 } |^{ \nicefrac{ 2 }{ 3 } } 
  \left(
  \sup_{ u \in [ 0, T ] }
  \big\|
      Y_u - \bar{Y}_u
  \big\|_{ \lpn{ 3 }{ \P }{ V_{ -\rho } } }
  \right),
\end{split}
\end{equation}
and 
\begin{equation}\label{eq:optimal_moment_3}
\begin{split}
&
    \ES\big[\!
      \max\!\big\{
        1 ,
        \| \bar{Y}_s \|^r_V ,
        \| Y_s \|^r_V
      \big\}
      \,
      g_2(
        Y_{ \floor{ s }{ h } }
      )
      \,
      \|
        Y_s - \bar{Y}_s
      \|_{ V_{ -\rho } }
    \big]
\\ & \leq
  \|\!
    \max\{
      1
      ,
      \| \bar{Y}_s \|^r_V
      ,
      \| Y_s \|^r_V
    \}
  \|_{ \lpn{ 1 + \nicefrac{ 3 }{ r } }{ \P }{ \R } } 
  \,
  \|
    g_2(
      Y_{ \floor{ s }{ h } }
    )
  \|_{ 
    \lpn{ \nicefrac{ ( r + 3 ) }{ 2 } }{ \P }{ \R } 
  } 
  \,
  \|
      Y_s - \bar{Y}_s
  \|_{ \lpn{ r + 3 }{ \P }{ V_{ -\rho } } }
\\ & \leq
  | K_{ r + 3 } |^{ \frac{ r + 2 }{ r + 3 } } 
  \left(
    \sup_{ u \in [ 0, T ] }
    \big\|
      Y_u - \bar{Y}_u
    \big\|_{ \lpn{ r + 3 }{ \P }{ V_{ -\rho } } }
  \right)
  .
\end{split}
\end{equation}
Combining 
\eqref{eq:absolute_sum_FB_tilde_2nd}--\eqref{eq:optimal_moment_3} 
with
Lemma~\ref{lem:euler_integrated_strong} 
and the fact that 
$ 
  1 - \max\{\frac{ 1 + \vartheta }{ 2 } - \rho,0\} > \rho 
$ 
yields that for all $ t \in ( 0, T ] $ it holds that 
\begin{equation}
\label{eq:integrated_absolute_sum_FB_tilde_2nd_conclude}
\begin{split}
&
  \int^{ t }_0
  \E\left[
    \big\|
      \tilde{F}_{ s, t }\big(
        \bar{Y}_s, Y_s, Y_{ \floor{ s }{ h } }
    \big)
    \big\|_{\mathcal{V}}
  \right]
  +
  \E\left[
    \big\|
      \tilde{B}_{ s, t }\big(
        \bar{Y}_s, Y_s, Y_{ \floor{ s }{ h } }
      \big)
    \big\|_{\mathcal{V}}
  \right]
  ds
\\&\leq
  \tfrac{
    \psiC
    \,
    | \SGchi{0} |^\power
    \,
    \SGchi{\rho} \,
    \varsigma_{ F, B } 
    \,
    K_{ \power + 3 } 
    \, 
    h^\rho
    \,
    t^{ ( 1 - \vartheta - \rho ) }
  }{ 
    ( 1 - \vartheta - \rho ) 
  }
  \Bigg[
    \SGchi{ \rho+\vartheta } + 2 \, \SGchi{ \nicefrac{ \vartheta }{ 2 } } \, \SGchi{ \rho+\nicefrac{ \vartheta }{ 2 } }
    +
    2 \, \SGchi{\rho} \,
    ( | \SGchi{ \nicefrac{ \vartheta }{ 2 } } |^2 + \SGchi{\vartheta} )
\\ & 
    \cdot
    \bigg(
      \tfrac{
        \SGchi{ \vartheta } \, t^{ ( 1 - \vartheta ) } \,
        \| F \|_{ \operatorname{Lip}^0( V, V_{ -\vartheta } ) }
      }{ 
        ( 1 - \vartheta ) 
      }
      +
      \tfrac{
        \BDG{\power+3}{} \,
        \SGchi{ \nicefrac{\vartheta}{2} } 
        \, 
        \sqrt{ 
          t^{ ( 1 - \vartheta ) } 
        }
        \,
        \| B \|_{ \operatorname{Lip}^0( V, \gamma( U, V_{ - \vartheta / 2 } ) ) }
      }{
        \sqrt{ 1 - \vartheta }
      }
    \bigg)
  \Bigg]
  .
\end{split}
\end{equation}
Putting~\eqref{eq:integrated_absolute_sum_FB_tilde_2nd_conclude} 
into 
\eqref{eq:mild_ito_2nd} 
proves \eqref{eq:weak_distance}.
This finishes
the proof of Proposition~\ref{prop:weak_temporal_regularity_2nd}.
\end{proof}

\section{Weak error estimates for exponential Euler approximations of 
SPDEs with mollified nonlinearities}
\label{sec:euler_integrated_mollified}

\subsection{Regularity properties for solutions of infinite dimensional Kolmogorov equations in Banach spaces}

\begin{lemma}
\label{lem:markov.lip}
Assume the setting in Section~\ref{sec:global_setting},
let
$
  \varphi \in \operatorname{Lip}^0(V,\mathcal{V})
$, 
$
  F \in \operatorname{Lip}^0(V,V)
$, 
$ 
  B \in \operatorname{Lip}^0(V,\gamma(U,V))
$, 
let
$
A \colon D(A)
\subseteq
V \rightarrow V
$
be a generator of a strongly continuous analytic semigroup
with 
$
\operatorname{spectrum}( A )
\subseteq
\{
z \in \mathbb{C}
\colon
\text{Re}( z ) < \eta
\}
$,
let 
$
  X^x \colon [0,T] \times \Omega \to V
$, 
$ x \in V $, 
be $(\mathcal{F}_t)_{t\in[0,T]}$-predictable stochastic processes 
which satisfy for all $x\in V$
that 
$
  \sup_{t\in[0,T]}
  \ES\big[\|X^x_t\|^2_V\big]
  < \infty
$
and which satisfy
that for all $x\in V$, $t\in[0,T]$ it holds $\P$-a.s.\ that
\begin{equation}
  X^x_t
  =
  e^{tA} x
  +
  \int^t_0
  e^{(t-s)A} F(X^x_s)
  \, ds
  +
  \int^t_0
  e^{(t-s)A} B(X^x_s)
  \, dW_s
  ,
\end{equation}
let $Y\colon[0,T]\times\Omega\to V$ be a continuous $(\mathcal{F}_t)_{t\in[0,T]}$-adapted stochastic process which satisfies for all $t\in[0,T]$ that 
$
  \ES\big[\|Y_t\|_V\big]
  < \infty
$
and which satisfies that for all $t\in[0,T]$ it holds $\P$-a.s.\ that
\begin{equation}
  Y_t
  =
  e^{tA} Y_0
  +
  \int^t_0
  e^{(t-s)A} F(Y_s)
  \, ds
  +
  \int^t_0
  e^{(t-s)A} B(Y_s)
  \, dW_s,
\end{equation}
and let $u\colon[0,T]\times V\to\mathcal{V}$
be the function which satisfies for all $x\in V$, $t\in[0,T]$
that 
$
  u(t,x)
  =
  \ES[\varphi(X^x_{T-t})]
$.
Then 
\begin{enumerate}[(i)]
	\item 
	\label{item:test.markov}
	it holds for all $s,t\in[0,T]$ that 
	$
	\ES\big[\|\varphi(Y_t)\|_{\mathcal{V}}+\|u(t,Y_s)\|_{\mathcal{V}}\big]
	< \infty
	$
	and 
	\item
	\label{item:simple.markov}
	it holds for all $t,h\in[0,T]$ with $t+h\leq T$
	that 
	\begin{equation}
	\ES[\varphi(Y_{T-t})]
	=
	\ES[u(t+h,Y_h)]
	.
	\end{equation}
\end{enumerate}
\end{lemma}
\begin{proof}
Throughout this proof let 
$\psi_n\colon\mathcal{V}\to\mathcal{V}$, $n\in\N$, 
be the functions which satisfy for all $n\in\N$, $v\in\mathcal{V}$ that 
\begin{equation}
  \psi_n(v)=
  \begin{cases}
  v &\colon \|v\|_{\mathcal{V}} \leq n\\
  \frac{nv}{\|v\|_{\mathcal{V}}}&\colon \|v\|_{\mathcal{V}} > n
  \end{cases},
\end{equation}
let 
$
  \varphi_n \colon V \to \mathcal{V}
$, 
$ n\in\N $, 
be the functions which satisfy for all $n\in\N$, $v\in\mathcal{V}$ that 
\begin{equation}
  \varphi_n(v)
  =
  \psi_n(\varphi(v))
  ,
\end{equation}
and let 
$
  u_n \colon [0,T] \times V \to \mathcal{V}
$, $n\in\N$, 
be the functions which satisfy for all $n\in\N$, $x\in V$, $t\in[0,T]$ 
that 
\begin{equation}
  u_n(t,x)
  =
  \ES[\varphi_n(X^x_{T-t})]
  .
\end{equation}
Observe that for all $n\in\N$ 
it holds that 
$\varphi_n \in \mathcal{M}(\mathcal{B}(V),\mathcal{B}(\mathcal{V}))$
and 
\begin{equation}
\label{eq:Bb.approximation}
  \sup_{v\in V}
  \|\varphi_n(v)\|_{\mathcal{V}}
  \leq n.
\end{equation}
We note that the Burkholder-Davis-Gundy type inequality in, e.g., \cite[Corollary~3.10]{vvw07}, Gronwall's lemma, Fatou's lemma, and the fact that 
$F\in\operatorname{Lip}^0(V,V)$ 
and 
$B\in\operatorname{Lip}^0(V,\gamma(U,V))$ 
ensure that 
for every probability space 
$
  (\tilde{\Omega},\tilde{\mathcal{F}},\tilde{\P})
$
with a normal filtration 
$(\tilde{\mathcal{F}}_t)_{t\in[0,T]}$, 
every $\operatorname{Id}_U$-cylindrical 
$
  (\tilde{\Omega},\tilde{\mathcal{F}},\tilde{\P},(\tilde{\mathcal{F}}_t)_{t\in[0,T]})
$-Wiener process 
$(\tilde{W}_t)_{t\in[0,T]}$, 
and all continuous $(\tilde{\mathcal{F}}_t)_{t\in[0,T]}$-adapted stochastic processes 
$\tilde{X}^{(i)}\colon[0,T]\times\tilde{\Omega}\to V$,
$i\in\{1,2\}$, 
which satisfy 
$\tilde{\P}(\tilde{X}^{(1)}_0=\tilde{X}^{(2)}_0)=1$
and which satisfy that for all $i\in\{1,2\}$, $t\in[0,T]$ it holds $\tilde{\P}$-a.s.\ that 
\begin{equation}
  \tilde{X}^{(i)}_t
  =
  e^{tA} \tilde{X}^{(i)}_0
  +
  \int^t_0
  e^{(t-s)A} F(\tilde{X}^{(i)}_s)
  \, ds
  +
  \int^t_0
  e^{(t-s)A} B(\tilde{X}^{(i)}_s)
  \, d\tilde{W}_s
\end{equation}
it holds that 
$
  \tilde{\P}(\sup_{t\in[0,T]} \|\tilde{X}^{(1)}_t-\tilde{X}^{(2)}_t\|_V=0)
  =1
$
(cf., e.g., Kunze~\cite[Theorem~5.6]{Kunze2010arXiv}).
This and, e.g., Kunze~\cite[Theorem~3.6, Theorem~5.3, \& Proposition~6.9]{Kunze2013} 
guarantee the uniqueness in law for solutions of the local martingale problem associated to $(A,F,B)$ (see, e.g., \cite[(3.2)]{Kunze2013}).
Moreover, note that, e.g., Theorem 6.2 in Van Neerven et al.~\cite{vvw08} ensure that for every probability space 
$
(\tilde{\Omega},\tilde{\mathcal{F}},\tilde{\P})
$
with a normal filtration 
$(\tilde{\mathcal{F}}_t)_{t\in[0,T]}$ 
and every $\operatorname{Id}_U$-cylindrical 
$
(\tilde{\Omega},\tilde{\mathcal{F}},\tilde{\P},(\tilde{\mathcal{F}}_t)_{t\in[0,T]})
$-Wiener process 
$(\tilde{W}_t)_{t\in[0,T]}$
there exist continuous $(\tilde{\mathcal{F}}_t)_{t\in[0,T]}$-adapted stochastic processes 
$\tilde{X}^x\colon[0,T]\times\tilde{\Omega}\to V$,
$x \in V$, 
which satisfy that for all $x\in V$, $t\in[0,T]$ it holds $\tilde{\P}$-a.s.\ that 
\begin{equation}
\tilde{X}^x_t
=
e^{tA} x
+
\int^t_0
e^{(t-s)A} F(\tilde{X}^x_s)
\, ds
+
\int^t_0
e^{(t-s)A} B(\tilde{X}^x_s)
\, d\tilde{W}_s
.
\end{equation}
This and, e.g., \cite[Theorem~3.6 \& Proposition~6.9]{Kunze2013}
assure that the local martingale problem associated to $(A,F,B)$ is well-posed (see, e.g., \cite[Definition~2.3]{Kunze2013}).
Combining this with~\eqref{eq:Bb.approximation} and, e.g., \cite[Theorem~4.2 and item~(4) of Theorem~2.2]{Kunze2013} implies that for all $n\in\N$, $t,h\in[0,T]$ with $t+h\leq T$ it holds that 
\begin{equation}
\label{eq:markov.Bb}
  \ES[\varphi_n(Y_{T-t})]
  =
  \ES\big[\ES[\varphi(Y_{(T-t-h)+h})\vert\mathcal{F}_h]\big]
  =
  \ES[u_n(t+h,Y_h)]
  .
\end{equation}
Next note that for all $n\in\N$, $v\in V$
it holds that 
\begin{equation}
  \|\varphi_n(v)-\varphi(v)\|_{\mathcal{V}}
  \leq
  2 \mathbbm{1}_{\{y\in\mathcal{V}\colon\|y\|_{\mathcal{V}}> n\}}(\varphi(v)) \,
  \|\varphi(v)\|_{\mathcal{V}}
  \leq
  2\|\varphi\|_{\operatorname{Lip}^0(V,\mathcal{V})} \,
  (1+\|v\|_V)
  .
\end{equation}
This implies that for all $n\in\N$, $x\in V$, $t\in[0,T]$
it holds that 
\begin{equation}
\label{eq:Bb.approx.majorante.Y}
  \ES\big[\|\varphi_n(Y_t)-\varphi(Y_t)\|_{\mathcal{V}}\big]
  \leq
  2\|\varphi\|_{\operatorname{Lip}^0(V,\mathcal{V})} \,
  \big(1+\ES\big[\|Y_t\|_V\big]\big)
  < \infty  
\end{equation}
and 
\begin{equation}
\label{eq:Bb.approx.majorante.X}
  \ES\big[\|\varphi_n(X^x_t)-\varphi(X^x_t)\|_{\mathcal{V}}\big]
  \leq
  2\|\varphi\|_{\operatorname{Lip}^0(V,\mathcal{V})} \,
  \big(1+\ES\big[\|X^x_t\|_V\big]\big)
  < \infty.
\end{equation}
Note that~\eqref{eq:Bb.approximation} and~\eqref{eq:Bb.approx.majorante.Y} show that for all $x\in V$, $t\in[0,T]$ it holds that 
\begin{equation}
\label{eq:test.Y}
\ES\big[\|\varphi(Y_t)\|_{\mathcal{V}}\big]
< \infty.
\end{equation}
Moreover, combining~\eqref{eq:Bb.approx.majorante.Y}--\eqref{eq:Bb.approx.majorante.X} with Lebesgue's theorem of dominated convergence and the fact that 
\begin{equation}
\label{eq:Bb.pointwise.approximation}
  \forall \, v \in V \colon
  \limsup_{n\to\infty}
  \|\varphi_n(v)-\varphi(v)\|_{\mathcal{V}}
  =0
\end{equation}
yields that for all $x\in V$, $t\in[0,T]$ it holds that 
\begin{equation}
\label{eq:Bb.pathwise.convergence}
  \limsup_{n\to\infty}
  \big\|\ES[\varphi_n(Y_t)]-\ES[\varphi(Y_t)]\big\|_{\mathcal{V}}
  +
  \limsup_{n\to\infty}
  \|u_n(t,x)-u(t,x)\|_{\mathcal{V}}
  =
  0.
\end{equation}
Next observe that for all $t\in[0,T]$ it holds that 
\begin{equation}
\label{eq:apriori.bound}
  \sup_{x\in V}
  \left[
  \frac{\ES\big[\|X^x_t\|_V\big]}{(1+\|x\|_V)}
  \right]
  \leq
  \sup_{x\in V}
  \left[
  \frac{\|X^x_t\|_{\lpn{2}{\P}{V}}}{\max\{1,\|x\|_V\}}
  \right]
  < \infty
\end{equation}
(cf., e.g., Cox \& Van Neerven~\cite[(2.1) and Theorem~2.7]{CoxVanNeerven2013}).
Next observe that~\eqref{eq:Bb.approx.majorante.X} imply that for all $n\in\N$, $x\in V$, $t\in[0,T]$ it holds that
\begin{equation}
\begin{split}
  \|u_n(t,x)-u(t,x)\|_{\mathcal{V}}
  &\leq
  \ES\big[\|\varphi_n(X^x_{T-t})-\varphi(X^x_{T-t})\|_{\mathcal{V}}\big]
  \\&\leq
  2\|\varphi\|_{\operatorname{Lip}^0(V,\mathcal{V})} \,
  \big(1+\ES\big[\|X^x_{T-t}\|_V\big]\big)
  \\&\leq
  2\|\varphi\|_{\operatorname{Lip}^0(V,\mathcal{V})} \,
  \bigg(1+  
  \sup_{v\in V}
    \left[
    \frac{\ES\big[\|X^v_{T-t}\|_V\big]}{(1+\|v\|_V)}
    \right]
    (1+\|x\|_V)
    \bigg)
    .
\end{split}
\end{equation}
This and~\eqref{eq:apriori.bound} yield that for all 
$n\in\N$, 
$s,t\in[0,T]$
it holds that 
\begin{equation}
\label{eq:u.majorante}
\begin{split}
&
  \ES\big[\|u_n(t,Y_s)-u(t,Y_s)\|_{\mathcal{V}}\big]
  \leq
  2\|\varphi\|_{\operatorname{Lip}^0(V,\mathcal{V})} \,
  \bigg(1+  
  \sup_{v\in V}
    \left[
    \frac{\ES\big[\|X^v_t\|_V\big]}{(1+\|v\|_V)}
    \right]
    (1+\ES\big[\|Y_s\|_V\big])
    \bigg)
    < \infty.
\end{split}
\end{equation}
This and~\eqref{eq:Bb.approximation} show that for all 
$s,t\in[0,T]$ it holds that 
\begin{equation}
\ES\big[\|u(t,Y_s)\|_{\mathcal{V}}\big]
< \infty.
\end{equation}
This and~\eqref{eq:test.Y} prove item~\eqref{item:test.markov}.
Next we combine~\eqref{eq:Bb.pathwise.convergence} and~\eqref{eq:u.majorante} with Lebesgue's theorem of dominated convergence to obtain that for all 
$s,t\in[0,T]$
it holds that 
\begin{equation}
  \limsup_{n\to\infty}
  \big\|\ES\big[u_n(t,Y_s)-u(t,Y_s)\big]\big\|_{\mathcal{V}}
  =0.
\end{equation}
This, \eqref{eq:markov.Bb}, and~\eqref{eq:Bb.pathwise.convergence} yield that for all $t,h\in[0,T]$ with $t+h\leq T$ it holds that
\begin{equation}
  \ES[\varphi(Y_{T-t})]
  =
  \ES[u(t+h,Y_h)]
  .
\end{equation}
This proves item~\eqref{item:simple.markov}.
The proof of Lemma~\ref{lem:markov.lip} is thus completed.
\end{proof}

\begin{lemma}
\label{lem:Kolmogorov}
Assume the setting in Section~\ref{sec:global_setting},
let $ \mathbb{U} \subseteq U $ be an orthonormal basis of $ U $, 
let
$
A \colon D(A)
\subseteq
V \rightarrow V
$
be a generator of a strongly continuous analytic semigroup
with 
$
\operatorname{spectrum}( A )
\subseteq
\{
z \in \mathbb{C}
\colon
\text{Re}( z ) < \eta
\}
$,
	let 
	$
	(
	V_r
	,
	\left\| \cdot \right\|_{ V_r }
	)
	$,
	$ r \in [0,\infty) $,
	be the $\R$-Banach spaces which satisfy for all 
	$r\in[0,\infty)$ that 
	$
	(
	V_r
	,
	\left\| \cdot \right\|_{ V_r }
	)
	=
	(
	D((\eta-A)^r)
	,
	\left\| (\eta-A)^r(\cdot) \right\|_V
	)	
	$, 
let 
$
  \varphi \in \operatorname{Lip}^4( V, \mathcal{V})
$, 
$
  F \in \operatorname{Lip}^4(V,V_1)
$, 
$ 
  B \in \operatorname{Lip}^4(V,\gamma(U,V_1))
$, 
let 
$ \Pi_k \in 
\mathcal{P}\big(\mathcal{P}\big(
\mathcal{P}( \N )
\big)\big)
$, 
$ k \in \N_0 $,
be the sets which satisfy for all $ k \in \N $
that
$ \Pi_0 = \emptyset $
and
\begin{equation}
\Pi_k =
\big\{
C \subseteq \mathcal{P}( \N )
\colon
\left[
\emptyset \notin C
\right]
\wedge
\left[
\cup_{ a \in C }
a
=
\left\{ 1, \dots, k \right\}
\right]
\wedge
\left[
\forall \, a, b \in C \colon
\left(
a \neq b
\Rightarrow
a \cap b = \emptyset
\right)
\right]
\big\}
,
\end{equation}
and for every $ k \in \N $,
$ \varpi \in \Pi_k $
let 
$
I^\varpi_i
\in
\varpi
$, 
$i\in\{1,\ldots,\#_\varpi\}$, 
be the sets which satisfy that 
$
\min( I^\varpi_1 ) < 
\dots < 
\min( 
I_{ \#_\varpi }^{ \varpi } 
)
$,
let 
$
I_{ i, j }^\varpi
\in
I_i^{ \varpi }
$, 
$j\in\{1,\ldots,\#_{I^\varpi_i}\}$, 
$ i \in \{ 1, \dots, \#_\varpi \} $, 
be the natural numbers which satisfy for all 
$ i \in \{ 1, \dots, \#_\varpi \} $ 
that 
$
I_{ i, 1 }^\varpi < I_{ i, 2 }^\varpi < \dots < I_{ i, \#_{ I_i^{ \varpi } } }^\varpi
$, 
and let 
$
[ \cdot ]_i^\varpi
\colon
V^{ k + 1 }
\to 
V^{ 
	\#_{I_i^\varpi} + 1
}
$, 
$
i \in \{ 1, \dots, \#_\varpi \} 
$, 
be the functions which satisfy for all 
$
i \in \{ 1, \dots, \#_\varpi \} 
$, 
$
\mathbf{v} = (v_0, v_1, \dots, v_k)
\in 
V^{ k + 1 }
$
that
$
[ \mathbf{v} ]_i^\varpi
= ( v_0, v_{ I_{ i, 1 }^\varpi } , \dots , v_{ I_{ i, \#_{I_i^\varpi} }^\varpi } )
$.
Then
\begin{enumerate}[(i)]
\item
\label{item:derivative.processes}
		there exist up-to-modifications unique $(\mathcal{F}_t)_{t\in[0,T]}$-predictable stochastic processes
		$
		X^{ k,\mathbf{v} }
		\colon
		[ 0 , T ] \times \Omega
		\to V
		$, 
		$
		\mathbf{v} \in V^{k+1}
		$, 
		$
		k \in \{ 0, 1, 2 \}
		$,
		which satisfy
		for all
		$
		k \in \{ 0, 1, 2 \}
		$,
		$
		\mathbf{v} \in V^{k+1}
		$, 
		$ p \in (0,\infty) $
		that
		$
		\sup_{t\in[0,T]}
		\ES\big[\|X^{k,\mathbf{v}}_t\|^p_V\big]
		< \infty
		$ 
		and 
		which satisfy that
		for all
		$
		k \in \{ 0, 1, 2 \}
		$,
		$
		\mathbf{v} = (v_0,v_1,\ldots,v_k) \in V^{k+1}
		$, 
		$ t \in [0,T] $
		it holds $\P$-a.s.\ that
		\begin{equation}
		\label{eq:derivative.processes}
		\begin{split}
		&\!\!\!
		X_t^{k,\mathbf{v}}
		=
		\mathbbm{1}_{ \{ 0, 1 \} }(k) \, 
		e^{tA}
		v_k  
		\\ &\!\!\!
		+
		\int_0^t
		e^{ ( t - s ) A }
		\Big[
		\mathbbm{1}_{ \{ 0 \} }(k)
		\,
		F(X_s^{0,v_0})
		+
		{\smallsum\limits_{ \varpi\in \Pi_k }}
		F^{ ( \#_\varpi ) }( X_s^{ 0,v_0 } )
		\big(
		X_s^{ \#_{I^\varpi_1}, [ \mathbf{v} ]_1^{ \varpi } }
		,
		\dots
		,
		X_s^{ \#_{I^\varpi_{\#_\varpi}}, [\mathbf{v} ]_{ \#_\varpi }^{ \varpi } }
		\big)
		\Big]
		\, ds
		\\ &\!\!\!
		+
		\int_0^t
		e^{ ( t - s ) A }
		\Big[
		\mathbbm{1}_{ \{ 0 \} }(k)
		\,
		B(X_s^{0,v_0})
		+
		{\smallsum\limits_{ \varpi\in \Pi_k }}
		B^{ ( \#_\varpi ) }( X_s^{ 0,v_0 } )
		\big(
		X_s^{ \#_{I^\varpi_1}, [ \mathbf{v} ]_1^{ \varpi } }
		,
		\dots
		,
		X_s^{ \#_{I^\varpi_{\#_\varpi}}, [\mathbf{v} ]_{ \#_\varpi }^{ \varpi } }
		\big)
		\Big]
		\, dW_s
		,
		\end{split}
		\end{equation}
\item
\label{item:transition.function}
		there exists a unique function
		$\phi\colon[0,T]\times V \to \mathcal{V}$ 
		which satisfies for all 
		$x\in V$, 
		$t\in[0,T]$ 
		that
		$
		\phi(t,x)
		=
		\ES[\varphi(X^{0,x}_t)]
		$,
\item
\label{item:kolmogorov.diff}
it holds for all $t\in[0,T]$ that 
$
  (V \ni x \mapsto \phi(t,x) \in \mathcal{V})
  \in C^4_b(V,\mathcal{V})
$, 
		\item
		\label{item:lem.trans.derivative.integrability}
		it holds for all 
		$ k \in \{1,2\} $, 
		$
		\mathbf{v} = (v_0 , v_1 , \dots , v_k) \in
		V^{k+1}
		$, 
		$ t \in [ 0 , T ] $ 
		that 
		\begin{equation}
		\sum\limits_{
			\varpi \in \Pi_k
		}
		\ES\Big[\big\|
		\varphi^{(\#_\varpi)}(X_t^{0,v_0})
		\big(
		X_t^{ \#_{I^\varpi_1}, [ \mathbf{v} ]_1^{ \varpi } }
		,
		\dots
		,
		X_t^{ \#_{I^\varpi_{\#_\varpi}}, [\mathbf{v} ]_{ \#_\varpi }^{ \varpi } }
		\big)
		\big\|_{\mathcal{V}}\Big]
		< \infty, 
		\end{equation}
		\item
		\label{item:lem.representation}
		it holds for all 
		$ k \in \{1,2\} $, 
		$
		\mathbf{v} \in
		V^k
		$, 
		$ x \in V $, 
		$ t \in [ 0 , T ] $ 
		that 
		\begin{equation}
		\begin{split}
		&
		\big(
		\tfrac{\partial^k}{\partial x^k}
		\phi
		\big)(t,x)
		\mathbf{v}
		=
		\sum\limits_{
			\varpi \in \Pi_k
		}
		\ES
		\Big[
		\varphi^{(\#_\varpi)}(X_t^{0,x})
		\big(
		X_t^{\#_{I^\varpi_1},[(x,\mathbf{v})]_1^{\varpi}}
		,
		\dots
		,
		X_t^{\#_{I^\varpi_{\#_\varpi}},[(x,\mathbf{v})]_{\#_\varpi}^\varpi}
		\big)
		\Big]
		,
		\end{split}
		\end{equation}
\item
\label{item:kolmogorov.c.delta}
it holds for all 
$ k \in \{ 1, 2, 3, 4 \} $, 
$ \delta_1, \dots, \delta_k \in (-\nicefrac{1}{2},0] $
with
$
  \sum^k_{i=1} \delta_i
  > -\nicefrac{1}{2}
$ 
that 
\begin{equation}
\begin{split} 
&
  \sup_{
    t \in (0,T]
  }
  \sup_{ 
    x \in V
  }
  \sup_{ 
    v_1, \dots, v_k \in V \setminus \{ 0 \}
  }
  \left[
  \frac{
    \big\|
      ( 
        \frac{ 
          \partial^k
        }{
          \partial x^k
        }
        \phi
      )( t, x )( v_1, \dots, v_k )
    \big\|_{\mathcal{V}}
  }{
    t^{ 
      (
        \delta_1 + \ldots + \delta_k
      ) 
    }
    \left\| (\eta-A)^{\delta_1} v_1 \right\|_V
    \cdot
    \ldots
    \cdot
    \left\| (\eta-A)^{\delta_k} v_k \right\|_V
  }
  \right]
  < \infty
  ,
\end{split}
\end{equation}
\item
\label{item:kolmogorov.lip.c.delta}
it holds for all 
$k\in\{1,2,3,4\}$, 
$ \delta_1, \dots, \delta_k \in (-\nicefrac{1}{2},0] $
with
$
\sum^k_{i=1} \delta_i
> -\nicefrac{1}{2}
$ 
that 
\begin{equation}
\begin{split} 
&
\!\!\!\!\!\!\!\!\!
\sup_{t \in (0,T] }
\sup_{\substack{x, y \in V, \\ x \neq y}}
  \sup_{ 
  	v_1, \dots, v_k \in V \setminus \{ 0 \}
  }
\left[
\frac{
	\big\|
	\big[
	\big(
	\tfrac{\partial^4}{\partial x^4}
	\phi
	\big)(t,x)
	-
	\big(
	\tfrac{\partial^4}{\partial x^4}
	\phi
	\big)(t,y)
	\big]
	( v_1, \dots, v_k )
	\big\|_{\mathcal{V}}    
}{
    t^{ 
    	(
    	\delta_1 + \ldots + \delta_k
    	) 
    } \,
    \|x-y\|_V
    \cdot
    \left\| (\eta-A)^{\delta_1} v_1 \right\|_V
    \cdot
    \ldots
    \cdot
    \left\| (\eta-A)^{\delta_k} v_k \right\|_V
}
\right]
< \infty
,
\end{split}
\end{equation}
\item
\label{item:base.Lipschitz}
it holds for all $p\in(0,\infty)$ that 
\begin{equation}
\sup_{\substack{x,y\in V, \\ x\neq y}}
\sup_{t\in[0,T]}
\left[
\frac{
	\|X^{0,x}_t-X^{0,y}_t\|_{\lpn{p}{\P}{V}}
}{
\|x-y\|_V
}
\right]
< \infty,
\end{equation}
\item
\label{item:derivative.process.linearity}
it holds for all $k\in\{1,2\}$, $p\in(0,\infty)$ that 
\begin{equation}
  \sup_{t\in[0,T]}
  \sup_{x\in V}
  \sup_{v_1,\ldots,v_k\in V\setminus\{0\}}
  \left[
  \frac{
  	\|X^{k,(x,v_1,\ldots,v_k)}_t\|_{\lpn{p}{\P}{V}}
  	}{
  	  \|v_1\|_V \cdot\ldots\cdot \|v_k\|_V
  	}
  \right]
  < \infty
  ,
\end{equation}
\item
\label{item:derivative.process.continuity}
it holds for all $k\in\{1,2\}$, $p\in(0,\infty)$ that 
\begin{equation}
 \sup_{\substack{x,y\in V, \\ x\neq y}}
\sup_{t\in[0,T]}
\sup_{v_1,\ldots,v_k\in V\setminus\{0\}}
\left[
\frac{
	\|X^{k,(x,v_1,\ldots,v_k)}_t-X^{k,(y,v_1,\ldots,v_k)}_t\|_{\lpn{p}{\P}{V}}
}{
\|x-y\|_V \cdot
\|v_1\|_V \cdot\ldots\cdot \|v_k\|_V
}
\right]
< \infty
,
\end{equation}
\item
\label{item:solution.V1}
it holds for all $x\in V_1$, $t\in[0,T]$ that 
$
  \P(X^{0,x}_t\in V_1)
  =1
$, 
\item
\label{item:V1.moment}
it holds for all 
$p\in(0,\infty)$,
$x\in V_1$, 
$t\in[0,T]$
that
$
  \ES\big[\|X^{0,x}_t\mathbbm{1}_{\{X^{0,x}_t \in V_1\}}\|^p_{V_1}\big]
  < \infty
$, 
\item
\label{item:time.continuity.V1}
it holds for all $l\in\{0,1\}$, $p\in[1,\infty)$, $x \in V_l$ that 
\begin{equation}
  ([0,T] \ni t \mapsto [X^{0,x}_t]_{\P,\mathcal{B}(V_l)} \in \lpnb{p}{\P}{V_l})
  \in C([0,T],\lpnb{p}{\P}{V_l})
  ,
\end{equation}
\item
\label{item:time.continuity.derivative}
it holds for all $k\in\{1,2\}$, $p,r\in(0,\infty)$, $x\in V$, $t\in[0,T]$ that 
\begin{equation}
\limsup_{[0,T] \ni s\to t}
\sup_{v_1,\ldots,v_k\in V_{r\mathbbm{1}_{\{1\}}(k)}\setminus\{0\}}
\left[
\frac{
	\|X^{k,(x,v_1,\ldots,v_k)}_s-X^{k,(x,v_1,\ldots,v_k)}_t\|_{\lpn{p}{\P}{V}}
}{
\|v_1\|_{V_{r\mathbbm{1}_{\{1\}}(k)}} \cdot\ldots\cdot \|v_k\|_{V_{r\mathbbm{1}_{\{1\}}(k)}}
}
\right]
=0
,
\end{equation}
\item
\label{item:space-time.diff}
it holds for all $x\in V_1$ that
$
  ( [0,T] \ni t \mapsto \phi(t,x) \in \mathcal{V} )
  \in C^1([0,T],\mathcal{V})
$, 
\item
\label{item:time.derivative.continuity}
it holds that 
$
  \big([0,T] \times V_1 \ni (t,x) \mapsto \big(\tfrac{\partial}{\partial t}\phi\big)(t,x) \in \mathcal{V}\big)
  \in C([0,T]\times V_1,\mathcal{V})
$, 
\item
\label{item:space.derivative.continuity}
it holds for all $k\in\{1,2\}$, $r\in(0,\infty)$ that 
\begin{multline}
\big([0,T] \times V_r \ni (t,x) \mapsto \big((V_r)^k\ni\mathbf{u}\mapsto\big(\tfrac{\partial^k}{\partial x^k}\phi\big)(t,x)\mathbf{u} \in \mathcal{V}\big) \in L^{(k)}(V_r,\mathcal{V})\big)
\\\in C([0,T]\times V_r,L^{(k)}(V_r,\mathcal{V})),
\end{multline}
\item
\label{item:space.derivative.bounded}
it holds for all $k\in\{1,2\}$ that 
$
  \sup_{t\in[0,T]}
  \sup_{x\in V}
  \big\|\big(\tfrac{\partial^k}{\partial x^k}\phi\big)(t,x)\big\|_{L^{(k)}(V,\mathcal{V})}
  < \infty
$, 
and 
\item
\label{item:kolmogorov.eq}
it holds for all $ x \in V_1 $, $ t \in (0,T] $ that 
\begin{equation}
  \big(\tfrac{\partial}{\partial t}\phi\big)(t,x)=
  \big(\tfrac{\partial}{\partial x}\phi\big)(t,x)(Ax+F(x))
  +\frac{1}{2}
  \sum_{b\in\mathbb{U}}
  \big(\tfrac{\partial^2}{\partial x^2}\phi\big)(t,x)(B(x)b,B(x)b)
  .
\end{equation}
\end{enumerate}
\end{lemma}

\begin{proof}
	Throughout this proof 
	let $\SGchi{}\in\R$ be the real number given by
	$
	\chi=
	\sup_{t\in[0,T]}
	\|e^{tA}\|_{L(V)}
	$. 
	The proof of items~\eqref{item:derivative.processes}--\eqref{item:kolmogorov.lip.c.delta} is entirely analogous to the proof of items~(i)--(v), (vii), \& (x) of Theorem~3.3 in Andersson et al.~\cite{AnderssonHefterJentzenKurniawan2016}. 
	Item~\eqref{item:transition.function} ensures that there exists a unique function
	$\psi\colon[0,T]\times V \to \mathcal{V}$ 
	which satisfies for all 
	$x\in V$, 
	$t\in[0,T]$ 
	that
	\begin{equation}
	\psi(t,x)
	=
	\phi(T-t,x)
	=
	\ES[\varphi(X^{0,x}_{T-t})]
	.
	\end{equation}
	The proof of item~\eqref{item:base.Lipschitz} is entirely analogous to the proof of item~(iii) of Corollary~2.10 in Andersson et al.~\cite{AnderssonJentzenKurniawan2015arXiv}.
	The proof of items~\eqref{item:derivative.process.linearity}--\eqref{item:derivative.process.continuity} is entirely analogous to the proof of items~(ii) \& (iv) of Theorem~2.1 in Andersson et al.~\cite{AnderssonJentzenKurniawan2016a}.
	The fact that 
	$F\in\operatorname{Lip}^0(V,V_1)$, 
	the fact that
	$B\in\operatorname{Lip}^0(V,\gamma(U,V_1))$, 
	and the fact that
	$
	\forall \, x \in V, \, p \in (0,\infty)
	\colon
	\sup_{t\in[0,T]}
	\ES\big[\|X^{0,x}_t\|^p_V\big]
	< \infty
	$ 
	show that for all $x\in V$, $t\in[0,T]$ it holds that 
	\begin{equation}
	\label{eq:integral.V1}
	  \int^t_0
	  \|e^{(t-s)A}F(X^{0,x}_s)\|_{\lpn{p}{\P}{V_1}}
	  +
	  \|e^{(t-s)A}B(X^{0,x}_s)\|^2_{\lpn{p}{\P}{\gamma(U,V_1)}}
	  \, ds
	  < \infty.
	\end{equation}
	This, \eqref{eq:derivative.processes}, Jensen's inequality, and the Burkholder-Davis-Gundy type inequality in, e.g., \cite[Corollary~3.10]{vvw07} prove items~\eqref{item:solution.V1}--\eqref{item:V1.moment}.
	Next note that~\eqref{eq:derivative.processes} implies that for all 
	$x\in V$, 
	$s,t\in[0,T]$ with $s\leq t$ 
	it holds $\P$-a.s.\ that
	\begin{equation}
	\begin{split}
	&
	X^{0,x}_t-X^{0,x}_s
	\\&=
	e^{sA}(e^{(t-s)A}-\operatorname{Id}_V) x
	+
	\int^t_s
	e^{(t-u)A} F(X^{0,x}_u) \,du
	+
	\int^t_s
	e^{(t-u)A} B(X^{0,x}_u) \,dW_u	
	\\&+
	\int^s_0
	e^{(s-u)A}(e^{(t-s)A}-\operatorname{Id}_V) F(X^{0,x}_u) \,du
	+
	\int^s_0
	e^{(s-u)A}(e^{(t-s)A}-\operatorname{Id}_V) B(X^{0,x}_u) \,dW_u	
	.	 
	\end{split}
	\end{equation}
	Combining the Burkholder-Davis-Gundy type inequality in, e.g., \cite[Corollary~3.10]{vvw07} with
	the fact that 
	$F\in\operatorname{Lip}^0(V,V_1)$, 
	the fact that
	$B\in\operatorname{Lip}^0(V,\gamma(U,V_1))$,
	the fact that
	$
	  \forall \, l \in \{0,1\}, \, x \in V_l
	  \colon
	  \limsup_{t\searrow 0}
	  \|(e^{tA}-\operatorname{Id}_V) (\eta-A)^lx\|_V
	  =0
	$,
	and the fact that
	$
	  \forall \, x \in V, \, p \in (0,\infty)
	  \colon
	  \sup_{t\in[0,T]}
	  $
	  $
	  \ES\big[\|X^{0,x}_t\|^p_V\big]
	  < \infty
	$ 
	hence proves item~\eqref{item:time.continuity.V1}.
	Next note that~\eqref{eq:derivative.processes}, the Burkholder-Davis-Gundy type inequality in, e.g., \cite[Corollary~3.10]{vvw07}, the fact that 
	$F\in\operatorname{Lip}^0(V,V_1)$, and the fact that
	$B\in\operatorname{Lip}^0(V,\gamma(U,V_1))$
	show that for all 
	$p\in[2,\infty)$, 
	$x,y\in V_1$, 
	$t\in[0,T]$ 
	it holds that 
	\begin{equation}
	\label{eq:bootstrapping.V1}
	\begin{split}
	\|X^{0,x}_t-X^{0,y}_t\|_{\lpn{p}{\P}{V_1}}
	&\leq
	\|e^{tA}(x-y)\|_{V_1}
	+
	\bigg\|
	\int^t_0
	e^{(t-s)A} (F(X^{0,x}_s)-F(X^{0,y}_s))
	\, ds
	\bigg\|_{\lpn{p}{\P}{V_1}}
	\\&\quad+
	\bigg\|
	\int^t_0
	e^{(t-s)A} (B(X^{0,x}_s)-B(X^{0,y}_s))
	\, dW_s
	\bigg\|_{\lpnb{p}{\P}{V_1}}	
	\\&\leq
	\SGchi{} \, \|x-y\|_{V_1}	
	+
	\int^t_0
	\|e^{(t-s)A} (F(X^{0,x}_s)-F(X^{0,y}_s))\|_{\lpn{p}{\P}{V_1}}
	\, ds
	\\&\quad+
	\BDG{p}{}
	\bigg[
	\int^t_0
	\|e^{(t-s)A} (B(X^{0,x}_s)-B(X^{0,y}_s))\|^2_{\lpn{p}{\P}{\gamma(U,V_1)}}
	\, ds
	\bigg]^{\nicefrac{1}{2}}
	\\&\leq
	\SGchi{} \, \|x-y\|_{V_1}
	+
	T \, \SGchi{} \, |F|_{\operatorname{Lip}^0(V,V_1)} \,
	\bigg[
	\sup_{s\in[0,T]}
	\|X^{0,x}_s-X^{0,y}_s\|_{\lpn{p}{\P}{V}}
	\bigg]
	\\&\quad+
	\BDG{p}{} \, \sqrt{T} \, \SGchi{} \, |B|_{\operatorname{Lip}^0(V,\gamma(U,V_1))} \,
	\bigg[
	\sup_{s\in[0,T]}
	\|X^{0,x}_s-X^{0,y}_s\|_{\lpn{p}{\P}{V}}
	\bigg].	
	\end{split}
	\end{equation}
	Moreover, item~\eqref{item:base.Lipschitz} and the fact that $V_1\subseteq V$ continuously imply that for all 
	$p\in[2,\infty)$ 
	it holds that 
	\begin{equation}
	\sup_{\substack{x,y\in V_1, \\ x\neq y}}
	\sup_{t\in[0,T]}
	\left[
	\frac{
		\|X^{0,x}_t-X^{0,y}_t\|_{\lpn{p}{\P}{V}}
	}{
	\|x-y\|_{V_1}
	}
	\right]
	< \infty.
	\end{equation}
	This, \eqref{eq:bootstrapping.V1}, and Jensen's inequality prove that for all $p\in(0,\infty)$ it holds that 
	\begin{equation}
	\label{eq:base.Lipschitz.V1}
	\sup_{\substack{x,y\in V_1, \\ x\neq y}}
	\sup_{t\in[0,T]}
	\left[
	\frac{
		\|X^{0,x}_t-X^{0,y}_t\|_{\lpn{p}{\P}{V_1}}
	}{
	\|x-y\|_{V_1}
	}
	\right]
	< \infty.
	\end{equation}
	In the next step observe that~\eqref{eq:derivative.processes} implies that for all 
	$k\in\{1,2\}$, 
	$x\in V$, 
	$\mathbf{v}=(v_1,\ldots,v_k)\in V^k$, 
	$s,t\in[0,T]$ with $s\leq t$ 
	it holds $\P$-a.s.\ that
	\begin{equation}
	\begin{split}
	&
	X^{k,(x,\mathbf{v})}_t-X^{k,(x,\mathbf{v})}_s
	=
	\mathbbm{1}_{\{1\}}(k) \,
	e^{sA}(e^{(t-s)A}-\operatorname{Id}_V) v_k
	\\&+
	\int^t_s
	e^{(t-u)A} 
		{\smallsum\limits_{ \varpi\in \Pi_k }}
		F^{ ( \#_\varpi ) }( X_u^{ 0,x } )
		\big(
		X_u^{ \#_{I^\varpi_1}, [ (x,\mathbf{v}) ]_1^{ \varpi } }
		,
		\dots
		,
		X_u^{ \#_{I^\varpi_{\#_\varpi}}, [(x,\mathbf{v}) ]_{ \#_\varpi }^{ \varpi } }
		\big)
	\,du
	\\&+
	\int^t_s
	e^{(t-u)A}
		{\smallsum\limits_{ \varpi\in \Pi_k }}
		B^{ ( \#_\varpi ) }( X_u^{ 0,x } )
		\big(
		X_u^{ \#_{I^\varpi_1}, [ (x,\mathbf{v}) ]_1^{ \varpi } }
		,
		\dots
		,
		X_u^{ \#_{I^\varpi_{\#_\varpi}}, [(x,\mathbf{v}) ]_{ \#_\varpi }^{ \varpi } }
		\big)	
	\,dW_u	
	\\&+
	\int^s_0
	e^{(s-u)A}(e^{(t-s)A}-\operatorname{Id}_V) 
		{\smallsum\limits_{ \varpi\in \Pi_k }}
		F^{ ( \#_\varpi ) }( X_u^{ 0,x } )
		\big(
		X_u^{ \#_{I^\varpi_1}, [ (x,\mathbf{v}) ]_1^{ \varpi } }
		,
		\dots
		,
		X_u^{ \#_{I^\varpi_{\#_\varpi}}, [(x,\mathbf{v}) ]_{ \#_\varpi }^{ \varpi } }
		\big)	
	\,du
	\\&+
	\int^s_0
	e^{(s-u)A}(e^{(t-s)A}-\operatorname{Id}_V)
		{\smallsum\limits_{ \varpi\in \Pi_k }}
		B^{ ( \#_\varpi ) }( X_u^{ 0,x } )
		\big(
		X_u^{ \#_{I^\varpi_1}, [ (x,\mathbf{v}) ]_1^{ \varpi } }
		,
		\dots
		,
		X_u^{ \#_{I^\varpi_{\#_\varpi}}, [(x,\mathbf{v}) ]_{ \#_\varpi }^{ \varpi } }
		\big)	
	\,dW_u	
	.	 
	\end{split}
	\end{equation}
	Combining Jensen's inequality and the Burkholder-Davis-Gundy type inequality in, e.g., \cite[Corollary~3.10]{vvw07} with item~\eqref{item:derivative.process.linearity}, H\"{o}lder's inequality, 
	the fact that 
	$F\in C^2_b(V,V)$, and the fact that
	$B\in C^2_b(V,\gamma(U,V))$
	therefore establish item~\eqref{item:time.continuity.derivative}
	and prove that for all $k\in\{1,2\}$, $p\in(0,\infty)$, $x\in V$, $t\in(0,T]$ it holds that 
	\begin{equation}
	\label{eq:derivative.time.continuous.exclude.0}
	\limsup_{[0,T] \ni s\to t}
	\sup_{v_1,\ldots,v_k\in V\setminus\{0\}}
	\left[
	\frac{
		\|X^{k,(x,v_1,\ldots,v_k)}_s-X^{k,(x,v_1,\ldots,v_k)}_t\|_{\lpn{p}{\P}{V}}
	}{
	\|v_1\|_V \cdot\ldots\cdot \|v_k\|_V
	}
	\right]
	=0
	.
	\end{equation}
	In the next step we combine items~\eqref{item:derivative.processes} \& \eqref{item:solution.V1}--\eqref{item:time.continuity.V1}, 
	the fact that 
	$\varphi\in C^2_b(V,\mathcal{V})$, 
	the fact that
	$F\in\operatorname{Lip}^0(V,V)$, 
	and the fact that
	$B\in\operatorname{Lip}^0(V,\gamma(U,V))$ 	
	with the standard It{\^o} formula in Theorem~2.4 in
	Brze\'{z}niak et al.~\cite{bvvw08} to obtain that for all 
	$ x \in V_1 $, 
	$
	  t\in[0,T]
	$
	it holds $\P$-a.s.\ that 
	\begin{equation}
	\label{eq:kolmogorov.ito}
\begin{split}
	\varphi(X^{0,x}_t)
	&=
	\varphi(x)
	+
	\int^t_0
	\varphi'(X^{0,x}_s)(AX^{0,x}_s+F(X^{0,x}_s))
	\, ds
	+
	\int^t_0
	\varphi'(X^{0,x}_s) B(X^{0,x}_s)
	\, dW_s
	\\&\quad+
	\int^t_0
	\frac{1}{2} \,
	{\smallsum\limits_{b\in\mathbb{U}}}
	\varphi''(X^{0,x}_s)(B(X^{0,x}_s)b,B(X^{0,x}_s)b)
	\, ds
	.
\end{split}
	\end{equation}
	Lemma~\ref{lem:gamma.estimate}, items~\eqref{item:solution.V1}--\eqref{item:time.continuity.V1}, the fact that 
	$\varphi\in C^2_b(V,\mathcal{V})$, 
	the fact that
	$F\in\operatorname{Lip}^0(V,V)$, and the fact that
	$B\in\operatorname{Lip}^0(V,\gamma(U,V))$ 
	show that for all 
	$x\in V_1$, $t\in[0,T]$ it holds that 
	\begin{equation}
	\begin{split}
	&
	\int^t_0
	\ES\bigg[
	\|\varphi'(X^{0,x}_s)(AX^{0,x}_s+F(X^{0,x}_s))\|_{\mathcal{V}}
	+
	\bigg\|
	{\smallsum\limits_{b\in\mathbb{U}}}
	\varphi''(X^{0,x}_s)(B(X^{0,x}_s)b,B(X^{0,x}_s)b)	
	\bigg\|_{\mathcal{V}}
	\\&+
	\|\varphi'(X^{0,x}_s) B(X^{0,x}_s)\|^2_{\gamma(U,\mathcal{V})}
	\bigg]
	\, ds
	< \infty.
	\end{split}
	\end{equation}
	This and~\eqref{eq:kolmogorov.ito} imply that for all 
	$ x \in V_1 $, 
	$
	t\in[0,T]
	$
	it holds that 
	\begin{equation}
	\label{eq:time.ito}
\begin{split}
	\phi(t,x)
	&=
	\ES[\varphi(X^{0,x}_t)]
	\\&=\varphi(x)
	+
	\ES\bigg[
	\int^t_0
	\varphi'(X^{0,x}_s)(AX^{0,x}_s+F(X^{0,x}_s))
	\, ds
	\bigg]
	+
	\ES\bigg[
	\int^t_0
	\varphi'(X^{0,x}_s)B(X^{0,x}_s)
	\, dW_s
	\bigg]	
	\\&\quad+
	\ES\bigg[
	\int^t_0
	\frac{1}{2} \,
	{\smallsum\limits_{b\in\mathbb{U}}}
	\varphi''(X^{0,x}_s)(B(X^{0,x}_s)b,B(X^{0,x}_s)b)
	\, ds
	\bigg]
	\\&=\phi(0,x)
	+
	\int^t_0
	\ES\big[\varphi'(X^{0,x}_s)(AX^{0,x}_s+F(X^{0,x}_s))\big]
	\, ds
	\\&\quad+
	\int^t_0
	\frac{1}{2} \,
	\ES\bigg[
	{\smallsum\limits_{b\in\mathbb{U}}}
	\varphi''(X^{0,x}_s)(B(X^{0,x}_s)b,B(X^{0,x}_s)b)
	\bigg]
	\, ds
	.
\end{split}
	\end{equation}	
	Furthermore, note that H\"{o}lder's inequality, 
	Lemma~\ref{lem:gamma.estimate}, 
	items~\eqref{item:base.Lipschitz} \& \eqref{item:solution.V1}--\eqref{item:time.continuity.V1},
	\eqref{eq:base.Lipschitz.V1},
	the fact that 
	$\varphi\in\operatorname{Lip}^2(V,\mathcal{V})$, 
	the fact that
	$F\in\operatorname{Lip}^0(V,V)$, 
	and the fact that
	$B\in\operatorname{Lip}^0(V,\gamma(U,V))$
	show that 
	\begin{multline}
	  \bigg(
	  [0,T]\times V_1 \ni (t,x) \mapsto \ES\big[\varphi'(X^{0,x}_t)(AX^{0,x}_t+F(X^{0,x}_t))\big]
	  \\+
	\frac{1}{2} \,
	\ES\bigg[
	{\smallsum\limits_{b\in\mathbb{U}}}
	\varphi''(X^{0,x}_t)(B(X^{0,x}_t)b,B(X^{0,x}_t)b)
	\bigg] \in \mathcal{V}	 \
	\bigg) \in C([0,T]\times V_1,\mathcal{V})
	. 
	\end{multline}
	This, \eqref{eq:time.ito}, and the fact that 
	\begin{equation}
	\label{eq:FTC}
	  \forall \, g \in C([0,T],\mathcal{V}), \, t \in [0,T]
	  \colon
	  \limsup_{[-t,T-t]\setminus\{0\}\ni h \to 0}
	  \bigg\|\frac{1}{h}\int^{t+h}_t g(s) \, ds - g(t)\bigg\|_{\mathcal{V}}
	  =0
	\end{equation}
	prove items~\eqref{item:space-time.diff}--\eqref{item:time.derivative.continuity}.	
	Moreover, H\"{o}lder's inequality, \eqref{eq:derivative.time.continuous.exclude.0},  items~\eqref{item:lem.trans.derivative.integrability}--\eqref{item:lem.representation} \& \eqref{item:base.Lipschitz}--\eqref{item:time.continuity.derivative}, the fact that 
	$\varphi\in\operatorname{Lip}^2(V,\mathcal{V})$,
	and the fact that
	$
	  \forall \, r \in (0,\infty) \colon
	  V_r \subseteq V 
	$
	continuously establish items~\eqref{item:space.derivative.continuity}--\eqref{item:space.derivative.bounded}
	and prove that for all $k\in\{1,2\}$ it holds that 
	\begin{equation}
	\label{eq:derivative.time.continuity.exclude.0}
	\big((0,T] \times V \ni (t,x) \mapsto \big(\tfrac{\partial^k}{\partial x^k}\phi\big)(t,x) \in L^{(k)}(V,\mathcal{V})\big)
	\in C((0,T]\times V,L^{(k)}(V,\mathcal{V}))
	.
	\end{equation}
	It thus remains to prove item~\eqref{item:kolmogorov.eq}.
	For this observe that Lemma~\ref{lem:markov.lip} and the fact that for every $x\in V$ it holds that $X^{0,x}$ has a continuous modification imply that for all 
	$x\in V$, 
	$t,h\in[0,T]$ 
	with 
	$t+h\leq T$
	it holds that 
	\begin{equation}
	\label{eq:markovs}
	\begin{split}
	  \psi(t,x)
	  &=
	  \ES[\varphi(X^{0,x}_{T-t})]
	  =
	  \ES[\psi(t+h,X^{0,x}_h)]
	  .
	\end{split}
	\end{equation}
	Moreover, Lemma~\ref{lem:gamma.estimate}, items~\eqref{item:kolmogorov.diff} \& \eqref{item:solution.V1}--\eqref{item:time.continuity.V1}, the fact that 
	$F\in\operatorname{Lip}^0(V,V)$,
	and the fact that
	$B\in\operatorname{Lip}^0(V,\gamma(U,V))$
	ensure that for all 
	$x\in V_1$, 
	$t\in[0,T]$, 
	$h\in[0,T-t]$
	it holds that 
	\begin{equation}
	\begin{split}
	&
	\int^h_0
	\ES\bigg[
	\Big\|\big(\tfrac{\partial}{\partial x}\psi\big)(t+h,X^{0,x}_s)(AX^{0,x}_s+F(X^{0,x}_s))\Big\|_{\mathcal{V}}
	+
	\bigg\|
	{\smallsum\limits_{b\in\mathbb{U}}}
	\big(\tfrac{\partial^2}{\partial x^2}\psi\big)(t+h,X^{0,x}_s)(B(X^{0,x}_s)b,B(X^{0,x}_s)b)	
	\bigg\|_{\mathcal{V}}
	\\&+
	\Big\|\big(\tfrac{\partial}{\partial x}\psi\big)(t+h,X^{0,x}_s)B(X^{0,x}_s)\Big\|^2_{\gamma(U,\mathcal{V})}
	\bigg]	
	\, ds
	< \infty. 
	\end{split}
	\end{equation}
	This, \eqref{eq:markovs}, 
	item~\eqref{item:kolmogorov.diff}, 
	and the standard It{\^o} formula in Theorem~2.4 in
	Brze\'{z}niak et al.~\cite{bvvw08} yield that for all 
	$x\in V_1$, 
	$t\in[0,T)$, 
	$h\in[0,T-t]$
	it holds that 
	\begin{equation}
	\label{eq:kolmogorov.ito.2}
	\begin{split}
	&
	\psi(t+h,x)-\psi(t,x)
	\\&=
	\psi(t+h,x)-\ES[\psi(t+h,X^{0,x}_h)]
	\\&=
	-\ES\bigg[\int^h_0	
	\big(\tfrac{\partial}{\partial x}\psi\big)(t+h,X^{0,x}_s)(AX^{0,x}_s+F(X^{0,x}_s))
	\, ds
	\bigg]
	\\&\quad-\ES\bigg[\int^h_0	
	\big(\tfrac{\partial}{\partial x}\psi\big)(t+h,X^{0,x}_s)B(X^{0,x}_s)
	\, dW_s
	\bigg]		
	\\&\quad-
	\ES\bigg[\int^h_0	
	\frac{1}{2} \,
	{\smallsum\limits_{b\in\mathbb{U}}}
	\big(\tfrac{\partial^2}{\partial x^2}\psi\big)(t+h,X^{0,x}_s)(B(X^{0,x}_s)b,B(X^{0,x}_s)b)
	\, ds
	\bigg]
	\\&=
	-\int^h_0	
	\ES\big[
	\big(\tfrac{\partial}{\partial x}\psi\big)(t+h,X^{0,x}_s)(AX^{0,x}_s+F(X^{0,x}_s))
	\big]
	\, ds
	\\&\quad-
	\int^h_0	
	\frac{1}{2} \,
	\ES\bigg[
	{\smallsum\limits_{b\in\mathbb{U}}}
	\big(\tfrac{\partial^2}{\partial x^2}\psi\big)(t+h,X^{0,x}_s)(B(X^{0,x}_s)b,B(X^{0,x}_s)b)
	\bigg]
	\, ds
	.	
	\end{split}
	\end{equation}
	Next observe that H\"{o}lder's inequality and Lemma~\ref{lem:gamma.estimate} show that for all 
	$x\in V_1$, 
	$t\in[0,T)$
	it holds that 
	\begin{equation}
	\label{eq:intermediate.kolmogorov}
	\begin{split}
	&
	\limsup_{(0,T-t]\ni h \to 0}
	\frac{1}{h}
	\bigg\|
	\int^h_0	
	\ES\Big[
	\big[\big(\tfrac{\partial}{\partial x}\psi\big)(t+h,X^{0,x}_s)-\big(\tfrac{\partial}{\partial x}\psi\big)(t,X^{0,x}_s)\big](AX^{0,x}_s+F(X^{0,x}_s))
	\Big]
	\, ds
	\\&\quad+
	\int^h_0	
	\frac{1}{2} \,
	\ES\bigg[
	{\smallsum\limits_{b\in\mathbb{U}}}
	\big[\big(\tfrac{\partial^2}{\partial x^2}\psi\big)(t+h,X^{0,x}_s)-\big(\tfrac{\partial^2}{\partial x^2}\psi\big)(t,X^{0,x}_s)\big](B(X^{0,x}_s)b,B(X^{0,x}_s)b)
	\bigg]
	\, ds	
	\bigg\|_{\mathcal{V}}
	\\&\leq
	\limsup_{(0,T-t]\ni h \to 0}
	\frac{1}{h}
	\bigg[
	\int^h_0	
	\Big(\ES\Big[\big\|\big(\tfrac{\partial}{\partial x}\psi\big)(t+h,X^{0,x}_s)-\big(\tfrac{\partial}{\partial x}\psi\big)(t,X^{0,x}_s)\big\|^2_{L(V,\mathcal{V})}\Big]\Big)^{\nicefrac{1}{2}}
	\\&\quad\cdot
	\|AX^{0,x}_s+F(X^{0,x}_s)\|_{\lpn{2}{\P}{V}}
	\, ds
	\\&\quad+
	\int^h_0	
	\frac{1}{2} \,
	\Big(\ES\Big[\big\|\big(\tfrac{\partial^2}{\partial x^2}\psi\big)(t+h,X^{0,x}_s)-\big(\tfrac{\partial^2}{\partial x^2}\psi\big)(t,X^{0,x}_s)\big\|^2_{L^{(2)}(V,\mathcal{V})}\Big]\Big)^{\nicefrac{1}{2}}
	\\&\quad\cdot
	\|B(X^{0,x}_s)\|^2_{\lpn{4}{\P}{\gamma(U,V)}}
	\, ds
	\bigg]
	\\&\leq
	\bigg[
	\limsup_{[0,T-t]\ni h \to 0}
	\sup_{s\in[0,h]}
	\Big(\ES\Big[\big\|\big(\tfrac{\partial}{\partial x}\psi\big)(t+h,X^{0,x}_s)-\big(\tfrac{\partial}{\partial x}\psi\big)(t,X^{0,x}_s)\big\|^2_{L(V,\mathcal{V})}\Big]\Big)^{\nicefrac{1}{2}}
	\bigg]
	\\&\quad\cdot
	\bigg[
	\sup_{s\in[0,T]}
	\|AX^{0,x}_s+F(X^{0,x}_s)\|_{\lpn{2}{\P}{V}}
	\bigg]
	\\&\quad+
	\bigg[
	\limsup_{[0,T-t]\ni h \to 0}
	\sup_{s\in[0,h]}
	\Big(\ES\Big[\big\|\big(\tfrac{\partial^2}{\partial x^2}\psi\big)(t+h,X^{0,x}_s)-\big(\tfrac{\partial^2}{\partial x^2}\psi\big)(t,X^{0,x}_s)\big\|^2_{L^{(2)}(V,\mathcal{V})}\Big]\Big)^{\nicefrac{1}{2}}
	\bigg]
	\\&\quad\cdot
	\bigg[
	\sup_{s\in[0,T]}
	\|B(X^{0,x}_s)\|^2_{\lpn{4}{\P}{\gamma(U,V)}}
	\bigg]
	.
	\end{split}
	\end{equation}
	Note that~\eqref{eq:derivative.time.continuity.exclude.0} and item~\eqref{item:time.continuity.V1}
	imply that for all 
	$k\in\{1,2\}$, 
	$x\in V$, 
	$\varepsilon\in(0,\infty)$,
	$t_0\in[0,T-t)$,
	$s_0\in[0,T]$
	it holds that 
	\begin{equation}	
	\limsup_{[0,T-t)\times[0,T]\ni(t,s)\to(t_0,s_0)}
	\P\Big(\big\|(\tfrac{\partial^k}{\partial x^k}\psi)(t,X^{0,x}_s)-(\tfrac{\partial^k}{\partial x^k}\psi)(t_0,X^{0,x}_{s_0})\big\|_{L^{(k)}(V,\mathcal{V})}\geq\varepsilon\Big)
	=0
	.
	\end{equation}
	Item~\eqref{item:space.derivative.bounded} and the Vitali convergence theorem in, e.g., Proposition~4.5 in 
	Hutzenthaler et al.~\cite{HutzenthalerJentzenSalimova2016arXiv} therefore imply that for all 
	$k\in\{1,2\}$, 
	$x\in V$, 
	$t_0\in[0,T-t)$,
	$s_0\in[0,T]$
	it holds that 
	\begin{equation}
	\limsup_{[0,T-t)\times[0,T]\ni(t,s)\to(t_0,s_0)}
	\ES\Big[\big\|\big(\tfrac{\partial^k}{\partial x^k}\psi\big)(t,X^{0,x}_s)-\big(\tfrac{\partial^k}{\partial x^k}\psi\big)(t_0,X^{0,x}_{s_0})\big\|^2_{L^{(k)}(V,\mathcal{V})}\Big]
	=0
	.	
	\end{equation}
	This shows that for all 
	$k\in\{1,2\}$, 
	$x\in V$, 
	$t\in[0,T)$ 
	it holds that 
	\begin{equation}
	\begin{split}
	&
	\limsup_{[0,T-t]\ni h \to 0}
	\sup_{s\in[0,h]}
	\ES\Big[\big\|\big(\tfrac{\partial^k}{\partial x^k}\psi\big)(t+h,X^{0,x}_s)-\big(\tfrac{\partial^k}{\partial x^k}\psi\big)(t,X^{0,x}_s)\big\|^2_{L^{(k)}(V,\mathcal{V})}\Big]
	\\&\leq
	\limsup_{[0,T-t]\ni h \to 0}
	\sup_{s\in[0,h]}
	\ES\Big[\big\|\big(\tfrac{\partial^k}{\partial x^k}\psi\big)(t+h,X^{0,x}_s)-\big(\tfrac{\partial^k}{\partial x^k}\psi\big)(t,X^{0,x}_0)\big\|^2_{L^{(k)}(V,\mathcal{V})}\Big]
	\\&+
	\limsup_{[0,T-t]\ni h \to 0}
	\sup_{s\in[0,h]}
	\ES\Big[\big\|\big(\tfrac{\partial^k}{\partial x^k}\psi\big)(t,X^{0,x}_0)-\big(\tfrac{\partial^k}{\partial x^k}\psi\big)(t,X^{0,x}_s)\big\|^2_{L^{(k)}(V,\mathcal{V})}\Big]
	=0.
	\end{split}
	\end{equation}
	Combining this with~\eqref{eq:intermediate.kolmogorov}, item~\eqref{item:time.continuity.V1}, 
	the fact that 
	$F\in\operatorname{Lip}^0(V,V)$, and the fact that
	$B\in\operatorname{Lip}^0(V,\gamma(U,V))$ 
	yields that for all 
	$x\in V_1$, 
	$t\in[0,T)$
	it holds that 
	\begin{equation}
	\label{eq:intermediate.kolmogorov.limit}
	\begin{split}
	&
	\limsup_{(0,T-t]\ni h \to 0}
	\frac{1}{h}
	\bigg\|
	\int^h_0	
	\ES\Big[
	\big[\big(\tfrac{\partial}{\partial x}\psi\big)(t+h,X^{0,x}_s)-\big(\tfrac{\partial}{\partial x}\psi\big)(t,X^{0,x}_s)\big](AX^{0,x}_s+F(X^{0,x}_s))
	\Big]
	\, ds
	\\&\quad+
	\int^h_0	
	\frac{1}{2} \,
	\ES\bigg[
	{\smallsum\limits_{b\in\mathbb{U}}}
	\big[\big(\tfrac{\partial^2}{\partial x^2}\psi\big)(t+h,X^{0,x}_s)-\big(\tfrac{\partial^2}{\partial x^2}\psi\big)(t,X^{0,x}_s)\big](B(X^{0,x}_s)b,B(X^{0,x}_s)b)
	\bigg]
	\, ds	
	\bigg\|_{\mathcal{V}}
	=
	0.
	\end{split}
	\end{equation}
	In the next step note that
	H\"{o}lder's inequality, 
	Lemma~\ref{lem:gamma.estimate}, 
	items~\eqref{item:kolmogorov.diff}, \eqref{item:kolmogorov.lip.c.delta}, 
	\& \eqref{item:time.continuity.V1},
	the fact that 
	$F\in\operatorname{Lip}^0(V,V)$, and the fact that
	$B\in\operatorname{Lip}^0(V,\gamma(U,V))$ 
	show that for all 
	$x\in V_1$, 
	$t\in[0,T]$
	it holds that 
	\begin{multline}
	\!\!
	[0,T] \ni s \mapsto 	
	\ES\big[
	\big(\tfrac{\partial}{\partial x}\psi\big)(t,X^{0,x}_s)(AX^{0,x}_s+F(X^{0,x}_s))
	\big]
	\\+
	\frac{1}{2} \,
	\ES\bigg[
	{\smallsum\limits_{b\in\mathbb{U}}}
	\big(\tfrac{\partial^2}{\partial x^2}\psi\big)(t,X^{0,x}_s)(B(X^{0,x}_s)b,B(X^{0,x}_s)b)
	\bigg] 
	\in \mathcal{V}	  
	\end{multline}
	is continuous.	
	This, \eqref{eq:FTC}, and the fact that 
	$
	  \forall \, x \in V
	  \colon
	  \P(X^{0,x}_0=x)=1
	$
	ensure that for all 
	$x\in V_1$, $t\in[0,T)$ it holds that 
	\begin{equation}
	\begin{split}
	&
	\limsup_{(0,T-t]\ni h \to 0}
	\bigg\|
	\frac{1}{h}
	\int^h_0	
	\ES\big[
	\big(\tfrac{\partial}{\partial x}\psi\big)(t,X^{0,x}_s)(AX^{0,x}_s+F(X^{0,x}_s))
	\big]
	\, ds
	\\&\quad+
	\frac{1}{h}
	\int^h_0	
	\frac{1}{2} \,
	\ES\bigg[
	{\smallsum\limits_{b\in\mathbb{U}}}
	\big(\tfrac{\partial^2}{\partial x^2}\psi\big)(t,X^{0,x}_s)(B(X^{0,x}_s)b,B(X^{0,x}_s)b)
	\bigg]
	\, ds	
	\\&-
	\big(\tfrac{\partial}{\partial x}\psi\big)(t,x)(Ax+F(x))
	-
	\frac{1}{2} \,
	{\smallsum\limits_{b\in\mathbb{U}}}
	\big(\tfrac{\partial^2}{\partial x^2}\psi\big)(t,x)(B(x)b,B(x)b)
	\bigg\|_{\mathcal{V}}
	=0.
	\end{split}
	\end{equation}
	Combining this with~\eqref{eq:kolmogorov.ito.2}, \eqref{eq:intermediate.kolmogorov.limit}, and the triangle inequality assures that for all 
	$x\in V_1$, $t\in[0,T)$
	it holds that 
	\begin{multline}
	\limsup_{(0,T-t]\ni h\to 0}
	\bigg\|
	\frac{
	  \psi(t+h,x)-\psi(t,x)
	}{h}
	+
	\big(\tfrac{\partial}{\partial x}\psi\big)(t,x)(Ax+F(x))
	\\+
	\frac{1}{2} \,
	{\smallsum\limits_{b\in\mathbb{U}}}
	\big(\tfrac{\partial^2}{\partial x^2}\psi\big)(t,x)(B(x)b,B(x)b)
	\bigg\|_{\mathcal{V}}
	=0.
	\end{multline}
	This, item~\eqref{item:space-time.diff}, 
	and the fact that
	$
	\forall \, k \in \{1,2\}, \, x \in V_1, \, t\in[0,T]
	\colon
	(\frac{\partial}{\partial t}\phi)(t,x)
	=
	-(\frac{\partial}{\partial t}\psi)(t,x)
	$
	and 
	$
	(\frac{\partial^k}{\partial x^k}\phi)(t,x)
	=
	(\frac{\partial^k}{\partial x^k}\psi)(t,x)	
	$
	establish item~\eqref{item:kolmogorov.eq}.
	The proof of Lemma~\ref{lem:Kolmogorov} is thus completed.
\end{proof}

\subsection{Setting}
\label{sec:setting_euler_integrated_mollified}

Assume the setting in Section~\ref{sec:global_setting},
let $ \mathbb{U} \subseteq U $ be an orthonormal basis of $ U $, 
let
$
A \colon D(A)
\subseteq
V \rightarrow V
$
be a generator of a strongly continuous analytic semigroup
with 
$
\operatorname{spectrum}( A )
\subseteq
\{
z \in \mathbb{C}
\colon
\text{Re}( z ) < \eta
\}
$,
let
$
(
V_r
,
\left\| \cdot \right\|_{ V_r }
)
$,
$ r \in \R $,
be a family of interpolation spaces associated to
$
\eta - A
$, 
let 
$ h \in (0,\infty) $, 
$ \vartheta \in [0,\frac{1}{2}) $,
$
  F \in 
  \operatorname{Lip}^4( V , V_2 ) 
$, 
$
  B \in 
  \operatorname{Lip}^4( 
    V, 
    \gamma( 
      U, 
      V_2
    ) 
  ) 
$, 
$
  \varphi \in \operatorname{Lip}^4( V, {\mathcal{V}})
$, 
let 
$ 
  ( B^b )_{ b \in \mathbb{U} } \subseteq C( V, V ) 
$ 
be the functions which satisfy for all 
$ b \in \mathbb{U} $, 
$ v \in V $ 
that 
$
     B^b( v ) 
    = 
      B( 
        v 
      )
      \,
      b
$, 
let
$ \varsigma_{ F, B } \in \R $
be the real number given by
$  
  \varsigma_{ F, B } 
  =
    \max\!\big\{
    1
    ,
    \|
      F
    \|_{ 
      C_b^3( V, V_{-\vartheta} )
    }^3
    ,
    \|
      B
    \|_{ 
      C_b^3( V, \gamma( U, V_{ - \vartheta / 2 } ) )
    }^6
    \big\}
$, 
let 
$
\SGchi{r}
\in [1,\infty)
$, 
$r\in[0,1]$, 
be the real numbers 
which satisfy for all 
$r\in[0,1]$
that 
$
\SGchi{r}
=
\max\{
1
,
\sup_{ t \in (0,T] }
t^r
\,
\|
( \eta - A )^r
e^{ t A }
\|_{ L( V ) }
,
\sup_{ t \in (0,T] }
t^{-r}
\|
( \eta - A )^{ -r }
( e^{ t A } - \operatorname{Id}_V )
\|_{ L( V ) }
\}
$, 
let
$
  X, Y \colon [0,T] \times \Omega \to V
$, 
$
  \bar{Y} \colon [0,T] \times \Omega \to V_2
$, 
and 
$
  X^x \colon [0,T] \times \Omega \to V
$, 
$ x \in V $, 
be 
$
  ( \mathcal{F}_t )_{ t \in [0,T] }
$-predictable stochastic processes 
which satisfy for all 
$ x \in V $ 
that 
$
  \sup_{ t \in [0,T] }
  \big[
  \| X_t \|_{ \lpn{5}{\P}{ V } } 
  +
  \| X^x_t \|_{ \lpn{5}{\P}{V} } 
  \big]
  < \infty
$, 
$ X^x_0 = x $, 
$
  \bar{Y}_0 \in \lpn{5}{\P}{V_2}
$,
and 
$
  Y_0 = X_0 = \bar{Y}_0 
$ 
and which satisfy that for all
$ x \in V $, 
$ t \in (0,T] $ 
it holds $ \P $-a.s.\ that 
\begin{equation}
  X_t
  = 
    e^{ t A }\, X_0 
  + 
    \int_0^t e^{ ( t - s )A }\, F( X_s ) \, ds
  + 
    \int_0^t e^{ ( t - s )A }\, B( X_s ) \, dW_s
    ,
\end{equation}
\begin{equation}
\label{eq:mollified.solution}
  X^x_t
  = 
    e^{ t A }\, x 
  + 
    \int_0^t e^{ ( t - s )A }\, F( X^x_s ) \, ds
  + 
    \int_0^t e^{ ( t - s )A }\, B( X^x_s ) \, dW_s
    ,
\end{equation}
\begin{equation}
  Y_t
  = 
    e^{tA}\, Y_0 
  + 
    \int_0^t e^{ (t-\lfloor s \rfloor_h)A }\, F( Y_{ \floor{ s }{ h } } ) \, ds
  + 
    \int_0^t e^{ (t-\lfloor s \rfloor_h)A }\, B( Y_{ \floor{ s }{ h } } ) \, dW_s 
    ,
\end{equation}
\begin{equation}
  \bar{Y}_t
  = 
    e^{ t A } \, \bar{Y}_0 
  + 
    \int_0^t e^{ ( t - s )A } \, F( Y_{ \floor{ s }{ h } } ) \, ds
  + 
    \int_0^t e^{ ( t - s )A } \, B( Y_{ \floor{ s }{ h } } ) \, dW_s 
\end{equation}
(cf., e.g.,  Theorem 4.3 in Brze{\'z}niak~\cite{b97b} and 
Theorem 6.2 in Van Neerven et al.~\cite{vvw08}), 
let 
$
  ( K_r )_{ r \in [ 0, \infty ) }
  \subseteq
  [ 0, \infty ]
$ 
be the extended real numbers which satisfy 
for all 
$ r \in [ 0, \infty ) $ 
that 
$
  K_r
  =
  \sup_{ s, t \in [ 0, T ] }
  \ES\big[\!
    \max\{
    1,
$
$
    \| \bar{Y}_s \|^r_V,
    \| Y_t \|^r_V
    \}
  \big]
$, 
let 
$
  u \colon [0,T] \times V \to {\mathcal{V}}  
$
be the function 
which satisfies for all $ x \in V $, $ t \in [0,T] $
that
$
  u( t, x ) = 
  \ES\big[ 
    \varphi( X^x_{ T - t } )
  \big]
$,
let 
$
  c_{ \delta_1, \dots, \delta_k }
  \in [0,\infty]	
$,
$ \delta_1, \dots, \delta_k \in (-\nicefrac{1}{2},0] $,
$ k \in \{ 1, 2, 3, 4 \} $,
be the extended real numbers 
which satisfy
for all $ k \in \{ 1, 2, 3, 4 \} $,
$ \delta_1, \dots, \delta_k \in (-\nicefrac{1}{2},0] $
that
\begin{equation}
\begin{split} 
&
  c_{ \delta_1, \dots, \delta_k }
=
  \sup_{
    t \in [0,T)
  }
  \sup_{ 
    x \in V
  }
  \sup_{ 
    v_1, \dots, v_k \in V \setminus \{ 0 \}
  }
  \left[
  \frac{
    \big\|
      ( 
        \frac{ 
          \partial^k
        }{
          \partial x^k
        }
        u
      )( t, x )( v_1, \dots, v_k )
    \big\|_{\mathcal{V}}
  }{
    ( T - t )^{ 
      (
        \delta_1 + \ldots + \delta_k
      ) 
    }
    \left\| v_1 \right\|_{ 
      V_{ \delta_1 } 
    }
    \cdot
    \ldots
    \cdot
    \left\| v_k \right\|_{ 
      V_{ \delta_k } 
    }
  }
  \right]
  ,
\end{split}
\end{equation}
let 
$
  \tilde{c}_{ \delta_1, \delta_2, \delta_3, \delta_4 }
  \in [0,\infty]	
$,
$ \delta_1, \delta_2, \delta_3, \delta_4 \in (-\nicefrac{1}{2},0] $, 
be the extended real numbers 
which satisfy for all 
$ \delta_1, \dots, \delta_4 \in (-\nicefrac{1}{2},0] $
that
\begin{equation}
\begin{split} 
&
  \tilde{c}_{ \delta_1, \delta_2, \delta_3, \delta_4 }
\\ & =
  \sup_{ t \in [ 0, T ) }\,
  \sup_{\substack{
    x_1, x_2 \in V, \\ x_1 \neq x_2
  }}\,
  \sup_{ 
    v_1, \ldots, v_4 \in V \setminus \{ 0 \}
  }
  \left[
  \frac{
    \big\|
    \big(
      \big( 
        \frac{ 
          \partial^4
        }{
          \partial x^4
        }
        u
      \big)( t, x_1 )
      -
      \big( 
        \frac{ 
          \partial^4
        }{
          \partial x^4
        }
        u
      \big)( t, x_2 )
      \big)
      ( v_1, \dots, v_4 )
    \big\|_{\mathcal{V}}
  }{
    ( T - t )^{ 
      (
        \delta_1 + \ldots + \delta_4
      ) 
    }
    \left\| x_1 - x_2 \right\|_V
    \left\| v_1 \right\|_{ V_{ \delta_1 } }
     \cdot
     \ldots
    \cdot
    \left\| v_4 \right\|_{ V_{ \delta_4 } }
  }
  \right]
  ,
\end{split}
\end{equation}
and let
$ 
  u_{1,0} \colon [0,T] \times V_1 \to {\mathcal{V}} 
$
and
$
  u_{0,k} \colon [0,T] \times V \to L^{ (k) }( V, {\mathcal{V}} ) 
$, 
$ k \in \{ 1, 2, 3, 4 \} $, 
be the functions 
which satisfy
for all
$ k \in \{ 1, 2, 3, 4 \} $, 
$ x, v_1, \ldots, v_k \in V_1 $, 
$ t \in [0,T] $ 
that 
$
  u_{1,0}( t, x )
  =
  \big( \frac{ \partial }{ \partial t } u \big)( t, x )
$
and 
$
  u_{0,k}( t, x )( v_1, \ldots, v_k )
  =
  \big( \frac{ \partial^k }{ \partial x^k } u \big)( t, x )( v_1, \ldots, v_k )
$
(cf.\ Lemma~\ref{lem:Kolmogorov}).

\subsection{Weak convergence rates for semilinear integrated exponential Euler
approximations of SPDEs with mollified nonlinearities}
\label{sec:euler_integrated_result_mollified}

\begin{lemma}
\label{lem:mollified.c.delta.finite}
Assume the setting in Section~\ref{sec:setting_euler_integrated_mollified}. Then 
\begin{enumerate}[(i)]
\item 
\label{item:mollified.c.delta.finite}
it holds for all 
$ k \in \{ 1, 2, 3, 4 \} $, 
$ \delta_1, \dots, \delta_k \in (-\nicefrac{1}{2},0] $
with
$
\sum^k_{i=1} \delta_i
> -\nicefrac{1}{2}
$ 
that 
$
c_{\delta_1,\ldots,\delta_k}
< \infty
$
and
\item
\label{item:mollified.lip.c.delta.finite}
it holds for all 
$ \delta_1, \dots, \delta_4 \in (-\nicefrac{1}{2},0] $
with
$
\sum^4_{i=1} \delta_i
> -\nicefrac{1}{2}
$ 
that 
$
\tilde{c}_{\delta_1,\delta_2,\delta_3,\delta_4}
< \infty
$. 
\end{enumerate}
\end{lemma}
\begin{proof}
	Items~\eqref{item:mollified.c.delta.finite}--\eqref{item:mollified.lip.c.delta.finite} are an immediate consequence of~\eqref{eq:mollified.solution},  of items~\eqref{item:kolmogorov.c.delta}--\eqref{item:kolmogorov.lip.c.delta} of Lemma~\ref{lem:Kolmogorov}, and of the fact that 
	$
	  \forall \, x \in V, \, t \in [0,T]
	  \colon
	  u(T-t,x)
	  =
	  \ES[\varphi(X^x_t)]
	$.
	The proof of Lemma~\ref{lem:mollified.c.delta.finite} is thus completed.
\end{proof}

\begin{lemma}\label{lem:weak_regularity_F_2nd_term}
Assume the setting in Section~\ref{sec:setting_euler_integrated_mollified} and let 
$ t \in [ 0, T ) $, 
$
  \psi = ( \psi(x,y) )_{ x, y \in V } \in \mathbb{M}( V \times V, {\mathcal{V}} )
$, 
$
  \phi \in \mathbb{M}( V, {\mathcal{V}} )
$
satisfy for all 
$ x, y \in V $ 
that 
$
  \psi( x, y )
  =
  u_{0,1}( t, x )
  F( y )
$
and 
$
  \phi(x)
  =
  \psi(x,x)
$.
Then it holds 
for all 
$ x, x_1, x_2, y, y_1, y_2 \in V $ 
that 
$ \psi \in C^3( V \times V, {\mathcal{V}} ) $, 
$ \phi \in C^3(V, {\mathcal{V}}) $, 
and 
\begin{equation}
\label{eq:psiF_Lipschitzx}
\begin{split}
&
    \max_{
    i, j \in \N_0
    ,\,
    i + j \leq 2
    }
    \big\|
    \big(
    \tfrac{ \partial^{(i+j)} }{ \partial x^i \partial y^j }
    \psi
    \big)
    ( x_1, y )
    -
    \big(
    \tfrac{ \partial^{(i+j)} }{ \partial x^i \partial y^j }
    \psi
    \big)
    ( x_2, y )
    \big\|_{ L^{ (i+j) }( V, {\mathcal{V}} ) }
\\ & \leq
  \tfrac{
    \| x_1 - x_2 \|_V
  }{
  ( T - t )^{ \vartheta }
  } \,
  \| F \|_{ C^2_b( V, V_{ -\vartheta } ) } \,
  \big[ c_{ -\vartheta, 0 } + c_{ -\vartheta, 0, 0 } + c_{ -\vartheta, 0, 0, 0 } \big]
  \max\{ 1, \| y \|_V \}
  ,
\end{split}
\end{equation}
\begin{equation}
\label{eq:psiF_Lipschitzy}
\begin{split}
&
    \max_{
    i, j \in \N_0
    ,\,
    i + j \leq 2
    }
    \big\|
    \big(
    \tfrac{ \partial^{(i+j)} }{ \partial x^i \partial y^j }
    \psi
    \big)
    ( x, y_1 )
    -
    \big(
    \tfrac{ \partial^{(i+j)} }{ \partial x^i \partial y^j }
    \psi
    \big)
    ( x, y_2 )
    \big\|_{ L^{ (i+j) }( V, {\mathcal{V}} ) }
\\ & \leq
  \tfrac{
    \| y_1 - y_2 \|_V
  }{
  ( T - t )^{ \vartheta }
  } \,
  \| F \|_{ C^3_b( V, V_{ -\vartheta } ) } \,
  \big[ c_{ -\vartheta } + c_{ -\vartheta, 0 } + c_{ -\vartheta, 0, 0 } \big]
  ,
\end{split}
\end{equation}
\begin{equation}
\label{eq:psiF_Lipschitzxx}
\begin{split}
&
    \max_{
    i \in \{ 0, 1, 2 \}
    }
    \big\|
    \phi^{ (i) }
    ( x_1 )
    -
    \phi^{ (i) }
    ( x_2 )
    \big\|_{ L^{ (i) }( V, {\mathcal{V}} ) }
\\ & \leq
  \tfrac{
    3 \,
    \| x_1 - x_2 \|_V
  }{
  ( T - t )^{ \vartheta }
  } \,
  \| F \|_{ C^3_b( V, V_{ -\vartheta } ) } \,
  \big[ c_{ -\vartheta } + c_{ -\vartheta, 0 } + c_{ -\vartheta, 0, 0 } + c_{ -\vartheta, 0, 0, 0 } \big]
  \max\{ 1, \| x_1 \|_V, \| x_2 \|_V \}
  .
\end{split}
\end{equation}
\end{lemma}
\begin{proof}
We first note that item~\eqref{item:kolmogorov.diff} of Lemma~\ref{lem:Kolmogorov} ensures that 
$
\big(
V \ni x \mapsto
u_{ 0, 1 }( t, x ) \in L( V, \mathcal{V} )
\big)
\in
C^3( V, L( V, \mathcal{V} ) ) 
$.
The assumption that
$
  F \in \operatorname{Lip}^4( V, V_2 )
$
therefore assures that
$ \psi \in C^3( V \times V, \mathcal{V} ) $ 
and 
$ \phi \in C^3( V, \mathcal{V} ) $. 
Next we observe that for all 
$ x, y, v_1, v_2, v_3 \in V $ 
with 
$ \max\{\|v_1\|_V, \|v_2\|_V, \|v_3\|_V\} \leq 1 $
it holds that 
\begin{equation}
\label{eq:psiF_xderivatives_begin}
\begin{split}
&
  \big\|
   \big(\tfrac{ \partial }{ \partial x }
   \psi\big)( x, y ) \, v_1
  \big\|_{\mathcal{V}}
=
  \left\|
    u_{0,2}( t, x )
    \big(
     F( y ), v_1
    \big)
  \right\|_{\mathcal{V}}
\leq
  \tfrac{
    c_{ -\vartheta, 0 }
  }{
  ( T - t )^{ \vartheta }
  } \,
  \| F( y ) \|_{ V_{ -\vartheta } },
\end{split}
\end{equation}
\begin{equation}
\begin{split}
&
  \big\|
   \big(\tfrac{ \partial^2 }{ \partial x^2 }
   \psi\big)( x, y ) \, ( v_1, v_2 )
  \big\|_{\mathcal{V}}
=
  \left\|
    u_{0,3}( t, x )
    \big(
     F( y ), v_1, v_2
    \big)
  \right\|_{\mathcal{V}}
\leq
  \tfrac{
    c_{ -\vartheta, 0, 0 }
  }{
  ( T - t )^{ \vartheta }
  } \,
  \| F( y ) \|_{ V_{ -\vartheta } },
\end{split}
\end{equation}
\begin{equation}
\label{eq:psiF_xderivatives_end}
\begin{split}
&
  \big\|
   \big(\tfrac{ \partial^3 }{ \partial x^3 }
   \psi\big)( x, y ) \, ( v_1, v_2, v_3 )
  \big\|_{\mathcal{V}}
=
  \left\|
    u_{0,4}( t, x )
    \big(
     F( y ), v_1, v_2, v_3
    \big)
  \right\|_{\mathcal{V}}
\leq
  \tfrac{
    c_{ -\vartheta, 0, 0, 0 }
  }{
  ( T - t )^{ \vartheta }
  } \,
  \| F( y ) \|_{ V_{ -\vartheta } },
\end{split}
\end{equation}
\begin{equation}
\label{eq:psiF_yderivatives_begin}
\begin{split}
&
  \big\|
   \big(\tfrac{ \partial }{ \partial y }
   \psi\big)( x, y ) \, v_1
  \big\|_{\mathcal{V}}
=
  \left\|
    u_{0,1}( t, x )
    \,
     F'( y ) \, v_1
  \right\|_{\mathcal{V}}
\leq
  \tfrac{
    c_{ -\vartheta }
  }{
  ( T - t )^{ \vartheta }
  } \,
  \| F'( y ) \|_{ L( V, V_{ -\vartheta } ) },
\end{split}
\end{equation}
\begin{equation}
\begin{split}
&
  \big\|
   \big(\tfrac{ \partial^2 }{ \partial y^2 }
   \psi\big)( x, y )( v_1, v_2 )
  \big\|_{\mathcal{V}}
=
  \left\|
    u_{0,1}( t, x )
    \big(
     F''( y ) ( v_1, v_2 )
    \big)
  \right\|_{\mathcal{V}}
\leq
  \tfrac{
    c_{ -\vartheta }
  }{
  ( T - t )^{ \vartheta }
  } \,
  \| F''( y ) \|_{ L^{ (2) }( V, V_{ -\vartheta } ) },
\end{split}
\end{equation}
\begin{equation}
\label{eq:psiF_yderivatives_end}
\begin{split}
  \big\|
   \big(\tfrac{ \partial^3 }{ \partial y^3 }
   \psi\big)( x, y )( v_1, v_2, v_3 )
  \big\|_{\mathcal{V}}
&=
  \left\|
    u_{0,1}( t, x )
    \big(
     F^{(3)}( y ) ( v_1, v_2, v_3 )
    \big)
  \right\|_{\mathcal{V}}
\\ & \leq
  \tfrac{
    c_{ -\vartheta }
  }{
  ( T - t )^{ \vartheta }
  } \,
  \| F^{(3)}( y ) \|_{ L^{ (3) }( V, V_{ -\vartheta } ) },
\end{split}
\end{equation}
\begin{equation}
\label{eq:psiF_mixedderivatives_begin}
\begin{split}
&
  \big\|
   \big(\tfrac{ \partial^2 }{ \partial x \partial y }
   \psi
   \big)( x, y )( v_1, v_2 )
  \big\|_{\mathcal{V}}
=
  \left\|
    u_{0,2}( t, x )
    \big(
     F'( y ) \, v_1,
     v_2
    \big)
  \right\|_{\mathcal{V}}
\leq
  \tfrac{
    c_{ -\vartheta, 0 }
  }{
  ( T - t )^{ \vartheta }
  } \,
  \| F'( y ) \|_{ L( V, V_{ -\vartheta } ) },
\end{split}
\end{equation}
\begin{equation}
\begin{split}
  \big\|
   \big(\tfrac{ \partial^3 }{ \partial x^2 \partial y }
   \psi\big)( x, y )( v_1, v_2, v_3 )
  \big\|_{\mathcal{V}}
&=
  \left\|
    u_{0,3}( t, x )
    \big(
     F'( y ) \, v_1,
     v_2, v_3
    \big)
  \right\|_{\mathcal{V}}
\\ & \leq
  \tfrac{
  c_{ -\vartheta, 0, 0 }
  }{
  ( T - t )^{ \vartheta }
  } \,
  \| F'( y ) \|_{ L( V, V_{ -\vartheta } ) },
\end{split}
\end{equation}
\begin{equation}
\label{eq:psiF_mixedderivatives_end}
\begin{split}
  \big\|
   \big(\tfrac{ \partial^3 }{ \partial x \partial y^2 }
   \psi\big)( x, y )( v_1, v_2, v_3 )
  \big\|_{\mathcal{V}}
&=
  \left\|
    u_{0,2}( t, x )
    \big(
     F''( y ) ( v_1, v_2 ),
     v_3
    \big)
  \right\|_{\mathcal{V}}
\\ & \leq
  \tfrac{
    c_{ -\vartheta, 0 }
  }{
  ( T - t )^{ \vartheta }
  } \,
  \| F''( y ) \|_{ L^{ ( 2 ) }( V, V_{ -\vartheta } ) }.
\end{split}
\end{equation}
Combining \eqref{eq:psiF_xderivatives_begin}--\eqref{eq:psiF_xderivatives_end} and 
\eqref{eq:psiF_mixedderivatives_begin}--\eqref{eq:psiF_mixedderivatives_end} 
with item~\eqref{item:mollified.c.delta.finite} of Lemma~\ref{lem:mollified.c.delta.finite} and the fundamental theorem of calculus in Banach spaces proves~\eqref{eq:psiF_Lipschitzx}.
Moreover, combining \eqref{eq:psiF_yderivatives_begin}--\eqref{eq:psiF_mixedderivatives_end} 
with item~\eqref{item:mollified.c.delta.finite} of Lemma~\ref{lem:mollified.c.delta.finite} and the fundamental theorem of calculus in Banach spaces shows~\eqref{eq:psiF_Lipschitzy}. 
It thus remains to prove~\eqref{eq:psiF_Lipschitzxx}. 
For this we observe that \eqref{eq:psiF_xderivatives_begin}--\eqref{eq:psiF_mixedderivatives_end} ensure that 
for all $ x, v_1, v_2, v_3 \in V $ with 
$ \max\{\|v_1\|_V, \|v_2\|_V, \|v_3\|_V\} \leq 1 $
it holds that 
\begin{equation}
\label{eq:psiF_xxderivatives_begin}
\begin{split}
&
  \big\|
   \phi'( x ) \, v_1
  \big\|_{\mathcal{V}}
  \leq
  \big\|
   \big(\tfrac{ \partial }{ \partial x }
   \psi\big)( x, x ) \, v_1
  \big\|_{\mathcal{V}}
  +
  \big\|
   \big(\tfrac{ \partial }{ \partial y }
   \psi\big)( x, x ) \, v_1
  \big\|_{\mathcal{V}}
\\ & \leq
  \tfrac{
    c_{ -\vartheta, 0 } \,
    \| F( x ) \|_{ V_{ -\vartheta } }
    +
    c_{ -\vartheta } \,
    \| F'( x ) \|_{ L( V, V_{ -\vartheta } ) }
  }{
  ( T - t )^{ \vartheta }
  }
  \leq
  \tfrac{
    [ c_{ -\vartheta }
    +
    c_{ -\vartheta, 0 } ]
  }{
  ( T - t )^{ \vartheta }
  }
  \,
  \| F \|_{ C^1_b( V, V_{ -\vartheta } ) } \,
  \max\{ 1, \| x \|_V \}
  ,
\end{split}
\end{equation}
\begin{equation}
\begin{split}
&
  \big\|
   \phi''( x ) \, ( v_1, v_2 )
  \big\|_{\mathcal{V}}
\\ & \leq
  \big\|
   \big(\tfrac{ \partial^2 }{ \partial x^2 }
   \psi\big)( x, x ) \, ( v_1, v_2 )
  \big\|_{\mathcal{V}}
  +
  2 \,
  \big\|
   \big(\tfrac{ \partial^2 }{ \partial x \partial y }
   \psi\big)( x, x ) \, ( v_1, v_2 )
  \big\|_{\mathcal{V}}
  +
  \big\|
   \big(\tfrac{ \partial^2 }{ \partial y^2 }
   \psi\big)( x, x ) \, ( v_1, v_2 )
  \big\|_{\mathcal{V}}
\\ & \leq
  \tfrac{
    c_{ -\vartheta, 0, 0 } \,
    \| F( x ) \|_{ V_{ -\vartheta } }
    +
    2 \,
    c_{ -\vartheta, 0 } \,
    \| F'( x ) \|_{ L( V, V_{ -\vartheta } ) }
    +
    c_{ -\vartheta } \,
    \| F''( x ) \|_{ L^{(2)}( V, V_{ -\vartheta } ) }
  }{
  ( T - t )^{ \vartheta }
  }
\\ & \leq
  \tfrac{
    2 \,
    [ c_{ -\vartheta }
    +
    c_{ -\vartheta, 0 } 
    + c_{ -\vartheta, 0, 0 } ]
  }{
  ( T - t )^{ \vartheta }
  }
  \,
  \| F \|_{ C^2_b( V, V_{ -\vartheta } ) } \,
  \max\{ 1, \| x \|_V \}
  ,
\end{split}
\end{equation}
and 
\begin{equation}
\label{eq:psiF_xxderivatives_end}
\begin{split}
&
  \big\|
   \phi^{(3)}( x ) \, ( v_1, v_2, v_3 )
  \big\|_{\mathcal{V}}
\leq
  \big\|
   \big(\tfrac{ \partial^3 }{ \partial x^3 }
   \psi\big)( x, x ) \, ( v_1, v_2, v_3 )
  \big\|_{\mathcal{V}}
+
  3 \,
  \big\|
   \big(\tfrac{ \partial^3 }{ \partial x^2 \partial y }
   \psi\big)( x, x ) \, ( v_1, v_2, v_3 )
  \big\|_{\mathcal{V}}
\\ & +
  3 \,
  \big\|
   \big(\tfrac{ \partial^3 }{ \partial x \partial y^2 }
   \psi\big)( x, x ) \, ( v_1, v_2, v_3 )
  \big\|_{\mathcal{V}}
  +
  \big\|
   \big(\tfrac{ \partial^3 }{ \partial y^3 }
   \psi\big)( x, x ) \, ( v_1, v_2, v_3 )
  \big\|_{\mathcal{V}}
\\ & \leq
  \tfrac{
    c_{ -\vartheta, 0, 0, 0 } \,
    \| F( x ) \|_{ V_{ -\vartheta } }
    +
    3 \,
    c_{ -\vartheta, 0, 0 } \,
    \| F'( x ) \|_{ L( V, V_{ -\vartheta } ) }
    +
    3 \,
    c_{ -\vartheta, 0 } \,
    \| F''( x ) \|_{ L^{(2)}( V, V_{ -\vartheta } ) }
    +
    c_{ -\vartheta } \,
    \| F^{(3)}( x ) \|_{ L^{(3)}( V, V_{ -\vartheta } ) }
  }{
  ( T - t )^{ \vartheta }
  }
\\ & \leq
  \tfrac{
    3 \,
    [ c_{ -\vartheta }
    +
    c_{ -\vartheta, 0 } 
    + c_{ -\vartheta, 0, 0 } 
    + c_{ -\vartheta, 0, 0, 0 } ]
  }{
  ( T - t )^{ \vartheta }
  }
  \,
  \| F \|_{ C^3_b( V, V_{ -\vartheta } ) } \,
  \max\{ 1, \| x \|_V \}
  .
\end{split}
\end{equation}
Combining \eqref{eq:psiF_xxderivatives_begin}--\eqref{eq:psiF_xxderivatives_end} 
with item~\eqref{item:mollified.c.delta.finite} of Lemma~\ref{lem:mollified.c.delta.finite} and the fundamental theorem of calculus in Banach spaces establishes~\eqref{eq:psiF_Lipschitzxx}. 
The proof of Lemma~\ref{lem:weak_regularity_F_2nd_term} is thus completed.
\end{proof}

\begin{lemma}
\label{lem:weak_regularity_B_2nd_term}
Assume the setting in Section~\ref{sec:setting_euler_integrated_mollified} and let 
$ t \in [ 0, T ) $, 
$
  \psi = ( \psi(x,y) )_{ x, y \in V } \in \mathbb{M}( V \times V, {\mathcal{V}} )
$, 
$
  \phi \in \mathbb{M}( V, {\mathcal{V}} )
$ 
satisfy for all 
$ x, y \in V $ 
that
$
  \psi( x, y )
  =
  \smallsum_{ b \in \mathbb{U} }
  u_{0,2}( t, x )
  \big(
    B^b( y ),
    B^b( y )
  \big)
$ 
and 
$
  \phi(x)
  =
  \psi(x,x)
$. 
Then it holds 
for all 
$ x, x_1, x_2, y, y_1, y_2 \in V $ 
that 
$ \psi \in C^2( V \times V, {\mathcal{V}} ) $, 
$ \phi \in C^2( V, {\mathcal{V}} ) $,  
and 
\begin{equation}
\label{eq:psiB_Lipschitzx}
\begin{split}
&
    \max_{
    i, j \in \N_0
    ,\,
    i + j \leq 2
    }
    \big\|
    \big(
    \tfrac{ \partial^{(i+j)} }{ \partial x^i \partial y^j }
    \psi
    \big)
    ( x_1, y )
    -
    \big(
    \tfrac{ \partial^{(i+j)} }{ \partial x^i \partial y^j }
    \psi
    \big)
    ( x_2, y )
    \big\|_{ L^{ (i+j) }( V, {\mathcal{V}} ) }
\leq
  \tfrac{
    2 \, \| x_1 - x_2 \|_V
  }{
    ( T - t )^{ \vartheta }
  }
\\ & \cdot
  \| B \|^2_{ C^2_b( V, \gamma( U, V_{ -\nicefrac{ \vartheta }{ 2 } } ) ) } \,
  \big[ 
    c_{ - \nicefrac{ \vartheta }{ 2 }, -\nicefrac{ \vartheta }{ 2 }, 0 } 
    + c_{ - \nicefrac{ \vartheta }{ 2 }, -\nicefrac{ \vartheta }{ 2 }, 0, 0 }
    + \tilde{c}_{ - \nicefrac{ \vartheta }{ 2 }, -\nicefrac{ \vartheta }{ 2 }, 0, 0 }
  \big]
  \max\{ 1, \| y \|_{ V }^2 \}
  ,
\end{split}
\end{equation}
\begin{equation}
\label{eq:psiB_Lipschitzy}
\begin{split}
&
    \max_{
    i, j \in \N_0
    ,\,
    i + j \leq 2
    }
    \big\|
    \big(
    \tfrac{ \partial^{(i+j)} }{ \partial x^i \partial y^j }
    \psi
    \big)
    ( x, y_1 )
    -
    \big(
    \tfrac{ \partial^{(i+j)} }{ \partial x^i \partial y^j }
    \psi
    \big)
    ( x, y_2 )
    \big\|_{ L^{ (i+j) }( V, {\mathcal{V}} ) }
\leq
  \tfrac{
    6 \, \| y_1 - y_2 \|_V
  }{
    ( T - t )^{ \vartheta }
  }
\\ & \cdot
  \| B \|^2_{ C^3_b( V, \gamma( U, V_{ -\nicefrac{ \vartheta }{ 2 } } ) ) } \,
  \big[ c_{ -\nicefrac{ \vartheta }{ 2 }, -\nicefrac{ \vartheta }{ 2 } } + c_{ -\nicefrac{ \vartheta }{ 2 }, -\nicefrac{ \vartheta }{ 2 }, 0 } + c_{ -\nicefrac{ \vartheta }{ 2 }, -\nicefrac{ \vartheta }{ 2 }, 0, 0 } \big]
  \max\{ 1, \| y_1 \|_{ V }, \| y_2 \|_{ V } \}
  ,
\end{split}
\end{equation}
\begin{equation}
\label{eq:psiB_Lipschitzxx}
\begin{split}
&
    \max_{
    i \in \{ 0, 1, 2 \}
    }
    \big\|
    \phi^{(i)}
    ( x_1 )
    -
    \phi^{(i)}
    ( x_2 )
    \big\|_{ L^{ (i) }( V, {\mathcal{V}} ) }
\leq
  \tfrac{
    8 \,
    \| x_1 - x_2 \|_V
  }{
  ( T - t )^{ \vartheta }
  }
  \,
  \| B \|^2_{ C^3_b( V, \gamma( U, V_{ -\nicefrac{ \vartheta }{ 2 } } ) ) } 
\\ & \cdot
  \big[ 
    c_{ 
      - \nicefrac{ \vartheta }{ 2 }, 
      - \nicefrac{ \vartheta }{ 2 } 
    } 
    + 
    c_{ 
      - \nicefrac{ \vartheta }{ 2 }, 
      - \nicefrac{ \vartheta }{ 2 }, 0 
    } 
    + 
    c_{ -\nicefrac{ \vartheta }{ 2 }, -\nicefrac{ \vartheta }{ 2 }, 0, 0 } 
    + 
    \tilde{c}_{ -\nicefrac{ \vartheta }{ 2 }, -\nicefrac{ \vartheta }{ 2 }, 0, 0 } 
  \big]
  \max\{ 1, \| x_1 \|^2_V, \| x_2 \|^2_V \}
  .
\end{split}
\end{equation}
\end{lemma}
\begin{proof}
We first note that item~\eqref{item:kolmogorov.diff} of Lemma~\ref{lem:Kolmogorov} ensures that 
$
\big(
V \ni x \mapsto
u_{0,2}( t, x ) \in L^{(2)}(V,{\mathcal{V}})
\big)
\in
C^2( V, L^{(2)}(V,{\mathcal{V}}) ) 
$.
Lemma~\ref{lem:gamma.estimate} and
the assumption that 
$ B \in \operatorname{Lip}^4( V, \gamma( U, V_2 ) ) $
therefore assure that 
\begin{equation}
  \psi \in 
  C^2( V \times V, {\mathcal{V}} ),
  \qquad
  \phi \in C^2( V, {\mathcal{V}} ), 
\end{equation}
\begin{equation}
  \big(
  V \times V \ni (x,y) 
  \mapsto
  \big( \tfrac{ \partial }{ \partial y } \psi\big)( x, y ) 
  \in L(V,{\mathcal{V}})
  \big)
  \in
  C^2( V \times V, L(V,{\mathcal{V}}) )
  , 
\end{equation}
and
\begin{equation}
  \forall \, x \in V \colon
  \big(
  V \ni y 
  \mapsto
  \big( \tfrac{ \partial^2 }{ \partial x^2 } \psi\big)( x, y ) 
  \in L^{(2)}(V,{\mathcal{V}})
  \big)
  \in
  C^1( V, L^{(2)}(V,{\mathcal{V}}) )
  .
\end{equation}
Next we use Lemma~\ref{lem:gamma.estimate} and Lemma~\ref{lem:mollified.c.delta.finite} to obtain that for all 
$ x, x_1, x_2, y, v_1, v_2, v_3 \in V $ with 
$ \max\{\|v_1\|_V, \|v_2\|_V, \|v_3\|_V\} \leq 1 $
it holds that 
\begin{equation}
\label{eq:psiB_xderivatives_begin}
\begin{split}
  \big\|
   \big(\tfrac{ \partial }{ \partial x }
   \psi\big)( x, y ) \, v_1
  \big\|_{\mathcal{V}}
&=
  \Big\|
  \smallsum\limits_{ b \in \mathbb{U} }
  u_{0,3}( t, x )
  \big(
  B^b( y ),
  B^b( y ), v_1
  \big)
  \Big\|_{\mathcal{V}}
\\ & \leq
  \tfrac{
    c_{ -\nicefrac{ \vartheta }{ 2 }, -\nicefrac{ \vartheta }{ 2 }, 0 }
  }{
  ( T - t )^{ \vartheta }
  } \,
  \| B( y ) \|^2_{ \gamma( U, V_{ -\nicefrac{ \vartheta }{ 2 } } ) },
\end{split}
\end{equation}
\begin{equation}
\begin{split}
  \big\|
   \big(\tfrac{ \partial^2 }{ \partial x^2 }
   \psi\big)( x, y ) \, ( v_1, v_2 )
  \big\|_{\mathcal{V}}
&=
  \Big\|
  \smallsum\limits_{ b \in \mathbb{U} }
  u_{0,4}( t, x )
  \big(
  B^b( y ),
  B^b( y ), v_1, v_2
  \big)
  \Big\|_{\mathcal{V}}
\\ & \leq
  \tfrac{
    c_{ -\nicefrac{ \vartheta }{ 2 }, -\nicefrac{ \vartheta }{ 2 }, 0, 0 }
  }{
  ( T - t )^{ \vartheta }
  } \,
  \| B( y ) \|^2_{ \gamma( U, V_{ -\nicefrac{ \vartheta }{ 2 } } ) },
\end{split}
\end{equation}
\begin{equation}
\label{eq:psiB_xderivatives_end}
\begin{split}
&
  \big\|
   \big(\tfrac{ \partial^2 }{ \partial x^2 }
   \psi\big)( x_1, y ) \, ( v_1, v_2 )
   -
   \big(\tfrac{ \partial^2 }{ \partial x^2 }
   \psi\big)( x_2, y ) \, ( v_1, v_2 )
  \big\|_{\mathcal{V}}
\\ & =
  \Big\|
  \smallsum\limits_{ b \in \mathbb{U} }
  \big(
    u_{0,4}( t, x_1 ) - u_{0,4}( t, x_2 )
  \big)
  \big(
  B^b( y ),
  B^b( y ), v_1, v_2
  \big)
  \Big\|_{\mathcal{V}}
\\ & \leq
  \tfrac{
    \tilde{c}_{ -\nicefrac{ \vartheta }{ 2 }, -\nicefrac{ \vartheta }{ 2 }, 0, 0 } \,
    \| x_1 - x_2 \|_{\mathcal{V}}
  }{
  ( T - t )^{ \vartheta }
  } \,
  \| B( y ) \|^2_{ \gamma( U, V_{ -\nicefrac{ \vartheta }{ 2 } } ) },
\end{split}
\end{equation}
\begin{equation}
\label{eq:psiB_yderivatives_begin}
\begin{split}
  \big\|
   \big(\tfrac{ \partial }{ \partial y }
   \psi\big)( x, y ) \, v_1
  \big\|_{\mathcal{V}}
&=
  2 \,
  \Big\|
  \smallsum\limits_{ b \in \mathbb{U} }
  u_{0,2}( t, x )
  \big(
  B^b( y ),
  ( B^b )'( y ) \, v_1
  \big)
  \Big\|_{\mathcal{V}}
\\ & \leq
  \tfrac{
    2 \, c_{ -\nicefrac{ \vartheta }{ 2 }, -\nicefrac{ \vartheta }{ 2 } }
  }{
  ( T - t )^{ \vartheta }
  } \,
  \| B( y ) \|_{ \gamma( U, V_{ -\nicefrac{ \vartheta }{ 2 } } ) } \,
  \| B'( y ) \|_{ L( V, \gamma( U, V_{ -\nicefrac{ \vartheta }{ 2 } } ) ) },
\end{split}
\end{equation}
\begin{equation}
\begin{split}
&
  \big\|
   \big(\tfrac{ \partial^2 }{ \partial y^2 }
   \psi\big)( x, y )( v_1, v_2 )
  \big\|_{\mathcal{V}}
\\ & =
  2 \,
  \Big\|
  \smallsum\limits_{ b \in \mathbb{U} }
  u_{0,2}( t, x )
  \big(
  ( B^b )'( y ) \, v_1,
  ( B^b )'( y ) \, v_2
  \big)
  +
    u_{0,2}( t, x )
    \big(
    B^b( y ),
    ( B^b )''( y )( v_1, v_2 )
    \big)
  \Big\|_{\mathcal{V}}
\\ & \leq
  \tfrac{
    2 \, c_{ -\nicefrac{ \vartheta }{ 2 }, -\nicefrac{ \vartheta }{ 2 } }
  }{
  ( T - t )^{ \vartheta }
  }
  \,
  \big(
  \| B'( y ) \|^2_{ L( V, \gamma( U, V_{ -\nicefrac{ \vartheta }{ 2 } } ) ) }
  +
  \| B( y ) \|_{ \gamma( U, V_{ -\nicefrac{ \vartheta }{ 2 } } ) } \,
  \| B''( y ) \|_{ L^{(2)}( V, \gamma( U, V_{ -\nicefrac{ \vartheta }{ 2 } } ) ) }
  \big)
  ,
\end{split}
\end{equation}
\begin{equation}
\label{eq:psiB_yderivatives_end}
\begin{split}
&
  \big\|
   \big(\tfrac{ \partial^3 }{ \partial y^3 }
   \psi\big)( x, y )( v_1, v_2, v_3 )
  \big\|_{\mathcal{V}}
\\ & =
  2 \,
  \Big\|
  \smallsum\limits_{ b \in \mathbb{U} }
  u_{0,2}( t, x )
  \big(
  ( B^b )'( y ) \, v_2,
  ( B^b )''( y ) ( v_1, v_3 )
  \big)
\\ & \quad
  +
  u_{0,2}( t, x )
  \big(
  ( B^b )'( y ) \, v_1,
  ( B^b )''( y ) ( v_2, v_3 )
  \big)
\\ & \quad
    +
    u_{0,2}( t, x )
    \big(
    ( B^b )'( y ) \, v_3,
    ( B^b )''( y )( v_1, v_2 )
    \big)
\\ & \quad
    +
    u_{0,2}( t, x )
    \big(
    B^b( y ),
    ( B^b )^{(3)}( y )( v_1, v_2, v_3 )
    \big)
  \Big\|_{\mathcal{V}}
\\ & \leq
  \tfrac{
    2 \, c_{ -\nicefrac{ \vartheta }{ 2 }, -\nicefrac{ \vartheta }{ 2 } }
  }{
  ( T - t )^{ \vartheta }
  }
  \,
  \big(
  3 \, \| B'( y ) \|_{ L( V, \gamma( U, V_{ -\nicefrac{ \vartheta }{ 2 } } ) ) } \,
  \| B''( y ) \|_{ L^{(2)}( V, \gamma( U, V_{ -\nicefrac{ \vartheta }{ 2 } } ) ) }
\\ & \quad
  +
  \| B( y ) \|_{ \gamma( U, V_{ -\nicefrac{ \vartheta }{ 2 } } ) } \,
  \| B^{(3)}( y ) \|_{ L^{(3)}( V, \gamma( U, V_{ -\nicefrac{ \vartheta }{ 2 } } ) ) }
  \big)
  ,
\end{split}
\end{equation}
\begin{equation}
\label{eq:psiB_mixedderivatives_begin}
\begin{split}
  \big\|
   \big(\tfrac{ \partial^2 }{ \partial x \partial y }
   \psi\big)( x, y )( v_1, v_2 )
  \big\|_{\mathcal{V}}
&=
  2\,
  \Big\|
  \smallsum\limits_{ b \in \mathbb{U} }
  u_{0,3}( t, x )
  \big(
  B^b( y ),
  ( B^b )'( y ) \, v_1,
  v_2
  \big)
  \Big\|_{\mathcal{V}}
\\ & \leq
  \tfrac{
    2 \, c_{ -\nicefrac{ \vartheta }{ 2 }, -\nicefrac{ \vartheta }{ 2 }, 0 }
  }{
  ( T - t )^{ \vartheta }
  }
  \,
  \| B( y ) \|_{ \gamma( U, V_{ -\nicefrac{ \vartheta }{ 2 } } ) } \,
  \| B'( y ) \|_{ L( V, \gamma( U, V_{ -\nicefrac{ \vartheta }{ 2 } } ) ) },
\end{split}
\end{equation}
\begin{equation}
\label{eq:psiB_mixedderivatives_middle}
\begin{split}
  \big\|
   \big(\tfrac{ \partial^3 }{ \partial x^2 \partial y }
   \psi\big)( x, y )( v_1, v_2, v_3 )
  \big\|_{\mathcal{V}}
& =
  2\, 
  \Big\|
  \smallsum\limits_{ b \in \mathbb{U} }
  u_{0,4}( t, x )
  \big(
  B^b( y ),
  ( B^b )'( y ) \, v_1,
  v_2, v_3
  \big)
  \Big\|_{\mathcal{V}}
\\ & \leq
  \tfrac{
    2 \, c_{ -\nicefrac{ \vartheta }{ 2 }, -\nicefrac{ \vartheta }{ 2 }, 0, 0 }
  }{
  ( T - t )^{ \vartheta }
  } \,
  \,
  \| B( y ) \|_{ \gamma( U, V_{ -\nicefrac{ \vartheta }{ 2 } } ) } \,
  \| B'( y ) \|_{ L( V, \gamma( U, V_{ -\nicefrac{ \vartheta }{ 2 } } ) ) },
\end{split}
\end{equation}
\begin{equation}
\label{eq:derivative.yxx}
  \big\|
  \big(\tfrac{ \partial^3 }{ \partial y \partial x^2 }
  \psi\big)( x, y )( v_1, v_2, v_3 )
  \big\|_{\mathcal{V}}
  \leq
  \tfrac{
  	2 \, c_{ -\nicefrac{ \vartheta }{ 2 }, -\nicefrac{ \vartheta }{ 2 }, 0, 0 }
  }{
  ( T - t )^{ \vartheta }
} \,
\,
\| B( y ) \|_{ \gamma( U, V_{ -\nicefrac{ \vartheta }{ 2 } } ) } \,
\| B'( y ) \|_{ L( V, \gamma( U, V_{ -\nicefrac{ \vartheta }{ 2 } } ) ) },
\end{equation}
\begin{equation}
\begin{split}
\label{eq:psiB_mixedderivatives_end}
&
  \big\|
   \big(\tfrac{ \partial^3 }{ \partial x \partial y^2 }
   \psi\big)( x, y )( v_1, v_2, v_3 )
  \big\|_{\mathcal{V}}
\\ & =
  2\,
  \Big\|
  \smallsum\limits_{ b \in \mathbb{U} }
  u_{0,3}( t, x )
  \big(
  ( B^b )'( y ) \, v_1,
  ( B^b )'( y ) \, v_2,
  v_3
  \big)
+
  u_{0,3}( t, x )
  \big(
  B^b( y ),
  ( B^b )''( y )( v_1, v_2 ),
  v_3
  \big)
  \Big\|_{\mathcal{V}}
\\ & \leq
  \tfrac{
    2 \, c_{ -\nicefrac{ \vartheta }{ 2 }, -\nicefrac{ \vartheta }{ 2 }, 0 }
  }{
  ( T - t )^{ \vartheta }
  }
  \,
  \big(
  \| B'( y ) \|^2_{ L( V, \gamma( U, V_{ -\nicefrac{ \vartheta }{ 2 } } ) ) }
  +
  \| B( y ) \|_{ \gamma( U, V_{ -\nicefrac{ \vartheta }{ 2 } } ) }
  \,
  \| B''( y ) \|_{ L^{(2)}( V, \gamma( U, V_{ -\nicefrac{ \vartheta }{ 2 } } ) ) }
  \big)
  ,
\end{split}
\end{equation}
and 
\begin{equation}
\label{eq:derivative.yxy}
\begin{split}
& 
  \big\|
  \big(\tfrac{ \partial^3 }{ \partial y \partial x \partial y }
  \psi\big)( x, y )( v_1, v_2, v_3 )
  \big\|_{\mathcal{V}}
\\ & \leq
\tfrac{
	2 \, c_{ -\nicefrac{ \vartheta }{ 2 }, -\nicefrac{ \vartheta }{ 2 }, 0 }
}{
( T - t )^{ \vartheta }
}
\,
\big(
\| B'( y ) \|^2_{ L( V, \gamma( U, V_{ -\nicefrac{ \vartheta }{ 2 } } ) ) }
+
\| B( y ) \|_{ \gamma( U, V_{ -\nicefrac{ \vartheta }{ 2 } } ) }
\,
\| B''( y ) \|_{ L^{(2)}( V, \gamma( U, V_{ -\nicefrac{ \vartheta }{ 2 } } ) ) }
\big)
.
\end{split}
\end{equation}
Combining \eqref{eq:psiB_xderivatives_begin}--\eqref{eq:psiB_xderivatives_end}, \eqref{eq:psiB_mixedderivatives_begin}, \eqref{eq:psiB_mixedderivatives_middle}, 
and \eqref{eq:psiB_mixedderivatives_end}
with Lemma~\ref{lem:mollified.c.delta.finite} and the fundamental theorem of calculus in Banach spaces proves~\eqref{eq:psiB_Lipschitzx}. 
Moreover, combining \eqref{eq:psiB_yderivatives_begin}--\eqref{eq:psiB_mixedderivatives_begin}, \eqref{eq:derivative.yxx}, and~\eqref{eq:derivative.yxy} with item~\eqref{item:mollified.c.delta.finite} of Lemma~\ref{lem:mollified.c.delta.finite} and the fundamental theorem of calculus in Banach spaces establishes~\eqref{eq:psiB_Lipschitzy}. 
It thus remains to prove~\eqref{eq:psiB_Lipschitzxx}.
For this we observe that 
\eqref{eq:psiB_xderivatives_begin}--\eqref{eq:psiB_mixedderivatives_end} ensure that 
for all 
$ x, v_1, v_2, v_3 \in V $ with 
$ \max\{\|v_1\|_V, \|v_2\|_V, \|v_3\|_V\} \leq 1 $
it holds that 
\begin{equation}
\label{eq:psiB_xxderivatives_begin}
\begin{split}
&
  \big\|
   \phi'( x ) \, v_1
  \big\|_{\mathcal{V}}
  \leq
  \big\|
   \big(\tfrac{ \partial }{ \partial x }
   \psi\big)( x, x ) \, v_1
  \big\|_{\mathcal{V}}
  +
  \big\|
   \big(\tfrac{ \partial }{ \partial y }
   \psi\big)( x, x ) \, v_1
  \big\|_{\mathcal{V}}
\\ & \leq
  \tfrac{
    c_{ -\nicefrac{ \vartheta }{ 2 }, -\nicefrac{ \vartheta }{ 2 }, 0 } \,
    \| B( x ) \|^2_{ \gamma( U, V_{ -\nicefrac{ \vartheta }{ 2 } } ) }
    +
    2 \,
    c_{ -\nicefrac{ \vartheta }{ 2 }, -\nicefrac{ \vartheta }{ 2 } } \,
    \| B( x ) \|_{ \gamma( U, V_{ -\nicefrac{ \vartheta }{ 2 } } ) } \,
    \| B'( x ) \|_{ L( V, \gamma( U, V_{ -\nicefrac{ \vartheta }{ 2 } } ) ) }
  }{
  ( T - t )^{ \vartheta }
  }
\\ & \leq
  \tfrac{
    2 \,
    [ c_{ -\nicefrac{ \vartheta }{ 2 }, -\nicefrac{ \vartheta }{ 2 } }
    +
    c_{ -\nicefrac{ \vartheta }{ 2 }, -\nicefrac{ \vartheta }{ 2 }, 0 } ]
  }{
  ( T - t )^{ \vartheta }
  }
  \,
  \| B \|^2_{ C^1_b( V, \gamma( U, V_{ -\nicefrac{ \vartheta }{ 2 } } ) ) } \,
  \max\{ 1, \| x \|^2_V \}
\end{split}
\end{equation}
and 
\begin{equation}
\label{eq:psiB_xxderivatives_end}
\begin{split}
&
  \big\|
   \phi''( x ) \, ( v_1, v_2 )
  \big\|_{\mathcal{V}}
\\ & \leq
  \big\|
   \big(\tfrac{ \partial^2 }{ \partial x^2 }
   \psi\big)( x, x ) \, ( v_1, v_2 )
  \big\|_{\mathcal{V}}
  +
  2 \,
  \big\|
   \big(\tfrac{ \partial^2 }{ \partial x \partial y }
   \psi\big)( x, x ) \, ( v_1, v_2 )
  \big\|_{\mathcal{V}}
  +
  \big\|
   \big(\tfrac{ \partial^2 }{ \partial y^2 }
   \psi\big)( x, x ) \, ( v_1, v_2 )
  \big\|_{\mathcal{V}}
\\ & \leq
  \tfrac{
    c_{ -\nicefrac{ \vartheta }{ 2 }, -\nicefrac{ \vartheta }{ 2 }, 0, 0 } \,
    \| B( x ) \|^2_{ \gamma( U, V_{ -\nicefrac{ \vartheta }{ 2 } } ) }
    +
    4 \,
    c_{ -\nicefrac{ \vartheta }{ 2 }, -\nicefrac{ \vartheta }{ 2 }, 0 } \,
    \| B( x ) \|_{ \gamma( U, V_{ -\nicefrac{ \vartheta }{ 2 } } ) } \,
    \| B'( x ) \|_{ L( V, \gamma( U, V_{ -\nicefrac{ \vartheta }{ 2 } } ) ) }
  }{
  ( T - t )^{ \vartheta }
  }
\\ & \quad +
  \tfrac{
    2 \,
    c_{ -\nicefrac{ \vartheta }{ 2 }, -\nicefrac{ \vartheta }{ 2 } } \,
  \big(
  \| B'( x ) \|^2_{ L( V, \gamma( U, V_{ -\nicefrac{ \vartheta }{ 2 } } ) ) }
  +
  \| B( x ) \|_{ \gamma( U, V_{ -\nicefrac{ \vartheta }{ 2 } } ) } \,
  \| B''( x ) \|_{ L^{(2)}( V, \gamma( U, V_{ -\nicefrac{ \vartheta }{ 2 } } ) ) }
  \big)
  }{
  ( T - t )^{ \vartheta }
  }
\\ & \leq
  \tfrac{
    4 \,
    [
    c_{ -\nicefrac{ \vartheta }{ 2 }, -\nicefrac{ \vartheta }{ 2 } }
    +
    c_{ -\nicefrac{ \vartheta }{ 2 }, -\nicefrac{ \vartheta }{ 2 }, 0 } 
    +
    c_{ -\nicefrac{ \vartheta }{ 2 }, -\nicefrac{ \vartheta }{ 2 }, 0, 0 }
    ]
  }{
  ( T - t )^{ \vartheta }
  }
  \,
  \| B \|^2_{ C^2_b( V, \gamma( U, V_{ -\nicefrac{ \vartheta }{ 2 } } ) ) } \,
  \max\{ 1, \| x \|^2_V \}
  .
\end{split}
\end{equation}
In the next step we observe that 
\eqref{eq:psiB_yderivatives_end}, \eqref{eq:psiB_mixedderivatives_middle}, 
\eqref{eq:psiB_mixedderivatives_end},
\eqref{eq:derivative.yxy}, 
and the fact that
$
  \big(
    V \ni x
    \mapsto
   \phi''( x )
   -
   \big(\tfrac{ \partial^2 }{ \partial x^2 }
   \psi\big)( x, x )
   \in L^{(2)}( V, {\mathcal{V}} )
  \big)
  \in C^1( V, L^{(2)}( V, {\mathcal{V}} ) )
$ 
show that for all $ x, x_1, x_2, v_1, v_2, v_3 \in V $ with 
$
  \max\{\|v_1\|_V, \|v_2\|_V,
  \|v_3\|_V\}
  \leq 1
$
it holds that 
\begin{equation}
\begin{split}
\label{eq:psiB_2ndxxLipschitz_begin}
&
  \big\|
   \tfrac{ \partial }{ \partial x }
   \big(
   \phi''( x )
   -
   \big(\tfrac{ \partial^2 }{ \partial x^2 }
   \psi\big)( x, x )
   \big)
   ( v_1, v_2, v_3 )
  \big\|_{\mathcal{V}}
\leq
  2 \,
  \big\|
   \big(\tfrac{ \partial^3 }{ \partial x^2 \partial y }
   \psi\big)( x, x ) \, ( v_1, v_2, v_3 )
  \big\|_{\mathcal{V}}
\\&+
  2 \,
  \big\|
  \big(\tfrac{ \partial^3 }{ \partial y \partial x \partial y }
  \psi\big)( x, x ) \, ( v_1, v_2, v_3 )
  \big\|_{\mathcal{V}}
+
  \big\|
   \big(\tfrac{ \partial^3 }{ \partial x \partial y^2 }
   \psi\big)( x, x ) \, ( v_1, v_2, v_3 )
  \big\|_{\mathcal{V}}
\\&+
  \big\|
   \big(\tfrac{ \partial^3 }{ \partial y^3 }
   \psi\big)( x, x ) \, ( v_1, v_2, v_3 )
  \big\|_{\mathcal{V}}
\\ & \leq
  \tfrac{
    4 \, c_{ -\nicefrac{ \vartheta }{ 2 }, -\nicefrac{ \vartheta }{ 2 }, 0, 0 } \,
    \| B( x ) \|_{ \gamma( U, V_{ -\nicefrac{ \vartheta }{ 2 } } ) } \,
    \| B'( x ) \|_{ L( V, \gamma( U, V_{ -\nicefrac{ \vartheta }{ 2 } } ) ) }
  }{
  ( T - t )^{ \vartheta }
  }
\\ & +
  \tfrac{
    6 \,
    c_{ -\nicefrac{ \vartheta }{ 2 }, -\nicefrac{ \vartheta }{ 2 }, 0 } \,
  \big(
  \| B'( x ) \|^2_{ L( V, \gamma( U, V_{ -\nicefrac{ \vartheta }{ 2 } } ) ) }
  +
  \| B( x ) \|_{ \gamma( U, V_{ -\nicefrac{ \vartheta }{ 2 } } ) }
  \,
  \| B''( x ) \|_{ L^{(2)}( V, \gamma( U, V_{ -\nicefrac{ \vartheta }{ 2 } } ) ) }
  \big)
  }{
  ( T - t )^{ \vartheta }
  }
\\ & +
  \tfrac{
    6 \, c_{ -\nicefrac{ \vartheta }{ 2 }, -\nicefrac{ \vartheta }{ 2 } }
    \,
    \big(
    \| B'( x ) \|_{ L( V, \gamma( U, V_{ -\nicefrac{ \vartheta }{ 2 } } ) ) } \,
    \| B''( x ) \|_{ L^{(2)}( V, \gamma( U, V_{ -\nicefrac{ \vartheta }{ 2 } } ) ) }
    +
    \| B( x ) \|_{ \gamma( U, V_{ -\nicefrac{ \vartheta }{ 2 } } ) } \,
    \| B^{(3)}( x ) \|_{ L^{(3)}( V, \gamma( U, V_{ -\nicefrac{ \vartheta }{ 2 } } ) ) }
    \big)
  }{
  ( T - t )^{ \vartheta }
  }
\\ & \leq
  \tfrac{
    6 \,
    [
    c_{ -\nicefrac{ \vartheta }{ 2 }, -\nicefrac{ \vartheta }{ 2 } }
    +
    c_{ -\nicefrac{ \vartheta }{ 2 }, -\nicefrac{ \vartheta }{ 2 }, 0 } 
    +
    c_{ -\nicefrac{ \vartheta }{ 2 }, -\nicefrac{ \vartheta }{ 2 }, 0, 0 }
    ]
  }{
  ( T - t )^{ \vartheta }
  }
  \,
  \| B \|^2_{ C^3_b( V, \gamma( U, V_{ -\nicefrac{ \vartheta }{ 2 } } ) ) } \,
  \max\{ 1, \| x \|_V \}
  .
\end{split}
\end{equation}
In addition, we combine~\eqref{eq:psiB_xderivatives_end} 
and~\eqref{eq:derivative.yxx}
with item~\eqref{item:mollified.c.delta.finite} of Lemma~\ref{lem:mollified.c.delta.finite} and the fundamental theorem of calculus in Banach spaces 
to obtain that for all $ x_1, x_2, v_1, v_2 \in V $ with 
$
  \max\{\|v_1\|_V, \|v_2\|_V\}
  \leq 1
$
it holds that 
\begin{equation}
\label{eq:psiB_2ndxxLipschitz_end}
\begin{split}
&
  \big\|
   \big(
   \big(\tfrac{ \partial^2 }{ \partial x^2 }
   \psi\big)( x_1, x_1 )
   -
   \big(\tfrac{ \partial^2 }{ \partial x^2 }
   \psi\big)( x_2, x_2 )
   \big)
   ( v_1, v_2 )
  \big\|_{\mathcal{V}}
\\ & \leq
  \big\|
   \big(
   \big(\tfrac{ \partial^2 }{ \partial x^2 }
   \psi\big)( x_1, x_1 )
   -
   \big(\tfrac{ \partial^2 }{ \partial x^2 }
   \psi\big)( x_2, x_1 )
   \big)
   ( v_1, v_2 )
  \big\|_{\mathcal{V}}
\\ & 
  +
  \big\|
   \big(
   \big(\tfrac{ \partial^2 }{ \partial x^2 }
   \psi\big)( x_2, x_1 )
   -
   \big(\tfrac{ \partial^2 }{ \partial x^2 }
   \psi\big)( x_2, x_2 )
   \big)
   ( v_1, v_2 )
  \big\|_{\mathcal{V}}
\\ & \leq
  \tfrac{
    \tilde{c}_{ -\nicefrac{ \vartheta }{ 2 }, -\nicefrac{ \vartheta }{ 2 }, 0, 0 } \,
    \| x_1 - x_2 \|_V
  }{
  ( T - t )^{ \vartheta }
  } \,
  \| B( x_1 ) \|^2_{ \gamma( U, V_{ -\nicefrac{ \vartheta }{ 2 } } ) }
\\ & 
  +
  \tfrac{
    2 \,
    c_{ -\nicefrac{ \vartheta }{ 2 }, -\nicefrac{ \vartheta }{ 2 }, 0, 0 } \,
    \| x_1 - x_2 \|_V
  }{
    ( T - t )^{ \vartheta }
  } 
  \,
  \| B \|^2_{ 
    C^1_b( V, \gamma( U, V_{ -\nicefrac{ \vartheta }{ 2 } } ) ) 
  } 
  \,
  \max\{ 1, \| x_1 \|_V, \| x_2 \|_V \}
\\ & \leq
  \tfrac{
    2 \,
    \| x_1 - x_2 \|_V
  }{
  ( T - t )^{ \vartheta }
  }
  \,
  \| B \|^2_{ C^1_b( V, \gamma( U, V_{ -\nicefrac{ \vartheta }{ 2 } } ) ) } \,
  \big[
    c_{ -\nicefrac{ \vartheta }{ 2 }, -\nicefrac{ \vartheta }{ 2 }, 0, 0 } 
    +
    \tilde{c}_{ -\nicefrac{ \vartheta }{ 2 }, -\nicefrac{ \vartheta }{ 2 }, 0, 0 }
  \big]
  \max\{ 1, \| x_1 \|^2_V, \| x_2 \|^2_V \}
  .
\end{split}
\end{equation}
Combining~\eqref{eq:psiB_xxderivatives_begin}--\eqref{eq:psiB_2ndxxLipschitz_end} with Lemma~\ref{lem:mollified.c.delta.finite} and the fundamental theorem of calculus in Banach spaces 
finally yields~\eqref{eq:psiB_Lipschitzxx}. 
The proof of Lemma~\ref{lem:weak_regularity_B_2nd_term} is thus completed.
\end{proof}

\begin{lemma}[Weak convergence of semilinear integrated exponential Euler
approximations of SPDEs with mollified nonlinearities]
\label{lem:mollified_weak_solution-integrated_num}
Assume the setting in Section~\ref{sec:setting_euler_integrated_mollified} 
and let $ \rho \in [ 0, 1 - \vartheta ) $. 
Then it holds that 
$
  \ES\big[
  \| \varphi(X_T) \|_{\mathcal{V}}
  +
  \| \varphi(\bar{Y}_T) \|_{\mathcal{V}}
  \big]
  < \infty
$ 
and 
\begin{equation}
\label{eq:mollified_weak_solution-integrated_num}
\begin{split}
&
  \left\|
    \ES\big[
      \varphi( X_T )
    \big]
    -
    \ES\big[
      \varphi( \bar{Y}_T )
    \big]
  \right\|_{\mathcal{V}}
\leq
  \tfrac{
    5 \, |\SGchi{0}|^3 \, | \SGchi{\rho} |^2 \, 
    T^{ ( 1 - \vartheta - \rho ) }
  }{ 
    ( 1 - \vartheta - \rho ) 
  } \, 
  \varsigma_{ F, B } \, K_5 \, h^\rho \,
\\ & \cdot
  \Big[
  c_{ -\vartheta } + c_{ -\vartheta, 0 }
  + c_{ -\vartheta, 0, 0 } + c_{ -\vartheta, 0, 0, 0 }
  +
  c_{ -\nicefrac{ \vartheta }{ 2 }, -\nicefrac{ \vartheta }{ 2 } }
  +
  c_{ -\nicefrac{ \vartheta }{ 2 }, -\nicefrac{ \vartheta }{ 2 }, 0 }
  +
  c_{ -\nicefrac{ \vartheta }{ 2 }, -\nicefrac{ \vartheta }{ 2 }, 0, 0 }
  +
  \tilde{c}_{ -\nicefrac{ \vartheta }{ 2 }, -\nicefrac{ \vartheta }{ 2 }, 0, 0 }
  \Big]
\\ & \cdot
  \Bigg[
    2^\rho
    +
    \tfrac{ T^{ ( 1 - \vartheta ) } }{ ( 1 - \vartheta - \rho ) }
  \bigg(
       3 \, \SGchi{\vartheta} + 2 \, \SGchi{ \rho + \vartheta } + 3 \, | \SGchi{ \nicefrac{\vartheta}{2} } |^2
       +
       4 \, \SGchi{ \rho + \nicefrac{\vartheta}{2} } \, \SGchi{ \nicefrac{\vartheta}{2} } 
+
  2 \, ( | \SGchi{ \nicefrac{ \vartheta }{ 2 } } |^2 + \SGchi{\vartheta} ) \, \SGchi{\rho}
\\&\cdot
  \Big[
    \tfrac{
    \SGchi{ \vartheta } \, T^{ ( 1 - \vartheta ) } \,
    }{ ( 1 - \vartheta ) }
    +
    \tfrac{
    \BDG{4}{} \, \SGchi{ \nicefrac{\vartheta}{2} } \, 
    T^{ ( 1 - \vartheta )/2 }
    }{
    \sqrt{ 1 - \vartheta }
    }
  \Big]
  \bigg)
  \Bigg]
  < \infty
  .
\end{split}
\end{equation}
\end{lemma}
\begin{proof}
We first observe that the assumption that 
$
  \sup_{ t \in [ 0, T ] }
  \ES\big[
  \| X_t \|^5_V
  \big]
  < \infty
$ 
implies that 
$
  \ES\big[
    \| \varphi(X_T) \|_{\mathcal{V}}
  \big]
$
$
  < \infty
$. 
Moreover, combining the assumption that 
$
  Y_0
  \in \lpn{5}{\P}{V_2}
$ 
with Lemma~\ref{lem:Kp_estimate} 
proves that $ K_5 < \infty $. 
This shows, in particular, that
\begin{equation}
\label{eq:moment.finite}
  \sup_{ s \in [ 0, T ] } 
  \ES\big[
    \| \varphi(\bar{Y}_T) \|_{\mathcal{V}}
    +
    \| \bar{Y}_s \|_{ V_2 } 
    +
      \int^T_0
      \|
      u_{ 0, 1 }( t, \bar{Y}_t )
      \,
      B( Y_{ \floor{t}{h} } )
      \|^2_{ \gamma(U,{\mathcal{V}}) }
      \, dt
  \big]
  < \infty
  .
\end{equation}
In addition, note that 
Kolmogorov-Chentsov's theorem,
the fact that 
$
  X_0
  \in \lpn{5}{\P}{V}
$, 
the fact that
$F\in\operatorname{Lip}^0(V,V)$, 
and the fact that
$B\in\operatorname{Lip}^0(V,\gamma(U,V))$
ensure that there exists a continuous modification of $X$.
This, Lemma~\ref{lem:markov.lip}, and the assumption that $X_0=\bar{Y}_0$ yield that
\begin{equation}
    \ES\big[ 
      \varphi( \bar{Y}_T )
    \big]
    -
    \ES\big[ 
      \varphi( X_T )
    \big]
    =
  \E\left[ 
    u( T, \bar{Y}_T )
    -
    u( 0, \bar{Y}_0 )
  \right]
  .
\end{equation}
Items~\eqref{item:time.derivative.continuity}--\eqref{item:space.derivative.continuity} of Lemma~\ref{lem:Kolmogorov}, \eqref{eq:moment.finite},
and the standard It{\^o} formula in Theorem~2.4 in
Brze\'{z}niak et al.~\cite{bvvw08}
therefore
prove that
\begin{equation}
\begin{split}
&
    \ES\big[ 
      \varphi( \bar{Y}_T )
    \big]
    -
    \ES\big[ 
      \varphi( X_T )
    \big]
    =
  \E\left[ 
    u( T, \bar{Y}_T )
    -
    u( 0, \bar{Y}_0 )
  \right]
\\ & =
  \int_0^T
  \E\left[
    u_{1,0}( t, \bar{Y}_t )
    +
    u_{0,1}( t, \bar{Y}_t )
    \!\left(
      A
      \bar{Y}_t
      +
      F(
        Y_{ \floor{t}{h} } 
      )
    \right)
  \right]
  dt
\\ & 
  +
  \frac{ 1 }{ 2 }
  \sum_{ b\in\mathbb{U} }
  \int_0^T
  \E\left[
    u_{0,2}( t, \bar{Y}_t )\!\left(
      B^b(
        Y_{ \floor{t}{h} }
      )
      ,
      B^b(
        Y_{ \floor{t}{h} }
      )
    \right)
  \right]
  dt
  .
\end{split}
\end{equation}
Item~\eqref{item:kolmogorov.eq} of Lemma~\ref{lem:Kolmogorov}
hence shows that
\begin{equation}
\label{eq:Kolmogorov}
\begin{split}
&
    \ES\big[ 
      \varphi( \bar{Y}_T )
    \big]
    -
    \ES\big[ 
      \varphi( X_T )
    \big]
\\ & =
  \int_0^T
  \E\left[
    u_{0,1}( t, \bar{Y}_t )
    \,
      F(
        Y_{ \floor{ t }{ h } } 
      )
  -
    u_{0,1}( t, \bar{Y}_t )
    \,
      F(
        \bar{Y}_t 
      )
  \right]
  dt
\\ & +
  \frac{ 1 }{ 2 }
  \sum_{ b\in\mathbb{U} }
  \int_0^T
  \E\left[
    u_{0,2}( t, \bar{Y}_t )\!\left(
      B^b\big(
        Y_{ \floor{ t }{ h } }
      \big)
      ,
      B^b\big(
        Y_{ \floor{ t }{ h } }
      \big)
    \right)
  -
    u_{0,2}( t, \bar{Y}_t )\!\left(
      B^b\big(
        \bar{Y}_t 
      \big)
      ,
      B^b\big(
        \bar{Y}_t 
      \big)
    \right)
  \right]
  dt
  .
\end{split}
\end{equation}
The triangle inequality hence shows that 
\begin{align}
\label{eq:mild_ito}
&
  \left\|
  \ES
  \big[
  \varphi( X_T )
  \big]
  -
  \ES
  \big[
  \varphi( \bar{Y}_T )
  \big]
  \right\|_{\mathcal{V}}
\nonumber
\\ & \leq
\nonumber
  \smallint_0^T
  \left\|\E\left[
    u_{0,1}( t, \bar{Y}_t )
    \,
      F(
        Y_{ \floor{ t }{ h } } 
      )
    -
    u_{0,1}( t, \bar{Y}_{ \floor{ t }{ h } } )
    \,
      F(
        Y_{ \floor{ t }{ h } } 
      )
  \right]\right\|_{\mathcal{V}}
  dt
\\ & +
\nonumber
  \smallint_0^T
  \left\|\E\left[
    u_{0,1}( t, \bar{Y}_{ \floor{ t }{ h } } )
    \,
      F(
        Y_{ \floor{ t }{ h } } 
      )
  -
    u_{0,1}( t, \bar{Y}_{ \floor{ t }{ h } } )
    \,
      F(
        \bar{Y}_{ \floor{ t }{ h } } 
      )
  \right]\right\|_{\mathcal{V}}
  dt
\\ & +
  \smallint_0^T
  \left\|\E\left[
    u_{0,1}( t, \bar{Y}_{ \floor{ t }{ h } } )
    \,
      F(
        \bar{Y}_{ \floor{ t }{ h } } 
      )
    -
    u_{0,1}( t, \bar{Y}_t )
    \,
      F(
        \bar{Y}_t 
      )
  \right]\right\|_{\mathcal{V}}
  dt
\\ & +
\nonumber
  \tfrac{ 1 }{ 2 } 
  \smallint_0^T
  \left\|\E\left[
  \,{\smallsum\limits_{ b\in\mathbb{U} }}
    u_{0,2}( t, \bar{Y}_t )\!\left(
      B^b\big(
        Y_{ \floor{ t }{ h } }
      \big)
      ,
      B^b\big(
        Y_{ \floor{ t }{ h } }
      \big)
    \right)
  -
  {\smallsum\limits_{ b\in\mathbb{U} }}
    u_{0,2}( t, \bar{Y}_{ \floor{ t }{ h } } )\!\left(
      B^b\big(
        Y_{ \floor{ t }{ h } }
      \big)
      ,
      B^b\big(
        Y_{ \floor{ t }{ h } }
      \big)
    \right)
  \right]\right\|_{\mathcal{V}}
  dt
\\ & +
\nonumber
  \tfrac{ 1 }{ 2 } 
  \smallint_0^T
  \left\|\E\left[
  \,{\smallsum\limits_{ b\in\mathbb{U} }}
    u_{0,2}( t, \bar{Y}_{ \floor{ t }{ h } } )\!\left(
      B^b\big(
        Y_{ \floor{ t }{ h } }
      \big)
      ,
      B^b\big(
        Y_{ \floor{ t }{ h } }
      \big)
    \right)
  -
  {\smallsum\limits_{ b\in\mathbb{U} }}
    u_{0,2}( t, \bar{Y}_{ \floor{ t }{ h } } )\!\left(
      B^b\big(
        \bar{Y}_{ \floor{ t }{ h } } 
      \big)
      ,
      B^b\big(
        \bar{Y}_{ \floor{ t }{ h } } 
      \big)
    \right)
  \right]\right\|_{\mathcal{V}}
  dt
\\ & +
\nonumber
  \tfrac{ 1 }{ 2 } 
  \smallint_0^T
  \left\|\E\left[
  \,{\smallsum\limits_{ b\in\mathbb{U} }}
    u_{0,2}( t, \bar{Y}_{ \floor{ t }{ h } } )\!\left(
      B^b\big(
        \bar{Y}_{ \floor{ t }{ h } }
      \big)
      ,
      B^b\big(
        \bar{Y}_{ \floor{ t }{ h } }
      \big)
    \right)
-
  {\smallsum\limits_{ b\in\mathbb{U} }}
    u_{0,2}( t, \bar{Y}_t )\!\left(
      B^b\big(
        \bar{Y}_t
      \big)
      ,
      B^b\big(
        \bar{Y}_t
      \big)
    \right)
  \right]\right\|_{\mathcal{V}}
  dt
  .
\end{align}
In the next step we combine 
Lemma~\ref{lem:weak_regularity_F_2nd_term} and Lemma~\ref{lem:weak_regularity_B_2nd_term} 
with Proposition~\ref{prop:weak_temporal_regularity_1st} 
to obtain that for all 
$ t \in ( 0, T ) $ 
it holds that 
\allowdisplaybreaks
\begin{align}
\label{eq:mild_ito_1st_term}
&
  \left\|\E\left[
    u_{0,1}( t, \bar{Y}_t )
    \,
      F(
        Y_{ \floor{ t }{ h } } 
      )
    -
    u_{0,1}( t, \bar{Y}_{ \floor{ t }{ h } } )
    \,
      F(
        Y_{ \floor{ t }{ h } } 
      )
  \right]\right\|_{\mathcal{V}}
\nonumber
\\ & +
\nonumber
  \left\|\E\left[
    u_{0,1}( t, \bar{Y}_{ \floor{ t }{ h } } )
    \,
      F(
        \bar{Y}_{ \floor{ t }{ h } } 
      )
    -
    u_{0,1}( t, \bar{Y}_t )
    \,
      F(
        \bar{Y}_t 
      )
  \right]\right\|_{\mathcal{V}}
\\ & +
\nonumber
  \tfrac{ 1 }{ 2 } 
  \left\|\E\left[
  \,{\smallsum\limits_{ b\in\mathbb{U} }}
    u_{0,2}( t, \bar{Y}_t )\!\left(
      B^b\big(
        Y_{ \floor{ t }{ h } }
      \big)
      ,
      B^b\big(
        Y_{ \floor{ t }{ h } }
      \big)
    \right)
  -
  {\smallsum\limits_{ b\in\mathbb{U} }}
    u_{0,2}( t, \bar{Y}_{ \floor{ t }{ h } } )\!\left(
      B^b\big(
        Y_{ \floor{ t }{ h } }
      \big)
      ,
      B^b\big(
        Y_{ \floor{ t }{ h } }
      \big)
    \right)
  \right]\right\|_{\mathcal{V}}
\\ & +
\nonumber
  \tfrac{ 1 }{ 2 } 
  \left\|\E\left[
  \,{\smallsum\limits_{ b\in\mathbb{U} }}
    u_{0,2}( t, \bar{Y}_{ \floor{ t }{ h } } )\!\left(
      B^b\big(
        \bar{Y}_{ \floor{ t }{ h } }
      \big)
      ,
      B^b\big(
        \bar{Y}_{ \floor{ t }{ h } }
      \big)
    \right)
  -
  {\smallsum\limits_{ b\in\mathbb{U} }}
    u_{0,2}( t, \bar{Y}_t )\!\left(
      B^b\big(
        \bar{Y}_t
      \big)
      ,
      B^b\big(
        \bar{Y}_t
      \big)
    \right)
  \right]\right\|_{\mathcal{V}}
\\ & \leq
  \tfrac{
    | \SGchi{0} |^3 \, 
    | \SGchi{\rho} |^2 \,}{ ( T - t )^\vartheta }
    \,
    K_5 
    \,
    h^\rho
    \,
  \max\!\big\{
    1,
    \|
      F
    \|_{ 
      \operatorname{Lip}^0( V, V_{-\vartheta} )
    }
    ,
    \|
      B
    \|^2_{ 
      \operatorname{Lip}^0( V, \gamma( U, V_{-\nicefrac{\vartheta}{2}} ) )
    }
  \big\}
\\ & \cdot
\nonumber
  \bigg[
  4 \,
  \big[
  c_{ -\vartheta } + c_{ -\vartheta, 0 }
  + c_{ -\vartheta, 0, 0 } + c_{ -\vartheta, 0, 0, 0 }
  \big] \,
  \| F \|_{ C^3_b( V, V_{ -\vartheta } ) }
\\ & +
\nonumber
  5 \,
  \big[
  c_{ -\nicefrac{ \vartheta }{ 2 }, -\nicefrac{ \vartheta }{ 2 } }
  +
  c_{ -\nicefrac{ \vartheta }{ 2 }, -\nicefrac{ \vartheta }{ 2 }, 0 }
  +
  c_{ -\nicefrac{ \vartheta }{ 2 }, -\nicefrac{ \vartheta }{ 2 }, 0, 0 }
  +
  \tilde{c}_{ -\nicefrac{ \vartheta }{ 2 }, -\nicefrac{ \vartheta }{ 2 }, 0, 0 }
  \big] \,
  \| B \|^2_{ C^3_b( V, \gamma( U, V_{ -\nicefrac{ \vartheta }{ 2 } } ) ) }
  \bigg]
\\ & \cdot
\nonumber
  \bigg[
    \tfrac{
      2^{ \rho }
    }{
      t^{ \rho }
    }
    +
    \tfrac{
      \left(
        2 \, 
        \SGchi{\vartheta} 
        +
        \SGchi{ \rho + \vartheta }
        +
        2 \, 
        | \SGchi{ \vartheta / 2 } |^2
        + 
        2 \,
        \SGchi{ \rho + \vartheta / 2 }
        \,
        \SGchi{ \vartheta / 2 }
      \right)
      \,
      |\floor{t}{h}|^{ ( 1 - \vartheta - \rho ) }
      +
      \left( 
        \SGchi{\vartheta} 
        +
        \frac{ 1 }{ 2 }
        | \SGchi{ \nicefrac{ \vartheta }{ 2 } } |^2 
      \right)
      \,
      ( t - \floor{t}{h} )^{ ( 1 - \vartheta - \rho ) }
    }{ 
      ( 1 - \vartheta - \rho ) }
  \bigg]
  .
\end{align}
In addition, we combine 
Lemma~\ref{lem:weak_regularity_F_2nd_term} and 
Lemma~\ref{lem:weak_regularity_B_2nd_term} 
with Proposition~\ref{prop:weak_temporal_regularity_2nd} 
to obtain that for all 
$ t \in ( 0, T ) $ 
it holds that 
\allowdisplaybreaks
\begin{align}
\label{eq:mild_ito_2nd_term}
\nonumber
&
  \left\|\E\left[
    u_{0,1}( t, \bar{Y}_{ \floor{ t }{ h } } )
    \,
      F(
        Y_{ \floor{ t }{ h } } 
      )
  -
    u_{0,1}( t, \bar{Y}_{ \floor{ t }{ h } } )
    \,
      F(
        \bar{Y}_{ \floor{ t }{ h } } 
      )
  \right]\right\|_{\mathcal{V}}
\\ & +
\nonumber
  \tfrac{ 1 }{ 2 } 
  \left\|\E\left[
  \,{\smallsum\limits_{ b\in\mathbb{U} }}
    u_{0,2}( t, \bar{Y}_{ \floor{ t }{ h } } )\!\left(
      B^b\big(
        Y_{ \floor{ t }{ h } }
      \big)
      ,
      B^b\big(
        Y_{ \floor{ t }{ h } }
      \big)
    \right)
  -
  {\smallsum\limits_{ b\in\mathbb{U} }}
    u_{0,2}( t, \bar{Y}_{ \floor{ t }{ h } } )\!\left(
      B^b\big(
        \bar{Y}_{ \floor{ t }{ h } } 
      \big)
      ,
      B^b\big(
        \bar{Y}_{ \floor{ t }{ h } } 
      \big)
    \right)
  \right]\right\|_{\mathcal{V}}
\\ & \leq
\nonumber
  \tfrac{
    \SGchi{0} \, \SGchi{ \rho }
  }{
    ( T - t )^{ \vartheta }
  }
  \,
    K_4 \, h^\rho \,
  \max\!\big\{
    1,
    \|
      F
    \|_{ 
      \operatorname{Lip}^0( V, V_{-\vartheta} )
    }
    ,
    \|
      B
    \|^2_{ 
      \operatorname{Lip}^0( V, \gamma( U, V_{-\nicefrac{\vartheta}{2}} ) )
    }
  \big\}
\\ & \cdot
  \max\!\big\{
    1,
    \|
      F
    \|_{ 
      \operatorname{Lip}^0( V, V_{-\vartheta} )
    }
    ,
    \|
      B
    \|_{ 
      \operatorname{Lip}^0( V, \gamma( U, V_{-\nicefrac{\vartheta}{2}} ) )
    }
  \big\}
  \,
  \Big(
  \big[
  c_{ -\vartheta } + c_{ -\vartheta, 0 }
  + c_{ -\vartheta, 0, 0 }
  \big] \,
  \| F \|_{ C^3_b( V, V_{ -\vartheta } ) }
\\ & +
\nonumber
  3 \,
  \big[
  c_{ -\nicefrac{ \vartheta }{ 2 }, -\nicefrac{ \vartheta }{ 2 } }
  +
  c_{ -\nicefrac{ \vartheta }{ 2 }, -\nicefrac{ \vartheta }{ 2 }, 0 }
  +
  c_{ -\nicefrac{ \vartheta }{ 2 }, -\nicefrac{ \vartheta }{ 2 }, 0, 0 }
  \big] \,
  \| B \|^2_{ C^3_b( V, \gamma( U, V_{ -\nicefrac{ \vartheta }{ 2 } } ) ) }
  \Big)
\\ & \cdot
\nonumber
    \tfrac{ | \floor{t}{h} |^{ ( 1 - \vartheta - \rho ) } }{ ( 1 - \vartheta - \rho ) }
  \bigg(
    \SGchi{ \rho+\vartheta } + 2 \, \SGchi{ \nicefrac{ \vartheta }{ 2 } } \, \SGchi{ \rho+\nicefrac{ \vartheta }{ 2 } }
+
  2 \, ( | \SGchi{ \nicefrac{ \vartheta }{ 2 } } |^2 + \SGchi{\vartheta} ) \, \SGchi{\rho}
  \Big[
    \tfrac{
    \SGchi{ \vartheta } \, T^{ ( 1 - \vartheta ) } \,
    }{ ( 1 - \vartheta ) }
    +
    \tfrac{
    \BDG{4}{} \, \SGchi{ \nicefrac{\vartheta}{2} } \, 
    T^{ ( 1 - \vartheta )/2 }
    }{
    \sqrt{ 1 - \vartheta }
    }
  \Big]
  \bigg)
  .
\end{align}
Combining \eqref{eq:mild_ito}--\eqref{eq:mild_ito_2nd_term} proves that 
\begin{equation}
\label{eq:mild_ito_final}
\begin{split}
&
\left\|
\ES
\big[
\varphi( X_T )
\big]
-
\ES
\big[
\varphi( \bar{Y}_T )
\big]
\right\|_{\mathcal{V}}
\leq
5 \, |\SGchi{0}|^3 \, | \SGchi{\rho} |^2 \, \varsigma_{ F, B } \, K_5 \, h^\rho \,
\smallint^T_0
\tfrac{1}{
	( T - t )^\vartheta \, t^\rho
}
\, dt
\\ & \cdot
\Big[
c_{ -\vartheta } + c_{ -\vartheta, 0 }
+ c_{ -\vartheta, 0, 0 } + c_{ -\vartheta, 0, 0, 0 }
+
c_{ -\nicefrac{ \vartheta }{ 2 }, -\nicefrac{ \vartheta }{ 2 } }
+
c_{ -\nicefrac{ \vartheta }{ 2 }, -\nicefrac{ \vartheta }{ 2 }, 0 }
+
c_{ -\nicefrac{ \vartheta }{ 2 }, -\nicefrac{ \vartheta }{ 2 }, 0, 0 }
+
\tilde{c}_{ -\nicefrac{ \vartheta }{ 2 }, -\nicefrac{ \vartheta }{ 2 }, 0, 0 }
\Big]
\\ & \cdot
\Bigg[
2^\rho
+
\tfrac{ T^{ ( 1 - \vartheta ) } }{ ( 1 - \vartheta - \rho ) }
\bigg(
3 \, \SGchi{\vartheta} + 2 \, \SGchi{ \rho + \vartheta } + 3 \, | \SGchi{ \nicefrac{\vartheta}{2} } |^2
+
4 \, \SGchi{ \rho + \nicefrac{\vartheta}{2} } \, \SGchi{ \nicefrac{\vartheta}{2} }
+
2 \, ( | \SGchi{ \nicefrac{ \vartheta }{ 2 } } |^2 + \SGchi{\vartheta} ) \, \SGchi{\rho}
\\ & \cdot
\Big[
\tfrac{
	\SGchi{ \vartheta } \, T^{ ( 1 - \vartheta ) } \,
}{ ( 1 - \vartheta ) }
+
\tfrac{
	\BDG{4}{arg2} \, \SGchi{ \nicefrac{\vartheta}{2} } \, 
	T^{ ( 1 - \vartheta )/2 }
}{
\sqrt{ 1 - \vartheta }
}
\Big]
\bigg)
\Bigg]
.
\end{split}
\end{equation}
This and the fact that for all 
$x,y\in(0,\infty)$ 
with 
$
(x-1)(y-1)
\geq 0
$
and 
$
x+y > 1
$
it holds that 
\begin{equation}
\label{eq:Beta.ub}
\int^1_0
(1-t)^{(x-1)} \,
t^{(y-1)}
\, dt
\leq
\frac{1}{(x+y-1)}
\end{equation}
establish the first inequality in~\eqref{eq:mollified_weak_solution-integrated_num}. 
The second inequality in~\eqref{eq:mollified_weak_solution-integrated_num} follows from Lemma~\ref{lem:mollified.c.delta.finite}.
The proof of Lemma~\ref{lem:mollified_weak_solution-integrated_num}
is thus completed.
\end{proof}

\subsection{Weak convergence rates for exponential Euler approximations of SPDEs with mollified nonlinearities}

The next result, Corollary~\ref{cor:weak_error_numerics-integrated},
provides a bound for the weak distance of the numerical approximation
and its semilinear integrated counterpart.
Corollary~\ref{cor:weak_error_numerics-integrated}
is an immediate consequence 
of Proposition~\ref{prop:weak_temporal_regularity_2nd}
and of Lemma~\ref{lem:mollified_weak_solution-integrated_num}.

\begin{corollary}[Weak distance between exponential Euler approximations of SPDEs 
with mollified nonlinearities and their semilinear integrated counterparts]
\label{cor:weak_error_numerics-integrated}
Assume the setting in Section~\ref{sec:setting_euler_integrated_mollified} 
and let $ \rho \in [ 0, 1 - \vartheta ) $. Then it holds that 
$
  \ES\big[
  \| \varphi(\bar{Y}_T) \|_{\mathcal{V}}
  +
  \| \varphi(Y_T) \|_{\mathcal{V}}
  \big]
  < \infty
$ 
and 
\begin{equation}
\begin{split}
&
\left\|
\ES\big[ \varphi( \bar{Y}_T ) \big]
-
\ES\big[ \varphi( Y_T ) \big]
\right\|_{\mathcal{V}}
\leq
\SGchi{ \rho } \,
\| \varphi \|_{ \operatorname{Lip}^2( V, {\mathcal{V}} ) }
\,
K_3 \, h^\rho \, \varsigma_{ F, B }
\\ & \cdot
\tfrac{ T^{ ( 1 - \vartheta - \rho ) } }{ ( 1 - \vartheta - \rho ) }
\bigg(
\SGchi{ \rho+\vartheta } + 2 \, \SGchi{ \nicefrac{ \vartheta }{ 2 } } \, \SGchi{ \rho+\nicefrac{ \vartheta }{ 2 } }
+
2 \, ( | \SGchi{ \nicefrac{ \vartheta }{ 2 } } |^2 + \SGchi{\vartheta} ) \, \SGchi{\rho}
\Big[
\tfrac{
	\SGchi{ \vartheta } \, T^{ ( 1 - \vartheta ) } \,
}{ ( 1 - \vartheta ) }
+
\tfrac{
	\BDG{3}{} \,
	\SGchi{ \nicefrac{\vartheta}{2} } \, 
	\sqrt{ T^{ ( 1 - \vartheta ) } }
}{
\sqrt{ 1 - \vartheta }
}
\Big]
\bigg)
.
\end{split}
\end{equation}
\end{corollary}

The next result is a direct consequence of the triangle inequality,
of Corollary~\ref{cor:weak_error_numerics-integrated}, 
and of Lemma~\ref{lem:mollified_weak_solution-integrated_num}.

\begin{corollary}[Weak convergence of exponential Euler approximations of SPDEs with mollified nonlinearities]
\label{cor:mollified_weak_solution-num}
Assume the setting in Section~\ref{sec:setting_euler_integrated_mollified} and let 
$ \rho \in [ 0, 1 - \vartheta ) $. 
Then it holds that 
$
  \ES\big[
  \| \varphi(X_T) \|_{\mathcal{V}}
  +
  \| \varphi(Y_T) \|_{\mathcal{V}}
  \big]
  < \infty
$ 
and 
\begin{equation}\label{eq:mollified_weak_solution-num}
\begin{split}
&
  \left\|
  \ES
  \big[
  \varphi( X_T )
  \big]
  -
  \ES
  \big[
  \varphi( Y_T )
  \big]
  \right\|_{\mathcal{V}}
\leq
  \tfrac{
    5 \, 
    | \SGchi{0} |^3 \, 
    | \SGchi{\rho} |^2
    \max\{ 1, T^{ ( 1 - \vartheta ) } \}
  }{ 
    ( 1 - \vartheta - \rho ) \, T^\rho 
  } 
  \, \varsigma_{ F, B } \, K_5 \, h^\rho 
\\ & \cdot
  \Bigg[
    2^\rho
    +
    \tfrac{ T^{ ( 1 - \vartheta ) } }{ ( 1 - \vartheta - \rho ) }
  \bigg(
       3 \, \SGchi{\vartheta} + 2 \, \SGchi{ \rho + \vartheta } + 3 \, | \SGchi{ \nicefrac{\vartheta}{2} } |^2
       +
       4 \, \SGchi{ \rho + \nicefrac{\vartheta}{2} } \, \SGchi{ \nicefrac{\vartheta}{2} }
\\ & +
  2 \, ( | \SGchi{ \nicefrac{ \vartheta }{ 2 } } |^2 + \SGchi{\vartheta} ) \, \SGchi{\rho}
  \Big[
    \tfrac{
    \SGchi{ \vartheta } \, T^{ ( 1 - \vartheta ) } \,
    }{ ( 1 - \vartheta ) }
    +
    \tfrac{
    \max\{\BDG{3}{},\BDG{4}{}\} \, \SGchi{ \nicefrac{\vartheta}{2} } \, 
    T^{ ( 1 - \vartheta )/2 }
    }{
    \sqrt{ 1 - \vartheta }
    }
  \Big]
  \bigg)
  \Bigg]
\\ & \cdot
  \Big[
  \| \varphi \|_{ \operatorname{Lip}^2( V, {\mathcal{V}} ) }
  +
  c_{ -\vartheta } + c_{ -\vartheta, 0 }
  + c_{ -\vartheta, 0, 0 } + c_{ -\vartheta, 0, 0, 0 }
  +
  c_{ -\nicefrac{ \vartheta }{ 2 }, -\nicefrac{ \vartheta }{ 2 } }
  +
  c_{ -\nicefrac{ \vartheta }{ 2 }, -\nicefrac{ \vartheta }{ 2 }, 0 }
\\ & +
  c_{ -\nicefrac{ \vartheta }{ 2 }, -\nicefrac{ \vartheta }{ 2 }, 0, 0 }
+
  \tilde{c}_{ -\nicefrac{ \vartheta }{ 2 }, -\nicefrac{ \vartheta }{ 2 }, 0, 0 }
  \Big]
  < \infty
  .
\end{split}
\end{equation}
\end{corollary}

\begin{corollary}[Weak convergence of exponential Euler approximations of SPDEs with mollified nonlinearities]
\label{cor:mollified_weak_solution-num_komplete}
Assume the setting in Section~\ref{sec:setting_euler_integrated_mollified} and let 
$ \theta \in [ 0, 1 ) $, 
$ \rho \in [ 0, 1 - \vartheta ) $. 
Then it holds that 
$
  \ES\big[
  \| \varphi(X_T) \|_{\mathcal{V}}
  +
  \| \varphi(Y_T) \|_{\mathcal{V}}
  \big]
  < \infty
$ 
and 
\begin{equation}
\begin{split}
&
  \left\|
  \ES
  \big[
  \varphi( X_T )
  \big]
  -
  \ES
  \big[
  \varphi( Y_T )
  \big]
  \right\|_{\mathcal{V}}
\leq
  \tfrac{
    57 \, |\SGchi{0}|^3 \, | \SGchi{\rho} |^2 \, 
    \max\{ 
      1 , T^{ ( 1 - \vartheta ) } 
    \}
  }{ 
    ( 1 - \vartheta - \rho ) \, T^\rho 
  } 
  \, 
  \varsigma_{ F, B } 
  \, 
  h^\rho 
\\ & \cdot
  \left[
    \SGchi{0} 
    \max\{ 1, \| X_0 \|_{ \lpn{5}{\P}{V} } \}
    +
    \tfrac{
      \SGchi{\theta} \,
      T^{ ( 1 - \theta ) } \,
      \| F \|_{ \operatorname{Lip}^0( V, V_{ -\theta } ) }
    }{
      ( 1 - \theta )
    }
      +
      \BDG{5}{} \,
      \SGchi{ \nicefrac{ \theta }{ 2 } }
      \sqrt{
      \tfrac{
      T^{ ( 1 - \theta ) }
      }{
      ( 1 - \theta )
      }
      } \,
      \| B \|_{ \operatorname{Lip}^0( V, \gamma( U, V_{ -\nicefrac{\theta}{2} } ) ) }
  \right]^{ 10 }
\\ & \cdot
  \left|\mathcal{E}_{ ( 1 - \theta ) }\!\left[
    \tfrac{
      \sqrt{ 2 }
      \,
      \SGchi{ \theta }
      \,
      T^{ ( 1 - \theta ) }
      \,
      |
        F
      |_{
        \operatorname{Lip}^0( V, V_{ - \theta } )
      }
    }{
      \sqrt{1 - \theta}
    }
    +
    \BDG{5}{} \, \SGchi{ 
      \theta / 2 
    }
    \sqrt{
      2 \, T^{ ( 1 - \theta ) }
    } \,
    |
      B
    |_{
      \operatorname{Lip}^0( V, \gamma( U, V_{ - \nicefrac{\theta}{2} } ) )
    }
  \right]\right|^5
\\ & \cdot
\Bigg[
2^\rho
+
\tfrac{ T^{ ( 1 - \vartheta ) } }{ ( 1 - \vartheta - \rho ) }
\bigg(
3 \, \SGchi{\vartheta} + 2 \, \SGchi{ \rho + \vartheta } + 3 \, | \SGchi{ \nicefrac{\vartheta}{2} } |^2
+
4 \, \SGchi{ \rho + \nicefrac{\vartheta}{2} } \, \SGchi{ \nicefrac{\vartheta}{2} }
\\ & +
2 \, ( | \SGchi{ \nicefrac{ \vartheta }{ 2 } } |^2 + \SGchi{\vartheta} ) \, \SGchi{\rho}
\Big[
\tfrac{
	\SGchi{ \vartheta } \, T^{ ( 1 - \vartheta ) } \,
}{ ( 1 - \vartheta ) }
+
\tfrac{
	\max\{\BDG{3}{},\BDG{4}{}\} \, \SGchi{ \nicefrac{\vartheta}{2} } \, 
	T^{ ( 1 - \vartheta )/2 }
}{
\sqrt{ 1 - \vartheta }
}
\Big]
\bigg)
\Bigg]
\\ & \cdot
  \Big[
  \| \varphi \|_{ \operatorname{Lip}^2( V, {\mathcal{V}} ) }
  +
  c_{ -\vartheta } + c_{ -\vartheta, 0 }
  + c_{ -\vartheta, 0, 0 } + c_{ -\vartheta, 0, 0, 0 }
  +
  c_{ -\nicefrac{ \vartheta }{ 2 }, -\nicefrac{ \vartheta }{ 2 } }
  +
  c_{ -\nicefrac{ \vartheta }{ 2 }, -\nicefrac{ \vartheta }{ 2 }, 0 }
\\ & +
  c_{ -\nicefrac{ \vartheta }{ 2 }, -\nicefrac{ \vartheta }{ 2 }, 0, 0 }
+
  \tilde{c}_{ -\nicefrac{ \vartheta }{ 2 }, -\nicefrac{ \vartheta }{ 2 }, 0, 0 }
  \Big]
  < \infty
  .
\end{split}
\end{equation}
\end{corollary}

\section{Weak error estimates for exponential Euler approximations of SPDEs}
\label{sec:weak_convergence_irregular}

\subsection{Setting}
\label{sec:setting_weak_convergence_irregular}

Assume the setting in Section~\ref{sec:global_setting},
let
$
A \colon D(A)
\subseteq
V \rightarrow V
$
be a generator of a strongly continuous analytic semigroup
with 
$
\operatorname{spectrum}( A )
\subseteq
\{
z \in \mathbb{C}
\colon
\text{Re}( z ) < \eta
\}
$,
let
$
(
V_r
,
\left\| \cdot \right\|_{ V_r }
)
$,
$ r \in \R $,
be a family of interpolation spaces associated to
$
\eta - A
$, 
let 
$ h \in (0,T] $, 
$ \theta \in [0,1) $, 
$ \vartheta \in [ 0, \nicefrac{ 1 }{ 2 } ) \cap [ 0, \theta ] $, 
$
  F \in 
  \operatorname{Lip}^4( V , V_{ - \theta } ) 
$, 
$
  B \in 
   \operatorname{Lip}^4( 
    V, 
    \gamma( 
      U, 
      V_{ - \theta / 2 }
    ) 
  ) 
$, 
$
  \varphi \in 
   \operatorname{Lip}^4( V, {\mathcal{V}} )
$, 
let
$ \varsigma_{ F, B } \in \R $
be the real number given by 
$  
  \varsigma_{ F, B } 
  =
    \max\!\big\{
    1
    ,
    \|
      F
    \|^3_{ 
      C_b^3( V, V_{-\theta} )
    }
    ,
    \|
      B
    \|_{ 
      C_b^3( V, \gamma( U, V_{-\nicefrac{\theta}{2}} ) )
    }^6
    \big\}
$, 
let 
$
\SGchi{r}
\in [1,\infty)
$, 
$r\in[0,1]$, 
be the real numbers 
which satisfy for all 
$r\in[0,1]$
that 
$
\SGchi{r}
=
\max\{
1
,
\sup_{ t \in (0,T] }
t^r
\,
\|
( \eta - A )^r
e^{ t A }
\|_{ L( V ) }
,
\sup_{ t \in (0,T] }
t^{-r}
\,
\|
( \eta - A )^{ -r }
( e^{ t A } - \operatorname{Id}_V )
\|_{ L( V ) }
\}
$, 
let 
$
  X, Y \colon [0,T] \times \Omega \to V
$
and
$
  X^{ \kappa, x } \colon [0,T] \times \Omega \to V
$, 
$ x \in V $, 
$ \kappa \in [0,T] $, 
be 
$
  ( \mathcal{F}_t )_{ t \in [0,T] }
$-predictable stochastic processes 
which satisfy for all 
$ \kappa \in [0,T] $, 
$ x \in V $ 
that 
$
  \sup_{ t \in [0,T] }
  \big[
  \| X_t \|_{ \lpn{5}{\P}{V} } 
  +
  \| X^{ \kappa, x }_t \|_{ \lpn{5}{\P}{V} } 
  \big]
  < \infty
$, 
$
  X^{ \kappa, x }_0 = x
$, 
and 
$ Y_0 = X_0 $ 
and which satisfy that
for all 
$ \kappa \in [0,T] $, 
$ x \in V $, 
$ t \in (0,T] $ 
it holds $ \P $-a.s.\ that
\begin{equation}
  X_t
  = 
    e^{ t A } X_0 
  + 
    \int_0^t e^{ ( t - s )A }\, F( X_s ) \, ds
  + 
    \int_0^t e^{ ( t - s )A }\, B( X_s ) \, dW_s
    ,
\end{equation}
\begin{equation}
  X^{ \kappa, x }_t
  = 
    e^{ t A } \, x 
  + 
    \int_0^t e^{ ( \kappa + t - s )A }\, F( X^{ \kappa, x }_s ) \, ds
  + 
    \int_0^t e^{ ( \kappa + t - s )A }\, B( X^{ \kappa, x }_s ) \, dW_s
    ,
\end{equation}
\begin{equation}
  Y_t
  = 
    e^{tA}\, Y_0 
  + 
    \int_0^t e^{ (t-\lfloor s \rfloor_h)A }\, F( Y_{ \floor{ s }{ h } } ) \, ds
  + 
    \int_0^t e^{ (t-\lfloor s \rfloor_h)A }\, B( Y_{ \floor{ s }{ h } } ) \, dW_s 
    ,
\end{equation}
let 
$
  u^{ ( \kappa ) } \colon [0,T] \times V \to {\mathcal{V}}, 
$
$
  \kappa \in ( 0, T ],
$
be the functions 
which satisfy
for all $ \kappa \in ( 0, T ] $, $ x \in V $, $ t \in [ 0, T ] $ 
that
$
  u^{ ( \kappa ) }( t, x ) = 
  \ES\big[ 
    \varphi( X^{ \kappa, x }_{ T - t } )
  \big]
$,
let 
$
  c^{ ( \kappa ) }_{ \delta_1, \dots, \delta_k }
  \in [0,\infty]	
$,
$ \delta_1, \dots, \delta_k \in (-\nicefrac{1}{2},0] $,
$ k \in \{ 1, 2, 3, 4 \} $, 
$ \kappa \in ( 0, T ] $, 
be the extended real numbers 
which satisfy for all 
$ \kappa \in ( 0, T ] $, 
$ k \in \{ 1, 2, 3, 4 \} $,
$ \delta_1, \dots, \delta_k \in (-\nicefrac{1}{2},0] $
that
\begin{equation}
\begin{split} 
&
  c^{ ( \kappa ) }_{ \delta_1, \dots, \delta_k }
=
  \sup_{
    t \in [0,T)
  }
  \sup_{ 
    x \in V
  }
  \sup_{ 
    v_1, \dots, v_k \in V \setminus \{ 0 \}
  }
  \left[
  \frac{
    \big\|
      \big( 
        \frac{ 
          \partial^k
        }{
          \partial x^k
        }
        u^{ ( \kappa ) }
      \big)( t, x )( v_1, \dots, v_k )
    \big\|_{\mathcal{V}}
  }{
    ( T - t )^{ 
      (
        \delta_1 + \ldots + \delta_k
      ) 
    }
    \left\| v_1 \right\|_{ V_{ \delta_1 } }
    \cdot
    \ldots
    \cdot
    \left\| v_k \right\|_{ V_{ \delta_k } }
  }
  \right]
  ,
\end{split}
\end{equation}
and let 
$
  \tilde{c}^{ ( \kappa ) }_{ \delta_1, \delta_2, \delta_3, \delta_4 }
  \in [0,\infty]	
$,
$ \delta_1, \delta_2, \delta_3, \delta_4 \in (-\nicefrac{1}{2},0] $, 
$ \kappa \in ( 0, T ] $, 
be the extended real numbers 
which satisfy
for all 
$ \kappa \in ( 0, T ] $, 
$ \delta_1, \delta_2, \delta_3, \delta_4 \in (-\nicefrac{1}{2},0] $
that
\begin{equation}
\begin{split} 
&
  \tilde{c}^{ ( \kappa ) }_{ \delta_1, \delta_2, \delta_3, \delta_4 }
\\ & =
  \sup_{ t \in [ 0, T ) } \,
  \sup_{\substack{
    x_1, x_2 \in V, \\ x_1 \neq x_2
  }} \,
  \sup_{ 
    v_1, \ldots, v_4 \in V \setminus \{ 0 \}
  }
  \left[
  \frac{
    \big\|
    \big(
      \big( 
        \frac{ 
          \partial^4
        }{
          \partial x^4
        }
        u^{ ( \kappa ) }
      \big)( t, x_1 )
      -
      \big( 
        \frac{ 
          \partial^4
        }{
          \partial x^4
        }
        u^{ ( \kappa ) }
      \big)( t, x_2 )
      \big)
      ( v_1, \dots, v_4 )
    \big\|_{\mathcal{V}}
  }{
    ( T - t )^{ 
      (
        \delta_1 + \ldots + \delta_4
      ) 
    }
    \left\| x_1 - x_2 \right\|_V
    \left\| v_1 \right\|_{ V_{ \delta_1 } }
     \cdot
     \ldots
    \cdot
    \left\| v_4 \right\|_{ V_{ \delta_4 } }
  }
  \right]
\end{split}
\end{equation}
(cf.\ Lemma~\ref{lem:Kolmogorov}).

\subsection{Weak convergence result}
\label{sec:weak_convergence_irregular_result}

\begin{lemma}
	\label{lem:c.delta.finite}
	Assume the setting in Section~\ref{sec:setting_weak_convergence_irregular}. Then 
	\begin{enumerate}[(i)]
		\item 
		\label{item:c.delta.finite}
		it holds for all 
		$ k \in \{ 1, 2, 3, 4 \} $, 
		$ \delta_1, \dots, \delta_k \in (-\nicefrac{1}{2},0] $
		with
		$
		\sum^k_{i=1} \delta_i
		> -\nicefrac{1}{2}
		$ 
		that 
		$
		\sup_{\kappa\in(0,T]}
		c^{(\kappa)}_{\delta_1,\ldots,\delta_k}
		< \infty
		$
		and
		\item
		\label{item:lip.c.delta.finite}
		it holds for all 
		$ \delta_1, \dots, \delta_4 \in (-\nicefrac{1}{2},0] $
		with
		$
		\sum^4_{i=1} \delta_i
		> -\nicefrac{1}{2}
		$ 
		that 
		$
		\sup_{\kappa\in(0,T]}
		\tilde{c}^{(\kappa)}_{\delta_1,\delta_2,\delta_3,\delta_4}
		< \infty
		$. 
	\end{enumerate}
\end{lemma}
\begin{proof}
	The proof of items~\eqref{item:c.delta.finite}--\eqref{item:lip.c.delta.finite} is entirely analogous to the proof of items~(iv)--(v) of Corollary~4.2 in Andersson et al.~\cite{AnderssonHefterJentzenKurniawan2016}.
	The proof of Lemma~\ref{lem:c.delta.finite} is thus completed.
\end{proof}

\begin{proposition}\label{prop:weak_convergence_irregular}
Assume the setting in Section~\ref{sec:setting_weak_convergence_irregular} 
and let 
$ r \in [ 0, 1 - \vartheta ) $, 
$ \rho \in ( 0, 1 - \theta ) $. 
Then it holds that 
$
  \ES\big[
  \| \varphi(X_T) \|_{\mathcal{V}}
  +
  \| \varphi(Y_T) \|_{\mathcal{V}}
  \big]
  < \infty
$
and 
\allowdisplaybreaks
\begin{align}
\label{eq:weak_convergence}
&
  \big\|
    \ES\big[ 
      \varphi( X_T )
    \big]
    -
    \ES\big[ 
      \varphi( Y_T )
    \big]
  \big\|_{\mathcal{V}}
\leq
  \left[
    57 
    \left|
      \max\{ T, \tfrac{ 1 }{ T } \} 
    \right|^{
      ( r + 3 ( \theta - \vartheta ) )
    }
    | \SGchi{0} |^{20} 
  \right]
  h^{
    \frac{ \rho \, r }{ ( \rho + 6 ( \theta - \vartheta ) ) }
  }
\nonumber
\\ & \cdot
\left[
\max\{
1
,
\| X_0 \|_{ \lpn{5}{\P}{V} }
\}
+
\tfrac{
	\SGchi{\theta} \, \SGchi{ \nicefrac{\rho}{2} + \theta } \,
	T^{ ( 1 - \theta ) }
	\,
	\| F \|_{ C^1_b( V, V_{ -\theta } ) }
}{
( 1 - \theta - \nicefrac{\rho}{2} )
}
+
\tfrac{
	\max\{\BDG{2}{},\BDG{5}{}\} \, \SGchi{ \nicefrac{\theta}{2} } \, \SGchi{ \nicefrac{( \rho + \theta )}{2} }
	\sqrt{ T^{ ( 1 - \theta ) } }
	\,
	\| B \|_{ C^1_b( V, \gamma( U, V_{ -\nicefrac{\theta}{2} } ) ) }
}{
\sqrt{ 1 - \theta - \rho }
}
\right]^{ 10 }
\nonumber
\\ & \cdot
\nonumber
\left|
\mathcal{E}_{ ( 1 - \theta ) }\!\left[
\tfrac{
	\sqrt{ 2 }
	\,
	\SGchi{0} \, \SGchi{ \theta }
	\,
	T^{ ( 1 - \theta ) }
	\,
	|
	F
	|_{
		C^1_b( V, V_{ - \theta } )
	}
}{
\sqrt{1 - \theta}
}
+
\max\{\BDG{2}{},\BDG{5}{}\} \, \SGchi{0} \, \SGchi{ 
	\theta / 2 
}
\sqrt{
	2 \, T^{ ( 1 - \theta ) }
} \,
|
B
|_{
	C^1_b( V, \gamma( U, V_{ - \nicefrac{\theta}{2} } ) )
}
\right]
\right|^5
\\ & \cdot
\Bigg[
2^r
+
\tfrac{ T^{ ( 1 - \vartheta ) } }{ ( 1 - \vartheta - r ) }
\bigg(
3 \, \SGchi{\vartheta} + 2 \, \SGchi{ r + \vartheta } + 3 \, | \SGchi{ \nicefrac{\vartheta}{2} } |^2
+
4 \, \SGchi{ r + \nicefrac{\vartheta}{2} } \, \SGchi{ \nicefrac{\vartheta}{2} }
\\ & +
\nonumber
2 \, ( | \SGchi{ \nicefrac{ \vartheta }{ 2 } } |^2 + \SGchi{\vartheta} ) \, \SGchi{r}
\Big[
\tfrac{
	\SGchi{ \vartheta } \, T^{ ( 1 - \vartheta ) } \,
}{ ( 1 - \vartheta ) }
+
\tfrac{
	\max\{\BDG{3}{},\BDG{4}{}\} \, \SGchi{ \nicefrac{\vartheta}{2} } \, 
	T^{ ( 1 - \vartheta )/2 }
}{
\sqrt{ 1 - \vartheta }
}
\Big]
\bigg)
\Bigg]
\\ & \cdot
\nonumber
  \bigg[
  \tfrac{
  |\SGchi{ \nicefrac{\rho}{2} }|^2
  }{
  T^{ \rho/2 }
  } \,
  | \varphi |_{ C^1_b( V, {\mathcal{V}} ) }
  +
  \tfrac{
    | \SGchi{0} |^3 \, 
    | \SGchi{r} |^2 \, 
    | \SGchi{ \theta - \vartheta } |^3 
    \,
    | \SGchi{ \nicefrac{ ( \theta - \vartheta ) }{ 2 } } |^6 
    \max\{ 1, T^{ ( 1 - \vartheta ) } \}
    \,
    \varsigma_{ F, B } 
  }{ 
    ( 1 - \vartheta - r ) \, T^r 
  } 
  \,
  \Big(
  \| \varphi \|_{ C^3_b( V, {\mathcal{V}} ) }
  +
  \sup_{ \kappa \in ( 0, T ] }
  \big[
  c^{ ( \kappa ) }_{ -\vartheta } 
\\ & +
\nonumber
  c^{ ( \kappa ) }_{ -\vartheta, 0 }
+ 
  c^{ ( \kappa ) }_{ -\vartheta, 0, 0 } + c^{ ( \kappa ) }_{ -\vartheta, 0, 0, 0 } 
  +
  c^{ ( \kappa ) }_{ -\nicefrac{ \vartheta }{ 2 }, -\nicefrac{ \vartheta }{ 2 } }
  +
  c^{ ( \kappa ) }_{ -\nicefrac{ \vartheta }{ 2 }, -\nicefrac{ \vartheta }{ 2 }, 0 }
+
  c^{ ( \kappa ) }_{ -\nicefrac{ \vartheta }{ 2 }, -\nicefrac{ \vartheta }{ 2 }, 0, 0 }
  +
  \tilde{c}^{ ( \kappa ) }_{ -\nicefrac{ \vartheta }{ 2 }, -\nicefrac{ \vartheta }{ 2 }, 0, 0 }
  \big]
  \Big)
  \bigg]
  < \infty
  .
\end{align}
\end{proposition}
\begin{proof}
We first note that, e.g., Theorem 2.7 in Cox \& Van Neerven~\cite{CoxVanNeerven2013} and 
Theorem 6.2 in Van Neerven et al.~\cite{vvw08} ensure that there exist 
up-to-modifications unique 
$
  ( \mathcal{F}_t )_{ t \in [0,T] }
$-predictable stochastic processes 
$
  \hat{Y}^{ \kappa, \delta } \colon [0,T] \times \Omega \to V
$, 
$ \kappa,\delta \in [0,T] $, 
and 
$
  \hat{X}^{ \kappa, \delta } \colon [0,T] \times \Omega \to V
$, 
$ \kappa,\delta \in [0,T] $, 
which satisfy for all 
$ \kappa, \delta \in [0,T] $ 
that 
$
  \sup_{ t \in [0,T] }
  \| \hat{X}^{ \kappa, \delta }_t \|_{ \lpn{5}{\P}{V} } 
  < \infty
$ 
and 
$
  \hat{X}^{ \kappa, \delta }_0
  =
  \hat{Y}^{ \kappa, \delta }_0
  =
  e^{ \delta A } X_0
$ 
and which satisfy that
for all 
$ \kappa, \delta \in [0,T] $, 
$ t \in (0,T] $ 
it holds $ \P $-a.s.\ that 
\begin{equation}
\label{eq:mollified_solution}
  \hat{X}_t^{ \kappa, \delta } 
= 
    e^{ t A } \hat{X}^{ \kappa, \delta }_0
  + 
    \int_0^t e^{ ( \kappa + t - s ) A } 
  F( \hat{X}_s^{ \kappa, \delta } ) \, ds
  +
    \int_0^t e^{ ( \kappa + t - s ) A } 
  B( \hat{X}_s^{ \kappa, \delta } ) \, dW_s 
  ,
\end{equation}
\begin{equation}
\label{eq:mollified_numerics}
  \hat{Y}_t^{ \kappa, \delta } 
= 
    e^{tA} \, \hat{Y}^{ \kappa, \delta }_0
  + 
    \int_0^t e^{ ( t - \floor{s}{h} )A } \, e^{ \kappa A } 
  F( \hat{Y}_{ \floor{ s }{ h } }^{ \kappa, \delta } ) \, ds
  +
    \int_0^t e^{ ( t - \floor{s}{h} )A } \, e^{ \kappa A } 
  B( \hat{Y}_{ \floor{ s }{ h } }^{ \kappa, \delta } ) \, dW_s 
  .
\end{equation}
In the next step we combine 
Lemma~\ref{lem:Kp_estimate} with the fact that 
$
  \forall \,
  \kappa, \delta \in [0,T]
  \colon
  \| \hat{Y}^{ \kappa, \delta }_0 \|_{ \lpn{5}{\P}{V} } 
  < \infty
$ 
to obtain that for all 
$ \kappa, \delta \in [ 0, T ] $ 
it holds that 
$
  \sup_{ t \in [0,T] }
  \| \hat{Y}^{ \kappa, \delta }_t \|_{ \lpn{5}{\P}{V} } 
  < \infty
$. 
This, the fact that
$
  \forall \, \kappa, \delta \in [0,T]
  \colon
  \sup_{ t \in [0,T] }
  \| 
    \hat{X}^{ \kappa, \delta }_t 
  \|_{ \lpn{5}{\P}{V} } 
  < \infty
$ 
and the assumption that 
$ \varphi \in \operatorname{Lip}^4( V, {\mathcal{V}} ) $ 
ensure that 
for all $ \kappa, \delta \in [ 0, T ] $ it holds that 
\begin{equation}
  \ES\big[
    \| \varphi(\hat{X}^{ \kappa, \delta }_T) \|_{\mathcal{V}}
    +
    \| \varphi(\hat{Y}^{ \kappa, \delta }_T) \|_{\mathcal{V}}
  \big]
  < \infty
  .
\end{equation}
This proves, in particular, that
$
  \ES\big[
  \| \varphi(X_T) \|_{\mathcal{V}}
  +
  \| \varphi(Y_T) \|_{\mathcal{V}}
  \big]
  < \infty
$. 
It thus remains to show \eqref{eq:weak_convergence}. 
For this we observe that the triangle inequality ensures that for all 
$ \kappa, \delta \in [ 0, T ] $ 
it holds that 
\begin{equation}\label{eq:mollified_decompose_solution}
\begin{split}
&
  \big\|
    \ES\big[ 
      \varphi( \hat{X}^{ 0, \delta }_T )
    \big]
    -
    \ES\big[ 
      \varphi( \hat{Y}_T^{ 0, \delta } )
    \big]
  \big\|_{\mathcal{V}}
\leq 
  \big\|
    \ES\big[ 
      \varphi( \hat{X}^{ 0, \delta }_T )
    \big]
    -
    \ES\big[ 
      \varphi( \hat{X}_T^{ \kappa, \delta } )
    \big]
  \big\|_{\mathcal{V}}
\\ & +
  \big\|
    \ES\big[ 
      \varphi( \hat{X}_T^{ \kappa, \delta } )
    \big]
    -
    \ES\big[ 
      \varphi( \hat{Y}_T^{ \kappa, \delta } )
    \big]
  \big\|_{\mathcal{V}}
  +
  \big\|
    \ES\big[ 
      \varphi( \hat{Y}_T^{ \kappa, \delta } )
    \big]
    -
    \ES\big[ 
      \varphi( \hat{Y}_T^{ 0, \delta } )
    \big]
  \big\|_{\mathcal{V}}.
\end{split}
\end{equation}
In the following we provide suitable bounds for the three summands 
on the right hand side of \eqref{eq:mollified_decompose_solution}.
For the first and the third summand on the right hand side of \eqref{eq:mollified_decompose_solution}
we observe that  
Proposition~\ref{prop:strong_convergence_numerics} 
together with H\"{o}lder's inequality and the fact that 
$
  \forall \, \kappa, \delta \in [ 0, T ]
  \colon
  \sup_{ t \in [0,T] }
  \| \hat{Y}^{ \kappa, \delta }_{ \floor{t}{h} } \|_{ \lpn{2}{\P}{V} } 
  \leq
  \sup_{ t \in [0,T] }
  \| \hat{Y}^{ \kappa, \delta }_t \|_{ \lpn{2}{\P}{V} } 
  < \infty
$
shows
that for all 
$ \kappa, \delta \in [0,T] $ 
it holds that
\begin{align}
\label{eq:mollified_strong_convergence}
&
  \big\|
    \ES\big[ 
      \varphi( \hat{X}^{ 0, \delta }_T )
    \big]
    -
    \ES\big[ 
      \varphi( \hat{X}_T^{ \kappa, \delta } )
    \big]
  \big\|_{\mathcal{V}}
  +
  \big\|
    \ES\big[ 
      \varphi( \hat{Y}_T^{ \kappa,\delta} )
    \big]
    -
    \ES\big[ 
      \varphi( \hat{Y}_T^{ 0, \delta } )
    \big]
  \big\|_{\mathcal{V}}
  \leq
  \tfrac{
    4 \, |\SGchi{ \nicefrac{\rho}{2} }|^2
  }{
    T^{ \rho/2 }
  } 
  \,
  | \varphi |_{ C^1_b( V, {\mathcal{V}} ) } \,
  \kappa^{ \frac{ \rho }{ 2 } }
\\ & \cdot
\nonumber
  \left[
    \SGchi{0} 
    \max\{
      1
      ,
      \| e^{ \delta A } X_0 \|_{ \lpn{2}{\P}{V} }
    \}
+
  \tfrac{
    \SGchi{\theta} \, \SGchi{ \nicefrac{\rho}{2} + \theta } \,
    T^{ ( 1 - \theta ) }
    \,
    \| F \|_{ C^1_b( V, V_{ -\theta } ) }
  }{
    ( 1 - \theta - \nicefrac{\rho}{2} )
  }
  +
  \tfrac{
    \BDG{2}{} \, \SGchi{ \nicefrac{\theta}{2} } \, \SGchi{ \nicefrac{( \rho + \theta )}{2} }
    \sqrt{ T^{ ( 1 - \theta ) } }
    \,
    \| B \|_{ C^1_b( V, \gamma( U, V_{ -\nicefrac{\theta}{2} } ) ) }
  }{
    \sqrt{ 1 - \theta - \rho }
  }
  \right]^2
\\ & \cdot
\nonumber
  \left|
  \mathcal{E}_{ ( 1 - \theta ) }\!\left[
    \tfrac{ 
      \sqrt{2} 
      \, 
      T^{ ( 1 - \theta ) } 
      \, 
      \SGchi{0} 
      \, 
      \SGchi{\theta} 
    }{ 
      \sqrt{ 1 - \theta } 
    }
    | F |_{ C^1_b( V, V_{ -\theta } ) }
    +
    \BDG{2}{} \, 
    \sqrt{ 
      2 \, T^{ ( 1 - \theta ) } 
    }
    \,
    \SGchi{0} \, 
    \SGchi{ \nicefrac{ \theta }{ 2 } } \,
    | B |_{ 
      C^1_b( V, \gamma( U, V_{ -\nicefrac{ \theta }{ 2 } } ) ) 
    }
  \right]
  \right|^2
  .
\end{align}
Next we bound the second summand on the right hand side 
of \eqref{eq:mollified_decompose_solution}. For this we note that 
for all $ \kappa \in (0,T] $ it holds that
\begin{equation}
\begin{split}
&
  \max\!\big\{
    1
    ,
    \|
      e^{ \kappa A } F
    \|_{ 
      C_b^3( V, V_{-\vartheta} )
    }^3
    ,
    \|
      e^{ \kappa A } B
    \|_{ 
      C_b^3( V, \gamma( U, V_{-\nicefrac{\vartheta}{2}} ) )
    }^6
  \big\}
\\ & \leq
  | \SGchi{ \theta - \vartheta } |^3 
  \,
  | \SGchi{ \nicefrac{ ( \theta - \vartheta ) }{ 2 } } |^6 
  \,
  \varsigma_{ F, B } 
  \max\!\big\{ 
    1, 
    \kappa^{ - 3 ( \theta - \vartheta ) } 
  \big\}
  .
\end{split}
\end{equation}
This and Corollary~\ref{cor:mollified_weak_solution-num_komplete} show that for all 
$ \kappa, \delta \in ( 0, T ] $ 
it holds that 
\allowdisplaybreaks
\begin{align}
\label{eq:mollified_weak_convergence}
&
  \big\|
    \ES\big[ 
      \varphi( \hat{X}_T^{ \kappa, \delta } )
    \big]
    -
    \ES\big[ 
      \varphi( \hat{Y}_T^{ \kappa, \delta } )
    \big]
  \big\|_{\mathcal{V}}
\nonumber
\\ & \leq
\nonumber
  \tfrac{
    57 \, 
    |\SGchi{0}|^3 \, 
    | \SGchi{r} |^2 \, 
    | \SGchi{ \theta - \vartheta } |^3 \,
    | \SGchi{ \nicefrac{ ( \theta - \vartheta ) }{ 2 } } |^6 
    \max\{ 1, T^{ ( 1 - \vartheta ) } \}
  }{ 
    ( 1 - \vartheta - r ) \; T^r 
  } 
  \,
  \varsigma_{ F, B } 
  \max\{ 1, \kappa^{ - 3 ( \theta - \vartheta ) } \} 
  \, 
  h^r
\\ & \cdot
\nonumber
  \left[
    \SGchi{0} 
    \max\{ 1, \| e^{ \delta A } X_0 \|_{ \lpn{5}{\P}{V} } \}
    +
    \tfrac{
      \SGchi{\theta} 
      \,
      T^{ ( 1 - \theta ) }
      \,
      \| e^{ \kappa A } F \|_{ C^1_b( V, V_{ -\theta } ) }
    }{
      ( 1 - \theta )
    }
      +
    \tfrac{
      \BDG{5}{} \,
      \SGchi{ \nicefrac{ \theta }{ 2 } }
      \sqrt{ T^{ ( 1 - \theta ) } } 
      \,
      \| e^{ \kappa A } B \|_{ C^1_b( V, \gamma( U, V_{ - \theta / 2 } ) ) }
    }{
      \sqrt{1 - \theta}
    }
  \right]^{ 10 }
\\ & \cdot
\nonumber
  \left|\mathcal{E}_{ ( 1 - \theta ) }\!\left[
    \tfrac{
      \sqrt{ 2 }
      \,
      \SGchi{ \theta }
      \,
      T^{ ( 1 - \theta ) }
      \,
      |
        e^{ \kappa A } F
      |_{
        C^1_b( V, V_{ - \theta } )
      }
    }{
      \sqrt{1 - \theta}
    }
    +
    \BDG{5}{} \, \SGchi{ 
      \theta / 2 
    }
    \sqrt{
      2 \, T^{ ( 1 - \theta ) }
    } \,
    |
      e^{ \kappa A } B
    |_{
      C^1_b( V, \gamma( U, V_{ - \nicefrac{\theta}{2} } ) )
    }
  \right]\right|^5
\\ & \cdot
  \Bigg[
    2^r
    +
    \tfrac{ T^{ ( 1 - \vartheta ) } }{ ( 1 - \vartheta - r ) }
  \bigg(
       3 \, \SGchi{\vartheta} + 2 \, \SGchi{ r + \vartheta } + 3 \, | \SGchi{ \nicefrac{\vartheta}{2} } |^2
       +
       4 \, \SGchi{ r + \nicefrac{\vartheta}{2} } \, \SGchi{ \nicefrac{\vartheta}{2} }
\\ & +
\nonumber
  2 \, 
  ( | \SGchi{ \nicefrac{ \vartheta }{ 2 } } |^2 + \SGchi{\vartheta} ) \, 
  \SGchi{r} \,
  \Big[
    \tfrac{
    \SGchi{ \vartheta } \, T^{ ( 1 - \vartheta ) } \,
    }{ ( 1 - \vartheta ) }
    +
    \tfrac{
    \max\{\BDG{3}{},\BDG{4}{}\} \, \SGchi{ \nicefrac{\vartheta}{2} } \, 
    T^{ ( 1 - \vartheta )/2 }
    }{
    \sqrt{ 1 - \vartheta }
    }
  \Big]
  \bigg)
  \Bigg]
\\ & \cdot
\nonumber
  \Big[
  \| \varphi \|_{ C^3_b( V, {\mathcal{V}} ) }
  +
  c^{ ( \kappa ) }_{ -\vartheta } + c^{ ( \kappa ) }_{ -\vartheta, 0 }
  + c^{ ( \kappa ) }_{ -\vartheta, 0, 0 } + c^{ ( \kappa ) }_{ -\vartheta, 0, 0, 0 }
+
  c^{ ( \kappa ) }_{ -\nicefrac{ \vartheta }{ 2 }, -\nicefrac{ \vartheta }{ 2 } }
  +
  c^{ ( \kappa ) }_{ -\nicefrac{ \vartheta }{ 2 }, -\nicefrac{ \vartheta }{ 2 }, 0 }
\\ & +
\nonumber
  c^{ ( \kappa ) }_{ -\nicefrac{ \vartheta }{ 2 }, -\nicefrac{ \vartheta }{ 2 }, 0, 0 }
+
  \tilde{c}^{ ( \kappa ) }_{ -\nicefrac{ \vartheta }{ 2 }, -\nicefrac{ \vartheta }{ 2 }, 0, 0 }
  \,
  \Big]
  .
\end{align}
Plugging \eqref{eq:mollified_strong_convergence} and \eqref{eq:mollified_weak_convergence} 
into \eqref{eq:mollified_decompose_solution} 
then shows that for all 
$ \kappa, \delta \in ( 0, T ] $ 
it holds that 
\begin{align}
\label{eq:mollified_combined}
&
  \big\|
    \ES\big[ 
      \varphi( \hat{X}^{ 0, \delta }_T )
    \big]
    -
    \ES\big[ 
      \varphi( \hat{Y}_T^{ 0, \delta } )
    \big]
  \big\|_{\mathcal{V}}
\leq
  \max\!\left\{
    4 \, \kappa^{ \frac{\rho}{2} }
    ,
    57 
    \max\!\big\{ 
      1 , 
      \kappa^{ - 3 ( \theta - \vartheta ) } 
    \big\} 
    \, h^r
  \right\}
  | \SGchi{0} |^{ 20 } 
\nonumber
\\ & \cdot
  \left[
    \max\{
      1
      ,
      \| X_0 \|_{ \lpn{5}{\P}{V} }
    \}
    +
  \tfrac{
    \SGchi{\theta} \, \SGchi{ \nicefrac{\rho}{2} + \theta } \,
    T^{ ( 1 - \theta ) }
    \,
    \| F \|_{ C^1_b( V, V_{ -\theta } ) }
  }{
    ( 1 - \theta - \nicefrac{\rho}{2} )
  }
  +
  \tfrac{
    \max\{\BDG{2}{},\BDG{5}{}\} \, \SGchi{ \nicefrac{\theta}{2} } \, \SGchi{ \nicefrac{( \rho + \theta )}{2} }
    \sqrt{ T^{ ( 1 - \theta ) } }
    \,
    \| B \|_{ C^1_b( V, \gamma( U, V_{ -\nicefrac{\theta}{2} } ) ) }
  }{
    \sqrt{ 1 - \theta - \rho }
  }
  \right]^{ 10 }
\nonumber
\\ & \cdot
\nonumber
  \left|
  \mathcal{E}_{ ( 1 - \theta ) }\!\left[
    \tfrac{
      \sqrt{ 2 }
      \,
      \SGchi{0} \, \SGchi{ \theta }
      \,
      T^{ ( 1 - \theta ) }
      \,
      |
        F
      |_{
        C^1_b( V, V_{ - \theta } )
      }
    }{
      \sqrt{1 - \theta}
    }
    +
    \max\{\BDG{2}{},\BDG{5}{}\} \, \SGchi{0} \, \SGchi{ 
      \theta / 2 
    }
    \sqrt{
      2 \, T^{ ( 1 - \theta ) }
    } \,
    |
      B
    |_{
      C^1_b( V, \gamma( U, V_{ - \nicefrac{\theta}{2} } ) )
    }
  \right]
  \right|^5
\\ & \cdot
  \Bigg[
    2^r
    +
    \tfrac{ T^{ ( 1 - \vartheta ) } }{ ( 1 - \vartheta - r ) }
  \bigg(
       3 \, \SGchi{\vartheta} + 2 \, \SGchi{ r + \vartheta } + 3 \, | \SGchi{ \nicefrac{\vartheta}{2} } |^2
       +
       4 \, \SGchi{ r + \nicefrac{\vartheta}{2} } \, \SGchi{ \nicefrac{\vartheta}{2} }
\\ & +
\nonumber
  2 \, ( | \SGchi{ \nicefrac{ \vartheta }{ 2 } } |^2 + \SGchi{\vartheta} ) \, \SGchi{r}
  \Big[
    \tfrac{
    \SGchi{ \vartheta } \, T^{ ( 1 - \vartheta ) } \,
    }{ ( 1 - \vartheta ) }
    +
    \tfrac{
    \max\{\BDG{3}{},\BDG{4}{}\} \, \SGchi{ \nicefrac{\vartheta}{2} } \, 
    T^{ ( 1 - \vartheta )/2 }
    }{
    \sqrt{ 1 - \vartheta }
    }
  \Big]
  \bigg)
  \Bigg]
\\ & \cdot
\nonumber
  \bigg[
  \tfrac{
  |\SGchi{ \nicefrac{\rho}{2} }|^2
  }{
  T^{ \rho/2 }
  } \,
  | \varphi |_{ C^1_b( V, {\mathcal{V}} ) }
  +
  \tfrac{
    |\SGchi{0}|^3 \, | \SGchi{r} |^2  \,
    | \SGchi{ \theta - \vartheta } |^3 \,
    | \SGchi{ \nicefrac{ ( \theta - \vartheta ) }{ 2 } } |^6 
    \max\{ 1, T^{ ( 1 - \vartheta ) } \}
  }{ 
    ( 1 - \vartheta - r ) \, T^r 
  } 
  \, 
  \varsigma_{ F, B } 
  \,
  \Big[
  \| \varphi \|_{ C^3_b( V, {\mathcal{V}} ) }
  +
  c^{ ( \kappa ) }_{ -\vartheta } + c^{ ( \kappa ) }_{ -\vartheta, 0 }
\\ & + 
\nonumber
  c^{ ( \kappa ) }_{ -\vartheta, 0, 0 } 
+
  c^{ ( \kappa ) }_{ -\vartheta, 0, 0, 0 }
+
  c^{ ( \kappa ) }_{ -\nicefrac{ \vartheta }{ 2 }, -\nicefrac{ \vartheta }{ 2 } }
  +
  c^{ ( \kappa ) }_{ -\nicefrac{ \vartheta }{ 2 }, -\nicefrac{ \vartheta }{ 2 }, 0 }
+
  c^{ ( \kappa ) }_{ -\nicefrac{ \vartheta }{ 2 }, -\nicefrac{ \vartheta }{ 2 }, 0, 0 }
+
  \tilde{c}^{ ( \kappa ) }_{ -\nicefrac{ \vartheta }{ 2 }, -\nicefrac{ \vartheta }{ 2 }, 0, 0 }
  \,
  \Big]
  \bigg]
  .
\end{align}
In addition, we observe that
\begin{equation}\label{eq:optimal_rate}
\begin{split}
&
  \inf_{ \kappa \in ( 0, T ] }
  \max\!\left\{
    4 \, 
    \kappa^{ \frac{ \rho }{ 2 } }
    ,
    57 
    \max\!\big\{
      1,
      \kappa^{ -3( \theta - \vartheta ) }
    \big\} 
    \,
    h^r
  \right\}
\\ & \leq
  \max\!\left\{
  4 
  \left[
    \min\{ 1, T \}
    \left|
      \tfrac{h}{T}
    \right|^{ \frac{ 2r }{ ( \rho + 6 ( \theta - \vartheta ) ) } }
  \right]^{ \frac{ \rho }{ 2 } }
  ,
    57  
    \max\!\left\{
    1,
    \left[
  \min\{ 1, T \}
  \left|
  \tfrac{h}{T}
  \right|^{ \frac{ 2r }{ ( \rho + 6 ( \theta - \vartheta ) ) } }
    \right]^{ -3( \theta - \vartheta ) }
    \right\} 
    h^r
  \right\}
\\ & =
  \max\!\left\{
  4 
  \left[
    \min\{ 1, T \}
    \left|
      \tfrac{h}{T}
    \right|^{ \frac{ 2r }{ ( \rho + 6 ( \theta - \vartheta ) ) } }
  \right]^{ \frac{ \rho }{ 2 } }
  ,
    57  
    \,
    h^r
    \left[
  \min\{ 1, T \}
  \left|
  \tfrac{h}{T}
  \right|^{ \frac{ 2r }{ ( \rho + 6 ( \theta - \vartheta ) ) } }
    \right]^{ -3( \theta - \vartheta ) }
  \right\}
\\ & =
  \max\!\left\{
  \tfrac{
    4 \,
    |\! \min\{ 1, T \} |^{ \frac{ \rho }{ 2 } }
  }{
    T^{
      \frac{
      \rho r
      }{
      ( \rho + 6 ( \theta - \vartheta ) )
      }
    }
  }
  ,
  \tfrac{
    57 \,
    T^{
      \frac{
      6 ( \theta - \vartheta ) r
      }{
      ( \rho + 6 ( \theta - \vartheta ) )
      }
    }
  }{
    |\!\min\{ 1, T \}|^{ 3( \theta - \vartheta ) }
  }
  \right\}
  h^{
    \frac{ \rho \, r }{ ( \rho + 6 ( \theta - \vartheta ) ) }
  }
\\ & 
\leq
  57
  \max\!\left\{
  \tfrac{
    1
  }{
    \left| \min\{ 1, T \} \right|^r
  }
  ,
  \tfrac{
    \left| \max\{ 1, T \} \right|^r
  }{
    |\!\min\{ 1, T \}|^{ 3( \theta - \vartheta ) }
  }
  \right\}
  h^{
    \frac{ \rho \, r }{ ( \rho + 6 ( \theta - \vartheta ) ) }
  }
\leq
  \frac{
    57 
    \,
    h^{
      \frac{ \rho \, r }{ ( \rho + 6 ( \theta - \vartheta ) ) }
    }
  }{
    \left| 
      \min\{ T, \frac{ 1 }{ T } \} 
    \right|^{ ( r + 3( \theta - \vartheta ) ) }
  } 
  .
\end{split}
\end{equation}
Combining \eqref{eq:mollified_combined} and \eqref{eq:optimal_rate} yields that for all 
$ \delta \in (0,T] $ it holds that 
\allowdisplaybreaks
\begin{align}
\label{eq:weak_convergence_conclude}
&\nonumber
  \big\|
    \ES\big[ 
      \varphi( \hat{X}^{ 0, \delta }_T )
    \big]
    -
    \ES\big[ 
      \varphi( \hat{Y}_T^{ 0, \delta } )
    \big]
  \big\|_{\mathcal{V}}
\leq
  \left[
    57 
    \left| 
      \max\{ T, \tfrac{ 1 }{ T } \} 
    \right|^{ ( r + 3( \theta - \vartheta ) ) }
    | \SGchi{0} |^{ 20 } 
  \right]
  h^{
    \frac{ \rho \, r }{ ( \rho + 6 ( \theta - \vartheta ) ) } 
  }
\\ & \cdot
\left[
\max\{
1
,
\| X_0 \|_{ \lpn{5}{\P}{V} }
\}
+
\tfrac{
	\SGchi{\theta} \, \SGchi{ \nicefrac{\rho}{2} + \theta } \,
	T^{ ( 1 - \theta ) }
	\,
	\| F \|_{ C^1_b( V, V_{ -\theta } ) }
}{
( 1 - \theta - \nicefrac{\rho}{2} )
}
+
\tfrac{
	\max\{\BDG{2}{},\BDG{5}{}\} \, \SGchi{ \nicefrac{\theta}{2} } \, \SGchi{ \nicefrac{( \rho + \theta )}{2} }
	\sqrt{ T^{ ( 1 - \theta ) } }
	\,
	\| B \|_{ C^1_b( V, \gamma( U, V_{ -\nicefrac{\theta}{2} } ) ) }
}{
\sqrt{ 1 - \theta - \rho }
}
\right]^{ 10 }
\nonumber
\\ & \cdot
\nonumber
\left|
\mathcal{E}_{ ( 1 - \theta ) }\!\left[
\tfrac{
	\sqrt{ 2 }
	\,
	\SGchi{0} \, \SGchi{ \theta }
	\,
	T^{ ( 1 - \theta ) }
	\,
	|
	F
	|_{
		C^1_b( V, V_{ - \theta } )
	}
}{
\sqrt{1 - \theta}
}
+
\max\{\BDG{2}{},\BDG{5}{}\} \, \SGchi{0} \, \SGchi{ 
	\theta / 2 
}
\sqrt{
	2 \, T^{ ( 1 - \theta ) }
} \,
|
B
|_{
	C^1_b( V, \gamma( U, V_{ - \nicefrac{\theta}{2} } ) )
}
\right]
\right|^5
\\ & \cdot
\Bigg[
2^r
+
\tfrac{ T^{ ( 1 - \vartheta ) } }{ ( 1 - \vartheta - r ) }
\bigg(
3 \, \SGchi{\vartheta} + 2 \, \SGchi{ r + \vartheta } + 3 \, | \SGchi{ \nicefrac{\vartheta}{2} } |^2
+
4 \, \SGchi{ r + \nicefrac{\vartheta}{2} } \, \SGchi{ \nicefrac{\vartheta}{2} }
\\ & +
\nonumber
2 \, ( | \SGchi{ \nicefrac{ \vartheta }{ 2 } } |^2 + \SGchi{\vartheta} ) \, \SGchi{r}
\Big[
\tfrac{
	\SGchi{ \vartheta } \, T^{ ( 1 - \vartheta ) } \,
}{ ( 1 - \vartheta ) }
+
\tfrac{
	\max\{\BDG{3}{},\BDG{4}{}\} \, \SGchi{ \nicefrac{\vartheta}{2} } \, 
	T^{ ( 1 - \vartheta )/2 }
}{
\sqrt{ 1 - \vartheta }
}
\Big]
\bigg)
\Bigg]
\\ & \cdot
\nonumber
  \bigg[
  \tfrac{
  |\SGchi{ \nicefrac{\rho}{2} }|^2
  }{
  T^{ \rho/2 }
  } \,
  | \varphi |_{ C^1_b( V, {\mathcal{V}} ) }
  +
  \tfrac{
  |\SGchi{0}|^3 \, | \SGchi{r} |^2 \,
  | \SGchi{ \theta - \vartheta } |^3 \,
  | \SGchi{ \nicefrac{ ( \theta - \vartheta ) }{ 2 } } |^6 \,
  \max\{ 1, T^{ ( 1 - \vartheta ) } \}
  }{ ( 1 - \vartheta - r ) \, T^r } \, \varsigma_{ F, B } \, 
  \Big[
  \| \varphi \|_{ C^3_b( V, {\mathcal{V}} ) }
  +
  \sup_{ \kappa \in ( 0, T ] }
  \big[
  c^{ ( \kappa ) }_{ -\vartheta } 
\\ & +
\nonumber
  c^{ ( \kappa ) }_{ -\vartheta, 0 }
+ 
  c^{ ( \kappa ) }_{ -\vartheta, 0, 0 } + c^{ ( \kappa ) }_{ -\vartheta, 0, 0, 0 } 
  +
  c^{ ( \kappa ) }_{ -\nicefrac{ \vartheta }{ 2 }, -\nicefrac{ \vartheta }{ 2 } }
  +
  c^{ ( \kappa ) }_{ -\nicefrac{ \vartheta }{ 2 }, -\nicefrac{ \vartheta }{ 2 }, 0 }
+
  c^{ ( \kappa ) }_{ -\nicefrac{ \vartheta }{ 2 }, -\nicefrac{ \vartheta }{ 2 }, 0, 0 }
  +
  \tilde{c}^{ ( \kappa ) }_{ -\nicefrac{ \vartheta }{ 2 }, -\nicefrac{ \vartheta }{ 2 }, 0, 0 }
  \big]
  \Big]
  \bigg]
  .
\end{align}
In the next step we 
note that Corollary~\ref{cor:initial_perturbation} yields that 
$
  \lim_{ \delta \to 0 }
  \ES\big[
    \varphi(
      \hat{X}^{ 0, \delta }_T
    )
  \big]
  =
  \ES\big[
    \varphi(
      X_T
    )
  \big]
$ 
and 
$
  \lim_{ \delta \to 0 }
  \ES\big[
    \varphi(
      \hat{Y}^{ 0, \delta }_T
    )
  \big]
  =
  \ES\big[
    \varphi(
      Y_T
    )
  \big]
$. 
Combining this with inequality~\eqref{eq:weak_convergence_conclude} proves the first 
inequality 
in~\eqref{eq:weak_convergence}. 
The second inequality 
in~\eqref{eq:weak_convergence}
follows from Lemma~\ref{lem:c.delta.finite}.
The proof of Proposition~\ref{prop:weak_convergence_irregular} is thus completed. 
\end{proof}

\begin{corollary}
\label{cor:weak_convergence_irregular}
Assume the setting in Section~\ref{sec:setting_weak_convergence_irregular} 
and let 
$ \rho \in ( 0, 1 - \theta ) \cap ( 6( \theta - \vartheta ), \infty ) $. 
Then it holds that 
$
  \ES\big[
  \| \varphi(X_T) \|_{\mathcal{V}}
  +
  \| \varphi(Y_T) \|_{\mathcal{V}}
  \big]
  < \infty
$
and 
\allowdisplaybreaks
\begin{align}
\label{eq:weak_convergence_simplified_rate}
&
  \big\|
    \ES\big[ 
      \varphi( X_T )
    \big]
    -
    \ES\big[ 
      \varphi( Y_T )
    \big]
  \big\|_{\mathcal{V}}
\leq
  \left[
    57  \, 
    |
    \!
    \max\{ T, \tfrac{ 1 }{ T } \} 
    |^{
    	3 ( \rho + \theta )
    } \,
    | \SGchi{0} |^{20} 
  \right]
    h^{ 
      \left(
        \rho - 6 ( \theta - \vartheta ) 
      \right)
    }
\nonumber
\\ & \cdot
\left[
\max\{
1
,
\| X_0 \|_{ \lpn{5}{\P}{V} }
\}
+
\tfrac{
	\SGchi{\theta} \, \SGchi{ \nicefrac{\rho}{2} + \theta } \,
	T^{ ( 1 - \theta ) }
	\,
	\| F \|_{ C^1_b( V, V_{ -\theta } ) }
}{
( 1 - \theta - \nicefrac{\rho}{2} )
}
+
\tfrac{
	\max\{\BDG{2}{},\BDG{5}{}\} \, \SGchi{ \nicefrac{\theta}{2} } \, \SGchi{ \nicefrac{( \rho + \theta )}{2} }
	\sqrt{ T^{ ( 1 - \theta ) } }
	\,
	\| B \|_{ C^1_b( V, \gamma( U, V_{ -\nicefrac{\theta}{2} } ) ) }
}{
\sqrt{ 1 - \theta - \rho }
}
\right]^{ 10 }
\nonumber
\\ & \cdot
\left|
\mathcal{E}_{ ( 1 - \theta ) }\!\left[
\tfrac{
	\sqrt{ 2 }
	\,
	\SGchi{0} \, \SGchi{ \theta }
	\,
	T^{ ( 1 - \theta ) }
	\,
	|
	F
	|_{
		C^1_b( V, V_{ - \theta } )
	}
}{
\sqrt{1 - \theta}
}
+
\max\{\BDG{2}{},\BDG{5}{}\} \, \SGchi{0} \, \SGchi{ 
	\theta / 2 
}
\sqrt{
	2 \, T^{ ( 1 - \theta ) }
} \,
|
B
|_{
	C^1_b( V, \gamma( U, V_{ - \nicefrac{\theta}{2} } ) )
}
\right]
\right|^5
\nonumber
\\ & \cdot
\Bigg[
2^\rho
+
\tfrac{ T^{ ( 1 - \vartheta ) } }{ ( 1 - \vartheta - \rho ) }
\bigg(
3 \, \SGchi{\vartheta} + 2 \, \SGchi{ \rho + \vartheta } + 3 \, | \SGchi{ \nicefrac{\vartheta}{2} } |^2
+
4 \, \SGchi{ \rho + \nicefrac{\vartheta}{2} } \, \SGchi{ \nicefrac{\vartheta}{2} }
\\ & +
\nonumber
2 \, ( | \SGchi{ \nicefrac{ \vartheta }{ 2 } } |^2 + \SGchi{\vartheta} ) \, \SGchi{\rho}
\Big[
\tfrac{
	\SGchi{ \vartheta } \, T^{ ( 1 - \vartheta ) } \,
}{ ( 1 - \vartheta ) }
+
\tfrac{
	\max\{\BDG{3}{},\BDG{4}{}\} \, \SGchi{ \nicefrac{\vartheta}{2} } \, 
	T^{ ( 1 - \vartheta )/2 }
}{
\sqrt{ 1 - \vartheta }
}
\Big]
\bigg)
\Bigg]
\\ & \cdot
\nonumber
\bigg[
\tfrac{
	|\SGchi{ \nicefrac{\rho}{2} }|^2
}{
T^{ \rho/2 }
} \,
| \varphi |_{ C^1_b( V, {\mathcal{V}} ) }
+
\tfrac{
	| \SGchi{0} |^3 \, 
	| \SGchi{\rho} |^2 \, 
	| \SGchi{ \theta - \vartheta } |^3 \,
	| \SGchi{ \nicefrac{ ( \theta - \vartheta ) }{ 2 } } |^6 
	\max\{ 
	1, T^{ ( 1 - \vartheta ) } 
	\}
	\, 
	\varsigma_{ F, B }
}{ 
( 1 - \vartheta - \rho ) \, T^\rho 
}  
\, 
\Big(
\| \varphi \|_{ C^3_b( V, {\mathcal{V}} ) }
+
\sup_{ \kappa \in ( 0, T ] }
\big[
c^{ ( \kappa ) }_{ -\vartheta } 
\\ & +
\nonumber
c^{ ( \kappa ) }_{ -\vartheta, 0 }
+ 
c^{ ( \kappa ) }_{ -\vartheta, 0, 0 } + c^{ ( \kappa ) }_{ -\vartheta, 0, 0, 0 } 
+
c^{ ( \kappa ) }_{ -\nicefrac{ \vartheta }{ 2 }, -\nicefrac{ \vartheta }{ 2 } }
+
c^{ ( \kappa ) }_{ -\nicefrac{ \vartheta }{ 2 }, -\nicefrac{ \vartheta }{ 2 }, 0 }
+
c^{ ( \kappa ) }_{ -\nicefrac{ \vartheta }{ 2 }, -\nicefrac{ \vartheta }{ 2 }, 0, 0 }
+
\tilde{c}^{ ( \kappa ) }_{ -\nicefrac{ \vartheta }{ 2 }, -\nicefrac{ \vartheta }{ 2 }, 0, 0 }
\big]
\Big)
\bigg]
< \infty
.
\end{align}
\end{corollary}
\begin{proof}
We first apply\footnote{with $ r = \rho $ in the notation of Proposition~\ref{prop:weak_convergence_irregular}} 
Proposition~\ref{prop:weak_convergence_irregular}
to obtain 
that 
$
  \ES\big[
  \| \varphi(X_T) \|_{\mathcal{V}}
  +
  \| \varphi(Y_T) \|_{\mathcal{V}}
  \big]
  < \infty
$
and 
\allowdisplaybreaks
\begin{align}
\label{eq:weak_convergence_simplified_rate_proof}
&
  \big\|
    \ES\big[ 
      \varphi( X_T )
    \big]
    -
    \ES\big[ 
      \varphi( Y_T )
    \big]
  \big\|_{\mathcal{V}}
\leq
  \left[
      57 \,
      |\! \max\{ T, \tfrac{ 1 }{ T } \} |^{ ( \rho + 3 ( \theta - \vartheta ) ) } \, 
      | \SGchi{0} |^{20} \,
  \right]
  h^{
    \frac{ \rho^2 }{ ( \rho + 6 ( \theta - \vartheta ) ) } 
  }
\nonumber
\\ & \cdot
  \left[
    \max\{
    1
    ,
    \| X_0 \|_{ \lpn{5}{\P}{V} }
    \}
+
  \tfrac{
    \SGchi{\theta} \, \SGchi{ \nicefrac{\rho}{2} + \theta } \,
    T^{ ( 1 - \theta ) }
    \,
    \| F \|_{ C^1_b( V, V_{ -\theta } ) }
  }{
    ( 1 - \theta - \nicefrac{\rho}{2} )
  }
  +
  \tfrac{
    \max\{\BDG{2}{},\BDG{5}{}\} \, \SGchi{ \nicefrac{\theta}{2} } \, \SGchi{ \nicefrac{( \rho + \theta )}{2} }
    \sqrt{ T^{ ( 1 - \theta ) } }
    \,
    \| B \|_{ C^1_b( V, \gamma( U, V_{ -\nicefrac{\theta}{2} } ) ) }
  }{
    \sqrt{ 1 - \theta - \rho }
  }
  \right]^{ 10 }
\nonumber
\\ & \cdot
  \left|
  \mathcal{E}_{ ( 1 - \theta ) }\!\left[
    \tfrac{
      \sqrt{ 2 }
      \,
      \SGchi{0} \, \SGchi{ \theta }
      \,
      T^{ ( 1 - \theta ) }
      \,
      |
        F
      |_{
        C^1_b( V, V_{ - \theta } )
      }
    }{
      \sqrt{1 - \theta}
    }
    +
    \max\{\BDG{2}{},\BDG{5}{}\} \, \SGchi{0} \, \SGchi{ 
      \theta / 2 
    }
    \sqrt{
      2 \, T^{ ( 1 - \theta ) }
    } \,
    |
      B
    |_{
      C^1_b( V, \gamma( U, V_{ - \nicefrac{\theta}{2} } ) )
    }
  \right]
  \right|^5
\nonumber
\\ & \cdot
  \Bigg[
    2^\rho
    +
    \tfrac{ T^{ ( 1 - \vartheta ) } }{ ( 1 - \vartheta - \rho ) }
  \bigg(
       3 \, \SGchi{\vartheta} + 2 \, \SGchi{ \rho + \vartheta } + 3 \, | \SGchi{ \nicefrac{\vartheta}{2} } |^2
       +
       4 \, \SGchi{ \rho + \nicefrac{\vartheta}{2} } \, \SGchi{ \nicefrac{\vartheta}{2} }
\\ & +
\nonumber
  2 \, ( | \SGchi{ \nicefrac{ \vartheta }{ 2 } } |^2 + \SGchi{\vartheta} ) \, \SGchi{\rho}
  \Big[
    \tfrac{
    \SGchi{ \vartheta } \, T^{ ( 1 - \vartheta ) } \,
    }{ ( 1 - \vartheta ) }
    +
    \tfrac{
    \max\{\BDG{3}{},\BDG{4}{}\} \, \SGchi{ \nicefrac{\vartheta}{2} } \, 
    T^{ ( 1 - \vartheta )/2 }
    }{
    \sqrt{ 1 - \vartheta }
    }
  \Big]
  \bigg)
  \Bigg]
\\ & \cdot
\nonumber
  \bigg[
  \tfrac{
  |\SGchi{ \nicefrac{\rho}{2} }|^2
  }{
  T^{ \rho/2 }
  } \,
  | \varphi |_{ C^1_b( V, {\mathcal{V}} ) }
  +
  \tfrac{
    | \SGchi{0} |^3 \, 
    | \SGchi{\rho} |^2 \, 
    | \SGchi{ \theta - \vartheta } |^3 \,
    | \SGchi{ \nicefrac{ ( \theta - \vartheta ) }{ 2 } } |^6 
    \max\{ 
      1, T^{ ( 1 - \vartheta ) } 
    \}
    \, 
    \varsigma_{ F, B }
  }{ 
    ( 1 - \vartheta - \rho ) \, T^\rho 
  }  
  \, 
  \Big(
  \| \varphi \|_{ C^3_b( V, {\mathcal{V}} ) }
  +
  \sup_{ \kappa \in ( 0, T ] }
  \big[
  c^{ ( \kappa ) }_{ -\vartheta } 
\\ & +
\nonumber
  c^{ ( \kappa ) }_{ -\vartheta, 0 }
+ 
  c^{ ( \kappa ) }_{ -\vartheta, 0, 0 } + c^{ ( \kappa ) }_{ -\vartheta, 0, 0, 0 } 
  +
  c^{ ( \kappa ) }_{ -\nicefrac{ \vartheta }{ 2 }, -\nicefrac{ \vartheta }{ 2 } }
  +
  c^{ ( \kappa ) }_{ -\nicefrac{ \vartheta }{ 2 }, -\nicefrac{ \vartheta }{ 2 }, 0 }
+
  c^{ ( \kappa ) }_{ -\nicefrac{ \vartheta }{ 2 }, -\nicefrac{ \vartheta }{ 2 }, 0, 0 }
  +
  \tilde{c}^{ ( \kappa ) }_{ -\nicefrac{ \vartheta }{ 2 }, -\nicefrac{ \vartheta }{ 2 }, 0, 0 }
  \big]
  \Big)
  \bigg]
  < \infty
  .
\end{align}
Next we note that 
\begin{equation}
\label{eq:simplified_rate}
\begin{split}
&
  h^{
    { \frac{ \rho^2 }{ ( \rho + 6 ( \theta - \vartheta ) ) } }
  }
  =
  h^{ 
      \rho 
      \left[
        \frac{ 1 }{ 1 + 6 ( \theta - \vartheta ) / \rho } 
        -
        1 
        + \frac{ 6 ( \theta - \vartheta ) }{ \rho }
      \right]
    }
    \,
    h^{ 
      \rho 
      \left[
        1 - \frac{ 6 ( \theta - \vartheta ) }{ \rho }
      \right]
    }
\\ & 
  \leq
  |\! \max\{ 1, T \} |^{
    \rho 
    \left[
      \frac{ 1 }{ 1 + 6 ( \theta - \vartheta ) / \rho } 
      -
      1 
      + \frac{ 6 ( \theta - \vartheta ) }{ \rho }
    \right]
  }
    \,
    h^{ 
      \left(
        \rho - 6 ( \theta - \vartheta ) 
      \right)
    }
\leq
  |\! \max\{ 1, T \} |^\rho
    \,
    h^{ 
      \left(
        \rho - 6 ( \theta - \vartheta ) 
      \right)
    }
    .
\end{split}
\end{equation}
Plugging \eqref{eq:simplified_rate} into \eqref{eq:weak_convergence_simplified_rate_proof} 
establishes~\eqref{eq:weak_convergence_simplified_rate}. 
This completes the proof of Corollary~\ref{cor:weak_convergence_irregular}.
\end{proof}

\section{Weak convergence rates for exponential Euler approximations of SPDEs}
\label{sec:sharp.weak.rates}
\begin{corollary}
	\label{cor:sharp.weak.rates}
	Consider the notation in Section~\ref{sec:notation}, 
	let 
	$ ( V, \left\| \cdot \right\|_V ) $ 
	and 
	$ ( \mathcal{V}, \left\| \cdot \right\|_{\mathcal{V}} ) $
	be separable UMD $\R$-Banach spaces with type 2,
	let
	$ ( U, \left< \cdot, \cdot \right>_U, \left\| \cdot \right\|_U ) $
	be a separable $ \R $-Hilbert space, 
	let 
	$ T \in (0,\infty) $,
	$ \eta \in \R $,  
	$\kappa\in[0,\nicefrac{4}{7})$,
	$\xi\in V$,
	$\varphi\in \operatorname{Lip}^4(V,\mathcal{V})$, 
	let $ ( \Omega, \mathcal{F}, \P ) $
	be a probability space with a normal filtration 
	$( \mathcal{F}_t )_{ t \in [0,T] }$,
	let $ ( W_t )_{ t \in [0,T] } $ be an $ \operatorname{Id}_U $-cylindrical
	$ ( \Omega, \mathcal{F}, \P, ( \mathcal{F}_t )_{ t \in [0,T] } ) $-Wiener process,
	let
	$
	A \colon D(A)
	\subseteq
	V \rightarrow V
	$
	be a generator of a strongly continuous analytic semigroup
	with 
	$
	\operatorname{spectrum}( A )
	\subseteq
	\{
	z \in \mathbb{C}
	\colon
	\text{Re}( z ) < \eta
	\}
	$,
	let
	$  
	( 
	V_r 
	,
	\left\| \cdot \right\|_{ V_r } 
	)
	$,
	$ r \in \R $,
	be a family of interpolation spaces associated to $ \eta - A $, 
	let 
	$F\in \operatorname{Lip}^4(V,V_{-\kappa})$, 
	$B\in \operatorname{Lip}^4(V,\gamma(U,V_{-\nicefrac{\kappa}{2}}))$,
	let 
	$ 
	X \colon [0,T] \times \Omega \to V
	$
	be a continuous 
	$
	( \mathcal{F}_t )_{ t \in [0,T] }
	$-adapted stochastic process
	which satisfies that for all 
	$ t \in [0,T] $
	it holds $\P$-a.s.\ that
	\begin{equation}
	\label{eq:lifting.SEE}
	X_t 
	= 
	e^{ tA } \xi
	+ 
	\int_0^t e^{ ( t - s )A } F( X_s ) \, ds
	+ 
	\int_0^t e^{ ( t - s )A } B( X_s ) \, dW_s
	, 
	\end{equation}
	and let $ Y^N \colon \{ 0, 1, \dots, N \} \times \Omega \to V $,
	$ N \in \N $, be stochastic processes which satisfy that for all $ N \in \N $,
	$ n \in \{ 0, 1, \dots, N - 1 \} $
	it holds $ \P $-a.s.\ that
	$
	Y^N_0 = \xi
	$
	and 
\begin{equation}
Y^N_{ n + 1 }
=
e^{ \frac{T}{N} A }
\Big(
Y_n
+
F( Y_n ) \tfrac{ T }{ N }
+
{\textstyle\int_{ \frac{ n T }{ N } }^{
	\frac{ ( n + 1 ) T }{ N } 
}}
B( Y_n )
\,
dW_s
\Big)
.
\end{equation}
	Then for every $\varepsilon\in(0,\infty)$ 
	there exists a real number $C\in\R$ 
	such that for all $N\in\N$ 
	it holds that 
	$
	\ES\big[
	\|\varphi(X_T)\|_{\mathcal{V}}
	+
	\|\varphi(Y^N_N)\|_{\mathcal{V}}
	\big]
	< \infty
	$
	and 
	\begin{equation}
	\label{eq:weak.rate}
	\left\|
	\ES\big[ 
	\varphi( X_T )
	\big]
	-
	\ES\big[ 
	\varphi( Y^N_N )
	\big]
	\right\|_{\mathcal{V}}
	\leq
	C \cdot
	N^{
		- ( 1 - \kappa - 6\max\{\kappa-\nicefrac{1}{2},0\} - \varepsilon )
	}
	.
	\end{equation}
\end{corollary}
\begin{proof}
Throughout this proof let 
$\tilde{Y}^N\colon[0,T]\times\Omega\to V$, $N\in\N$, 
be $(\mathcal{F}_t)_{t\in[0,T]}$-predictable stochastic processes 
which satisfy for all $N\in\N$ that 
$\tilde{Y}^N_0=\xi$ and which satisfy that for all $N\in\N$, $t\in(0,T]$ 
it holds $\P$-a.s.\ that
\begin{equation}
\label{eq:interpolated.exp}
  \tilde{Y}^N_t
  =
  e^{tA} \, \tilde{Y}^N_0
  +
  \int^t_0
  e^{(t-\floor{s}{\nicefrac{T}{N}})A} \, F(\tilde{Y}^N_{\floor{s}{\nicefrac{T}{N}}})
  \, ds
  +
  \int^t_0
  e^{(t-\floor{s}{\nicefrac{T}{N}})A} \, B(\tilde{Y}^N_{\floor{s}{\nicefrac{T}{N}}})
  \, dW_s
  .
\end{equation} 
Observe that for all $N\in\N$ it holds that 
\begin{equation}
\label{eq:discrete.interpolated.scheme}
  \P\big(\tilde{Y}^N_T
  =
  Y^N_N\big)=1.
\end{equation}
Next note that for all 
$
  \varepsilon \in (0,\min\{1-\kappa,4-7\kappa\})
$
it holds that 
\begin{equation}
\label{eq:parameter.simplification}
\begin{split}
&
  1-\kappa-(1+\mathbbm{1}_{[0,\nicefrac{1}{2})}(\kappa)) \tfrac{\varepsilon}{2}
  -
  6\big(\kappa-\big[\!\min\!\big\{\kappa,\tfrac{1}{2}\big\}-\tfrac{\varepsilon}{12}\mathbbm{1}_{[\nicefrac{1}{2},1)}(\kappa)\big]\big)
\\&=
  1-\kappa-\tfrac{\varepsilon}{2}-\tfrac{\varepsilon}{2} \, \mathbbm{1}_{[0,\nicefrac{1}{2})}(\kappa)
  -
  6\kappa
  +
  6\big[\!\min\!\big\{\kappa,\tfrac{1}{2}\big\}-\tfrac{\varepsilon}{12}\mathbbm{1}_{[\nicefrac{1}{2},1)}(\kappa)\big]
\\&=
  1-7\kappa-\tfrac{\varepsilon}{2}-\tfrac{\varepsilon}{2} \, \mathbbm{1}_{[0,\nicefrac{1}{2})}(\kappa)
  +
  6\min\!\big\{\kappa,\tfrac{1}{2}\big\}
  -
  \tfrac{\varepsilon}{2} \, \mathbbm{1}_{[\nicefrac{1}{2},1)}(\kappa)
\\&=
  1-7\kappa-\varepsilon
  +
  6\min\!\big\{\kappa,\tfrac{1}{2}\big\}
\\&=
  1-\kappa-6\big[\kappa-\min\!\big\{\kappa,\tfrac{1}{2}\big\}\big]
  -\varepsilon
\\&=
  1-\kappa-6\max\!\big\{0,\kappa-\tfrac{1}{2}\big\}-\varepsilon
\\&=
  1-\kappa-6\max\!\big\{\kappa-\tfrac{1}{2},0\big\}-\varepsilon
  .
\end{split}
\end{equation}
Corollary~\ref{cor:weak_convergence_irregular} 
(with
$V=V$,
$\mathcal{V}=\mathcal{V}$, 
$U=U$,
$T=T$,
$\eta=\eta$,
$W=W$,
$A=A$,
$V_r=V_r$,
$h=\nicefrac{T}{N}$, 
$\theta=\kappa$, 
$\vartheta=\min\{\kappa,\nicefrac{1}{2}\}-\nicefrac{\varepsilon}{12} \, \mathbbm{1}_{[\nicefrac{1}{2},1)}(\kappa)$, 
$F=F$,
$B=B$,
$\varphi=\varphi$,
$X=X$,
$Y=\tilde{Y}^N$,
$\rho=1-\kappa-(1+\mathbbm{1}_{[0,\nicefrac{1}{2})}(\kappa))\,\nicefrac{\varepsilon}{2}$
for 
$\varepsilon\in(0,\min\{1-\kappa,4-7\kappa\})$,
$N\in\N$, 
$r\in\R$
in the notation of Corollary~\ref{cor:weak_convergence_irregular}), 
\eqref{eq:interpolated.exp},
\eqref{eq:discrete.interpolated.scheme}, 
items~\eqref{item:derivative.processes}--\eqref{item:kolmogorov.lip.c.delta} of Lemma~\ref{lem:Kolmogorov}, e.g., 
Kunze~\cite[Theorem~5.6]{Kunze2010arXiv}, and, e.g., Van Neerven et al.~\cite[Theorem 6.2]{vvw08}
hence ensure that for all 
$\varepsilon\in(0,\min\{1-\kappa,4-7\kappa\})$ 
there exists a real number 
$C\in[0,\infty)$
such that for all $N\in\N$ it holds that
	$
	\ES\big[
	\|\varphi(X_T)\|_{\mathcal{V}}
	+
	\|\varphi(Y^N_N)\|_{\mathcal{V}}
	\big]
	< \infty
	$
and 
\begin{equation}
\begin{split}
  \big\|
    \ES\big[
    \varphi(X_T)
    \big]
    -
    \ES\big[
    \varphi(Y^N_N)
    \big]
  \big\|_{\mathcal{V}}
  =
  \big\|
  \ES\big[
  \varphi(X_T)
  \big]
  -
  \ES\big[
  \varphi(\tilde{Y}^N_T)
  \big]
  \big\|_{\mathcal{V}}
&\leq
  C\cdot N^{-(1-\kappa-6\max\{\kappa-\nicefrac{1}{2},0\}-\varepsilon)}
  .
\end{split}
\end{equation}
This implies that for all 
$ \varepsilon \in (0,\infty) $, 
$\epsilon\in(0,\min\{1-\kappa,4-7\kappa,\varepsilon\})$
there exists a real number 
$C\in[0,\infty)$
such that for all $N\in\N$ it holds that
	$
	\ES\big[
	\|\varphi(X_T)\|_{\mathcal{V}}
	+
	\|\varphi(Y^N_N)\|_{\mathcal{V}}
	\big]
	< \infty
	$
and 
\begin{equation}
\begin{split}
  &
  \big\|
  \ES\big[
  \varphi(X_T)
  \big]
  -
  \ES\big[
  \varphi(Y^N_N)
  \big]
  \big\|_{\mathcal{V}}
  \leq
  C\cdot N^{-(1-\kappa-6\max\{\kappa-\nicefrac{1}{2},0\}-\epsilon)}
  \\&=
  C\cdot N^{-(\varepsilon-\epsilon)}\cdot N^{-(1-\kappa-6\max\{\kappa-\nicefrac{1}{2},0\}-\varepsilon)}
  =
  \frac{C}{
  	N^{(\varepsilon-\epsilon)}
  	}
  	\cdot N^{-(1-\kappa-6\max\{\kappa-\nicefrac{1}{2},0\}-\varepsilon)}
  \\&\leq
  C
  \cdot N^{-(1-\kappa-6\max\{\kappa-\nicefrac{1}{2},0\}-\varepsilon)}
  .
\end{split}
\end{equation}
The proof of Corollary~\ref{cor:sharp.weak.rates} is thus completed. 
\end{proof}

\section{Weak convergence rates for nonlinear stochastic heat equations}
\label{sec:temporal.rate.stoch.heat}
\begin{corollary}
	\label{cor:neymitskii}
	Consider the notation in Section~\ref{sec:notation}, 
	let $n,d,k\in\N$, 
	$ \alpha \in (-\infty,0) $, 
	$p\in(\max\{n,
	$
	$
	\nicefrac{d(n-1)}{(2|\alpha|)}\},\infty)$, 
	$(V,\left\|\cdot\right\|_V)=(\lpnb{p}{\lambda_{(0,1)^d}}{\R^k},\left\|\cdot\right\|_{\lpnb{p}{\lambda_{(0,1)^d}}{\R^k}})$,
	let $f\colon\R^k\to\R^k$ 
	be an $n$-times continuously differentiable function with globally bounded derivatives, 
	let $A\colon D(A)\subseteq V \to V$ be the Laplacian with Dirichlet boundary conditions on $V$, 
	and let
	$  
	( 
	V_r 
	,
	\left\| \cdot \right\|_{ V_r } 
	)
	$,
	$ r \in \R $,
	be a family of interpolation spaces associated to $ - A $ (cf., e.g., \cite[Section~3.7]{sy02}).
	Then
	\begin{enumerate}[(i)]
		\item 
		\label{item:neymitskii.exist}
		there exists a unique function $F\colon V\to V_\alpha$ 
		which satisfies for all $ v \in \lpn{p}{\lambda_{(0,1)^d}}{\R^k} $ 
		that 
		\begin{equation}
		  F\big([v]_{\lambda_{(0,1)^d},\mathcal{B}(\R^k)}\big)
		  =
		  \big[\{f(v(x))\}_{x\in(0,1)^d}\big]_{\lambda_{(0,1)^d},\mathcal{B}(\R^k)}
		  =
		  [f \circ v]_{\lambda_{(0,1)^d},\mathcal{B}(\R^k)}
		  ,
		\end{equation}
		\item
		\label{item:neymitskii.smooth}
		it holds that $F$ is $n$-times continuously Fr\'{e}chet differentiable with globally bounded derivatives, 
		\item
		\label{item:neymitskii.derivative}
		it holds for all $ m \in \{1,\ldots,n\} $, 
		$ v \in \lpn{p}{\lambda_{(0,1)^d}}{\R^k} $, 
		$ u_1, \ldots, u_m \in \lpn{pm}{\lambda_{(0,1)^d}}{\R^k} $
		that
		\begin{equation}
		\begin{split}
		&
		F^{( m )}( [ v ]_{\lambda_{(0,1)^d}, \mathcal{B}(\R^k ) } )
		( [ u_1 ]_{\lambda_{(0,1)^d}, \mathcal{B}(\R^k ) }, \ldots, 
		[ u_m ]_{\lambda_{(0,1)^d}, \mathcal{B}(\R^k ) } )
		\\
		&
		=
		[ \{ f^{(m)}(v(x)) ( u_1(x), \ldots, u_m(x) ) \}_{ x \in (0,1)^d } ]_{\lambda_{(0,1)^d}, \mathcal{B}(\R^k)}
		,
		\end{split}
		\end{equation}
		\item
		\label{item:neymitskii.embedding}
		it holds for all 
		$ q \in [\max\{1,\frac{dp}{2p|\alpha|+d}\},\frac{p}{n}) $
		that 
		\begin{equation}
		\sup_{v\in V\setminus\{0\}}
		\bigg[
		\frac{
		  \|v\|_{V_\alpha}
		}{
		  \|v\|_{\lpnb{q}{\lambda_{(0,1)^d}}{\R^k}}
		}
		\bigg]
		< \infty,
		\end{equation}
		\item
		\label{item:neymitskii.cb}
		it holds for all $m\in\{1,\ldots,n\}$, 
		$ q \in [\max\{1,\frac{dp}{2p|\alpha|+d}\},\frac{p}{n}) $, 
		$r\in[mq,\infty)$ 
		that 
		\begin{equation}
		\begin{split}
		&
		\sup_{v\in V}
		\sup_{u_1,\ldots,u_m\in\lpnb{\max\{r,p\}}{\lambda_{(0,1)^d}}{\R^k}\setminus\{0\}}
		\left[
		\frac{
			\|F^{(m)}(v)(u_1,\ldots,u_m)\|_{V_\alpha}
			}{
			\prod^m_{i=1}
			\|u_i\|_{\lpnb{r}{\lambda_{(0,1)^d}}{\R^k}}
			}
		\right]
		\\&\leq
		|f|_{C^m_b(\R^k,\R^k)} \,
		\bigg[
		\sup_{v\in V\setminus\{0\}}
		\frac{
		  \|v\|_{V_\alpha}
		}{
		  \|v\|_{\lpnb{q}{\lambda_{(0,1)^d}}{\R^k}}
		}
		\bigg]
		< \infty,
		\end{split}
		\end{equation}
		and
		\item
		\label{item:neymitskii.lip}
		it holds for all $m\in\{1,\ldots,n\}$, 
		$ q \in [\max\{1,\frac{dp}{2p|\alpha|+d}\},\frac{p}{n}) $, 
		$r\in[(m+1)q,\infty)$ 
		that 
		\begin{equation}
		\begin{split}
		&
		\sup_{\substack{v,w\in\lpnb{\max\{r,p\}}{\lambda_{(0,1)^d}}{\R^k}, \\ v \neq w}}
		\sup_{\substack{u_1,\ldots,u_m\in\\\lpnb{\max\{r,p\}}{\lambda_{(0,1)^d}}{\R^k}\setminus\{0\}}}
		\left[
		\frac{
			\|(F^{(m)}(v)-F^{(m)}(w))(u_1,\ldots,u_m)\|_{V_\alpha}
			}{
			\|v-w\|_{\lpnb{r}{\lambda_{(0,1)^d}}{\R^k}} \cdot
			\prod^m_{i=1}
			\|u_i\|_{\lpnb{r}{\lambda_{(0,1)^d}}{\R^k}}
			}
		\right]
		\\&\leq
		|f|_{\operatorname{Lip}^m(\R^k,\R^k)} \,
		\bigg[
		\sup_{v\in V\setminus\{0\}}
		\frac{
		  \|v\|_{V_\alpha}
		}{
		  \|v\|_{\lpnb{q}{\lambda_{(0,1)^d}}{\R^k}}
		}
		\bigg]
		.
		\end{split}
		\end{equation}
	\end{enumerate}
\end{corollary}
\begin{proof}
Throughout this proof let 
$ q \in [\max\{1,\frac{dp}{2p|\alpha|+d}\},\frac{p}{n}) $, 
let 
$
  G \colon V \to \lpnb{q}{\lambda_{(0,1)^d}}{\R^k}
$ 
be the function which satisfies for all 
$ v \in \lpn{p}{\lambda_{(0,1)^d}}{\R^k} $ 
that 
\begin{equation}
\label{eq:neymitskii.def}
		  G\big([v]_{\lambda_{(0,1)^d},\mathcal{B}(\R^k)}\big)
		  =
		  [f \circ v]_{\lambda_{(0,1)^d},\mathcal{B}(\R^k)}
		  ,
\end{equation}
and let 
$ \iota \colon V \to V_\alpha $ 
be the function which satisfies for all 
$ v \in V $ 
that 
$ \iota(v) = v $.
Observe that item~\eqref{item:neymitskii.exist} is an immediate consequence of the fact that 
$
  \forall \, v \in \lpn{p}{\lambda_{(0,1)^d}}{\R^k}
  \colon
  f \circ v \in \lpn{p}{\lambda_{(0,1)^d}}{\R^k}
$.
It thus remains to prove items~\eqref{item:neymitskii.smooth}--\eqref{item:neymitskii.lip}.
For this note that the Sobolev embedding theorem and the fact that 
\begin{equation}
  0-2\alpha
  =
  -2\alpha
  =
  2|\alpha|
  \geq
  d
  \max\!\big\{
  0
  ,
  \nicefrac{1}{q}
  -
  \nicefrac{1}{p}
  \big\}
\end{equation}
show that 
\begin{equation}
\label{eq:q.alpha.embed.def}
  \sup_{ v \in V \setminus \{0\} }
  \bigg[
  \frac{
    \|\iota(v)\|_{V_\alpha}
  }{
    \|v\|_{\lpnb{q}{\lambda_{(0,1)^d}}{\R^k}}
  }
  \bigg]
  =
  \sup_{ v \in V \setminus \{0\} }
  \bigg[
  \frac{
    \|v\|_{V_\alpha}
  }{
    \|v\|_{\lpnb{q}{\lambda_{(0,1)^d}}{\R^k}}
  }
  \bigg]
  < \infty
  .
\end{equation}
Combining this with the fact that 
\begin{equation}
  \overline{V}^{\lpnb{q}{\lambda_{(0,1)^d}}{\R^k}}
  =
  \lpnb{q}{\lambda_{(0,1)^d}}{\R^k}
\end{equation}
proves that there exists a unique continuous linear function 
$
  \mathcal{I}
  \colon
  \lpnb{q}{\lambda_{(0,1)^d}}{\R^k}
  \to V_\alpha
$
which satisfies for all 
$
  v \in V \subseteq \lpnb{q}{\lambda_{(0,1)^d}}{\R^k}
$
that 
\begin{equation}
\label{eq:q.alpha.embed}
  \mathcal{I}(v)
  =
  \iota(v)
  =
  v
  .
\end{equation}
Observe that~\eqref{eq:neymitskii.def} and~\eqref{eq:q.alpha.embed} ensure that 
\begin{equation}
  F=
  \mathcal{I} \circ G
  .
\end{equation}
Combining Proposition~2.6 in~\cite{CoxJentzenKurniawanPusnik2016} 
(with
$k=k$, 
$l=k$, 
$d=d$, 
$n=n$, 
$p=q$, 
$q=p$, 
$\mathcal{O}=(0,1)^d$, 
$f=f$, 
$F=G$
in the notation of Proposition~2.6 in~\cite{CoxJentzenKurniawanPusnik2016}), \eqref{eq:q.alpha.embed.def},
and the fact that 
$
  \mathcal{I}
  \in L(\lpnb{q}{\lambda_{(0,1)^d}}{\R^k},V_\alpha)
$
hence establishes items~\eqref{item:neymitskii.smooth}--\eqref{item:neymitskii.lip}.
The proof of Corollary~\ref{cor:neymitskii} is thus completed.
\end{proof}
\begin{corollary}
\label{cor:nonlinearity.F}
	Consider the notation in Section~\ref{sec:notation}, 
	let $n,d,k\in\N$, 
	$ \alpha \in (-\infty,0) $, 
	$p\in[\max\{n+1,\nicefrac{dn}{(2|\alpha|)}\},\infty)$, 
	$(V,\left\|\cdot\right\|_V)=(\lpnb{p}{\lambda_{(0,1)^d}}{\R^k},\left\|\cdot\right\|_{\lpnb{p}{\lambda_{(0,1)^d}}{\R^k}})$,
	let $f\colon\R^k\to\R^k$ 
	be an $n$-times continuously differentiable function with globally Lipschitz continuous and  globally bounded derivatives, 
	let $A\colon D(A)\subseteq V \to V$ be the Laplacian with Dirichlet boundary conditions on $V$, 
	and let
	$  
	( 
	V_r 
	,
	\left\| \cdot \right\|_{ V_r } 
	)
	$,
	$ r \in \R $,
	be a family of interpolation spaces associated to $ - A $.
	Then
	\begin{enumerate}[(i)]
		\item 
		\label{item:F.exist}
		there exists a unique function $F\colon V\to V_\alpha$ 
		which satisfies for all $ v \in \lpn{p}{\lambda_{(0,1)^d}}{\R^k} $ 
		that 
		\begin{equation}
		  F\big([v]_{\lambda_{(0,1)^d},\mathcal{B}(\R^k)}\big)
		  =
		  \big[\{f(v(x))\}_{x\in(0,1)^d}\big]_{\lambda_{(0,1)^d},\mathcal{B}(\R^k)}
		  =
		  [f \circ v]_{\lambda_{(0,1)^d},\mathcal{B}(\R^k)}
		\end{equation}
		and 
		\item
		\label{item:F.smooth}
		it holds that $F$ is $n$-times continuously Fr\'{e}chet differentiable with globally Lipschitz continuous and  globally bounded derivatives. 
	\end{enumerate}
\end{corollary}
\begin{proof}
First, we note that the assumption that 
$
  p \geq
  \max\{n+1,\frac{dn}{2|\alpha|}\}
$
ensures that 
\begin{equation}
\label{eq:cor9.1.parameter.check}
  1 \leq \tfrac{p}{(n+1)}
  \qquad\text{and}\qquad
  dn\leq 2p|\alpha|
  .
\end{equation}
This implies that 
$
  d(n+1) \leq 2p|\alpha|+d
$.
Hence, we obtain that 
\begin{equation}
  \tfrac{d}{2p|\alpha|+d}
  \leq
  \tfrac{1}{(n+1)}
  .
\end{equation}
Combining this with~\eqref{eq:cor9.1.parameter.check} assures that 
\begin{equation}
  \max\!\big\{1,
  \tfrac{dp}{2p|\alpha|+d}
  \big\}
  \leq
  \tfrac{p}{(n+1)}
  .
\end{equation}
Therefore, we obtain that 
$
  \frac{p}{(n+1)}
  \in
  [\max\{1,\frac{dp}{2p|\alpha|+d}\},\frac{p}{n})
$.
Items~\eqref{item:neymitskii.exist}, \eqref{item:neymitskii.smooth}, \& \eqref{item:neymitskii.lip} of Corollary~\ref{cor:neymitskii} 
(with
$n=n$, 
$d=d$, 
$k=k$, 
$\alpha=\alpha$, 
$p=p$, 
$f=f$, 
$A=A$, 
$m=m$, 
$q=\frac{p}{(n+1)}$, 
$r=p$
for 
$m\in\{1,\ldots,n\}$
in the notation of items~\eqref{item:neymitskii.exist}, \eqref{item:neymitskii.smooth}, \& \eqref{item:neymitskii.lip} Corollary~\ref{cor:neymitskii})
hence prove items~\eqref{item:F.exist}--\eqref{item:F.smooth}. 
The proof of Corollary~\ref{cor:nonlinearity.F} is thus completed.
\end{proof}
\begin{theorem}
\label{thm:stochastic.heat.eq}
Consider the notation in Section~\ref{sec:notation}, let 
$T,\varepsilon\in(0,\infty)$, 
$\beta\in(\nicefrac{1}{4},\nicefrac{1}{4}+\nicefrac{\min\{\varepsilon,1\}}{28})$, 
$p\in(\frac{5}{2(\beta-\nicefrac{1}{4})},\infty)$, 
$(V,\left\|\cdot\right\|_V)=(\lpnb{p}{\lambda_{(0,1)}}{\R},\left\|\cdot\right\|_{\lpnb{p}{\lambda_{(0,1)}}{\R}})$, 
$( U, \left< \cdot, \cdot \right>_U, 
\left\| \cdot \right\|_U )=( \lpnb{2}{\lambda_{(0,1)}}{\R}, 
\left< \cdot, \cdot \right>_{\lpnb{2}{\lambda_{(0,1)}}{\R}}, 
$
$
\left\| \cdot \right\|_{\lpnb{2}{\lambda_{(0,1)}}{\R}} )$, 
$\xi\in V$, 
let $f,b\colon\R\to\R$ 
and 
$\varphi\colon V\to\R$ be four times continuously differentiable functions with globally Lipschitz continuous and globally bounded derivatives, 
	 let $ ( \Omega, \mathcal{F}, \P ) $ be a probability space with a normal filtration $( \mathcal{F}_t )_{ t \in [0,T] }$, 
	 let $ ( W_t )_{ t \in [0,T] } $ be an $\operatorname{Id}_U$-cylindrical $ ( \Omega, \mathcal{F}, \P, ( \mathcal{F}_t )_{ t \in [0,T] } ) $-Wiener process, 
let $A\colon D(A)\subseteq V \to V$ be the Laplacian with Dirichlet boundary conditions on $V$, 
and let
$  
( 
V_r 
,
\left\| \cdot \right\|_{ V_r } 
)
$,
$ r \in \R $,
be a family of interpolation spaces associated to $ - A $.
Then 
\begin{enumerate}[(i)]
\item
\label{item:exist.F}
there exists a unique continuous function
$F\colon V \to V_{-2\beta}$ 
which satisfies for all 
$v\in\lpn{p}{\lambda_{(0,1)}}{\R}$  
that 
\begin{equation}
  F\big([v]_{\lambda_{(0,1)},\mathcal{B}(\R)}\big)
  =
  \big[\{f(v(x))\}_{x\in(0,1)}\big]_{\lambda_{(0,1)},\mathcal{B}(\R)}
  =
  [f\circ v]_{\lambda_{(0,1)},\mathcal{B}(\R)}
  ,
\end{equation}
\item
\label{item:exist.B}
there exists a unique continuous function
$B\colon V \to \gamma(U,V_{-\beta})$ 
which satisfies for all 
$v,u\in\lpn{2p}{\lambda_{(0,1)}}{\R}$  
that 
\begin{equation}
B\big([v]_{\lambda_{(0,1)},\mathcal{B}(\R)}\big)[u]_{\lambda_{(0,1)},\mathcal{B}(\R)}
=
\big[\{b(v(x))\cdot u(x)\}_{x\in(0,1)}\big]_{\lambda_{(0,1)},\mathcal{B}(\R)}
,
\end{equation}
\item
\label{item:exist.SEE}
there exists an up-to-indistinguishability unique continuous 
$(\mathcal{F}_t)_{t\in[0,T]}$-adapted stochastic process 
$X\colon[0,T]\times\Omega\to V$
which satisfies that for all $t\in[0,T]$ it holds $\P$-a.s.\ that 
	\begin{equation}
	X_t 
	= 
	e^{ tA } \, \xi
	+ 
	\int_0^t e^{ ( t - s )A } F( X_s ) \, ds
	+ 
	\int_0^t e^{ ( t - s )A } B( X_s ) \, dW_s
	, 
	\end{equation}
\item
\label{item:exist.exp.Euler}
there exist up-to-indistinguishability unique stochastic processes
$ Y^N \colon \{ 0, 1, \dots, N \} \times \Omega \to V $,
$ N \in \N $, 
which satisfy that for all $ N \in \N $,
$ n \in \{ 0, 1, \dots, N - 1 \} $
it holds $ \P $-a.s.\ that
$
  Y^N_0 = \xi
$ 
and 
\begin{equation}
Y^N_{ n + 1 }
=
e^{ \frac{T}{N} A }
\Big(
Y_n
+
F( Y_n ) \tfrac{ T }{ N }
+
{\textstyle\int_{ \frac{ n T }{ N } }^{
	\frac{ ( n + 1 ) T }{ N } 
}}
B( Y_n )
\,
dW_s
\Big)
,
\end{equation}
and
\item
\label{item:cor.convergence.rates}
	there exists a real number $C\in\R$ 
	such that for all $N\in\N$ 
	it holds that 
	$
	\ES\big[
	|\varphi(X_T)|
	+
	|\varphi(Y^N_N)|
	\big]
	< \infty
	$
	and 
	\begin{equation}
	\left|
	\ES\big[ 
	\varphi( X_T )
	\big]
	-
	\ES\big[ 
	\varphi( Y^N_N )
	\big]
	\right|
	\leq
	C \cdot
	N^{(\varepsilon-\nicefrac{1}{2})}
	.
	\end{equation}
\end{enumerate}
\end{theorem}
\begin{proof}
First, note that the assumption that 
$
  p >
  \frac{5}{2(\beta-\nicefrac{1}{4})}
$
and the assumption that 
$
  \beta \in
  ( \nicefrac{1}{4}, \nicefrac{1}{4} + \nicefrac{\min\{\varepsilon,1\}}{28} )
$
ensure that 
\begin{equation}
\label{eq:cor.9.2.parameter.check}
\begin{split}
  p 
  &>
  \frac{5}{2(\beta-\nicefrac{1}{4})}
  >
  \frac{5}{2(\nicefrac{\min\{\varepsilon,1\}}{28})}
  =
  \frac{5}{(\nicefrac{\min\{\varepsilon,1\}}{14})}
  =
  \frac{70}{\min\{\varepsilon,1\}}
  \\&\geq
  70\geq 5
  =
  \max\{5,\beta\}
  =
  \max\!\big\{5,\tfrac{4}{2|2\beta|}\big\}
  .
\end{split}
\end{equation}
Corollary~\ref{cor:nonlinearity.F} 
(with 
$n=4$, 
$d=1$, 
$k=1$, 
$\alpha=-2\beta$, 
$p=p$, 
$f=f$, 
$A=A$
in the notation of Corollary~\ref{cor:nonlinearity.F})
hence establishes that item~\eqref{item:exist.F} holds and that 
\begin{equation}
\label{eq:neymitskii.lip}
  F \in \operatorname{Lip}^4(V,V_{-2\beta}).
\end{equation}
Moreover, observe that Corollary~4.9 in~\cite{CoxJentzenKurniawanPusnik2016} 
(with 
$n=4$, 
$\beta=-\beta$, 
$p=p$, 
$b=b$, 
$A=A$
in the notation of Corollary~4.9 in~\cite{CoxJentzenKurniawanPusnik2016})
together with~\eqref{eq:cor.9.2.parameter.check}
proves that item~\eqref{item:exist.B} holds and that 
\begin{equation}
\label{eq:multiplication.lip}
  B \in \operatorname{Lip}^4(V,\gamma(U,V_{-\beta})).
\end{equation}
In addition, note that~\eqref{eq:neymitskii.lip} and~\eqref{eq:multiplication.lip} assure that there exists an up-to-indistinguishability unique continuous 
$(\mathcal{F}_t)_{t\in[0,T]}$-adapted stochastic process 
$X\colon[0,T]\times\Omega\to V$
which satisfies that for all $t\in[0,T]$ it holds $\P$-a.s.\ that 
	\begin{equation}
	X_t 
	= 
	e^{ tA } \, \xi
	+ 
	\int_0^t e^{ ( t - s )A } F( X_s ) \, ds
	+ 
	\int_0^t e^{ ( t - s )A } B( X_s ) \, dW_s
	\end{equation}
(cf., e.g., Theorem 6.2 in Van Neerven et al.~\cite{vvw08}). 
This establishes item~\eqref{item:exist.SEE}.
Next note that induction shows that there exist up-to-indistinguishability unique stochastic processes
$ Y^N \colon \{ 0, 1, \dots, N \} \times \Omega \to V $,
$ N \in \N $, 
which satisfy that for all $ N \in \N $,
$ n \in \{ 0, 1, \dots, N - 1 \} $
it holds $ \P $-a.s.\ that
$
  Y^N_0 = \xi
$ 
and 
\begin{equation}
Y^N_{ n + 1 }
=
e^{ \frac{T}{N} A }
\Big(
Y_n
+
F( Y_n ) \tfrac{ T }{ N }
+
{\textstyle\int_{ \frac{ n T }{ N } }^{
	\frac{ ( n + 1 ) T }{ N } 
}}
B( Y_n )
\,
dW_s
\Big)
.
\end{equation}
This demonstrates item~\eqref{item:exist.exp.Euler}.
It thus remains to prove item~\eqref{item:cor.convergence.rates}. For this observe that~\eqref{eq:neymitskii.lip}, \eqref{eq:multiplication.lip},
Corollary~\ref{cor:sharp.weak.rates} 
(with 
$V=V$, 
$\mathcal{V}=\R$, 
$U=U$, 
$T=T$, 
$\eta=0$, 
$\kappa=2\beta$, 
$\xi=\xi$, 
$\varphi=\varphi$, 
$W=W$, 
$A=A$, 
$F=F$, 
$B=B$, 
$X=X$, 
$Y^N=Y^N$, 
$\varepsilon=\nicefrac{\varepsilon}{2}$
for 
$N\in\N$
in the notation of Corollary~\ref{cor:sharp.weak.rates}),
and the fact that 
$
  \beta \in (\nicefrac{1}{4},\nicefrac{1}{4}+\nicefrac{\varepsilon}{28})
$
show that there exists a real number $C\in[0,\infty)$ 
such that for all 
$N\in\N$ 
it holds that 
	$
	\ES\big[
	|\varphi(X_T)|
	+
	|\varphi(Y^N_N)|
	\big]
	< \infty
	$
	and 
\begin{equation}
\begin{split}
	\left|
	\ES\big[ 
	\varphi( X_T )
	\big]
	-
	\ES\big[ 
	\varphi( Y^N_N )
	\big]
	\right|
	&\leq
	C \cdot
	N^{-(1-2\beta-6\max\{2\beta-\nicefrac{1}{2},0\}-\nicefrac{\varepsilon}{2})}
	=
	C \cdot
	N^{-(1-2\beta-6(2\beta-\nicefrac{1}{2})-\nicefrac{\varepsilon}{2})}
	\\&=
	C \cdot N^{-(1-2\beta-12\beta+3-\nicefrac{\varepsilon}{2})}
	=
	C \cdot
	N^{(14\beta+\nicefrac{\varepsilon}{2}-4)}	
	\\&\leq
	C \cdot
	N^{(14(\nicefrac{1}{4}+\nicefrac{\varepsilon}{28})+\nicefrac{\varepsilon}{2}-4)}	
	=
	C \cdot
	N^{(\varepsilon-\nicefrac{1}{2})}	
	.
\end{split}
\end{equation}
This establishes item~\eqref{item:cor.convergence.rates}.
The proof of Theorem~\ref{thm:stochastic.heat.eq} is thus completed.
\end{proof}

\section*{Acknowledgements}

Special thanks are due to Sonja Cox for a series of fruitful discussions on this work.
This project has been supported through the SNSF-Research project 200021\_156603 
``Numerical approximations of nonlinear stochastic ordinary and partial differential equations".

\bibliographystyle{acm}
\bibliography{Bib/bibfile}
\end{document}